\documentclass[11pt,a4paper]{book}
\usepackage{mythesiscommands}

\title{
{\Huge \textbf{\MakeUppercase{On the growth of}}} \\[5pt]
{\Huge \textbf{\MakeUppercase{permutation classes}}} \\[84pt]
\textbf{\MakeUppercase{David Ian Bevan}~} \\
M.A. M.Sc. B.A. \\[96pt]
a thesis submitted to \\
THE OPEN UNIVERSITY \\
for the degree of \\
DOCTOR OF PHILOSOPHY \\
in Mathematics
}

\author{}

\date{\LARGE May 2015}

\begin{document}

\frontmatter

\maketitle
\cleardoublepage

\newcommand{\myref}[1]{\ref{#1}}
\newcommand{\mypageref}[1]{\pageref{#1}}


\chapter{Abstract}

We study aspects of the enumeration of permutation classes, sets of permutations closed downwards under the subpermutation order.

First, we consider monotone {grid classes} of permutations. We present procedures for calculating the generating function
of any class whose matrix has dimensions $m \times 1$ for some~$m$, and of acyclic and unicyclic classes of {gridded permutations}.
We show that almost all large permutations in a grid class have the same shape, and determine this limit shape.

We prove that the growth rate of a grid class is given by the square
of the spectral radius of an associated graph and deduce some facts relating to the set of grid class growth rates.
In the process, we establish a new result concerning tours on graphs.
We also prove a similar result relating
the growth rate of a {geometric} grid class to the {matching polynomial} of a graph, and
determine the effect of edge
subdivision on the matching polynomial.
We characterise the growth rates of
geometric grid classes in terms of the spectral radii of trees.

We then investigate the set of growth rates of permutation classes and establish a new upper bound on the value above which every real number is the growth rate of some permutation class.
In the process, we prove new results concerning expansions of real numbers in non-integer bases in which the digits are drawn from sets of allowed values.

Finally, we introduce a new enumeration technique,
based on associating a graph with each permutation, and
determine the generating functions for some previously unenumerated classes.
We conclude by using this approach to provide an improved lower bound on the growth rate of the class of permutations avoiding the pattern $\pdiamond$.
In the process, we prove that, asymptotically, patterns in
{\L}uka\-sie\-wicz paths exhibit a concentrated Gaussian distribution.

\cleardoublepage
 

\chapter{Acknowledgements}

\begin{center}
\emph{J\'esus-Christ est l'objet de tout, et le centre o\`u tout tend. \\
Qui le conna\^it conna\^it la raison de toutes choses.}\footnote{Jesus Christ is the goal of everything, and the centre to which everything tends. Whoever knows Him knows the reason behind all things.}
\end{center}
\begin{flushright}
\vspace{-9pt}
  --- Pascal, \emph{Pens\'ees}~\cite{Pascal1670}
\end{flushright}
\vspace{27pt}

My journey (back) into mathematics research began with reading \emph{Proofs from THE BOOK}, the delightful book by Aigner \& Ziegler~\cite{AZ2014}. I am grateful to G\"unter Ziegler for his encouragement at that time, which resulted in my first published mathematics paper~\cite{Bevan2006}.

Lunchtimes at the Open University have been enlivened by numerous interesting stories
from Phil Rippon. I am thankful to both him and Gwyneth Stallard for much helpful advice, and also for a couple of hundred cups of coffee!
Thanks too, to other members of the department for friendship and support, including
Robert, Toby, Ian, Ben, Tim and Jozef, and the other research students: Grahame, Rob,
Matthew, Mairi, Rosie, Vasso and David.

I received a warm welcome and encouragement from many members of the
permutation patterns community and other combinatorialists.
Particular thanks are due to Vince Vatter for detailed feedback on some of my papers and for suggesting that it might be worthwhile investigating
whether his 
upper bound on $\lambda$ could be improved.
I am also grateful to Mireille Bousquet-M\'elou
for valuable discussions relating to my (at that time, very woolly) ideas about using Hasse graphs for enumeration, during the Cardiff Workshop on Combinatorial Physics. 
Thanks are also due to
Mike Atkinson,
Michael Albert, Einar Steingr\'imsson,
Jay Pantone,
Dan Kr{\'a}l', and
Doron \mbox{Zeilberger}, 
among others.

Robert Brignall has been an excellent supervisor, guiding me with a light touch, and doing a first-rate job of teaching me how to write mathematics so that it can be understood by those trying to read it.
Thank you!

Michael Albert's \emph{PermLab} software~\cite{PermLab} was of particular benefit in helping to visualise and
explore the structure of permutations in different permutation classes.
I have also made extensive use of
\emph{Mathematica}~\cite{Mathematica} throughout my research.

Three years ago, my wife Penny encouraged me to pursue my dream,
and
has accompanied me on this latest adventure in our life together, bearing with good grace the (frequent) times when
I'm in
``maths world'' and present only physically.
She has also willingly taken on the primary bread-winning role during this season.
Words are insufficient to express my gratitude.

\vspace{9pt}
Above all else,
like Blaise Pascal, I am convinced that the crux to understanding
life, the universe and everything --- 
including the ontology, epistemology and
``unreasonable effectiveness''\footnote{See Wigner~\cite{Wigner1960}}
of mathematics --- 
lies in the life, death and resurrection
of a certain itinerant Jewish rabbi and miracle-worker.
I thus conclude:

\vspace{9pt}
\begin{center}
\emph{To the only wise God \\
be glory forever \\
through Jesus Christ!}\footnote{Romans 16:27}
\end{center}

\newpage

\begin{figure}[p]
\begin{center}
  \vspace{15pt}
  \includegraphics[scale=0.349]{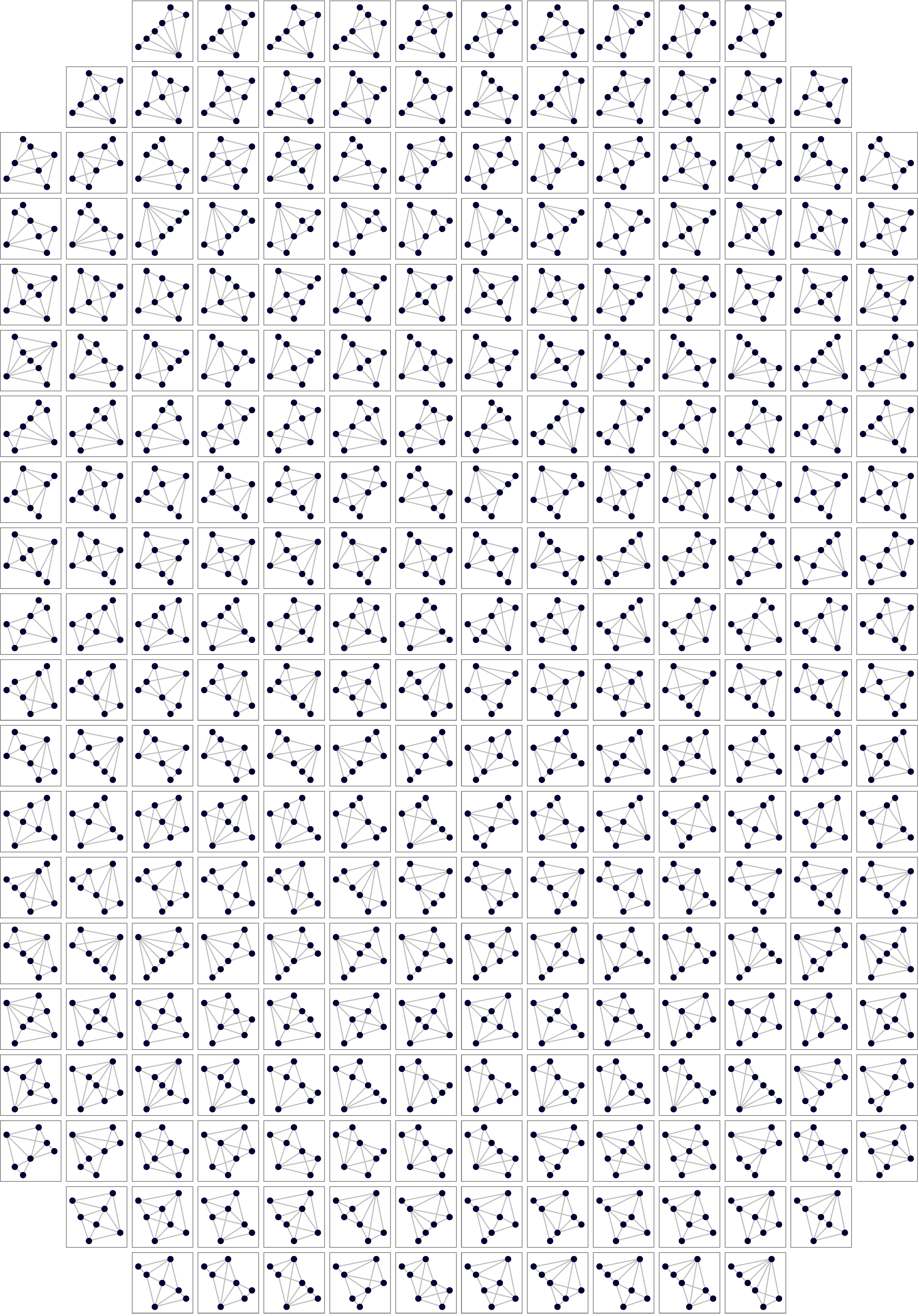}\label{figFrontispiece}

  \vspace{15pt}
  {\small
  Permutation class art \\
  ~\\
  The permutations of length 7 in $\av(\mathbf{21\bar{3}54},\mathbf{45\bar{3}12})$ \\
  that are both sum and skew indecomposable}
\end{center}
\end{figure}

\cleardoublepage
 

\begin{spacing}{0}
\setcounter{tocdepth}{1}
\tableofcontents
\end{spacing}
\cleardoublepage

\mainmatter


\chapter{Enumerating permutation classes}\label{chap01}

Doron Zeilberger begins his expository article~\cite{Zeilberger2008}
on enumerative and algebraic combinatorics
in
\emph{The Princeton
Companion to Mathematics} by observing that
\begin{quote}
\vspace{-11pt}
\begin{flushleft}
  ``\emph{Enumeration}, otherwise known as \emph{counting}, is the oldest mathematical subject''. 
\end{flushleft}
\vspace{-11pt}
\end{quote}

Subsequently,
he declares,
\begin{quote}
\vspace{-11pt}
\begin{flushleft}
  ``The fundamental theorem of enumeration, independently discovered by several anonymous cave dwellers, states that
  \begin{equation}\label{eqCounting}
  |A| \;=\; \sum_{a\in A} 1 .
  \end{equation}
  \textellipsis{} While this formula is still useful after all these years, enumerating specific finite sets is no longer considered mathematics.''
\end{flushleft}
\vspace{-11pt}
\end{quote}
In this study, we participate in the millennia-old pursuit of counting,
and contribute to the growth of mathematical knowledge,
by addressing the question of the enumeration of certain infinite sets,
called permutations classes.

In the rest of this introductory section, we first consider what it means to enumerate a class of combinatorial objects and describe the techniques we use.
Then we introduce permutation classes and give an overview of research concerning their enumeration.
Finally, we present an outline of the contents of the three parts of this thesis.

\section{Enumeration}

A \emph{combinatorial class} $\AAA$ is a countable (i.e.~finite or countably infinite) set endowed with a \emph{size} function, such that
\begin{bullets}
  \item the size of each element of $\AAA$ is a non-negative integer, and
  \item the number of elements of $\AAA$ of any given size is finite.
\end{bullets}
The size of of an element $\alpha\in\AAA$ is denoted by $|\alpha|$.
The finite set of elements in $\AAA$ whose size is $n$ is written $\AAA_n$.
We consistently use calligraphic uppercase letters, e.g.~$\RRR,\SSS,\TTT$, for combinatorial classes and lowercase Greek letters, e.g.~$\rho,\sigma,\tau$, for their members.

Given some combinatorial class, or family of combinatorial classes, the goal of enumerative combinatorics is to
determine the number of elements of each size in the class or classes.
As discussed by Wilf~\cite{Wilf1982}, there are a number of possible ways of answering the question, ``How many things are there?''.
While claiming that~\eqref{eqCounting}, applied to each $\AAA_n$, provides
a simple formula that ``answers'' all such questions at once, Wilf rejects such an answer since it is just a restatement of the question.

\subsection*{Generating functions}

For the most part, the answers we give make use of \emph{generatingfunctionology}. According to Wilf, who coined this neologism as the title of his classic book~\cite{Wilf2006},
\begin{quote}
\vspace{-11pt}
\begin{flushleft}
  ``A generating function is a clothesline on which we hang up a sequence
of numbers for display.''
\end{flushleft}
\vspace{-11pt}
\end{quote}

The (ordinary univariate) \emph{generating function} for a combinatorial class $\AAA$ is defined to be the formal power series
$$
A(z)
\;=\;
\sum_{n\geqslant0} |\AAA_n| z^n
\;=\;
\sum_{\alpha\in\AAA} z^{|\alpha|} .
$$
Thus, each element $\alpha\in\AAA$ makes a contribution of $z^{|\alpha|}$, the result being that, for each~$n$, the coefficient of $z^n$ is the number of elements of size $n$.
In generating functions, we consistently use the variable $z$ to mark the size of objects.

As an elementary example, the number of distinct ways of tossing a coin $n$ times is clearly $2^n$, hence the generating function for sequences of coin tosses, in which the size of a sequence is the number of tosses, is
$$
S(z) \;=\;
2z + 4\+z^2 +8\+z^3 + \ldots \;=\; \frac{2\+z}{1-2\+z} .
$$

Given the generating function $F(z)$ for a combinatorial class, we use the notation $\big[z^n\big]F(z)$ to denote the operation of extracting the coefficient of $z^n$ (i.e. the number of elements of the class with size $n$) from the formal power series $F(z)$. Thus, in our example, 
$$
\big[z^n\big]S(z) \;=\; 2^n , \qquad \text{if~~} n\geqslant1 .
$$

In addition to recording the size of objects, it is often useful to keep track of additional parameters in a \emph{multivariate} generating function.
For example, if $\omega:\AAA\rightarrow\mathbb{N}_0$ is a parameter of elements of combinatorial class $\AAA$, then
the (ordinary) \emph{bivariate} generating function
for $\AAA$, in which $w$ marks the parameter $\omega$, is:
    $$
    A(z,w)
    \;=\;
    \sum_{\alpha\in\AAA}z^{|\alpha|}\+w^{\omega(\alpha)}
    \;=\;
    \sum_{n\,\geqslant\,0} \+ \sum_{k\,\geqslant\,0} a_{n,k}\+ z^n\+w^k ,
    $$
where $a_{n,k}$ is the number of elements $\alpha\in\AAA$ of size $n$ for which $\omega(\alpha)=k$.
This can be generalised for multiple parameters.
Observe that $A(z,1)$ is the ordinary univariate generating function.

For example, if we use $t$ to mark the number of tails in a sequence of coin tosses, the corresponding bivariate generating function is
$$
S(z,t)
\;=\;
\sum_{n=1}^\infty \+ \sum_{k=0}^n \tbinom{n}{k}\+ z^n\+t^k
\;=\;
\frac{(1+t)\+z}{1-(1+t)\+z} .
$$

Why do we choose to use generating functions?
Their extraordinary
utility is
beautifully elucidated in the consummate \emph{{magnum} opus} of Flajolet \& Sedgewick, \emph{Analytic Combinatorics}~\cite{FS2009}.
Three aspects are pre-eminent in this work:

Firstly, it is possible to translate \emph{directly} from a structural definition of a combinatorial class to equations that define
the generating function for the class. It may then be possible to solve these functional equations to yield an explicit form for the generating function.

Secondly, generating functions enable us to answer questions concerning the \emph{growth} of combinatorial classes.
Given a generating function, $F(z)=\sum a_n\+z^n$, the asymptotic behaviour of the sequence $(a_n)$ can be determined by treating
$F(z)$ analytically as a complex function.
The singularities of $F$ provide full information on the asymptotic behaviour of its coefficients.
Typically, for large $n$, $\big[z^n\big]F(z)$ behaves like $\gamma^n\+\theta(n)$, for some $\gamma>0$ and
subexponential function $\theta$.
The location of the singularities of $F$ dictates the exponential growth rate, $\gamma$, of its coefficients.
The nature of the singularities of $F$ then determines the subexponential factor, $\theta(n)$.

Thirdly, it is possible to extract from multivariate generating functions precise information concerning the limiting distribution of parameter values. Thus, if $\omega$ is some parameter of elements of a combinatorial class, the asymptotic probability distribution
$\mathbb{P}_n[\omega(\alpha)\leqslant k]$ for large objects of size $n$ can be established.

A further benefit of possessing a generating function for a combinatorial class
is that the nature of the generating function immediately reveals something of the
structure of objects in the class.
Specifically, the extent to which a class is ``well-behaved''
depends on whether its generating function is
rational,
algebraic,
D-finite, or 
not D-finite.
\vspace{6pt}

A \emph{rational} function is the ratio of two polynomials, such as $S(z)$ above.
Rational functions are the
generating functions of deterministic finite state automata
or, equivalently, of
\emph{regular} languages (see~\cite[Section I.4]{FS2009}).
Objects in a class with a
rational generating function therefore tend to
exhibit a structure
similar to the linear structure of words in a regular language.
A function $F(z)=\sum f_nz^n$ is rational if and only if
its coefficients satisfy a linear recurrence relation with constant coefficients.

\begin{figure}[t] 
\begin{center}
  $\qquad$ \includegraphics[scale=0.225]{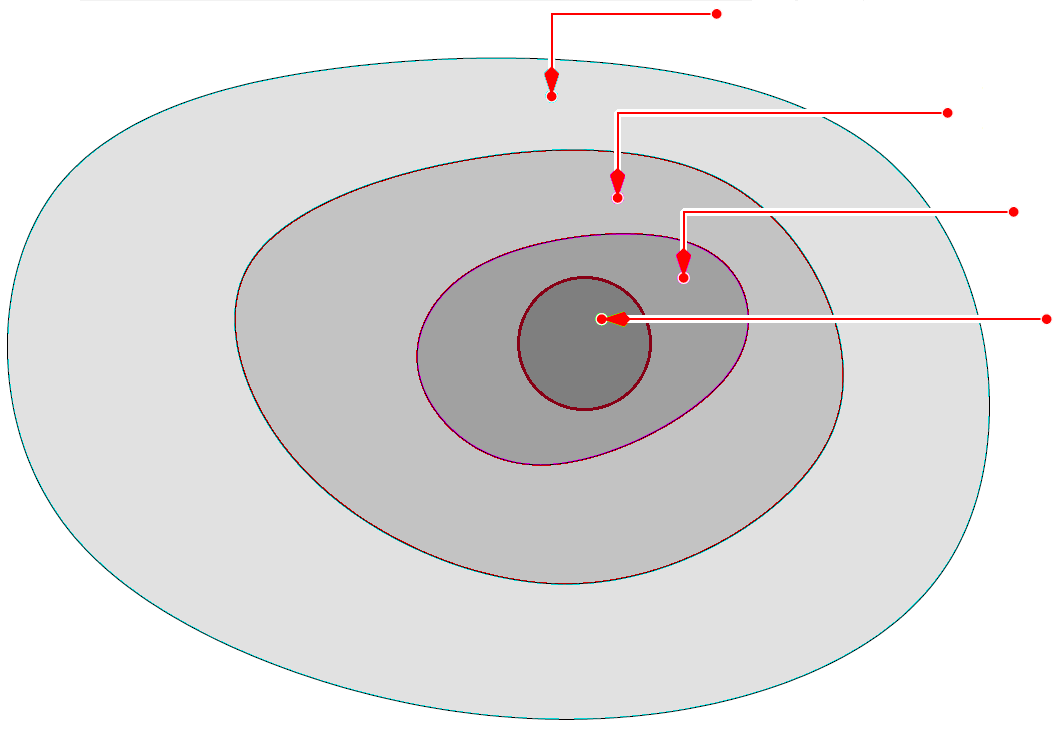}
  \raisebox{118pt}{\hspace{-56pt}\small All power series}
  \raisebox{101.5pt}{\hspace{-35pt}\small D-finite}
  \raisebox{84.5pt}{\hspace{-26pt}\small Algebraic}
  \raisebox{66.5pt}{\hspace{-40pt}\small Rational}
  \vspace{-9pt}
\end{center}
\caption{The hierarchy of families of generating functions}
\end{figure}
A larger family than the rational functions is the family of \emph{algebraic} functions.
A function $F(z)$ is {algebraic} if it can be defined as the root of a polynomial equation.
That is, there exists a 
bivariate polynomial $P(z, y)$ such that $P(z,F(z)) = 0$.
Algebraic functions are the
generating functions of unambiguous \emph{context-free} languages (Chomsky \& Sch\"utzenberger~\cite{CS1963}).
Objects in a class with an algebraic generating function
tend to
exhibit a branching structure
similar to that of trees
(see Bousquet-M\'elou~\cite{Bousquet-Melou2006}).
There is no simple characterization of the coefficients of algebraic generating functions.
However, the coefficients of an algebraic function can be expressed in closed form as a finite linear combination of multinomial coefficients, which can be determined directly from the defining polynomial~$P(z,y)$.
This result is due to Flajolet \& Soria; a proof
can be found in~\cite[Theorem 8.10]{FS2001}; see also the presentation by
Banderier \& Drmota~\cite[Theorem 1]{BD2013}, and~\cite[Note VII.34, p.495]{FS2009}.

A more general family than that of rational or algebraic functions is the family of \emph{D-finite}\label{defDFinite} (``differentiably finite'') functions, also known as \emph{holonomic} functions.
A function $F(z)$ is D-finite
if it satisfies a differential equation with coefficients that are polynomials in $z$.
A power series is D-finite if and only if its
coefficients satisfy a linear recurrence relation with polynomial coefficients.
A sequence of numbers satisfying such a recurrence is said to be \emph{P-recursive} (``polynomially recursive'').

Beyond D-finite power series lie those that are not D-finite.
Given that there are only countably many D-finite functions,
almost all of the uncountably many combinatorial classes are enumerated by non-D-finite generating functions.
Such functions are considerably less amenable to analysis.
One way of determining whether a generating function is D-finite or not is to make use of the fact that
a D-finite function has only finitely many singularities.
For more on ways of determining D-finiteness,
see the papers by Guttmann~\cite{Guttmann2005} and
Flajolet, Gerhold \& Salvy~\cite{FGS2005}.



\subsection*{The symbolic method}

There is a natural direct correspondence between the structure of combinatorial classes
and their generating functions, as reflected in Table~\ref{tabSymbolic}.

\begin{table}[ht]
\mybox{
\begin{tabular}{l|l|l|l}
 $\;$ \emph{Structure}   $\;$ & $\;$ \emph{Construction} $\;$ & $\;$ \emph{Generating function}              $\;$ & $\;$ \emph{Condition} \\[8pt]
 $\;$ Atom               $\;$ & $\;$ $\ZZZ\;=\;\{\bullet\}$     $\;$  & $\;$ $Z(z)\;=\;z$                    $\;$ &                       \\[8pt]
 $\;$ Disjoint union     $\;$ & $\;$ $\AAA\;=\;\BBB+\CCC$       $\;$  & $\;$ $A(z)\;=\;B(z)+C(z)$            $\;$ &                       \\[8pt]
 $\;$ Cartesian product  $\;$ & $\;$ $\AAA\;=\;\BBB\times\CCC$  $\;$  & $\;$ $A(z)\;=\;B(z)\+C(z)$           $\;$ &                       \\[6pt]
 $\;$ Sequence           $\;$ & $\;$ $\AAA\;=\;\seq{\BBB}$      $\;$  & $\;$ $A(z)\;=\;\dfrac{1}{1-B(z)}$    $\;$ & $\;$ $B(0)=0$         \\[12pt]
 $\;$ Non-empty sequence $\;$ & $\;$ $\AAA\;=\;\seqplus{\BBB}$  $\;$  & $\;$ $A(z)\;=\;\dfrac{B(z)}{1-B(z)}$ $\;$ & $\;$ $B(0)=0$         \\[12pt]
 $\;$ Pointing           $\;$ & $\;$ $\AAA\;=\;\Theta\BBB$      $\;$  & $\;$ $A(z)\;=\;z\+\partial_z B(z)$   $\;$ &                       \\[8pt]
 $\;$ Marking            $\;$ & $\;$ $\AAA\;=\;u\+\BBB$         $\;$  & $\;$ $A(z)\;=\;u\+B(z)$              $\;$ &
\end{tabular}
} 
\caption{The correspondence between structure and generating functions}
\label{tabSymbolic}
\end{table}

We always use $\ZZZ$ to denote the \emph{atomic class} consisting of a single element of size 1.

We use $\BBB+\CCC$ or $\BBB\uplus\CCC$ to denote the \emph{disjoint union} of classes $\BBB$ and $\CCC$.

The \emph{Cartesian product} of two classes contains all ordered pairs of elements of the classes.
For example, suppose elements of class $\AAA$ consist of
ordered pairs of objects, $\alpha=(\beta,\gamma)$, where
$\beta$ and $\gamma$ drawn from classes $\BBB$ and $\CCC$ respectively, and
the size of $\alpha$ is given  by
$|\alpha|=|\beta|+|\gamma|$.
Then the generating function for $\AAA$ is given by
  $$
    A(z)
    \;=\;
    \sum_{\alpha\in\BBB\times\CCC} z^{|\alpha|}
    \;=\;
    \sum_{\beta\in\BBB} \sum_{\gamma\in\CCC} z^{|\beta|+|\gamma|}
    \;=\;
    \sum_{\beta\in\BBB} z^{|\beta|} \sum_{\gamma\in\CCC} z^{|\gamma|}
    \;=\;
    B(z)\+C(z),
  $$
where $B(z)$ and $C(z)$ are the generating functions for $\BBB$ and $\CCC$ respectively.

We use
$\seq{\BBB}$ to represent the class of (possibly empty) sequences of elements of $\BBB$, and
$\seqplus{\BBB}$ to represent the class of non-empty sequences of elements of~$\BBB$.
The size of such a sequence is the sum of the sizes of its components.
For example, if $\AAA$ consists of non-empty sequences of elements of $\BBB$, then
$$
\AAA\;=\;\seqplus{\BBB}\;=\;\BBB\:+\:\BBB^2\:+\:\BBB^3\:+\:\ldots,
$$
where we write $\BBB^2$ for $\BBB\times\BBB$, etc.
Thus,
$$
A(z) \;=\; B(z)\:+\:B(z)^2\:+\:B(z)^3\:+\:\ldots \;=\; \dfrac{B(z)}{1-B(z)}.
$$

The \emph{pointing} construction\label{defPointing}, denoted here by $\Theta$, represents the idea of ``pointing at a distinguished atom''.
For example, if $\AAA=\Theta\BBB$, then the generating function of $\AAA$ is 
$$
A(z)
  \;=\;
  \sum_{n\geqslant0}a_n\+ z^n
  \;=\;
  \sum_{n\geqslant0}n\+b_n\+ z^n
  \;=\;
  z\+\partial_z B(z),
$$
where $\partial_z$ is used to denote the differential operator $\frac{d}{dz}$.

\emph{Marking} enables us to record combinatorial parameters.
We use a lowercase letter for marking in structural equations, corresponding to the variable in the generating function that is used to mark the parameter.
For example, the class of sequences of coin tosses, with $t$ marking tails has the following structure:
$$
\SSS \;=\; \seqplus{\ZZZ + t\+\ZZZ},
$$
where the first term in the disjoint union corresponds to a throw of heads, and the second to a (marked) throw of tails.

\subsection*{Functional equations}

Typically, to determine the generating function of a combinatorial class, we define the structure recursively.
The prototypical example is that of \emph{rooted plane trees}, where the size is the number of vertices.

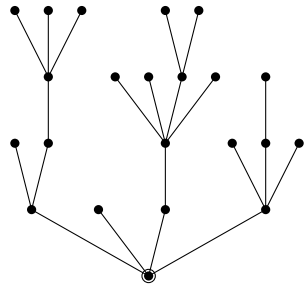
\begin{figure}[ht]
  $$
  \begin{tikzpicture}[scale=0.22]
    \plotpermnobox{}{12,8,12,0,20,8,16,0,4,8,16,20,16,12,0,8,0,12}
    \plotpermnobox{}{20,0,16,0,0,0,0,0,16,12,0,0,0,0,0,12}
    \plotpermnobox{}{0,0,20,0,0,0,0,0,0,20,0,0,0,0,0,16}
    \draw[thin] (6,8)--(9,4)--(2,8)--(1,12);
    \draw[thin] (2,8)--(3,12)--(3,20);
    \draw[thin] (1,20)--(3,16)--(5,20);
    \draw[thin] (16,16)--(16,8)--(9,4)--(10,8)--(10,12)--(7,16);
    \draw[thin] (14,12)--(16,8)--(18,12);
    \draw[thin] (9,16)--(10,12)--(13,16);
    \draw[thin] (11,16)--(10,12);
    \draw[thin] (10,20)--(11,16)--(12,20);
    \draw [thin] (9,4) circle [radius=0.4];
  \end{tikzpicture}
  \vspace{-6pt}
  $$
  \caption{A rooted plane tree}
\end{figure}
Rooted plane trees consist of a root vertex joined to a (possibly empty) sequence of subtrees.
Thus the class, $\TTT$, satisfies the recursive structural equation
$$
\TTT \;=\; \ZZZ \times \seq{\TTT} .
$$
So, by the correspondence between structure and generating functions, we know that the generating function for rooted plane trees, $T(z)$, satisfies the equation
$$
T(z) \;=\; z \+ \frac{1}{1-T(z)} .
$$
This equation has two roots, but one of them has negative coefficients and so can be rejected. Hence,
$$
T(z) \;=\; \thalf  \big(1-\sqrt{1-4\+z}\big). 
$$
Thus, extracting coefficients, $T(z) = z+z^2+2\+z^3+5\+z^4+14\+z^5+42\+z^6+\ldots$,
rooted plane trees being enumerated by the Catalan numbers.

Often,
in order to derive the univariate generating function
for a combinatorial class, multivariate functions are used,
involving additional ``catalytic'' variables that record certain parameters of the objects in the class.
These additional variables make it possible to establish functional equations, which can sometimes be solved to yield the required
generating function.
Typically, when employing a multivariate generating function, it is common to treat it simply as a function of the relevant catalytic variable, writing, for example, $F(u)$ rather than $F(z,u)$.

Another common technique is the use of \emph{linear operators} on generating functions, for which we use the symbol $\oper$.
Let us illustrate this by briefly considering the class of \emph{stepped \mbox{parallelogram} polyominoes}, which we denote $\PPP$.

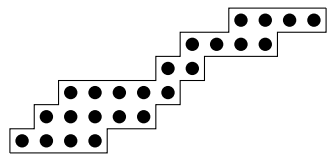
\begin{figure}[ht]
  $$
  \begin{tikzpicture}[scale=0.32]
    \plotpermnobox{}{0,0,0,0,0,0,0,0,0,6,6,6,6}
    \plotpermnobox{}{0,0,0,0,0,0,0,5,5,5,5}
    \plotpermnobox{}{0,0,0,0,0,0,4,4}
    \plotpermnobox{}{0,0,3,3,3,3,3}
    \plotpermnobox{}{0,2,2,2,2,2}
    \plotpermnobox{}{1,1,1,1}
    \draw[thin] (.5,.5)--(.5,1.5)--(1.5,1.5)--(1.5,2.5)--(2.5,2.5)--(2.5,3.5)--
                (6.5,3.5)--(6.5,4.5)--(7.5,4.5)--(7.5,5.5)--(9.5,5.5)--(9.5,6.5)--(13.5,6.5);
    \draw[thin] (.5,.5)--(4.5,.5)--(4.5,1.5)--(6.5,1.5)--(6.5,2.5)--(7.5,2.5)--
                (7.5,3.5)--(8.5,3.5)--(8.5,4.5)--(11.5,4.5)--(11.5,5.5)--(13.5,5.5)--(13.5,6.5);
  \end{tikzpicture}
  \vspace{-6pt}
  $$
  \caption{A stepped parallelogram polyomino} 
  \label{figParaPolyomino}
\end{figure}
A stepped parallelogram polyomino is constructed from one or more rows, each consisting of a contiguous sequence of cells.
Except for the bottom row,
the leftmost cell of a row must occur to the right of the leftmost cell of the previous row but not to the right of the rightmost cell of the previous row, and the rightmost cell of a row must occur to the right of the rightmost cell of the previous row.
See Figure~\ref{figParaPolyomino} for an example.

For our illustration, we consider the size of a polyomino to be given by its \emph{width}.
Let $P(u)=P(z,u)$ be the bivariate generating function for stepped parallelogram polyominoes, in which $z$ marks the width and $u$ marks the number of cells in the top row.
Thus a polyomino that has $k$ cells in its top row contributes a term to
$P(u)$ in which $u$ has exponent~$k$.
The generating function for elements of $\PPP$ consisting of a single row is
$$
P_1(u) \;=\;
\seqplus{u\+\ZZZ} \;=\; \frac{z\+u}{1-z\+u},
$$
where, by mild abuse of notation, we identify the structure and the generating function.

The addition of a new row on top of a row with length $k$ is reflected by the linear operator $\oper_\RR$ defined by
$$
\oper_\RR\big[u^k\big]
\;=\;
(u^{k-1} + u^{k-2} + \ldots + u)\+\seqplus{u\+\ZZZ}
\;=\;
\frac{z\+u}{(1-u)(1-z\+u)}\+(u-u^k),
$$
in which $\seqplus{u\+\ZZZ}$ corresponds to the ``overhanging'' cells to the right.

Hence,
the bivariate generating function $P(u)$ for stepped parallelogram polyominoes is defined by the recursive functional equation
$P(u)=P_1(u) + \oper_\RR\big[P(u)\big]$. That is,
\begin{equation}\label{eqPolyoFunEq}
P(u) \;=\; \frac{z\+u}{1-z\+u} \:+\: \frac{z\+u}{(1-u)(1-z\+u)}\+\big(u\+P(1)-P(u)\big) ,
\end{equation}
an equation that relates $P(u)$ to $P(1)$.

\subsection*{The kernel method}\label{sectKernelMethod}
Equations such as \eqref{eqPolyoFunEq} can sometimes be solved by making use of what is known as the \emph{kernel method}.
For examples of its use, see
the paper by Banderier, Bousquet-M\'elou, Denise, Flajolet, Gardy \& Gouyou-Beauchamps~\cite{BBDFGG2002}
and the expository article by Prodinger~\cite{Prodinger2004}.
We illustrate how the kernel method works by deriving the generating function for the class of stepped parallelogram polyominoes.

To start, we express $P(u)$ in terms of $P(1)$,
by expanding and rearranging~\eqref{eqPolyoFunEq} to give
\begin{equation*}
P(u)  \;=\;   \frac{z\+u \+\big(1-u+u\+P(1)\big)}{1-u+z\+u^2}.
\end{equation*}
Equivalently, we have the equation
$$
(1-u+z\+u^2)\+P(u)  \;=\;   z\+u \+\big(1-u+u\+P(1)\big) .
$$
Now, if we set $u$ to be a root of the multiplier of $P(u)$ on the left, then we obtain a linear equation for $P(1)$.
This is known as ``cancelling the kernel'' (the multiplier being the kernel).
The appropriate root to use can be identified from
the combinatorial requirement that the series expansion of $P(1)$ contains no negative exponents and
has only non-negative coefficients.

In this case, the correct root is $u=(1-\sqrt{1-4\+z})/2\+z$, which yields
the univariate generating function for $\PPP$,
$$
P(1) \;=\; \thalf(1-\sqrt{1-4\+ z}).
$$
This turns out to be the same as that for rooted plane trees, since stepped parallelogram polyominoes, counted by width, are also enumerated by the Catalan numbers (see Stanley~\cite{Stanley2013}).

\subsection*{Analytic combinatorics}\label{sectAnalyticCombin}
It is a remarkable fact that the asymptotic behaviour of the coefficients of a generating function $F(z)$ can be determined by considering the analytic properties of $F$ considered as a complex function, that is, as a mapping of the complex plane to itself.
We briefly present here the key facts.

We are often interested in determining how the number of objects in a combinatorial class grows with size.
A fundamental quantity of interest is the exponential \emph{growth rate}.
The {growth rate} of a class $\CCC$ is defined to be the limit
$$
\gr(\CCC) \;=\;
\liminfty{|\CCC_n|}^{1/n},
$$
if it exists.
We define the \emph{upper} and \emph{lower} growth rates similarly:
$$
\grup(\CCC)=\limsupinfty{|\CCC_n|}^{1/n}
\qquad\qquad
\grlow(\CCC)=\liminfinfty\+{|\CCC_n|}^{1/n} .
$$
Observe that combinatorial classes whose enumeration differs only by a polynomial factor have the same (upper/lower) growth rate. This fact
follows directly from the definition of the growth rate.

Fundamental to determining exponential growth rates is Pringsheim's Theorem.
This result is concerned with the location of the singularities of
functions with non-negative coefficients. Such functions include all enumerative generating functions.
\begin{lemma}[{Pringsheim~\cite[Theorem IV.6]{FS2009}}]\label{lemmaPringsheim}
  If the power series for $F(z)$ has non-negative coefficients and radius of convergence $\rho$, then $z=\rho\in \mathbb{R}^+$ is a singularity of $F(z)$. 
\end{lemma}
The least singularity on the positive real axis is known as the
\emph{dominant singularity}.
It is all that is required to determine the (upper) growth rate of a combinatorial class, as the following lemma reveals.
\begin{lemma}[{Exponential Growth Formula~\cite[Theorem IV.7]{FS2009}}]\label{lemmaExponentialGrowthFormula}
  If $F(z)$ is analytic at 0 with non-negative coefficients and $\rho$ is its dominant singularity, then \vspace{-3pt}
  $$
  \limsupinfty \left(\big[z^n\big]F(z)\right)^{1/n} \;=\; \rho^{-1} . 
  $$
\end{lemma}
Thus, the upper growth rate of a combinatorial class is equal to the
reciprocal of the dominant singularity of its generating function.

Analytic combinatorics can give us more information than this.
The \emph{nature} of the dominant singularity prescribes the {subexponential factor}.
To state the relevant results, we use the notation $F(n)\sim G(n)$ to denote the fact that asymptotically $F(n)$ and $G(n)$ are ``approximately equal''. Formally,
$$
F(n)\;\sim\; G(n) \qquad \text{if} \quad \liminfty \tfrac{F(n)}{G(n)} \;=\; 1.
$$
For
meromorphic functions (i.e., functions whose singularities are poles), we have the following:
\begin{lemma}[{see \cite[Theorems IV.10 and VI.1]{FS2009}}]
  When the dominant singularity $\rho$ of $F(z)$ is a pole of order~$r$, then 
  $$
  [z^n]F(z) \;\sim\; c{\rho^{-n}}\+n^{r-1} 
  \qquad \text{where} \quad
  c \;=\; \dfrac{\rho^{-r}}{(r-1)!} \+ \lim\limits_{z\rightarrow\rho} \+ (\rho-z)^r\+F(z) .
  $$
\end{lemma}
More generally, when the dominant singularity is not a pole, the following result can be employed:
\begin{lemma}[{see \cite[Theorem VI.1]{FS2009}}]
  $$
  [z^n](1-z/\rho)^{-\alpha} \;\sim\; \frac{\rho^{-n}\+n^{\alpha-1}}{\Gamma(\alpha)} ,
  \qquad\quad
  \alpha\neq0,-1,-2,\ldots ,
  $$
  where $\Gamma(k)=(k-1)!$ and $\Gamma\big(k+\thalf\big)=\frac{(2k)!}{4^k\+k!}\sqrt{\pi}$, if $k\in\mathbb{N}$.
\end{lemma}

This concludes our brief exposition of combinatorial enumeration. As mentioned above, it is also possible to extract distributional information about parameter values from generating functions. We make use of this in Chapter~\myref{chap1324}. The relevant results are presented there.

\HIDE{ 
{Distribution of parameter values}

  {Mean}

    For a class $\SSS$ and parameter $\omega$ with BGF $S(z,w)$, the expected value of $\omega$ for an object of size $n$ is
    $$
    \mathbb{E}_n[\omega] \;=\;
    \frac{[z^n]\partial_w S(z,w)|_{w=1}}{[z^n]S(z,1)}
    \;=\;
   \frac{ \sum_{k\geqslant0}k\+s_{n,k}}{s_n}.
    $$

 {Variance}

    The variance of $\omega$ for an object of size $n$ is
    $$
    \mathbb{V}_n[\omega] \;=\;
    \frac{[z^n]\partial^2_w S(z,w)|_{w=1}}{[z^n]S(z,1)}
    + \mathbb{E}_n[\omega] - \mathbb{E}_n[\omega]^2.
    $$
} 

\section{Classes of permutations}

The combinatorial classes that we study in this thesis are classes of permutations.
We begin by presenting some standard definitions, before giving a brief historical overview of research
addressing
enumerative questions concerning
permutation classes.

\subsection*{Permutations}
We consider a permutation to be simply an arrangement of the numbers $1,2,\ldots,k$ for some positive $k$.
We use $|\sigma|$ to denote the length of permutation $\sigma$.
It can be helpful to consider permutations graphically.
If $\sigma=\sigma_1,\ldots,\sigma_k$ is a permutation, its \emph{plot} consists of the the points $(i,\sigma_i)$ in the Euclidean plane, for $i=1,\ldots,k$.
Often a permutation is identified with its plot.

\begin{figure}[ht]
  $$
  \begin{tikzpicture}[scale=0.3]
    \plotpermgrid{8}{3,1,5,6,7,4,8,2}
    \draw [thin] (1,3) circle [radius=0.36];
    \draw [thin] (4,6) circle [radius=0.36];
    \draw [thin] (6,4) circle [radius=0.36];
    \draw [thin] (7,8) circle [radius=0.36];
  \end{tikzpicture}
  \vspace{-3pt}
  $$
  \caption{The plot of permutation $\mathbf{31567482}$ with a $\pdiamond$ subpermutation marked}\label{figPermutation}
\end{figure}
A permutation $\tau$ is said to be \emph{contained} in, or to be a \emph{subpermutation} of, another permutation $\sigma$ if $\sigma$ has a subsequence whose terms are order isomorphic to (i.e.~have the same relative ordering as)
$\tau$.
From the graphical perspective,
$\sigma$ contains $\tau$ if the plot of $\tau$ results from erasing some points from the plot of $\sigma$ and then ``shrinking'' the axes appropriately.
We write $\tau\leqslant\sigma$ if $\tau$ is a subpermutation of $\sigma$.
For example, $\mathbf{31567482}$ contains $\pdiamond$
because the subsequence $\mathbf{3648}$ (among others) is ordered
in the same way as $\pdiamond$ (see Figure~\ref{figPermutation}).
If $\sigma$ does not contain $\tau$, we say that $\sigma$ \emph{avoids} $\tau$.
For example, $\mathbf{31567482}$ avoids $\mathbf{1243}$ since it has no subsequence ordered in the same way as $\mathbf{1243}$.
In the context of containment and avoidance, a permutation is often called a \emph{pattern}.

\begin{figure}[ht]
  $$
  \begin{tikzpicture}[scale=0.21]
    \plotpermnobox{14}{7,10,1,4,9,14,2,11,3,13,12,6,8,5}
    \draw[] (0.2,0.2) rectangle (14.8,14.8);
    \draw [thick] (1,7) circle [radius=0.4];
    \draw [thick] (2,10) circle [radius=0.4];
    \draw [thick] (6,14) circle [radius=0.4];
    \draw [thick] (3,1) circle [radius=0.4];
    \draw [thick] (7,2) circle [radius=0.4];
    \draw [thick] (9,3) circle [radius=0.4];
    \draw [thick] (14,5) circle [radius=0.4];
  \end{tikzpicture}
  \qquad\qquad\qquad
  \begin{tikzpicture}[scale=0.21]
    \plotpermnobox{14}{7,10,1,4,9,14,2,11,3,13,12,6,8,5}
    \draw[] (0.2,0.2) rectangle (14.8,14.8);
    \draw [thick] (1,7) circle [radius=0.4];
    \draw [thick] (3,1) circle [radius=0.4];
    \draw [thick] (6,14) circle [radius=0.4];
    \draw [thick] (10,13) circle [radius=0.4];
    \draw [thick] (11,12) circle [radius=0.4];
    \draw [thick] (13,8) circle [radius=0.4];
    \draw [thick] (14,5) circle [radius=0.4];
  \end{tikzpicture}
  \vspace{-3pt}
  $$
  \caption{The three left-to-right maxima and four right-to-left minima, and the two left-to-right minima and five right-to-left maxima, of a permutation}
  \label{figPermExtrema}
\end{figure}
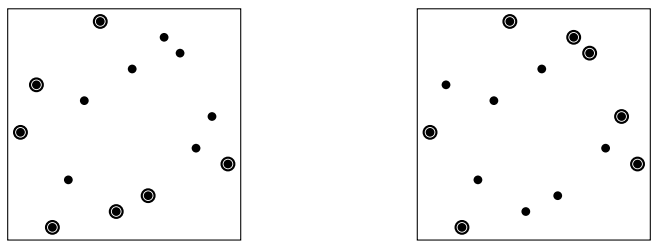
Sometimes we want to refer to the extremal points in a permutation.
A value in a permutation is called a \emph{left-to-right maximum} if it is larger than all the values to its left.
\emph{Left-to-right minima}, \emph{right-to-left maxima} and \emph{right-to-left minima} are defined analogously.
See Figure~\ref{figPermExtrema} for an illustration.

\begin{figure}[ht]
  $$
  \begin{tikzpicture}[scale=0.275]
    \plotpermnobox{}{2,4,1,3,8,6,7,5}
    \draw[] (0.5,0.5) rectangle (4.5,4.5);
    \draw[] (4.5,4.5) rectangle (8.5,8.5);
  \end{tikzpicture}
  \qquad\qquad
  \begin{tikzpicture}[scale=0.275]
    \plotpermnobox{}{6,8,5,7,4,2,3,1}
    \draw[] (4.5,0.5) rectangle (8.5,4.5);
    \draw[] (0.5,4.5) rectangle (4.5,8.5);
  \end{tikzpicture}
  \qquad\qquad
  \begin{tikzpicture}[scale=0.275]
    \plotpermnobox{}{2,1,3,6,5,4,8,7}
    \draw[] (0.5,0.5) rectangle (2.5,2.5);
    \draw[] (2.5,2.5) rectangle (3.5,3.5);
    \draw[] (3.5,3.5) rectangle (6.5,6.5);
    \draw[] (6.5,6.5) rectangle (8.5,8.5);
  \end{tikzpicture}
  \vspace{-3pt}
  $$
  \caption{The direct sum $\mathbf{2413}\oplus\mathbf{4231}$, the skew sum $\mathbf{2413}\ominus\mathbf{4231}$, and the layered permutation $\mathbf{21}\oplus\mathbf{1}\oplus\mathbf{321}\oplus\mathbf{21}$}
  \label{figPermSums}
\end{figure}
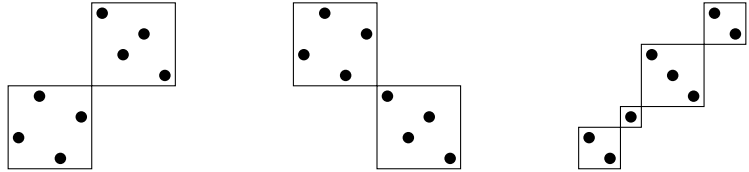
Given two permutations $\sigma$ and $\tau$ with lengths $k$ and $\ell$ respectively, their \emph{direct sum}\label{defDirectSum} $\sigma\oplus\tau$ is the permutation of length $k+\ell$ consisting of $\sigma$ followed by a shifted copy of $\tau$:
$$
(\sigma\oplus\tau)(i) \;=\;
\begin{cases}
  \sigma(i)   & \text{if~} i\leqslant k , \\
  k+\tau(i-k) & \text{if~} k+1 \leqslant i\leqslant k+\ell .
\end{cases}
$$
The \emph{skew sum} $\sigma\ominus\tau$ is defined analogously.
See Figure~\ref{figPermSums} for an illustration.

A permutation is called \emph{sum indecomposable} or just \emph{indecomposable} if it cannot be expressed as the direct sum of two shorter permutations.
For brevity, we call an indecomposable permutation simply an \emph{indecomposable}.
A permutation is \emph{skew indecomposable} if it cannot be expressed as the skew sum of two shorter permutations.
Note that every permutation has a unique representation as the direct sum of a sequence of indecomposables (and also as the skew sum of a sequence of skew indecomposables). 
If each indecomposable in a permutation is a decreasing sequence, then we say that the permutation is \emph{layered}.
See Figure~\ref{figPermSums} for an example.

\newcommand{\myArc}[2] {\draw [thick] (#1,1.25) arc (0:180:#2);} 
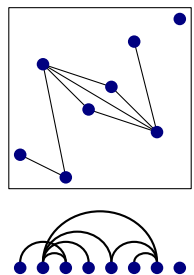
\begin{figure}[ht]
  $$
  \begin{tikzpicture}[scale=0.3,line join=round]
    \draw[] (0.5,4.5) rectangle (8.5,12.5);
    \draw [thin] (1,6)--(3,5)--(2,10)--(4,8)--(7,7)--(2,10)--(5,9)--(7,7)--(6,11);
    \plotpermnobox[blue!50!black]{}{6,10,5,8,9,11,7,12}
    \myArc{3}{1 and .9}
    \myArc{3}{.5 and .4}
    \myArc{4}{1 and .9}
    \myArc{5}{1.5 and 1.35}
    \myArc{7}{2.5 and 2.25}
    \myArc{7}{1 and .9}
    \myArc{7}{.5 and .4}
    \plotpermnobox[blue!50!black]{8}{1,1,1,1,1,1,1,1}
  \end{tikzpicture}
  \vspace{-3pt}
  $$
  \caption{A permutation and its ordered inversion graph}\label{figPermGraph}
\end{figure}
Sometimes it helps to consider permutations in a slightly broader context.
A permutation can be considered to be a particular type of \emph{ordered graph}.
An ordered graph is a graph with a linear order on its vertices (it is natural to number the vertices from $1$ to $n$).
To each permutation we associate an ordered graph.\label{defGraphOfPerm}
The (ordered) (inversion) \emph{graph} of a permutation $\sigma$ of length~$n$ has
vertex set $\{1,\ldots,n\}$
with an edge between vertices $i$ and~$j$ 
if $i<j$ and $\sigma(i)>\sigma(j)$.
A pair of terms in a permutation with this property is called an \emph{inversion}.
Thus, the graph of a permutation contains one edge for each inversion in the permutation.
Note that a permutation graph is transitively closed.
See Figure~\ref{figPermGraph} for an illustration.
We use $G_\sigma$ to denote the graph of $\sigma$.
An \emph{induced ordered subgraph} of an ordered graph $G$ is an induced subgraph of $G$ that inherits
its vertex ordering.
It is easy to see that
$\tau\leqslant\sigma$
if and only if
$G_\tau$ is an induced ordered subgraph of $G_\sigma$.

\subsection*{Permutation classes}
The subpermutation relation is a partial order on the set of all permutations.
A classical \emph{permutation class}, sometimes called a \emph{pattern class}, is a set of permutations closed downwards (a {down-set}) under this partial order.
Thus, if $\sigma$ is a member of a permutation class $\CCC$ and $\tau$ is contained in $\sigma$, then it must be the case that $\tau$ is also a member of $\CCC$.
From a graphical perspective, this means that erasing points from
the plot of
a permutation in $\CCC$ always results in
the plot of
another permutation in $\CCC$
when the axes are rescaled appropriately.

It is common in the study of
classical permutation classes to reserve the word ``class'' for 
sets of permutations closed downwards under containment, and to use the mathematical synonyms ``set'', ``collection'' and ``family'' for other combinatorial classes. We do not follow this convention rigorously in this thesis.

It is natural to define a permutation class ``negatively'' by stating the minimal set of permutations that it avoids.
This minimal forbidden set of patterns is known as the \emph{basis} of the class.
The class with basis $B$ is denoted $\av(B)$, and we use $\av_n(B)$ for the permutations of length $n$ in $\av(B)$.
As a trivial example, $\av(\mathbf{21})$ is the class of increasing permutations (i.e.~the identity permutation of each length).
As another simple example, the $\mathbf{123}$-avoiders, $\av(\mathbf{123})$, consists of those permutations that can be partitioned into two decreasing subsequences.
The basis of a permutation class is an {antichain} (a set of pairwise incomparable elements)
under the containment order,
and may be infinite.
Classes for which the basis consists of a single permutation are called \emph{principal} classes.

Suppose that $\SSS$ is a set of indecomposables 
that is downward closed under the subpermutation 
order (so if $\sigma\in\SSS$, $\tau\leqslant\sigma$ and $\tau$ is indecomposable, then $\tau\in\SSS$).
The \emph{sum closure}\label{defSumClosure} of $\SSS$, denoted~$\sumclosed\SSS$, is then the class of permutations of the form
$\sigma_1\oplus\sigma_2\oplus\ldots\oplus\sigma_r$, where each $\sigma_i\in\SSS$.
It is simple to check that $\sumclosed\SSS$ is, indeed, a permutation class.
Such a class is \emph{sum-closed}\label{defSumClosed}: if $\sigma,\tau\in\sumclosed\SSS$ then $\sigma\oplus\tau\in\sumclosed\SSS$.
Furthermore, every sum-closed class is the sum closure of its set of indecomposables.
It can easily be seen that a class is sum-closed if and only if all its basis elements are indecomposable.

In the context of graphs, a set
closed under taking induced 
subgraphs is known as a \emph{hereditary class}.
It is an easy exercise using mathematical induction to prove that the graph of a permutation is an ordered graph that avoids the two induced ordered subgraphs
  $
  \!\raisebox{-0pt}{
  \begin{tikzpicture}[scale=0.225,line join=round]
    \draw [thick] (1,1)--(3,1);
    \plotpermnobox{}{1,1,1}
  \end{tikzpicture}
  }\!
  $
  and
  $
  \!\raisebox{-0pt}{
  \begin{tikzpicture}[scale=0.225,line join=round]
    \myArc{3}{1 and .8}
    \plotpermnobox{}{1,1,1}
  \end{tikzpicture}
  }\!
  $.
Each permutation class is thus isomorphic to a hereditary class of ordered graphs, and results can be transferred between the two domains.

We are interested in the enumeration of permutation classes. One natural question is to determine whether two classes, $\CCC$ and $\DDD$, are equinumerous, i.e.~$|\CCC_n|=|\DDD_n|$ for every~$n$.
Two classes that are equinumerous are said to be \emph{Wilf-equivalent}
and the equivalence classes are called \emph{Wilf classes}.
From the graphical perspective, it is clear that classes related by symmetries of the square are Wilf-equivalent.
Thus, for example, $\av(\mathbf{132})$, $\av(\mathbf{231})$, $\av(\mathbf{213})$ and $\av(\mathbf{312})$ are equinumerous.
However, not all Wilf-equivalences are a result of these symmetries. Indeed $\av(\mathbf{123})$ and $\av(\mathbf{132})$ are Wilf-equivalent.

Classical permutation classes are not the only sets of permutations that are of interest. A specific focus is classes avoiding certain \emph{barred patterns}.\label{defBarredPattern}
A barred pattern is specified by a permutation with some entries barred ($\mathbf{4\bar{3}12}$, for example).
For a permutation $\sigma$ to avoid a barred pattern $\hat{\pi}$, whose underlying permutation is $\pi$, every occurrence in $\sigma$ of the
permutation order isomorphic to the non-barred entries in $\hat{\pi}$ ($\mathbf{312}$ in the example) must feature 
as part of an occurrence of $\pi$.
Note that the class of permutations avoiding a barred pattern is not normally closed downwards under the subpermutation order.
A combinatorial class of permutations that is not closed downwards
is called a \emph{non-classical} class.


\section{Historical overview}

We now present a brief, and somewhat selective, overview of the development of research into enumerative aspects of permutation patterns, a subject which now has a voluminous bibliography.
For a more detailed survey of the same area, see Vatter's chapter~\cite{Vatter2014} in the forthcoming \emph{Handbook of Enumerative Combinatorics}. The topic is presented as part of a broader picture in the books by Kitaev~\cite{Kitaev2011} and B\'ona~\cite{Bona2012}. Other useful sources include the volume of papers from the Permutation Patterns conference held in 2007~\cite{LRV2010} and
Steingr\'imsson's survey article produced for the 2013 British Combinatorial Conference~\cite{Steingrimsson2013}.
\vspace{6pt}

\subsubsection*{Stacks, queues and deques}
The study
of permutation classes can reasonably be said to have begun with Knuth's investigations
in the 1960s
into sorting using stacks and deques (double-ended queues), published
in the first volume of his encyclopedic monograph, \emph{The Art of Computer Programming}~\cite{Knuth1969}.
Knuth observed that a permutation can be sorted by a stack if and only if it does not contain the pattern $\mathbf{231}$, and showed that this class is counted by the Catalan numbers.
He also proved that the
class of permutations sorted by an \emph{input-restricted deque} (i.e.~a deque with the push operation restricted to one end)
is $\av(\mathbf{4231},\mathbf{3241})$ and enumerated the class with what was possibly the first use of the kernel method.
Knuth nicely described his approach in terms of railway ``switchyard networks'', and this perspective was taken up and developed in subsequent papers by Even \& Itai~\cite{EI1971}, Tarjan~\cite{Tarjan1972}, and Pratt~\cite{Pratt1973}, which investigated networks of stacks, queues and deques.

\begin{figure}[ht] 
  \begin{center}
  \includegraphics[scale=0.33]{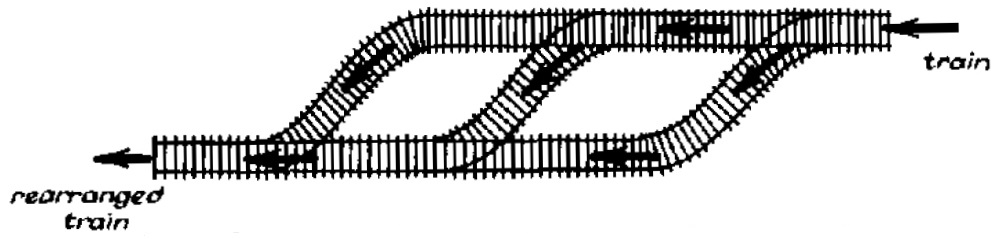}
  \vspace{-15pt}
  \end{center}
  \caption{A figure of a switchyard network from Tarjan's paper~\cite{Tarjan1972}}
\end{figure}
The investigation of stack sorting was continued in the work of West~\cite{West1990,West1993}.
He considered the class of permutations that could be sorted by passing twice through a stack
while requiring the contents of the stack to remain ordered (as in the Tower of Hanoi puzzle).
These permutations, which now tend to be referred to as the West-2-stack-sortable permutations, constitute the non-classical class $\av(\mathbf{2341},\mathbf{3\bar{5}241})$.
West conjectured that this class had the same enumeration as non-separable planar maps.
This conjecture was first proved
by Zeilberger~\cite{Zeilberger1992}, using a computer to solve a complicated functional equation.
Subsequently, Dulucq, Gire \& West~\cite{DGW1996} and Goulden \& West~\cite{GW1996}
found bijective proofs.
An alternative approach to sorting with two ordered stacks was subsequently investigated by Atkinson, Murphy \& Ru\v{s}kuc~\cite{AMR2002a}, who determined both the infinite basis of the class and its algebraic generating function.

Despite this activity, most problems related to Knuth's switchyard
networks have turned out to be very hard.
Fundamental questions 
remain unanswered,
including the enumeration of permutations sortable by two stacks in parallel,
the enumeration of permutations sortable by two stacks in series, and
the enumeration of permutations sortable by a general deque.
Albert, Atkinson \& Linton~\cite{AAL2010} calculated lower and upper bounds on the growth rates of each of these classes.
More recently, Albert \& Bousquet-M\'elou~\cite{ABM2015}, in a paper that is a \emph{tour de force} of analytic combinatorics, gave a pair of
functional equations that characterise the generating function of
permutations that can be sorted with two parallel stacks.
For more on sorting with stacks, queues and deques, see the surveys by B\'ona~\cite{Bona2003} and Atkinson~\cite{Atkinson2010}.

\subsubsection*{Conjectures}
Much of the research into permutation patterns has been driven by certain conjectures.
The first of these was the {Stanley--Wilf} conjecture that every permutation class $\CCC$ (excluding the class of all permutations) has a finite upper growth rate $\grup(\CCC)$, i.e.~that for any permutation $\sigma$ there exists a constant $c_\sigma$ such that, for every~$n$, $|\av_n(\sigma)|\leqslant {c_\sigma}^{\!n}$.
Arratia~\cite{Arratia1999} observed that the Stanley--Wilf conjecture implies that every sum-closed class $\CCC$, and hence every principal class, has a growth rate $\gr(\CCC)$.
Alon \& Friedgut~\cite{AF2000} managed to prove a result very close to the conjecture: that for any permutation $\sigma$ there exists a constant $c_\sigma$ such that, for every~$n$, $\av_n(\sigma)\leqslant {c_\sigma}^{n\gamma(n)}$, where $\gamma$ is a function that grows extremely slowly.
B\'ona~\cite{Bona1999,Bona2004} established that the conjecture was true for layered patterns.
Meanwhile, Klazar~\cite{Klazar2000} showed that the Stanley--Wilf conjecture was implied by a conjecture of F\"uredi \& Hajnal~\cite{FH1992} concerning
0-1 matrices.
Finally, Marcus \& Tardos~\cite{MT2004} gave a proof of the F\"uredi--Hajnal conjecture, thus confirming the Stanley--Wilf conjecture.
Thus, every principal class has a growth rate.
There are no known examples of permutation classes that do not have a growth rate
and it is widely believed that $\gr(\CCC)$ exists for every permutation class $\CCC$ (see the first conjecture in~\cite{Vatter2014}).

A second conjecture is that of Noonan \& Zeilberger~\cite{NZ1996} that every {finitely based} permutation
class has a D-finite generating function.
Clearly, this is not the case for every permutation
class since there are uncountably many permutation classes with
distinct generating functions, but only countably many D-finite generating functions.
This conjecture remains open. However, it is now generally believed to be false, Zeilberger (see~\cite{EV2005}) counter-conjecturing
that the Noonan--Zeilberger conjecture is false, and, in
particular, the sequence $\av_n(\pdiamond)$ is not P-recursive.
Recent work of Conway \& Guttmann~\cite{CG2015} strongly suggests that $\av(\pdiamond)$ does indeed have a non-D-finite generating
function.\footnote{Following submission of this thesis, at the AMS meeting in Washington, DC, in March 2015, Scott Garrabrant announced a proof that the conjecture is false, the result of joint work with Igor Pak~\cite{GP2015a}.}

How fast can a permutation class grow?
The proof, by Marcus and Tardos, of the Stanley--Wilf conjecture means that every principal class has a growth rate.
Marcus and Tardos' proof yielded a doubly exponential upper bound on $\gr(\av(\beta))$ in terms of the length of $\beta$.
Cibulka~\cite{Cibulka2009} was able to reduce it to the order of $2^{k \log k}$. 
Arratia~\cite{Arratia1999} conjectured that the upper bound was much lower and
that, for every permutation $\beta$ of length $k$, $\gr(\av(\beta)) \leqslant (k - 1)^2$.
B\'ona~\cite{Bona2005} then strengthened this, by conjecturing that equality holds
if and only if $\beta$ is layered.
Arratia's conjecture was subsequently refuted by Albert, Elder, Rechnitzer, Westcott \& Zabrocki~\cite{AERWZ2006} who showed that $\gr(\av(\pdiamond))$ exceeded 9.47.

However, evidence suggested that the fastest growing principal classes were those of layered permutations, and B\'ona
conjectured (see~\cite{Bona2007}) that the most easily avoided permutation of length
$k$ 
was $1 \oplus 21 \oplus 21 \oplus\ldots\oplus 21$ for odd $k$ and $1 \oplus 21 \oplus 21 \oplus\ldots\oplus 21 \oplus 1$ for even~$k$.
Claesson, Jel\'inek, and Steingr\'imsson~\cite{CJS2012} then proved that for every layered permutation $\beta$ of length
$k$, the growth rate of $\av(\beta)$ is less than $4k^2$, and B\'ona~\cite{Bona2012a+} showed that for his conjectured
easiest-to-avoid permutations, the growth rates were at most $9k^2/4$.
It thus came as somewhat of a shock when Fox~\cite{Fox2013}
recently proved (by considering the problem in the context of 0-1 matrices) that for almost all permutations $\beta$ of length $k$, $\gr(\av(\beta))$ is, in fact, of the order
of $2^k$.
Therefore, in general, layered permutations are very far from being the easiest to avoid.

\subsubsection*{Specific classes}
Another major strand in permutation class research has concerned the enumeration of classes with a few small basis elements.
An up-to-date table of results in this area is recorded on the Wikipedia page~\cite{WikiEnumPermClassesThin}.
Knuth's investigation of $\av(\mathbf{231})$ was not the first study of a permutation class.
Half a century previously, MacMahon~\cite{MacMahon1915} had shown that $\av(\mathbf{123})$ was counted by the Catalan numbers.
Knuth's matching result for $\av(\mathbf{231})$ thus gave the first example of a Wilf-equivalence not induced by symmetry, and revealed that there was only a single Wilf class for permutations of length 3.

One important class that was considered soon after Knuth's work
was
the \emph{Baxter permutations}, previously studied by Baxter~\cite{Baxter1964}
in connection with an investigation into the behaviour of fixed points of commuting continuous functions.
The Baxter permutations constitute the non-classical class $\av(\mathbf{25\bar{3}14},\mathbf{41\bar{3}52})$ (see Gire~\cite{Gire1993}).
They were enumerated by
Chung, Graham, Hoggatt \& Kleiman~\cite{CGHK1978}, who introduced \emph{generating trees} as an enumerative mechanism,
a technique later to be used more widely.
Another significant
early enumerative result was the proof by Regev~\cite{Regev1981} that $\gr(\av(\mathbf{12\ldots k}))=(k-1)^2$ for every~$k$.
Gessel~\cite{Gessel1990} later gave an explicit enumeration in terms of determinants.

The first in-depth 
study of classical permutation classes
was by Simion \& Schmidt~\cite{SS1985} who enumerated classes avoiding two patterns of length 3 and
established that there were three Wilf classes.
Subsequent results made heavy use of generating trees:
West~\cite{West1995} showed that $\av(\textbf{3142},\textbf{2413})$ is enumerated by the Schr\"oder numbers.
West also~\cite{West1996} enumerated classes avoiding one pattern of length 3 and one of length 4, and, in collaboration with Chow~\cite{CW1999}, those avoiding one pattern of length 3 and an increasing or decreasing sequence.
B\'ona~\cite{Bona1997a} enumerated $\av(\mathbf{1342})$ by establishing a bijection between the class and $\beta(0,1)$ trees.

Various other techniques have been used. Zeilberger~\cite{Zeilberger1998} introduced \emph{enumeration schemes} for automatic enumeration.
These were later employed by Kremer \& Shiu~\cite{KS2003} to count several classes that avoid pairs of length 4 patterns.
Enumeration schemes were further developed by Vatter~\cite{Vatter2008}, and extended by Pudwell~\cite{Pudwell2010} for use with barred pattern classes,
and by Baxter \& Pudwell~\cite{BP2012} to enumerate classes avoiding
\emph{vincular} patterns, a type of non-classical pattern
introduced by
Babson \& Steingr\'imsson~\cite{BS2000}.
An alternative method that has been successful in some contexts
is the \emph{insertion encoding} of permutations introduced by Albert, Linton \& Ru\v{s}kuc~\cite{ALR2005}.
Albert, Elder, Rechnitzer, Westcott, \& Zabrocki~\cite{AERWZ2006}
made use of the insertion encoding to determine a lower bound on the growth rate of $\av(\pdiamond)$.
This technique was also utilized by Vatter~\cite{Vatter2012} to enumerate two classes avoiding two patterns of length 4.

Another important approach has been the use of \emph{grid classes}, an introduction to which we postpone until Chapter~\myref{chap02}.
Atkinson~\cite{Atkinson1998} determined the generating function
for \emph{skew-merged} permutations, which can be partitioned into an increasing sequence and a decreasing sequence, a class that
Stankova~\cite{Stankova1994} had previously determined to
have the basis $\{\mathbf{2143},\mathbf{3412}\}$.
Atkinson~\cite{Atkinson1999} also made use of grid classes to enumerate a number of classes, including $\av(\mathbf{132},\mathbf{4321})$.
Much more recently,
grid classes have been used
to enumerate various classes avoiding two patterns of length 4 by
Albert, Atkinson \& Brignall~\cite{AAB2011,AAB2012}, Pantone~\cite{Pantone2013}, and Albert, Atkinson \& Vatter~\cite{AAV2014}.
Finally, a paper of Albert \& Brignall~\cite{AB2014} utilizes grid classes to enumerate
$\av(\mathbf{4231},\mathbf{35142},\mathbf{42513},\mathbf{351624})$, a class which arises in connection with algebraic geometry, specifically the categorization of Schubert varieties.


\label{sectWilfEquiv}
Parallel to the enumeration of specific classes went work on determining the Wilf classes.
The first result of this sort was by West~\cite{West1990}, who showed that, for any permutation $\sigma$,
$\av(\mathbf{12}\oplus\sigma)$ and $\av(\mathbf{21}\oplus\sigma)$ are Wilf-equivalent.
Hence, in particular, $\av(\mathbf{1234})$, $\av(\mathbf{1243})$ and $\av(\mathbf{2143})$ are in the same Wilf class.
This result was subsequently generalised. 
Babson \& West~\cite{BW2000} proved the Wilf-equivalence of
$\av(\mathbf{123}\oplus\sigma)$ and $\av(\mathbf{321}\oplus\sigma)$.
Then, Backelin, West \& Xin~\cite{BWX2007} demonstrated that
$\av(\mathbf{12\ldots k}\oplus\sigma)$ and $\av(\mathbf{k\ldots21}\oplus\sigma)$ were in the same Wilf class for every $k$.
It was also established
by Stankova \& West~\cite{SW2002} that $\av(\mathbf{231}\oplus\sigma)$ was Wilf-equivalent to $\av(\mathbf{312}\oplus\sigma)$.
In addition, Stankova~\cite{Stankova1994} showed the Wilf-equivalence of
$\av(\mathbf{1342})$ and $\av(\mathbf{2413})$.
As well as these results on principal classes, papers by
B\'ona~\cite{Bona1998a},
Kremer~\cite{Kremer2000,Kremer2003}, and
Le~\cite{Le2005}
together accomplished the Wilf classification of classes avoiding pairs of patterns of length four.


There is one obvious gap in this record of the enumeration of classes with small bases: $\av(\pdiamond)$.
This class has been the \emph{b\^ete noire} of permutation class enumeration. Very little concrete progress has been made on it. We consider the $\pdiamond$-avoiders
in Chapter~\myref{chap1324}, and
postpone a historical overview until there.

Another important strand of research which we have ignored here is the question of
determining the structure of the subset of the real line consisting of
the growth rates of a permutation classes.
We investigate this subject in Chapter~\myref{chap08}, where we present the background to our work in this area.

\section{Outline and list of main results}

The rest of this thesis is divided into three parts.
In
Part~\myref{partI}, we consider the enumeration of monotone grid classes of permutations, a family of permutation classes defined in terms of the permitted shape for plots of permutations in a class.
The most important result in this part is an explicit formula for the growth rate of every permutation grid class.
Part~\myref{partII} is much shorter and concerns the structure of the set of growth rates of permutation classes.
We establish a new upper bound on the value above which every real number is the growth rate of some permutation class.
Finally, in Part~\myref{partIII}, we introduce a new technique that can be used for the enumeration of permutation classes, based on a
graph associated with each permutation, which we call its Hasse graph.
As well as using this method to determine the generating function for some previously unenumerated classes, we conclude the thesis by making use of our approach to provide an improved lower bound on the growth rate of $\av(\pdiamond)$.

The chapters in this thesis are of very unequal length, each one addressing a specific enumerative question.
The work
in Chapter~\myref{chap05} has been published (see~\cite{Bevan2013}), as has that
in Chapter~\myref{chap07} (see~\cite{Bevan2013b}).
The research in Chapters~\myref{chapF} and~\myref{chapE} has been accepted for publication (see~\cite{Bevan2014a}), as has that in Chapter~\myref{chap1324} (see~\cite{Bevan2014Thin}).
The work in Chapter~\myref{chap08} has been submitted for publication (see~\cite{Bevan2014b}).
Here, for reference, is a list of the primary results in this thesis:

\subsubsection*{Part~\myref{partI}: Grid Classes}
\begin{bullets}
\item Exact enumeration of skinny grid classes (Theorem~\myref{thmSkinny}).
\item Exact enumeration of acyclic classes of gridded permutations (Theorem~\myref{thmAcyclicGriddingsGF}).
\item Exact enumeration of unicyclic classes of gridded permutations (Theorem~\myref{thmUnicyclicGriddingsGF}).
\item The generating function of any unicyclic class of gridded permutations is algebraic (Theorem~\myref{thmUnicyclicAlgebraic}).
\item The generating function of any class of gridded permutations is D-finite (Theorem~\myref{thmGridPermsDFinite}).
\item The growth rate of the family of balanced tours on a connected graph is the same as that of the family of all tours of even length on the graph (Theorem~\myref{thmBalancedEqualsEven}).
\item The growth rate of a monotone grid class of permutations is equal to the square of the spectral radius of its row-column graph (Theorem~\myref{thmGrowthRate}).
\item The growth rate of each monotone grid class is an algebraic integer (Corollary~\myref{corAlgebraicInteger}).
\item A monotone grid class whose row-column graph is a cycle has growth rate~4 (Corollary~\myref{corCycle}).
\item If the growth rate of a monotone grid class is less than 4, it is equal to $4\cos^2\!\left(\frac{\pi}{k}\right)$ for some $k\geqslant 3$ (Corollary~\myref{corSmallGrowthRates}).
\item For every $\gamma \geqslant 2+\sqrt{5}$ there is a monotone grid class with growth rate arbitrarily close to $\gamma$ (Corollary~\myref{corAccumulationPoints}).
\item Almost all large permutations in a monotone grid class have the same shape (Proposition~\myref{propGriddedPermsLimitShape}).
\item A technique for determining the limit shape of a permutation in a monotone grid class (Proposition~\myref{propAsymptLagrange}).
\item The growth rate of
a geometric grid class
is equal to
the square
of the largest root of
the matching polynomial
of the row-column graph of the double refinement of its gridding matrix (Theorem~\myref{thmGeomClassGrowthRate}).
\item The set of growth rates of geometric grid classes consists of the squares of the spectral radii of trees (Corollary~\myref{corGeomSqRhoTrees}).
\item If $G(M)$ is connected,
  then $\gr(\Geom(M))<\gr(\Grid(M))$
  if and only if $G(M)$ contains a cycle (Corollary~\myref{corGeomGridGRIneq}).
\item A specification of the effect of the subdivision of an edge on the largest root of
the matching polynomial of a graph (Lemma~\myref{lemmaSubdividing1}).
\end{bullets}
\vspace{-12pt}

\subsubsection*{Part~\myref{partII}: The Set of Growth Rates}
\begin{bullets}
\item Let $\theta_B\approx2.35526$ be the unique real root of $x^7-2\+x^6-x^4-x^3-2\+x^2-2\+x-1$.
The set of growth rates of permutation classes includes an infinite sequence of intervals whose infimum is $\theta_B$ (Theorem~\myref{thmTheta}).
\item Let $\lambda_B\approx2.35698$ be the unique positive root of $x^8-2\+x^7-x^5-x^4-2\+x^3-2\+x^2-x-1$.
Every value at least $\lambda_B$ is the growth rate of a permutation class (Theorem~\myref{thmLambda}).
\item A specification of when the set of representations of numbers in non-integer bases, where each digit is drawn from a different set, is an interval
(Lemma~\myref{lemmaBaseNotation1}).
\end{bullets}
\vspace{-12pt}

\subsubsection*{Part~\myref{partIII}: Hasse Graphs}
\begin{bullets}
\item A new derivation of the algebraic generating function of $\av(\pdiamond,\mathbf{2314})$, first enumerated by Kremer~\cite{Kremer2000,Kremer2003} (Theorem~\myref{thmR}).
\item A functional equation for the generating function of $\av(\pdiamond,\mathbf{1432})$ (Theorem~\myref{thmGCrec}).
\item A new derivation of the algebraic generating function of {forest-like} permutations, $\av(\pdiamond,\mathbf{21\bar{3}54})$,
first enumerated by Bousquet-M\'elou \& Butler~\cite{BMB2007} (Theorem~\myref{thmL}).
\item A functional equation for the generating function of plane permutations, $\av(\mathbf{21\bar{3}54})$ (Theorem~\myref{thmM}).
\item The algebraic generating function of $\av(\mathbf{1234},\mathbf{2341})$ (Theorem~\myref{thmF}).
\item The algebraic generating function of
$\av(\mathbf{1243},\mathbf{2314})$ (Theorem~\myref{thmE}).
\item The growth rate of $\av(\pdiamond)$ exceeds $9.81$ (Theorem~\myref{thm1324LowerBound}).
\item The number of occurrences of a fixed pattern
in a {\L}uka\-sie\-wicz path of length $n$
exhibits
a concentrated Gaussian limit distribution
(Theorem~\myref{thmLukaPatternsGaussian}).
\end{bullets}

\cleardoublepage


\part{\textsc{Grid Classes}}\label{partI} 
\setcounter{chapter}{1}

\chapter{Introducing grid classes}\label{chap02}



\section{Grid classes and griddings}
\label{sectGridClassDefs}
One approach to investigating permutation classes
that has proven particularly fruitful
involves the use of certain classes that are defined constructively, rather than in terms of their basis.
The monotone \emph{grid class} $\Grid(M)$\footnote{Huczynska \& Vatter~\cite{HV2006} were the first to use the term ``grid class'' and the notation $\Grid(M)$.}
is defined by a matrix $M$, all of whose entries are in $\{0,1,-1\}$. This \emph{gridding matrix} specifies the permitted shape for plots of permutations in the class. Each entry of $M$ corresponds to a \emph{cell} in a \emph{gridding} of a permutation. If the
entry is $1$, any points in the cell must form an increasing sequence; if the entry is $-1$, any points in the cell must form a decreasing sequence; if the entry is $0$, the cell must be empty.

\begin{figure}[ht]
  $$
  \begin{tikzpicture}[scale=0.20]
    \plotperm{8}{3,1,5,6,7,4,8,2}
      \draw[thick] (5.5,.5) -- (5.5,8.5);
      \draw[thick] (.5,2.5) -- (8.5,2.5);
  \end{tikzpicture}
  \quad\;
  \begin{tikzpicture}[scale=0.20]
    \plotperm{8}{3,1,5,6,7,4,8,2}
      \draw[thick] (5.5,.5) -- (5.5,8.5);
      \draw[thick] (.5,3.5) -- (8.5,3.5);
  \end{tikzpicture}
  \quad\;
  \begin{tikzpicture}[scale=0.20]
    \plotperm{8}{3,1,5,6,7,4,8,2}
      \draw[thick] (5.5,.5) -- (5.5,8.5);
      \draw[thick] (.5,4.5) -- (8.5,4.5);
  \end{tikzpicture}
  \quad\;
  \begin{tikzpicture}[scale=0.20]
    \plotperm{8}{3,1,5,6,7,4,8,2}
      \draw[thick] (4.5,.5) -- (4.5,8.5);
      \draw[thick] (.5,4.5) -- (8.5,4.5);
  \end{tikzpicture}
  \quad\;
  \begin{tikzpicture}[scale=0.20]
    \plotperm{8}{3,1,5,6,7,4,8,2}
      \draw[thick] (3.5,.5) -- (3.5,8.5);
      \draw[thick] (.5,4.5) -- (8.5,4.5);
  \end{tikzpicture}
  \quad\;
  \begin{tikzpicture}[scale=0.20]
    \plotperm{8}{3,1,5,6,7,4,8,2}
      \draw[thick] (2.5,.5) -- (2.5,8.5);
      \draw[thick] (.5,4.5) -- (8.5,4.5);
  \end{tikzpicture}
  \quad\;
  \begin{tikzpicture}[scale=0.20]
    \plotperm{8}{3,1,5,6,7,4,8,2}
      \draw[thick] (2.5,.5) -- (2.5,8.5);
      \draw[thick] (.5,5.5) -- (8.5,5.5);
  \end{tikzpicture}
  \vspace{-6pt}
  $$
  \caption{The seven griddings of permutation $\mathbf{31567482}$ in
  \protect\gctwo{2}{1,1}{-1,-1}
  }\label{figGriddings}
\end{figure}
For greater clarity, we denote grid classes by \emph{cell diagrams} rather than by their matrices; for example, $\gctwo{3}{1,-1,0}{0,-1,1}=\begin{gridmx}1&-1&0\\0&-1&1\end{gridmx}$.
Sometimes, with a slight abuse of notation, we use a cell diagram to denote the gridding matrix itself.
We say that a grid class $\Grid(M)$ has \emph{size} $m$ if its matrix $M$ has $m$ non-zero entries.
Note that a permutation may have multiple possible griddings in a grid class. See Figure~\ref{figGriddings} for an example.

When defining grid classes, to match the way we view permutations graphically, we index matrices from the lower left corner, with the order of the indices reversed from the normal convention. For example,
a matrix with dimensions $t\times u$ has $t$ columns and $u$~rows, and
$M_{2,1}$ is the entry in the second column from the left and in the bottom row of~$M$.

\label{defGridding}
If $M$ is a $0/\!\pm\!1$ matrix with $t$ columns and $u$ rows, then an \emph{$M$-gridding} of a permutation $\sigma$ of length $k$ is a pair of integer
sequences
$0=c_0\leqslant c_1\leqslant \ldots\leqslant c_t=k$ (the \emph{column dividers})
and $0=r_0\leqslant r_1\leqslant \ldots\leqslant r_u=k$ (the \emph{row dividers})
such that
for all $i\in\{0,\ldots,t\}$ and $j\in\{0,\ldots,u\}$,
the subsequence of $\sigma$ with indices in $(c_{i-1},c_i]$ and values in $(r_{j-1},r_j]$ is
increasing if $M_{i,j}=1$,
decreasing if $M_{i,j}=-1$, and
empty if $M_{i,j}=0$.
For example, in the leftmost gridding in Figure~\ref{figGriddings}, $c_1=5$ and $r_1=2$.

The \emph{grid class} $\Grid(M)$ is then defined to be the set of all permutations that have an $M$-gridding.
The \emph{griddings} of a permutation in $\Grid(M)$ are its $M$-griddings.
We use $\Grid_k(M)$ to denote the permutations in $\Grid(M)$ of length $k$.

Sometimes we need to consider a permutation along with a specific gridding. In this case, we refer to a permutation together with an $M$-gridding as an \emph{$M$-gridded permutation}. We use $\Gridhash(M)$ to denote the class of all $M$-gridded permutations, every permutation in $\Grid(M)$ being present once with each of its griddings.
We use $\Gridhash_k(M)$ for the set of $M$-gridded permutations of length $k$.
An important observation is the following:
\begin{lemma}[{Vatter~\cite[Proposition~2.1]{Vatter2011}}]\label{lemmaGRGriddings}
The upper/lower growth rate of a monotone grid class $\Grid(M)$ is equal to the upper/lower growth rate of the corresponding class of gridded permutations $\Gridhash(M)$.
\end{lemma}
\begin{proof}
Suppose that $M$ has dimensions $t\times u$.
Every permutation in $\Grid(M)$ has at least one gridding in $\Gridhash(M)$, but no permutation in $\Grid(M)$ of length $k$ can have more than $P(k)=\binom{k+t-1}{t-1}\binom{k+u-1}{u-1}$ griddings in $\Gridhash(M)$ because $P(k)$ is the number of possible choices for the
row and column dividers.
Since $P(k)$ is a polynomial
in $k$, the result follows from the definition of the upper/lower growth rate.
\end{proof}


\subsection*{Row-column graphs}\label{sectRowCol}

To each grid class, we associate a bipartite graph, which we call its \emph{row-column graph}\footnote{Vatter~\cite{Vatter2011} was the first to use the term ``row-column graph''.}.
If $M$ has $t$ rows and $u$ columns, the {row-column graph}, $G(M)$, of $\Grid(M)$ is the graph
with vertices $r_1,\ldots,r_t,c_1,\ldots,c_u$ and an edge between $r_i$ and $c_j$ if and only if $M_{i,j}\neq0$ (see Figure~\ref{figRowColumnGraphA} for an example).
Note that any bipartite graph is
the row-column graph of some grid class, and that the size (number of edges) of the row-column graph is the same as the size (number of non-zero cells) of the grid class.
\begin{figure}[ht]
  $$
  \gclass[0.3125]{4}{3}{
    \gcrow{2}{-1, 1,0, 0}
    \gcrow{1}{ 0, 1,0,-1}
    \gcrow{0}{ 1,-1,1, 0}
  }
  \qquad\qquad
  \raisebox{-.15in}{
  \begin{tikzpicture}[scale=0.5]
      \draw [thick] (4,0)--(1,0)--(1,1)--(2,1)--(2,0);
      \draw [thick] (0,0)--(1,0);
        \fill[black!35!blue,radius=0.17] (0,0) circle ;
        \fill[black!35!red,radius=0.17] (1,0) circle ;
        \fill[black!35!blue,radius=0.17] (2,0) circle;
        \fill[black!35!red,radius=0.17] (3,0) circle;
        \fill[black!35!blue,radius=0.17] (4,0) circle;
        \fill[black!35!red,radius=0.17] (2,1) circle ;
        \fill[black!35!blue,radius=0.17] (1,1) circle ;
        \node[below]at(0,-.1){$c_3$};
        \node[below]at(1,-.1){$r_3$};
        \node[below]at(2,-.1){$c_2$};
        \node[below]at(3,-.1){$r_2$};
        \node[below]at(4,-.1){$c_4$};
        \node[above]at(1,1.1){$c_1$};
        \node[above]at(2,1.1){$r_1$};
    \end{tikzpicture}}
    \vspace{-6pt}
    $$
    \caption{A grid class and its row-column graph}
    \label{figRowColumnGraphA}
\end{figure}
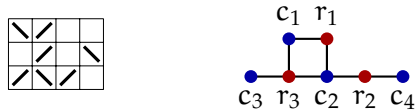

The row-column graph of a grid class captures a great deal of important structural information about the class, so it is common to apportion properties of the row-column graph directly to the grid class itself, for example speaking of a
connected, acyclic or unicyclic
grid class
rather than of a grid class whose row-column graph is
connected, acyclic or unicyclic.
We follow this convention.

Of particular note,
Murphy \& Vatter~\cite{MV2003}
have shown that
a grid class is \emph{partially well-ordered} (contains no infinite antichains) if and only if its row-column graph has no cycles.
A simpler proof of this result was subsequently given by Vatter \& Waton~\cite{VW2011}.
Furthermore, Albert, Atkinson, Bouvel, Ru\v{s}kuc \& Vatter~\cite{AABRV2011} proved a result that implies that
if a grid class has an acyclic row-column graph
then the class has a rational generating function.

It is generally believed that every grid class has a finite basis (see~\cite[Conjecture 2.3]{HV2006}), but this has so far resisted proof.
Atkinson~\cite{Atkinson1999}
proved that \emph{skinny}
grid classes (whose matrices have dimensions $m\times1$ for some $m$) 
have a finite basis.
Waton~\cite{WatonThesis} established the same for
$\!\gctwo{2}{1,1}{1,1}\!$,
a result
which has since been extended by
Albert, Atkinson \& Brignall
to all $2\times 2$
grid classes (in an unpublished note~\cite{AAB2010}).
More recently, Albert, Atkinson, Bouvel, Ru\v{s}kuc \& Vatter~\cite{AABRV2011} have proved that every grid class with an acyclic row-column graph is finitely based.

The concept of a grid class of permutations has been generalised, permitting arbitrary permutation classes in each cell (see Vatter~\cite[Section~2]{Vatter2011}). We only consider \emph{monotone} grid classes here.
Both monotone and generalised grid classes have played a key role in investigations of the set of permutation class growth rates (see~\cite{HV2006,Vatter2011}).
We explore this topic in Part~\ref{partII} below.
An interactive \emph{Mathematica} demonstration of monotone grid classes is available online~\cite{Bevan2012}.

\section{Outline of Part I}

In the next five chapters, we explore
the enumeration of grid classes from various angles.

Chapter~\myref{chap03} concerns \emph{skinny} grid classes, whose matrices have dimensions $m\times1$ for some $m$.
A permutation in such a class consists of the juxtaposition of ascending and descending sequences.
We present an effective procedure for calculating the generating function for any skinny grid class.

In chapter~\myref{chap04}, we investigate the enumeration of classes of gridded permutations.
We exhibit an effective method for determining the generating function for any acyclic or unicyclic class of gridded permutations.
In the process, we prove that unicyclic classes have algebraic generating functions.
We also prove that the generating function
of every class of gridded permutations is D-finite.

In Chapter~\myref{chap05}, we turn away from exact enumeration and focus on determining the exponential growth rate of grid classes.
We prove that the growth rate of a monotone grid class is given by the square of the spectral radius of its row-column graph.
This is
the first general result
concerning the exact growth rates of a family of permutation classes.
Our proof depends on relating classes of gridded permutations to certain families of tours on graphs, and in the process we establish a new result concerning these tours that is of independent interest.
As a consequence of our theorem, we deduce a number of facts relating to the set of growth rates of grid classes.
This work has been published in~\cite{Bevan2013}.

Chapter~\myref{chap06} concerns the shape of a ``typical'' large permutation in a
monotone grid class. We show that almost all large permutations in a grid class do indeed look the same, and explain how to determine the limit shape.

The final chapter in this part of the thesis, Chapter~\myref{chap07}, concerns \emph{geometric} grid classes, a family of permutation classes closely related to monotone grid classes.
We investigate the growth rates of these classes, and
prove a result
which relates
the growth rate of a geometric grid class to the largest root of the \emph{matching polynomial} of a graph.
In the process we establish a new result concerning the effect of edge
subdivision on the largest root of the matching polynomial.
We also deduce a number of consequences
including providing a characterisation of the growth rates of
geometric grid classes in terms of the spectral radii of trees.
This work was published in~\cite{Bevan2013b}.

\cleardoublepage


\newcommand{\Lzero}{\oper_0}
\newcommand{\Lone}{\oper_1}
\newcommand{\Lstar}{\oper_\star}
\newcommand{\Lplus}{\oper_+}
\newcommand{\Lminus}{\oper_-}

\newcommand{\drawnxt}[2][+-]{\draw #2;}

\newcommand{\skinnyclass}[2][0.35]   
{
  \raisebox{-2pt}
  {
  \begin{tikzpicture}[scale=#1]
    \foreach \d [count=\i] in {#2}
    {
      \draw (\i-1,0) grid (\i,1);
      \draw[thick] (\i-.85,.5-.35*\d)--(\i-.15,.5+.35*\d);
    }
  \end{tikzpicture}
  }
}

\newcommand{\gridclassrow}[3][1]  
{
  \foreach \d [count=\i] in {#3}
  {
    \draw (#1-1+\i-1,#2) grid (#1-1+\i,#2+1);
    \draw[thick] (#1-1+\i-.85,#2+.5-.35*\d)--(#1-1+\i-.15,#2+.5+.35*\d);
  }
}
\newcommand{\gridclasszero}[2]  
{
  \draw (#1-1,#2) grid (#1,#2+1);
}

\chapter{Skinny grid classes}\label{chap03}

We say that the permutation grid class $\Grid(V)$ is \emph{skinny}
if $V$
is a $\pm1$ row vector.
Thus, a permutation in a skinny grid class consists of the juxtaposition of ascending and descending sequences.
For an illustration, see Figure~\ref{figSkinnyA}.
Skinny grid classes were previously studied by Atkinson, Murphy \& Ru\v{s}kuc~\cite{AMR2002} (under the name ``monotone segment sets''). They proved that every skinny grid class can be described by a regular language and so has a rational generating function.
In this chapter we present a way of
determining the generating function for any skinny grid class $\Grid(V)$.

\begin{figure}[ht]
  \vspace{3pt}
  $$
      \begin{tikzpicture}[scale=0.2]
        \plotperm{15}{4,8,9,11,15,10,5,1,7,6,2,3,12,13,14}
        \drawnxt[2]{[thin] ( 4.5,.5) -- ( 4.5,15.5)}
        \drawnxt[2]{[thin] ( 8.5,.5) -- ( 8.5,15.5)}
        \drawnxt[2]{[thin] (11.5,.5) -- (11.5,15.5)}
      \end{tikzpicture}
  $$
  \caption{A gridded permutation in skinny grid class $\!\!\!$ \skinnyclass{1,-1,-1,1}}
  \label{figSkinnyA}
\end{figure}

To state our result, we need some definitions.
Suppose $V=(v_1,\ldots,v_k)$ is a $\pm1$ vector of length $k$.
Then we define $V^+$ to be the vector $(v_1,\ldots,v_k,v_k)$ of length $k+1$, created from $V$ by repeating its last entry, and define
$V^-$ to be $(v_1,\ldots,v_k,-v_k)$, created from $V^+$ by negating its last entry.
For example, 
$\skinnyclass{-1,1}^+ = \skinnyclass{-1,1,1}\!$
and~$\skinnyclass{-1,1}^- = \skinnyclass{-1,1,-1}\!$.
Also, given a vector $V=(v_1,\ldots,v_k)$ and some $j\leqslant k$, we use $V{[\![j]\!]}$ to denote $(v_1,\ldots,v_j)$, the prefix of $V$ of length $j$.
For example, $\skinnyclass{1,-1,-1,1}\!{[\![2]\!]} = \skinnyclass{1,-1}\!$.

\newpage  
We prove that the generating function for a skinny grid class can be computed as follows:

\thmbox{
\begin{thm}\label{thmSkinny}
If $V$ is a $\pm1$ vector of length $k$, then
the ordinary generating function for
skinny grid class
$\Grid(V)$ is given by
$$
G_V(z) \;=\; \frac{1}{z} \+ \sum_{j=1}^k H_{V{[\!\!\!\;[j]\!\!\!\;]}}(z,z)
$$
where, for any $\pm1$ vector $V$, $H_V(x,y)$ is defined iteratively as follows:
\begin{align*}
  H_{(1)}(x,y) &\;=\; H_{(-1)}(x,y) \;=\; \frac{x\+y}{1-x} , \\[6pt]
  H_{V^+}(x,y) &\;=\; \frac{x\+y}{x\+y+x-y} \+ \Big( H_V\big(x,y\big) \:-\: H_V\big(y,y\big) \:-\: H_V\big(x,\tfrac{x}{1-x}\big) \:+\: H_V\big(\tfrac{x}{1-x},\tfrac{x}{1-x}\big)\Big) , \\[6pt]
  H_{V^-}(x,y) &\;=\; \frac{x\+y}{x\+y+x-y} \+ \Big( H_V\big(\tfrac{x}{1-x},x\big) \:-\: H_V\big(y,x\big) \Big) .
\end{align*}
\end{thm}
}

One way of gridding a permutation in a skinny grid class is to process the permutation from left to right, placing points in cells
as far to the left as possible while honouring the cell constraints.
A new cell is used only when the next point ``changes direction''.
We call this process \emph{greedy gridding},
and also call the resulting gridding of a permutation a \emph{greedy gridding}.
See Figure~\ref{figGreedyGridding} for an illustration.
\begin{figure}[ht]
  $$
      \begin{tikzpicture}[scale=0.22]
        \plotperm{15}{2,6,7,8,11,3,9,13,15,14,12,10,5,4,1}
        \drawnxt[3-]{[thick] ( 5.5,.5) -- ( 5.5,15.5) }
        \drawnxt[4-]{[thick] ( 6.5,.5) -- ( 6.5,15.5) }
        \drawnxt[5-]{[thick] ( 9.5,.5) -- ( 9.5,15.5) }
        \drawnxt[6-]{[thick] (10.5,.5) -- (10.5,15.5) }
        \drawnxt[8-]{[thick] (15.25,17) -- (16.25,16) }
        \drawnxt[3-]{[thick] ( 2.5,16) -- ( 3.5,17)}
        \drawnxt[4-]{[thick] ( 5.5,17) -- ( 6.5,16)}
        \drawnxt[5-]{[thick] ( 7.5,16) -- ( 8.5,17)}
        \drawnxt[6-]{[thick] ( 9.5,16) -- (10.5,17)}
        \drawnxt[7-]{[thick] (12.5,17) -- (13.5,16)}
        \drawnxt[8-]{[thick] (15.5,.5) -- (16,.5) -- (16,15.5) -- (15.5,15.5)}
      \end{tikzpicture}
  $$
  \caption{The greedy gridding of a permutation in skinny grid class $\!\!\!$ \skinnyclass{1,-1,1,1,-1,-1}}
  \label{figGreedyGridding}
\end{figure}

Greedy gridding is guaranteed to produce a valid gridding for any permutation in a skinny grid class because
  placing
  a point in a cell as far to the left as possible can never make it harder to place subsequent points.
Since there is exactly one greedy gridding of each permutation in a skinny grid class, we can enumerate skinny grid classes
by counting greedy gridded permutations.

Given a permutation $\pi\in\Grid(V)$, it may be the case that
$\pi$ has points in every cell of $V$
when greedy gridded.
We call such permutations \emph{tightly gridded} and use $\Gridstar(V)$ to denote
the set of tightly gridded permutations in $\Grid(V)$.
For example, the permutation in Figure~\ref{figGreedyGridding} is not in $\Gridstar(\!\skinnyclass{1,-1,1,1,-1,-1}\!)$, but is a member of
$\Gridstar(\!\skinnyclass{1,-1,1,1,-1}\!)$.
Clearly, a skinny grid class is the disjoint union of the tightly gridded permutations in each of its prefixes:
$$
\Grid(V) \;= \biguplus_{1\leqslant j\leqslant \mathrm{len}(V)}\!\! \Gridstar(V[\![j]\!]) .
$$
For example,
$\Grid(\!\skinnyclass{-1,1,1}\!) = \Gridstar(\!\skinnyclass{-1}\!) \+\uplus\+ \Gridstar(\!\skinnyclass{-1,1}\!) \+\uplus\+ \Gridstar(\!\skinnyclass{-1,1,1}\!)$.

To enumerate skinny grid classes, we only need to enumerate tightly gridded permutations.
This we do by using a bivariate generating function in which we parameterise twice by the {position of the last point} in the plot of the permutation. We use $x$ to mark the position of the last point counting from the \emph{bottom} and $y$ to mark its position counting from the \emph{top}:
$$
G^\star_V(x,y)
\;= \sum_{\pi\+\in\+\Gridstar(V)}\!\!x^{\pi(|\pi|)}\+y^{|\pi|+1-\pi(|\pi|)}
\;=\; \sum_{r,s\+\geqslant\+ 1}g^\star_{r,s}\+x^r\+y^s
,
$$
where
$g^\star_{r,s}$ is the number of permutations $\pi$ in $\Gridstar(V)$
of length $r+s-1$
such that 
$\pi(|\pi|)=r$.
For example, the permutation $\mathbf{314652}$ would be represented by a contribution of $x^2\+y^5$.

Clearly, $\Gridstar(V)$ and $\Gridstar(-V)$ are Wilf-equivalent.
It helps to restrict our attention to the case in which the last cell of $V$ is increasing. Thus we define
\begin{equation}\label{eqHv}
H_V(x,y) \;=\;
\begin{cases}
G^\star_V(x,y), & \text{if the last cell of $V$ is increasing,} \\[2pt]
G^\star_V(y,x), & \text{if the last cell of $V$ is decreasing.}
\end{cases}
\end{equation}
Note that
$\frac{1}{z}\+ H_V(z,z)$
is the univariate generating function for $\Gridstar(V)$, and hence the generating function for $\Grid(V)$ is given by
$$
G_V(z) \;=\; \frac{1}{z} \+ \sum_{j=1}^k H_{V{[\!\!\!\;[j]\!\!\!\;]}}(z,z)
$$
as required.

We now need to determine functional equations for $H_V(x,y)$.
To begin with, as the base case, we have
$H_{(1)}(x,y) = H_{(-1)}(x,y) = \frac{xy}{1-x}$ since there is a single increasing permutation of each length in $\Gridstar(1)$.

We now investigate the effect on the generating function of adding points when greedy gridding.

\subsubsection*{Adding a new point}
Let us first consider
how a greedy gridded permutation $\pi\in\Gridstar(V)$ can be extended by adding a single point to its right.
We assume, without loss of generality, that the last cell of $V$ is increasing.
If the new point is added above the last point of $\pi$, then
the extended permutation is also in $\Gridstar(V)$. On the other hand, if the new point is added below the last point of $\pi$,
then it must be placed in
a new cell and the extended permutation is
in $\Gridstar(V^+)$ or $\Gridstar(V^-)$.
See Figure~\ref{figAddPoint} for an illustration.
\begin{figure}[ht]
  $$
      \begin{tikzpicture}[scale=0.3]
      \plotperm{9}{8,7,4,2,3,5,6}
      \drawnxt[1-]{[thick] (4.5,.5) -- (4.5,9.5)}
      \drawnxt[4]{[thick] (8.5,.5) -- (8.5,9.5)}
      {
      \draw[radius=0.3,thick] (9,1.5) circle;
      \draw[radius=0.3,thick] (9,2.5) circle;
      \draw[radius=0.3,thick] (9,3.5) circle;
      \draw[radius=0.3,thick] (9,4.5) circle;
      \draw[radius=0.3,thick] (9,5.5) circle;
      \draw[radius=0.3,thick] (8,6.5) circle;
      \draw[radius=0.3,thick] (8,7.5) circle;
      \draw[radius=0.3,thick] (8,8.5) circle;
      }
      \end{tikzpicture}
  $$
\caption{Possibilities for adding a single point to a 
permutation in $\Gridstar(\!\skinnyclass{-1,1}\!)$}
\label{figAddPoint}
\end{figure}

For each of these three cases
(remaining in $\Gridstar(V)$, expanding to $\Gridstar(V^+)$, and expanding to $\Gridstar(V^-)$)
we define a linear operator ($\Lzero$, $\Lplus$ and $\Lminus$, respectively) acting on $H_V(x,y)$, that reflects in each case the effect of adding a single new point.

For a specific permutation, represented by $x^r\+y^s$, we have
\begin{alignat*}{3}
    \text{(i)} \quad& \Lzero[x^r\+y^s]  &\;=\; \sum_{i=1}^s x^{r+i}\+y^{s+1-i} &\;=\; \frac{x\+y}{x-y}(x^{r+s}-x^r\+y^s) \\[3pt]
   \text{(ii)} \quad& \Lplus[x^r\+y^s]  &\;=\; \sum_{i=1}^r x^{r+1-i}\+y^{s+i} &\;=\; \frac{x\+y}{x-y}(x^r\+y^s-y^{r+s}) \\[3pt]
  \text{(iii)} \quad& \Lminus[x^r\+y^s] &\;=\; \sum_{i=1}^r x^{s+i}\+y^{r+1-i} &\;=\; \frac{x\+y}{x-y}(x^{r+s}-x^s\+y^r).
\end{alignat*}
Thus, the three operators are defined by
\begin{alignat*}{3}
  \Lzero[f(x,y)]  &\;=\; x\+y\+ \frac{f(x,x)-f(x,y)}{x-y}, \\[3pt]
  \Lplus[f(x,y)]  &\;=\; x\+y\+ \frac{f(x,y)-f(y,y)}{x-y}, \\[3pt]
  \Lminus[f(x,y)] &\;=\; x\+y\+ \frac{f(x,x)-f(y,x)}{x-y}.
\end{alignat*}
Note that $\Lminus[f(x,y)]=\Lplus[f(y,x)]$ as expected from
\eqref{eqHv}.

\subsubsection*{Adding a new cell}
Let us now consider how a permutation in
$\Gridstar(V^+)$
can be built by extending a permutation $\pi\in\Gridstar(V)$.
First, we need to add a point below the last point of $\pi$,
and then add zero or more increasing points above this first new point.
See Figure~\ref{figAddCell} for an illustration.
\begin{figure}[ht]
  $$
      \begin{tikzpicture}[scale=0.3]
      \plotperm{9}{8,1,2,5,6}
      \drawnxt[3-]{[thick] ( 2.5,.5) -- ( 2.5,9.5)}
      \drawnxt[3-]{[thick] ( 5.5,.5) -- ( 5.5,9.5)}
      \draw[radius=0.3,thick] (6,3) circle;
      \draw[radius=0.3,thick] (7,4) circle;
      \draw[radius=0.3,thick] (8,7) circle;
      \draw[radius=0.3,thick] (9,9) circle;
      \end{tikzpicture}
  $$
\caption{Building a permutation in
$\Gridstar(\!\skinnyclass{-1,1,1}\!)$
from one in
$\Gridstar(\!\skinnyclass{-1,1}\!)$}
\label{figAddCell}
\end{figure}

Thus, in this case, adding a new cell is reflected by a single application of $\Lplus$ followed by zero or more applications of $\Lzero$.
Analogously, building a permutation in
$\Gridstar(V^-)$
from one in
$\Gridstar(V)$ is reflected by a single application of $\Lminus$ followed by zero or more applications of $\Lzero$.

Let $\Lstar$ be the linear operator that reflects the action of adding zero or more increasing points (i.e.~it is equivalent to zero or more applications of $\Lzero$). Then $\Lstar$ satisfies the following functional equation:
\begin{equation*}
  \Lstar\big[f(x,y)\big] \;=\; f(x,y) \:+\: \Lzero\big[\Lstar\big[f(x,y)\big]\big].
\end{equation*}
This can be solved for $\Lstar$ by using the kernel method
(see Section~\ref{sectKernelMethod}).
Expanding for $\Lzero$ and rearranging 
gives
\begin{equation}\label{eqOmegaStarKernel}
  (x\+y+x-y)\+ \Lstar[f(x,y)] \;=\; (x-y)f(x,y) \:+\: x\+ y\+ \Lstar[f(x,x)].
\end{equation}
Cancelling the kernel, $(x\+y+x-y)$, by setting $y=\frac{x}{1-x}$, then yields
$$
\Lstar[f(x,x)] \;=\; f(x,\tfrac{x}{1-x}),
$$
which we can use to substitute for $\Lstar[f(x,x)]$ in~\eqref{eqOmegaStarKernel}, which we then solve for $\Lstar$:
$$
\Lstar[f(x,y)]\;=\;\frac{x\+y\+ f(x,\frac{x}{1-x})\,+\,(x-y)f(x,y)}{x\+y+x-y}.
$$

We now have all we need to express the relationship between $H_{V^+}(x,y)$, $H_{V^-}(x,y)$ and $H_V(x,y)$:
\begin{equation*}
\begin{array}{rcl}
  H_{V^+}(x,y)&\!\!=\!\!&\Lstar\big[\Lplus\big[H_V(x,y)\big]\big] \\[10pt]
  H_{V^-}(x,y)&\!\!=\!\!&\Lstar\big[\Lminus\big[H_V(x,y)\big]\big].
\end{array}
\end{equation*}
Expansion using the definitions of the linear operators then completes the proof of Theorem~\ref{thmSkinny}.

The resulting generating functions for a few small skinny grid classes are listed in Table~\ref{tabSkinny}.
\begin{table}
\renewcommand{\arraystretch}{2.25}
\centering
\begin{tabular}{cc}
\skinnyclass{1} &       $\dfrac{z}{1-z}$ \\[6pt]
\hline
\skinnyclass{-1,1} &    $\dfrac{z}{1-2\+z}$ \\[6pt]
\hline
\skinnyclass{1,1} &     $\dfrac{z-2\+z^2+2\+z^3}{(1-z)^2\+(1-2\+z)}$ \\[6pt]
\hline
\skinnyclass{1,-1,1} &  $\dfrac{z-3\+z^2+3\+z^3}{(1-z)^2\+(1-3\+z)}$ \\[6pt]
\hline
\skinnyclass{-1,1,1} &  $\dfrac{z-6\+z^2+13\+z^3-9\+z^4}{(1-z)\+(1-2\+z)^2\+(1-3\+z)}$ \\[6pt]
\hline
\skinnyclass{1,1,1} &   $\dfrac{z-8\+z^2+26\+z^3-39\+z^4+30\+z^5-12\+z^6}{(1-z)^3\+(1-2\+z)^2\+(1-3\+z)}$ \\[6pt]
\end{tabular}
\caption{Generating functions for small skinny grid classes}\label{tabSkinny}
\end{table}

Based on an inspection of the results for vectors whose entries are are all ones,
we conclude our considerations of skinny grid classes
by proposing the following conjecture concerning permutations with no more than $k-1$ descents:
\begin{conj}
The ordinary generating function for the $k\times 1$ skinny grid class $\Grid(1,1,\dots,1)$ is given by
$$
G_{(\underbrace{1,1,\,\dots\,,1}_k)}(z) \;=\; -1 \:+\: \sum_{r=1}^k \, \frac{1}{1-r\+ z} \left( \frac{r\+ z}{r\+ z-1} \right)^{k-r}.
$$
\end{conj}

\subsection*{Beyond skinny grid classes}
Individual non-skinny monotone grid classes have been enumerated using various \emph{ad hoc} methods.
Nonetheless,
finding a \emph{general} procedure for
the exact enumeration of monotone grid classes remains an open problem.
The primary challenge is that, whereas
each permutation in a skinny grid class has
a unique greedy gridding,
there is no apparent way to choose a canonical gridding for permutations in an arbitrary grid class.
Futhermore, in general, the generating function for such a class may not be rational or even algebraic.

There is, however, one family of permutation classes, closely related to grid classes, that \emph{has} recently been enumerated.
In~\cite{HV2013+}, Hom\-berger \& Vatter present a way of exactly enumerating any \emph{polynomial} permutation class.\label{defPolynomialClass}
A class $\CCC$ is said to be polynomial if, for all sufficiently large $n$, $|\CCC_n|$ is given by a polynomial. Equivalently, $\CCC$ is polynomial if $\gr(\CCC)\in\{0,1\}$.

\label{pegPermGridClasses}
Each polynomial class can 
be expressed as the finite union of what Homberger \& Vatter call \emph{peg permutation grid classes}.
Such a class
can be defined by a matrix
whose entries are drawn from $\{0,1,-1,\bullet\}$ and in which each row and each column contains exactly one non-zero entry.
As with monotone grid classes,
the entries in the matrix specify the permitted pattern of
points in the corresponding cell. If the entry is $\bullet$, the cell must contain only a single point or remain empty.
An example of a peg permutation grid class is
$$
\setgcptgridscale{2}
\setgcptsize{0.3}
\av(321,2134)
\;=\;
\gcseven[-1,-1,-1,-1,-1,-1,-1,-1,-1,13,-1,-1,-1,9]{7}
  {0,0,0,0,0,0,0}
  {0,0,1,0,0,0,0}
  {0,0,0,0,0,0,0}
  {0,1,0,0,0,0,0}
  {0,0,0,0,0,1,0}
  {0,0,0,1,0,0,0}
  {1,0,0,0,0,0,0}
  ,
$$
an identity first proved by
Atkinson in~\cite{Atkinson1999}.

Homberger \& Vatter present an effective method for enumerating any peg permutation grid class. It may be possible to extend the techniques they use so as to enable the enumeration of some non-polynomial non-skinny monotone grid classes.

In the next chapter, we turn to an investigation of how we can enumerate classes of \mbox{$M$-gridded} permutations, and explore
how the enumeration of gridded permutations may be of use in enumerating monotone grid classes.

\cleardoublepage


\newcommand{\gdn}{\mathbin{\raisebox{.99ex}{$\hspace{.25pt}{}_{{}_{\diagdown\mkern-14mu\diagdown}}\hspace{-1.3pt}$}}}
\newcommand{\gupi}{\mathbin{\raisebox{.91ex}{$\hspace{.25pt}{}_{{}_{\diagup\mkern-14mu\diagup}}\hspace{-1.3pt}$}}}
\newcommand{\gup}{\mathbin{\ooalign{$\phantom{\gdn}$\cr$\gupi$}}}
\newcommand{\gxx}{\mathbin{\ooalign{$\gdn$\cr$\gupi$}}}
\newcommand{\bgdn}{\mathbin{\raisebox{.55ex}{$\hspace{.38pt}{}_{\diagdown\mkern-15mu\diagdown}\hspace{-.15pt}$}}}
\newcommand{\bgupi}{\mathbin{\raisebox{.55ex}{$\hspace{.38pt}{}_{\diagup\mkern-15mu\diagup}\hspace{-.15pt}$}}}
\newcommand{\bgup}{\mathbin{\ooalign{$\phantom{\bgdn}$\cr$\bgupi$}}}
\newcommand{\bgxx}{\mathbin{\ooalign{$\bgdn$\cr$\bgupi$}}}
\newcommand{\hgdn}{\mathbin{\hspace{.57pt}\diagdown\hspace{.43pt}}}
\newcommand{\hgupi}{\mathbin{\hspace{.57pt}\diagup\hspace{.43pt}}}
\newcommand{\hgup}{\mathbin{\ooalign{$\phantom{\hgdn}$\cr$\hgupi$}}}
\newcommand{\hgxx}{\mathbin{\ooalign{$\hgdn$\cr$\hgupi$}}}
\newcommand{\xX}{\hspace{-.6pt}\times\hspace{-.5pt}}
\newcommand{\oox}{    &     & \xX}
\newcommand{\oxo}{    & \xX &    }
\newcommand{\oxx}{    & \xX & \xX}
\newcommand{\xoo}{\xX &     &    }
\newcommand{\xox}{\xX &     & \xX}
\newcommand{\xxo}{\xX & \xX &    }
\newcommand{\xxx}{\xX & \xX & \xX}
\newcommand{\nl}{\\ \hline}
\newenvironment{garray}[1]  {\def\arraystretch{0}    \begin{array}{|*#1{@{}c@{}|}}\hline} {\nl\end{array}}
\newenvironment{bgarray}[1] {\def\arraystretch{0.21} \begin{array}{|*#1{@{}c@{}|}}\hline} {\nl\end{array}}
\newenvironment{hgarray}[1] {\def\arraystretch{0.67} \begin{array}{|*#1{@{}c@{}|}}\hline} {\nl\end{array}}
\newenvironment{xarray}[1]  {\def\arraystretch{0.25} \begin{array}{|*#1{@{}c@{}|}}\hline} {\nl\end{array}}

\chapter{Gridded permutations}\label{chap04}
\label{chapGriddedPerms}

Success in finding general methods for the exact enumeration of permutation grid classes
is currently limited to skinny and polynomial classes.
However,
it is possible to exactly enumerate a somewhat broader family of
classes of \emph{gridded permutations}.

\begin{figure}[ht]
$$
    \begin{tikzpicture}[scale=0.18,rotate=270]
      \plotperm{24}{16,2,5,18,15,14,6,13,12,24,8,11,9,10, 7,23,22,4,3,21,20,19,1,17}
      \draw[] (14.5,  .5) -- (14.5,24.5);
      \draw[] (  .5, 9.5) -- (24.5, 9.5);
      \draw[] (  .5,15.5) -- (24.5,15.5);
    \end{tikzpicture}
$$
\caption{A $\!\gctwo{3}{-1,1,-1}{1,0,1}\!$-gridded permutation}
\end{figure}

In this chapter we present a procedure for determining the generating function
for any class of gridded permutations whose row-column graph has no more than one cycle in any connected component.
We show that, if $G(M)$ is acyclic, then the generating function of $\Gridhash(M)$
is rational. This actually follows from the proof by Albert, Atkinson, Bouvel, Ru\v{s}kuc \& Vatter~\cite{AABRV2011}
that acyclic grid classes have rational generating functions.
Moreover, we prove that the generating function of a unicyclic class of gridded permutations is always algebraic.
We also prove that, for any $M$, the generating function of $M$-gridded permutations is D-finite.
Finally, we explore how the enumeration of gridded permutations can help when trying to enumerate grid classes.

It is possible to give an explicit expression for the number of
gridded permutations of length $k$ in any specified grid class.
We do this by summing over the number of configurations with a specified number of points in each cell.

\newpage 
To express the result, we make use of
\emph{multinomial coefficients},
with their normal combinatorial interpretation,
for which we
use the standard notation
$$
\qquad\qquad\qquad\qquad
\binom{n}{k_1,k_2,\ldots,k_r}
\;=\; \frac{n!}{k_1!\+k_2!\+\ldots \+k_r!},
\qquad\quad
\text{where~}
\sum_{i=1}^r k_i =n,
$$
to denote
the number of ways of distributing $n$ distinguishable objects between $r$ (distinguishable) bins,
such that bin $i$ contains exactly $k_i$ objects ($1\leqslant i\leqslant r$).\label{MultinomialCoeffs}

\begin{lemma}\label{lemmaCountGriddings1}
If $M$ has dimensions $r\times s$, then
the number of gridded permutations of length $k$ in $\Gridhash(M)$ is
given by
$$
\big|\Gridhash_k(M)\big|
\;=\;
\raisebox{-4.5pt}{\fontsize{24.88pt}{0pt}\selectfont $\sum$}
\prod_{i=1}^r \binom{k_{i,1}+k_{i,2}+\ldots+k_{i,s}}{k_{i,1},k_{i,2},\ldots,k_{i,s}}
\prod_{j=1}^s \binom{k_{1,j}+k_{2,j}+\ldots+k_{r,j}}{k_{1,j},k_{2,j},\ldots,k_{r,j}} ,
$$
where the sum is over all combinations of non-negative $k_{i,j}$ ($1\leqslant i\leqslant r$, $1\leqslant j\leqslant s$) such that $\sum\limits_{i,j}k_{i,j}=k$ and $k_{i,j}=0$ if $M_{i,j}=0$.
\end{lemma}
\begin{proof}
An $M$-gridded permutation consists of a number of points in each of the cells that correspond to a non-zero entry of $M$. For every permutation, the relative ordering of points (increasing or decreasing) within a particular cell is fixed by the value of the corresponding matrix entry.
However, the relative interleaving between points in distinct cells in the same row or column
can be chosen arbitrarily and independently for each row and column.

Each term in the sum is thus
the number of
gridded permutations in which there are $k_{i,j}$ points in the cell corresponding to $M_{i,j}$,
the terms in the first product representing the number of ways of interleaving points in each column,
and terms in the second product representing the number of ways of interleaving points in each row.
The result follows by summing over configurations with a total of $k$ points and no points in cells that correspond to a zero entry of $M$.
\end{proof}
As a consequence of this result, it is clear that
$\big|\Gridhash_k(M)\big|=\big|\Gridhash_k(M')\big|$
whenever $M'$ results from
permuting the rows and/or columns of $M$.
Indeed, it is not hard to see that the enumeration of a class of gridded
permutations depends only on its row-column graph, a fact that we prove formally later (Corollary~\ref{corGRHashEqForSameRCGraph}).

We can ``translate'' Lemma~\ref{lemmaCountGriddings1}
to give us an expression for the generating function of any class of gridded permutations.
Note that, when enumerating gridded permutations, it turns out to be simpler if we include the zero-length permutation.
However, when enumerating grid classes, we chose to exclude the zero-length permutation.

Central to this and following results is the \emph{diagonalisation} of generating functions.
Given a bivariate power series $f(x,y)=\sum a_{r,s} x^r y^s$, the \emph{diagonal}, $\Delta(f)$, of $f$ is the univariate series defined by $\Delta(f)(z)=\sum a_{n,n} z^n$.
Equivalently, if $q(x,y)$ consists of the sum of just those terms of $f(x,y)$ for which the exponent of $x$ is the same as that of $y$, then we have $\Delta(f)(z)=q(\sqrt{z},\sqrt{z})$. Alternatively, $\Delta(f)(z)=\big[w^0\big]f(w\sqrt{z},\sqrt{z}/w)$.

\thmbox{
\begin{lemma}\label{lemmaGridHashGF}
The generating function for $\Gridhash(M)$ is given by
$$
G^\#_M(z)\;=\;
\Big[\!\prod_{\substack{i,j\\
                        M_{i,j}\neq0}} \!\! {x_{i,j}}^0\Big]
\prod_i \Big(1-\sqrt{z}\!\sum_{\substack{j\\
                        M_{i,j}\neq0}} \!\! x_{i,j} \Big)^{-1} \+
\prod_j \Big(1-\sqrt{z}\!\sum_{\substack{i\\
                        M_{i,j}\neq0}} \!\! x_{i,j}\!{}^{-1} \Big)^{-1} ,
$$
in which there is one variable $x_{i,j}$ for each non-zero cell of $M$, and we extract only the constant terms as far as the $x_{i,j}$ are concerned.
\end{lemma}
} 

\begin{proof}
Suppose that, for each non-zero entry $M_{i,j}$, we let the variable $u_{i,j}$ mark the number of points in the cell corresponding to $M_{i,j}$.
Then the multivariate generating function for the
ways in which the points in column $i$ can be interleaved is
$$
\frac{1}{1-(u_{i,j_1}+u_{i,j_2}+\ldots+u_{i,j_{t}})}
\;=\;
\Big(1-\sum_{\substack{j\\
                        M_{i,j}\neq0}} \!\! u_{i,j} \Big)^{-1} ,
$$
where $j_1, j_2, \ldots, j_{t}$ are the values of $j$ for which $M_{i,j}\neq0$.

Now, let us \emph{also} use $v_{i,j}$ to mark the number of points in the cell corresponding to non-zero entry $M_{i,j}$. If we use the $v_{i,j}$, rather than the $u_{i,j}$, when considering the interleaving of points in the same row, then the multivariate generating function for $M$-gridded permutations is given by those terms in
\begin{equation}\label{eqGriddedGFDiagonal}
\prod_i \Big(1-\!\sum_{\substack{j\\
                        M_{i,j}\neq0}} \!\! u_{i,j} \Big)^{-1} \+
\prod_j \Big(1-\!\sum_{\substack{i\\
                        M_{i,j}\neq0}} \!\! v_{i,j} \Big)^{-1}
\end{equation}
for which the exponent of $u_{i,j}$ is the same as that of $v_{i,j}$ for all $i$ and $j$.

The result follows
by replacing each $u_{i,j}$ with $x_{i,j}\sqrt{z}$ and replacing each $v_{i,j}$ with $\sqrt{z}/x_{i,j}$, so that $z$ marks the length of the permutation (which is the total number of points).
Using $\sqrt{z}$ in this way ensures that half of each point is counted when considering the interleaving of points in a column, and another half of each point is counted when considering the interleaving of points in a row.
After this substitution, the terms required are those
for which the exponent of each of $x_{i,j}$ is zero.
\end{proof}

In general, coefficient extraction is hard, so
Lemma~\ref{lemmaGridHashGF} does not give us an effective procedure for determining explicit expressions for the generating functions of arbitrary classes of gridded permutations.
Nevertheless, it can be used to yield closed forms for the generating functions when the row-column graphs of the classes are acyclic or unicyclic.


\section{Acyclic classes}\label{sectAcyclicGridhash}
We begin with acyclic classes.
First, we describe a process for adding cells to a grid class that can be used repeatedly to construct any class whose row-column graph is a forest.

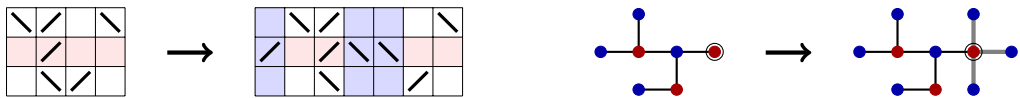
\begin{figure}[ht]
  $$
  \setgcscale{0.39}
  \setgcpreextra{
    \path [fill=red!10] (0,1) rectangle (4,2);
  }
  \gcthree{4}{-1, 1,0,-1}
             { 0, 1,0, 0}
             { 0,-1,1, 0}
  \;\;\;\;
  \raisebox{8pt}{\begin{tikzpicture}\draw [->,ultra thick](0,0)--(.6,0);\end{tikzpicture}}
  \;\;\;\;
  \setgcpreextra{
    \path [fill=red!10]  (0,1) rectangle (7,2);
    \path [fill=blue!15] (0,0) rectangle (1,3);
    \path [fill=blue!15] (3,0) rectangle (5,3);
  }
  \gcthree{7}{0,-1, 1, 0, 0,0,-1}
             {1, 0, 1,-1,-1,0, 0}
             {0, 0,-1, 0, 0,1, 0}
  \qquad
  \qquad
  \raisebox{-7.5pt}
  {
  \begin{tikzpicture}[scale=0.5]
    \draw [thick] (0,1)--(3,1);
    \draw [thick] (1,0)--(2,0);
    \draw [thick] (1,1)--(1,2);
    \draw [thick] (2,0)--(2,1);
    \foreach \x in {0,2}
      \draw [black!35!blue,fill] (\x,1) circle [radius=0.15];
    \foreach \x in {1,3}
      \draw [black!35!red,fill] (\x,1) circle [radius=0.15];
    \draw [black!35!blue,fill] (1,2) circle [radius=0.15];
    \draw [black!35!blue,fill] (1,0) circle [radius=0.15];
    \draw [black!35!red,fill] (2,0) circle [radius=0.15];
    \draw[radius=0.225] (3,1) circle;
  \end{tikzpicture}
  }
  \;\;\;\;
  \raisebox{8pt}{\begin{tikzpicture}\draw [->,ultra thick](0,0)--(.6,0);\end{tikzpicture}}
  \;\;\;\;
  \raisebox{-7.5pt}
  {
  \begin{tikzpicture}[scale=0.5]
    \draw [thick] (0,1)--(3,1);
    \draw [thick] (1,0)--(2,0);
    \draw [thick] (1,1)--(1,2);
    \draw [thick] (2,0)--(2,1);
    \draw [ultra thick,gray] (3,1)--(4,1);
    \draw [ultra thick,gray] (3,0)--(3,2);
    \foreach \x in {0,2,4}
      \draw [black!35!blue,fill] (\x,1) circle [radius=0.15];
    \foreach \x in {1,3}
    {
      \draw [black!35!blue,fill] (\x,2) circle [radius=0.15];
      \draw [black!35!red,fill] (\x,1) circle [radius=0.15];
      \draw [black!35!blue,fill] (\x,0) circle [radius=0.15];
    }
    \draw [black!35!red,fill] (2,0) circle [radius=0.15];
    \draw[radius=0.225] (3,1) circle;
  \end{tikzpicture}
  }
  $$
  \caption{Extending an acyclic grid class by inserting three new columns, each with a single non-zero entry in the second row, and the corresponding extension of the row-column graph}
  \label{figExtendingTree}
\end{figure}
In the context of graphs, it is clear that any forest can be constructed by
starting with
a 1-regular graph (a number of disconnected edges) and then
repeatedly attaching some pendant edges to a leaf vertex of the current graph.
The analogous method for grid classes consists of starting with a class
in which each row and each column contains exactly one non-zero entry, and then repeatedly applying the following extension process:
Choose a row or column that contains a single non-zero entry and form a new class by inserting additional columns or rows respectively, each containing a single non-zero entry in the chosen row or column.
This grid class method corresponds exactly to the graph method applied to row-column graphs.
See Figure~\ref{figExtendingTree} for an illustration of this extension process.

The following theorem enables us to
use this extension process to enumerate any class of
$M$-gridded permutations when $G(M)$ is a forest. For clarity, its statement only covers the case in which new columns are added to the class.

\thmbox{
\begin{thm}\label{thmAcyclicGriddingsGF}
Suppose $G(M)$ is acyclic and that $M$ has a single non-zero entry in some row $r$.
Let $M^+$ be a matrix formed from $M$ by inserting $k$ additional columns, each containing a single non-zero entry in row $r$.

If $G^\#_M(z,x)$ is the bivariate generating function for $\Gridhash(M)$, where $z$ marks length and $x$ marks the number of points in the cell corresponding to the non-zero entry of $M$ in row $r$,
then there exist polynomials $P$ and $Q$ such that
$$
G^\#_M(z,x) \;=\; \frac{1}{1+z\+(P(z)+x\+Q(z))}.
$$
Moreover, if $G^\#_{M^+}(z,x_1,\dots,x_k)$ is the multivariate generating function for $\Gridhash(M^+)$ where $x_1,\dots,x_k$ mark the number of points in each of the $k$ new cells in row $r$, then
$$
G^\#_{M^+}(z,x_1,\dots,x_k) \;=\; \frac{1}{1+z\+\big(P(z)+Q(z)-(x_1+\ldots+x_k)\+(1+z\+P(z)) \big)}.
$$
\end{thm}
} 

In proving this theorem, we make use of the following standard diagonalisation result concerning the extraction of coefficients from generating functions:
\begin{lemma}[{Stanley~\cite[Section 6.3]{Stanley1999}}; see Furstenberg~\cite{Furstenberg1967}]\label{lemmaDiagonalGF}
  If $g(x)=g(z,x)$ is a formal Laurent series, then the constant term $[x^0]g(x)$ is given by the sum of the residues of $x^{-1}g(x)$ at those poles $\alpha$ of $g(x)$ for which $\lim\limits_{z\rightarrow0}\alpha(z)=0$.
\end{lemma}

\begin{proof}[Proof of Theorem~\ref{thmAcyclicGriddingsGF}]
For the base case, if each row and each column of $M$ contains exactly one non-zero entry,
then
$$
G^\#_M(z,x) \;=\; \frac{1}{(1-z\+x)(1-z)^{m-1}},
$$
where $m$ is the number of non-zero entries in $M$.
This can certainly be expressed in the form specified in the statement of the theorem.

We now consider the extension process. The argument is analogous to that used in the proof of Lemma~\ref{lemmaGridHashGF}.
The multivariate generating function for the ways in which the points in row $r$ of $\Gridhash(M^+)$ can be interleaved is
$$
R(z,x) \;=\;
R(z,x,x_1,\dots,x_k) \;=\;
\frac{1}{1-z\+(x+x_1+\ldots+x_k)}
.$$

Therefore,
$$
G^\#_{M^+}(z,x_1,\dots,x_k)
\;=\;
\big[x^0\big]G^\#_M\big(z,\tfrac{x}{\sqrt{z}}\big)R\big(z,\tfrac{1}{x\sqrt{z}}\big)
\;=\;
\big[x^0\big]\frac{G^\#_M(z,x/\sqrt{z})}{1-x^{-1}\sqrt{z}-z\+s},
$$
where $s=x_1+\ldots+x_k$.

If we assume that
$$
G^\#_M(z,x) \;=\; \frac{1}{1+z\+(P(z)+x\+Q(z))},
$$
then
$$
G^\#_{M^+}(z,x_1,\dots,x_k)
\;=\;
\big[x^0\big]\frac{x}{\big(1+z\+P(z)+x\+\sqrt{z}\+Q(z)\big)\+\big(x\+(1-z\+s)-\sqrt{z}\big)}.
$$

We now apply Lemma~\ref{lemmaDiagonalGF}.
The expression at the right has
two poles.
The root of the first factor of the denominator
diverges at $z=0$.
The other pole,
$\alpha=\sqrt{z}/(1-z\+s)$, satisfies the necessary criterion.

Since $\alpha$ is a simple pole,
$$
G^\#_{M^+}(z,x_1,\dots,x_k)
\;=\;
\frac{1}{\big(1+z\+P(z)+\alpha\+\sqrt{z}\+Q(z)\big)\+\big(1-z\+s\big)},
$$
which simplifies to yield the desired result.

Since every acyclic class can be constructed by multiple applications of the extension process,
and the generating function resulting from the extension process can also be expressed in the required form, the result holds.
\end{proof}

\section{Unicyclic classes}\label{sectUnicyclicGridhash}
We call a graph unicyclic if it is not acyclic and no connected component contains more than one cycle.
A permutation class is unicyclic if its row-column graph is unicyclic.

\begin{figure}[ht]
  $$
  \setgcscale{0.39}
  \setgcpreextra{
    \path [fill=black!15] (1,1) rectangle (3,3);
  }
  \gcthree{5}{-1, 0,1, 1,0}
             { 0, 1,0, 0,0}
             { 0,-1,0,-1,1}
  \;\;\;\;
  \raisebox{8pt}{\begin{tikzpicture}\draw [->,ultra thick](0,0)--(.6,0);\end{tikzpicture}}
  \;\;\;\;
  \setgcpreextra{
    \path [fill=black!15] (1,1) rectangle (2,2);
  }
  \gctwo{4}{-1, 1, 1,0}
           { 0,-1,-1,1}
  \qquad
  \qquad
  \qquad
  \raisebox{-7.5pt}
  {
  \begin{tikzpicture}[scale=0.5]
    \draw [thick] (1,0)--(3,0);
    \draw [thick] (1,2)--(1,1)--(2,1)--(2,0);
    \draw [ultra thick,gray] (0,0)--(1,0);
    \draw [ultra thick,gray] (0,1)--(1,1);
    \foreach \x in {0,2}
    {
      \draw [black!35!red,fill] (\x,0) circle [radius=0.15];
      \draw [black!35!blue,fill] (\x,1) circle [radius=0.15];
    }
    \foreach \x in {1,3}
      \draw [black!35!blue,fill] (\x,0) circle [radius=0.15];
    \draw [black!35!blue,fill] (1,2) circle [radius=0.15];
    \draw [black!35!red,fill] (1,1) circle [radius=0.15];
  \end{tikzpicture}
  }
  \;\;\;\;
  \raisebox{8pt}{\begin{tikzpicture}\draw [->,ultra thick](0,0)--(.6,0);\end{tikzpicture}}
  \;\;\;\;
  \raisebox{-7.5pt}
  {
  \begin{tikzpicture}[scale=0.5]
    \draw [thick] (1,0)--(3,0);
    \draw [thick] (1,2)--(1,1)--(2,1)--(2,0);
    \draw [ultra thick,gray] (1,0)--(1,1);
    \foreach \x in {2}
    {
      \draw [black!35!red,fill] (\x,0) circle [radius=0.15];
      \draw [black!35!blue,fill] (\x,1) circle [radius=0.15];
    }
    \foreach \x in {1,3}
      \draw [black!35!blue,fill] (\x,0) circle [radius=0.15];
    \draw [black!35!blue,fill] (1,2) circle [radius=0.15];
    \draw [black!35!red,fill] (1,1) circle [radius=0.15];
  \end{tikzpicture}
  }
  $$
  \caption{Constructing a unicyclic grid class by identifying two cells of an acyclic class, and the corresponding 
  operation on the row-column graph}
  \label{figCreateUnicyclic}
\end{figure}
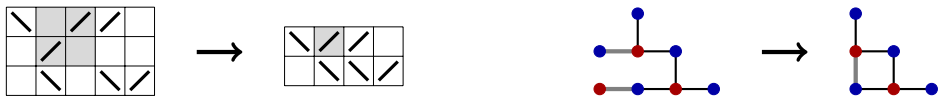
In the context of graphs, a connected unicyclic graph can be constructed
from a tree
by identifying two pendant edges and their endvertices, oriented such that the two leaf vertices are not identified.
We construct a unicyclic grid class in an analogous way.
We begin with
an acyclic class whose matrix $M$
contains either
$\!\gctwo{2}{0,1}{1}\!$ or
$\!\gctwo{2}{-1}{0,-1}\!$ as a submatrix $M'$,
such that one of the two non-zero cells in $M'$ is the only non-zero cell in its column in $M$,
and the other non-zero cell in $M'$ is the only non-zero cell in its row in $M$.
A unicyclic class is then constructed by combining the two columns of $M$ that contain $M'$ keeping all the non-zero entries,
and similarly combining the two rows of $M$ that contain $M'$.
This corresponds to identifying the (pendant) edges of $G(M)$ that correspond to the two non-zero entries of $M'$.
See Figure~\ref{figCreateUnicyclic} for an illustration of this operation.

The following theorem enables us to use this operation to enumerate any unicyclic class of gridded permutations.

\thmbox{
\begin{thm}\label{thmUnicyclicGriddingsGF}
Suppose $G(M)$ is acyclic and that $e_1$ and $e_2$ are two pendant edges of $G(M)$ that correspond to the non-zero entries of a
$\!\gctwo{2}{0,1}{1}\!$ or
$\!\gctwo{2}{-1}{0,-1}\!$ submatrix of $M$.
Let $M^\circ$ be the unicyclic matrix that results from combining the columns and rows of $M$ containing the cells corresponding to $e_1$ and $e_2$.

If $G^\#_M(z,u,v)$ is the trivariate generating function for $\Gridhash(M)$, where $z$ marks length, and $u$ and $v$ mark the number of points in the cells corresponding to $e_1$ and $e_2$,
then there exist polynomials $R$, $S$, $T$ and $U$ such that
$$
G^\#_M(z,u,v) \;=\; \frac{1}{1+z\+(R(z)+u\+S(z)+v\+T(z)+u\+v\+z\+U(z))}.
$$
Moreover, the generating function of $\Gridhash(M^\circ)$ is then given by
$$
G^\#_{M^\circ}(z)
\;=\;
\frac{1}{\sqrt{(1+z\+R(z)+z\+U(z))^2-4\+z\+S(z)\+T(z)}}.
$$
\end{thm}
} 

We have as an immediate consequence:

\thmbox{
\begin{thm}\label{thmUnicyclicAlgebraic}
  The generating function of any unicyclic class of gridded permutations is algebraic.
\end{thm}
} 

\begin{proof}[Proof of Theorem~\ref{thmUnicyclicGriddingsGF}]
  The first part follows from the fact that the required form is preserved by the extension process of Theorem~\ref{thmAcyclicGriddingsGF}.

  By construction, $G^\#_{M^\circ}(z)$ is given by those terms of $G^\#_M(z,u,v)$ for which $u$ and $v$ have the same exponent. Thus, if we assume that $G^\#_M(z,u,v)$ has the form specified in the statement of the theorem,
  $$
  G^\#_{M^\circ}(z)
  \;=\;
  \big[w^0\big]G^\#_M\big(z,\tfrac{w}{\sqrt{z}},\tfrac{1}{w\sqrt{z}}\big)
  \;=\;
  \big[w^0\big]\frac{w}{w^2\+\sqrt{z}\+S(z) + w\+(1 + z\+R(z)+z\+U(z)) + \sqrt{z}\+T(z)} .
  $$

We now apply Lemma~\ref{lemmaDiagonalGF}.
The expression at the right has
two poles.
One
diverges at $z=0$.
The other,
$$
\alpha \;=\;
\frac{-1-z\+R(z) -z\+ U(z)+\sqrt{(1+z\+R(z) +z\+ U(z))^2-4\+ z\+ S(z)\+ T(z)}}{2\+ \sqrt{z}\+ S(z)},
$$
satisfies the necessary criterion.
Algebraic manipulation then yields the required result.
\end{proof}

\section{Exploitation}\label{sectExploitGridhash}
Theorems~\ref{thmAcyclicGriddingsGF} and~\ref{thmUnicyclicGriddingsGF} can be used to help in enumerating acyclic and unicyclic monotone grid classes.
One possible approach to the enumeration of a permutation grid class is to partition it into parts
in such a way that each permutation in the class has a unique gridding in exactly one of the parts.
The parts are then essentially sets of gridded permutations.

\subsection*{Acyclic classes}
For example, it is not difficult to confirm that
the acyclic grid class
$\Grid\big(\!\gctwo{2}{1,1}{-1}\!\big)$ can be partitioned as follows:
$$
\setgcscale{0.3}
\gctwo{2}{1,1}{-1}
\;=\;      \gctwo[1]{1}{1}{-1}
\;\uplus\; \gcfive[2,4,3]{4}{0,1,0,1}{0,0,1}{1}{-1}{0,-1}
\;\uplus\; \gcfour[3,1,2]{4}{1,0,0,1}{0,0,1}{}{0,-1}
\;\uplus\; \gcsix[4,2,5,3]{6}{0,0,1,0,0,1}{1,0,0,0,1}{0,0,0,1}{}{0,-1}{0,0,-1} .
$$
We call diagrams like these \emph{entanglement diagrams}.
They are similar to the peg permutation grid class diagrams described
in the previous chapter
(see page~\pageref{pegPermGridClasses}).
However, in an entanglement diagram, disks ($\bullet$) are placed on the intersections of grid lines, rather than in the cells, and they specify that any permutation in the set denoted by the diagram \emph{must} have a (single) point at that location.
Entanglement diagrams thus do not represent permutation classes closed under containment.
With a slight abuse of terminology,
we also use the term \emph{entanglement diagram} to
refer to the set of permutations represented by such a diagram.

Let us call a partition of a grid class into disjoint entanglement diagrams
such that each permutation has a unique gridding a \emph{proper partition} of the class.
If we have a proper partition of a grid class,
then the enumeration of the grid class can be achieved simply by
counting gridded permutations in each of the diagrams.
So, in our example, we have
$$
G_{\!\gctwo{2}{1,1}{-1}\!\!}(z)
  \;=\;
  z\+G^\#_{\!\gctwo{1}{1}{-1}\!\!}(z) \:+\:
  z^3\+G^\#_{\!\gcfive{4}{0,1,0,1}{0,0,1}{1}{-1}{0,-1}\!\!}(z) \:+\:
  z^3\+G^\#_{\!\gcthree{4}{1,0,0,1}{0,0,1}{0,-1}\!\!}(z) \:+\:
  z^4\+G^\#_{\!\gcfive{6}{0,0,1,0,0,1}{1,0,0,0,1}{0,0,0,1}{0,-1}{0,0,-1}\!\!}(z) ,
$$
which, after numerous applications of Theorem~\ref{thmAcyclicGriddingsGF}, yields
$$
G_{\!\gctwo{2}{1,1}{-1}\!\!}(z)
  \;=\;
\frac{1-5\+ z+8\+ z^2-3\+ z^3}{(1-z)\+ (1-2\+ z) \+(1-3\+ z+z^2)} .
$$

It is elementary to determine a general technique for the proper
partitioning of \emph{skinny} grid classes into entanglement diagrams.
This yields an alternative approach to their enumeration to that in the previous chapter.
We leave the details as an exercise for the reader.
However, attempts at discovering an effective procedure for proper partitioning applicable to all acyclic grid classes
have so far been unsuccessful.
However, it does seem reasonable to assume that every acyclic grid class can be enumerated in this way:

\thmbox{
\begin{conj}
  Every acyclic monotone grid class can be properly partitioned into a finite number of acyclic entanglement diagrams.
\end{conj}
} 

\newpage
It is to be hoped that some effective procedure can be discovered for the partitioning so that it can be used to enumerate all unicyclic grid classes.

\subsection*{Unicyclic classes}
When we consider unicyclic classes, the situation seems to be somewhat less straightforward.
For example, it would appear that the ``double chevron'' grid class $\Grid\big(\!\gctwo{2}{1,1}{-1,-1}\!\big)$ can be properly partitioned as follows:
$$
\setgcscale{0.3}
\gctwo{2}{1,1}{-1,-1}
\;=\;      \gctwo[1]{1}{1}{-1}
\;\uplus\; 
           \gcsix[2,5,4]{4}{0,1,0,1}{0,0,1}{1}{0,0,-1}{-1}{0,-1,0,-2}
\;\uplus\; \gcsix[4,1,2]{4}{0,1}{1,0,0,2}{0,0,1}{-1}{0,0,-1}{0,-1,0,-1}.
$$

To achieve a proper partition of a unicyclic class, we seem to need to generalise our concept of an entanglement diagram, allowing some lines to cross cell boundaries.\footnote{Such diagrams have row-column \emph{hypergraphs}.}
It is not too hard to work out how to enumerate the gridded permutations in such generalised diagrams.

Assuming our partition of $\!\gctwo{2}{1,1}{-1,-1}\!$ is correct, it is then possible to deduce that
$$
f_{\!\gctwo{2}{1,1}{-1,-1}\!\!}(z)
\;=\;
\frac{1}{\sqrt{1-4\+z}} \:-\: \frac{1-4\+z+5\+z^2}{(1-2\+z)\+(1-3\+z)}.
$$

In the light of this, we tentatively make the following conjecture, which would hold if every unicyclic monotone grid class could be properly partitioned into a finite number of
acyclic and
unicyclic generalised entanglement diagrams:

\thmbox{
  \begin{conj}
     The generating function of any unicyclic monotone grid class is algebraic.
  \end{conj}
} 

\subsection*{Multicyclic classes}
What can we say about classes with components having more than one cycle?

From
\eqref{eqGriddedGFDiagonal} in the proof of Lemma~\ref{lemmaGridHashGF}, we know that
the generating function of any class of gridded permutations results from the repeated {diagonalisation} of a rational function.
The diagonal of a bivariate rational power series is always algebraic, a result first proved by Furstenberg~\cite{Furstenberg1967}. Indeed the converse is also true, every algebraic function comes about by the diagonalisation of some rational function.
However, diagonalisation does not preserve algebraicity.
Nevertheless, Lipshitz~\cite{Lipshitz1988} proved that \emph{D-finite} power series \emph{are} closed under taking diagonals. (See 
Section~\ref{defDFinite}
for the definition of a D-finite power series.)
As a consequence, we have the following general result concerning gridded permutations:

\thmbox{
\begin{thm}\label{thmGridPermsDFinite}
   The generating function of any class of gridded permutations is D-finite.
\end{thm}
} 

As a consequence, to conclude, we dare to propose the following conjecture, which would hold if
every monotone grid class could be properly partitioned into a finite number of suitably generalised entanglement diagrams, each of whose enumeration was D-finite:

\thmbox{
  \begin{conj}
     The generating function of any monotone grid class is D-finite.
  \end{conj}
} 

\cleardoublepage


\newcommand{\limkinfty}{\lim\limits_{k\rightarrow\infty}}
\newcommand{\tr}{\mathrm{tr}}
\newcommand{\dg}{d}

\newcommand{\WB}{W^{{}^{\text{\textsf{\textbf{\textup{B}}}}}}}
\newcommand{\WWWB}{\WWW^{{}^{\text{\textsf{\textbf{\textup{B}}}}}}}
\newcommand{\WE}{W^{{}^{\text{\textsf{\textbf{\textup{E}}}}}}}
\newcommand{\WWWE}{\WWW^{{}^{\text{\textsf{\textbf{\textup{E}}}}}}}

\newcommand{\Hzero}{H_0}
\newcommand{\vis}[2]{\Psi(#1,#2)}

\newcommand{\xx}{\mbox{$\ast$-$\ast$}}
\newcommand{\ox}{\mbox{1-$\ast$}}
\newcommand{\xo}{\mbox{$\ast$-1}}
\newcommand{\oo}{\mbox{1-1}}
\newcommand{\xinit}{\mbox{$\ast$-initial}}

\chapter{Growth rates of grid classes}\label{chap05}

In this chapter and the next, we turn away from exact enumeration to the (slightly) easier question of determining the asymptotic growth rate of monotone grid classes of permutations.
We
prove that the
exponential
growth rate of $\Grid(M)$ is equal to the square of the spectral radius of its row-column graph $G(M)$.
Consequently,
we
utilize
spectral graph
theoretic results
to
characterise all slowly growing grid classes
and
to
show
that for every
$\gamma\geqslant2+\sqrt{5}$
there is a grid class with growth rate arbitrarily close to $\gamma$.

To prove our main result, we
establish bounds on the size of certain families of tours on graphs.
In the process,
we
prove
that
the family of
tours of even length on
a connected graph
grows at the same rate as
the family of
``balanced''
tours on
the graph
(in which
the number of times an edge is traversed in one direction is the same as the number of times it is traversed in the other direction).

\section{Introduction}

Our focus in this chapter is on the
growth rates
of grid classes.
We prove the following theorem:

\thmbox{
\begin{repthm}{thmGrowthRate}
The growth rate of a monotone grid class of permutations exists and is equal to the square of the spectral radius of its row-column graph.
\end{repthm}
} 

The bulk of the work required to prove this theorem 
is concerned with carefully counting certain families of tours on graphs, in order to give bounds on their sizes.
In particular, we consider ``balanced'' tours, in which the number of times an edge is
traversed in one direction is the same as the number of times it is traversed in the other
direction.
As a consequence, we 
prove the following new result concerning tours on graphs:

\thmbox{
\begin{repthm}{thmBalancedEqualsEven}
The growth rate of the family of balanced tours on a connected graph is the same as that of the family of all tours of even length on the graph.
\end{repthm}
} 

As a consequence of Theorem~\ref{thmGrowthRate}, 
by using the machinery of spectral graph theory, we are able to deduce
a variety of supplementary results.
We give a characterisation of grid classes whose growth rates are no greater than
$\frac{9}{2}$ 
(in a similar fashion to Vatter's characterisation of ``small'' permutation classes in~\cite{Vatter2011}).
We also fully characterise all \emph{accumulation points} of grid class growth rates, the least of which occurs at 4.
Other results include:
\begin{repcor}{corAlgebraicInteger}
The growth rate of every monotone grid class is an algebraic integer.
\end{repcor}
\begin{repcor}{corCycle}
A monotone grid class whose row-column graph is a cycle has growth rate~4.
\end{repcor}
\begin{repcor}{corSmallGrowthRates}
If the growth rate of a monotone grid class is less than 4, it is equal to $4\cos^2\!\left(\frac{\pi}{k}\right)$ for some integer $k\geqslant 3$.
\end{repcor}
\begin{repcor}{corAccumulationPoints}
For every $\gamma \geqslant 2+\sqrt{5}$ there is a monotone grid class with growth rate arbitrarily close to $\gamma$.
\end{repcor}


The remainder of this chapter is structured as follows:
In Section~\ref{sectTours},
we introduce the particular families of tours on graphs that we study and
present our results concerning these tours, culminating in the proof of Theorem~\ref{thmBalancedEqualsEven}.
This is followed, in Section~\ref{sectGridClasses}, by the application of these results to prove our grid class growth rate result, Theorem~\ref{thmGrowthRate}.
To conclude, in Section~\ref{sectConsequences}, we present a number of consequences of Theorem~\ref{thmGrowthRate} that follow from known spectral graph theoretic results.

\section{Tours on graphs}\label{sectTours}
In this section, we investigate families of tours on graphs, parameterised by the number of times each edge is traversed.
We determine a lower bound on the size of families of ``balanced'' tours
and an upper bound on families of arbitrary tours.
Applying the upper bound to tours of even length gives us an expression compatible with the lower bound. 
Combining this with the fact that any balanced tour has even length enables us to
prove Theorem~\ref{thmBalancedEqualsEven}
which reveals that even-length tours and balanced tours grow at the same rate.
These bounds are subsequently used in Section~\ref{sectGridClasses} to relate
tours on graphs to permutation grid classes.

To establish the lower and upper bounds, we first enumerate tours on trees. We then present a way of associating tours on an arbitrary connected graph $G$ with tours on a
related 
``partial covering'' tree,
which we employ to determine bounds for families of tours on arbitrary graphs.
Let us begin by introducing the tours that we consider below.

\subsection{Notation and definitions}
A \emph{walk}, of length $k$, on a graph
is a non-empty alternating sequence of
vertices and edges $v_0, e_1, v_1, e_2, v_2, \ldots, e_k, v_k$ in which the endvertices of $e_i$ are $v_{i-1}$ and $v_i$.
Neither the edges nor the vertices need be distinct.
We say that such a walk \emph{traverses} edges $\{e_1,\ldots,e_k\}$ and \emph{visits} vertices $\{v_1,\ldots,v_{k-1}\}$.
A \emph{tour} (or \emph{closed} walk) is a walk which starts and ends at the same vertex (i.e. $v_k=v_0$).
Our interest is restricted to tours.

Below,
when considering a graph with $m$ edges, we denote its edges $e_1,e_2,\ldots,e_m$. In any particular context, we can choose
the
ordering of the edges
so as
to simplify our presentation.
We denote the
edges incident to a given vertex $v$
by
${e^v_1},{e^v_2},\ldots,e^v_{\dg(v)}$, where $\dg(v)$ is the \emph{degree} of $v$ (number of edges incident to $v$).
Again, we are free to choose
the order of the edges incident to a vertex
so as
to clarify our arguments.

\subsubsection*{Families of tours}
Our interest is in families of tours that are parameterised by the number of times each edge is traversed.
Given non-negative integers $h_1,h_2,\ldots,h_m$
and some vertex $u$ of a graph $G$, we use
$$\WWW_G((h_i);u)\;=\;\WWW_G(h_1,h_2,\ldots,h_m;u)$$ to denote the family of tours on $G$ which start and end at $u$ and traverse each edge $e_i$ exactly $h_i$ times. (We use $\WWW$ rather than $\TTT$ for families of tours to avoid confusion when considering tours on trees.) 

We use ${h^v_1},{h^v_2},\ldots,h^v_{\dg(v)}$ for the number of traversals of edges incident to a vertex $v$ in $\WWW_G((h_i);u)$.
So, if $v$ and $w$ are the endvertices of $e_i$, $h_i$ has two aliases $h^v_j$ and $h^w_{j'}$ for some $j$ and $j'$.

We use $W_G((h_i);u)=|\WWW_G((h_i);u)|$ to denote the number of these tours.

Note that for some values of $h_1,\ldots,h_m$, the family $\WWW_G((h_i);u)$ is empty.
In particular, if $E^+=\{e_i\in E(G):h_i>0\}$ is the set of edges traversed by tours in the family,
and $G^+=G[E^+]$ is the subgraph of $G$ induced by these edges,
then
if $G^+$ is disconnected or does not contain $u$,
we have $\WWW_G((h_i);u)=\varnothing$.
A family of tours may also be empty for ``parity'' reasons; for example, if $T$ is a tree, then $\WWW_T((h_i);u)=\varnothing$ if any of the $h_i$ are odd. Our counting arguments must remain valid for these empty families.

Of particular interest to us are tours in which
the number of times an edge is traversed in one direction is the same as
the number of times it is traversed in the other direction.
We call such tours {\emph{balanced}}.

Given non-negative integers $k_1,k_2,\ldots,k_m$
and some vertex $u$ of a graph $G$, we use
$$\WWWB_G((k_i);u)\;=\;\WWWB_G(k_1,k_2,\ldots,k_m;u)$$ to denote the family of balanced tours on $G$ which start and end at $u$, and traverse each edge $e_i$ exactly $k_i$ times \emph{in each direction}.
Note that we parameterise balanced tours by \emph{half} the number of traversals of each edge.

We use ${k^v_1},{k^v_2},\ldots,k^v_{\dg(v)}$ for the number of traversals in either direction of edges incident to a vertex $v$ in $\WWWB_G((k_i);u)$.
So, if $v$ and $w$ are the endvertices of $e_i$, $k_i$ has two aliases $k^v_j$ and $k^w_{j'}$ for some $j$ and $j'$

We use $\WB_G((k_i);u)=|\WWWB_G((k_i);u)|$ to denote the number of these balanced tours.

As with $\WWW_G((h_i);u)$, $\WWWB_G((k_i);u)$ may be empty.
Observe also that, since any tour on a forest is balanced, $\WWW_F((2k_i);u)=\WWWB_F((k_i);u)$ for any forest $F$ and $u\in V(F)$.
Moreover, for any graph $G$,
we have
$\WB_G((k_i);u)\leqslant W_G((2k_i);u)$, with equality if and only if the component of $G^+$ containing $u$, if present, is acyclic, where $G^+$ is the subgraph of $G$ induced by the edges that are actually traversed by tours
in the family.

\subsubsection*{Visits and excursions}
We use $\vis{G}{v}$ to denote the number of {visits} to $v$ of any tour on $G$ in some
family (specified by the context).
In practice, this notation is unambiguous because we only consider one family of tours on a particular graph at a time.
Observe that any tour in $\WWW_G((h_i);u)$ visits vertex $v\neq u$ exactly $\half({h^v_1}+{h^v_2}+\ldots+{h^v_{\dg(v)}})$ times, and
that for balanced tours in $\WWWB_G((k_i);u)$ we have
$\vis{G}{v}={k^v_1}+{k^v_2}+\ldots+{k^v_{\dg(v)}}$.

If $\vis{G}{v}$ is positive, then
separating the visits to $v$ are $\vis{G}{v}-1$ ``subtours''
starting and ending at $v$;
we refer to these subtours as \emph{excursions} from $v$.

\subsubsection*{Multinomial coefficients}
In our calculations, we make considerable use of multinomial coefficients as described on page~\pageref{MultinomialCoeffs}.
We make repeated use of the fact that a multinomial coefficient can be decomposed into a product of binomial coefficients:
$
\binom{n}{k_1,\ldots,k_r}
=
\binom{k_1}{k_1} \binom{k_1+k_2}{k_2} \ldots \binom{k_1+\ldots +k_r}{k_r}
.
$
We consider a multinomial coefficient that has one or more negative terms to be \emph{zero}.
This guarantees
that the monotonicity condition
$
\binom{n}{k_1,\ldots,k_r}
\leqslant
\binom{n+1}{k_1+1,\ldots,k_r}
$
holds
for
all possible sets of values.

\subsection{Tours on trees}\label{sectTrees}
We begin by establishing bounds on the size of families of tours on \emph{trees}.
As we noted above, all such tours are balanced.
We start with star graphs, giving an exact enumeration of any family:
\begin{lemma}\label{lemmaStar}
If $S_m$ is the star graph $K_{1,m}$ with central vertex $u$, then
$$
\WB_{S_m}((k_i);u)\;=\;
\binom{k_1+k_2+\ldots+k_m}{k_1,\,k_2,\,\ldots,\,k_m}
\;=\;
\binom{\vis{S_m}{u}}{k^u_1,\,k^u_2,\,\ldots,\,k^u_{\dg(u)}}.
$$
\end{lemma}
\begin{proof}
$\WWWB_{S_m}((k_i);u)$ consists of all possible interleavings of $k_i$ excursions from $u$ out-and-back along each $e_i$.
\end{proof}
It is possible to extend our exact enumeration to those families of balanced tours on trees in which every internal (non-leaf) vertex is visited at least once. These families are never empty.
\begin{lemma}\label{lemmaTree}
If $T$ is a tree,
$u\in V(T)$
and, for each $v\neq u$, $e^v_1$ is the edge incident to $v$ that is on the unique path between $u$ and $v$, and
if $k^v_1$ is positive for all internal vertices $v$ of $T$,
then
$$
\WB_T((k_i);u)\;=\;
\binom{\vis{T}{u}}{k^u_1,\,k^u_2,\,\ldots,\,k^u_{\dg(u)}}
\prod_{v\neq u}\binom{\vis{T}{v}-1}{{k^v_1}\!-\!1,\,{k^v_2},\,\ldots,\,{k^v_{\dg(v)}}}.
$$
\end{lemma}
\begin{proof}
We use induction on the number of internal vertices. Note that the multinomial coefficient for a leaf vertex simply contributes a factor of 1 to the product.
Lemma~\ref{lemmaStar} provides the base case.

Given a tree $T$ with $m$ edges $e_1,\ldots,e_m$,
and a leaf $v$ of $T$,
let $T'$ be the tree ``grown'' from $T$ by attaching $r$ new pendant edges $e_{m+1},\ldots,e_{m+r}$ to $v$.

If ${k^v_1}$ is positive,
since $v$ is a leaf, each tour in $\WWWB_T(k_1,\ldots,k_m;u)$ visits $v$ exactly ${k^v_1}$ times, with ${k^v_1}-1$ excursions from $v$ along ${e^v_1}$ separating these visits.
Any such tour can be extended to a tour in $\WWWB_{T'}(k_1,\ldots,k_{m+r};u)$
by arbitrarily interleaving $k_{m+i}$ new excursions out-and-back along each new pendant edge $e_{m+i}$
($i=1,\ldots,r$)
with the existing ${k^v_1}-1$ excursions from $v$ along ${e^v_1}$.
\end{proof}
This exact enumeration can be used to generate 
the following general bounds on the number of tours on trees:
\begin{cor}\label{corTreeBounds}
If $T$ is a tree, then for any vertex
$u\in V(T)$, $\WB_T((k_i);u)$ satisfies the following bounds:
$$
\prod_{v\in V(T)}\binom{\vis{T}{v}-\dg(v)}{{k^v_1}\!-\!1,\,{k^v_2}\!-\!1,\,\ldots,\,{k^v_{\dg(v)}}\!-\!1}
\;\leqslant\;
\WB_T((k_i);u)
\;\leqslant\;
\prod_{v\in V(T)}\binom{\vis{T}{v}}{{k^v_1},\,{k^v_2},\,\ldots,\,{k^v_{\dg(v)}}}.
$$
\end{cor}
\begin{proof}
If all the $k_i$ are positive, then this follows directly from Lemma~\ref{lemmaTree}.

If one or more of the $k_i$ is zero, then the lower bound is trivially true, because one of the multinomial coefficients is zero. The upper bound
also holds trivially
if there are no tours in the family. Otherwise,
let $T^+$ be the subtree of $T$ induced by the vertices actually visited by tours in $\WWWB_T((k_i);u)$. Then $\WB_T((k_i);u)=\WB_{T^+}((k_i);u)$. But we know that
$$
\WB_{T^+}((k_i);u)
\;\leqslant\;
\prod_{v\in V(T^+)}\binom{\vis{T^+}{v}}{{k^v_1},\,{k^v_2},\,\ldots,\,{k^v_{\dg(v)}}}
\;=\;
\prod_{v\in V(T)}\binom{\vis{T}{v}}{{k^v_1},\,{k^v_2},\,\ldots,\,{k^v_{\dg(v)}}}
$$
as a result of Lemma~\ref{lemmaTree} and the fact that $k^v_i=0$ for all
edges $e^v_i$ incident to
unvisited vertices
$v\in V(T)\setminus V(T^+)$.
\end{proof}

\subsection{Treeification} 
In order to establish the lower and upper bounds for tours on arbitrary connected graphs, we relate tours on a connected graph $G$ to (balanced) tours on a related
tree which we call a \emph{treeification} of $G$.
The process of treeification
consists of repeatedly breaking cycles until
the resulting graph is acyclic.
This creates a sequence of graphs $G=G_0,G_1,\ldots,G_t=T$ 
where
$T$ is a tree.
We call this sequence a \emph{treeification sequence}.

Formally, we define a treeification of a
connected
graph to be the result of 
the following
(nondeterministic)
process that transforms a connected graph
into a
tree
with
the same number of
edges. 

\begin{samepage}
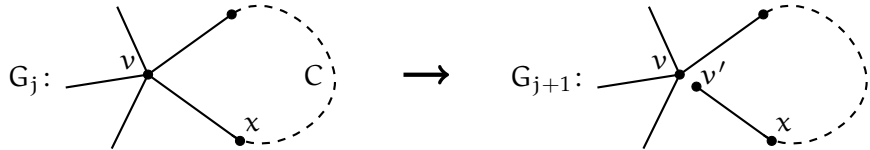
\begin{figure}[ht]
  $$
    \raisebox{.4in}{$G_j\!:\;\:$}
    \begin{tikzpicture}[scale=1.36,rotate=-36]
      \draw [thick] (1.1,0)--(0,0);
      \draw [thick] (0,0) -- (.31,.95);
      \draw [thick] (-.63,.36)--(0,0)--(-.57,-.57);
      \draw [thick] (0,0)--(.14,-.79);
      \draw [dashed,thick,] (1,0) to [out=0,in=-45] (1.46,1.06) node[left]{$C$} to [out=135,in=72] (.31,.95);
      \draw [fill] (0,0) circle [radius=0.044];
      \draw [fill] (1.1,0) circle [radius=0.044];
      \draw [fill] (.31,.95) circle [radius=0.044];
      \node at (-.2,0) {$v\:$};
      \node at (1.1,0.2) {$x$};
    \end{tikzpicture}
    \qquad
    \raisebox{.45in}
    {\begin{tikzpicture}\draw [->,ultra thick](0,0)--(.6,0);\end{tikzpicture}}
    \qquad
    \raisebox{.4in}{$G_{j+1}\!:\;\:$}
    \begin{tikzpicture}[scale=1.36,rotate=-36]
      \draw [thick] (1.1,0)--(.2,0);
      \draw [thick] (0,0) -- (.31,.95);
      \draw [thick] (-.63,.36)--(0,0)--(-.57,-.57);
      \draw [thick] (0,0)--(.14,-.79);
      \draw [dashed,thick,] (1,0) to [out=0,in=-45] (1.46,1.06) to [out=135,in=72] (.31,.95);
      \draw [fill] (0,0) circle [radius=0.044];
      \draw [fill] (.2,0) circle [radius=0.044];
      \draw [fill] (1.1,0) circle [radius=0.044];
      \draw [fill] (.31,.95) circle [radius=0.044];
      \node at (-.2,0) {$v\:$};
      \node at (.26,.21) {$v'$};
      \node at (1.1,0.2) {$x$};
    \end{tikzpicture}
    \vspace{-3pt}
  $$
  \caption{Splitting vertex $v$}\label{figVertexSplitting}
\end{figure}
To treeify a connected graph $G=G_0$,
we first give an (arbitrary) order to
its vertices. Then we 
apply the following
vertex-splitting operation in turn
to each $G_j$
to create $G_{j+1}$ ($j=0,1,\ldots$), 
until no cycles remain:
\vspace{-9pt}
\begin{enumerate}
  \itemsep0pt
  \item
  Let $v$ be the first vertex (in the ordering)
  that occurs in some cycle $C$ of $G_j$.
  \item Split vertex $v$ by doing the following (see Figure~\ref{figVertexSplitting}): \vspace{-3pt}
  \begin{enumerate}
  \itemsep0pt
  \item Delete an edge $xv$ from $E(C)$ (there are two choices for vertex $x$).
  \item Add a new vertex $v'$ (to the end of the vertex ordering).
  \item Add the pendant edge $xv'$ (making $v'$ a leaf).
  \end{enumerate}
\end{enumerate}
\end{samepage}

\begin{figure}[ht]
$$
  \begin{tikzpicture}[scale=0.8]
      \draw [thick] (0,2)--(0,0)--(2,0)--(2,1);
      \draw [thick] (1,0)--(1,2)--(2,2);
      \draw [thick] (1,2)--(0,2);
      \draw [thick] (1,1)--(2,1);
      \draw [thick] (2,2)--(2,1);
      \draw [thick] (0,1)--(1,1);
      \draw [fill] (0,0) circle [radius=0.075];
      \draw [fill] (0,1) circle [radius=0.075];
      \draw [fill] (0,2) circle [radius=0.075];
      \draw [fill] (1,0) circle [radius=0.075];
      \draw [fill] (1,1) circle [radius=0.075];
      \draw [fill] (1,2) circle [radius=0.075];
      \draw [fill] (2,0) circle [radius=0.075];
      \draw [fill] (2,1) circle [radius=0.075];
      \draw [fill] (2,2) circle [radius=0.075];
      \node at (1,-.5) {$G=G_0$};
      \node at (-.25,2) {${}_1$};
      \node at (2.25,1) {${}_2$};
      \node at (2.25,0) {${}_3$};
      \node at (1.2,.75) {${}_4$};
    \end{tikzpicture}
    \qquad\!
    \begin{tikzpicture}[scale=0.8]
      \draw [thick] (0,2)--(0,0)--(2,0)--(2,1);
      \draw [thick] (1,0)--(1,2)--(2,2);
      \draw [thick] (1,2)--(0.2,2.2);
      \draw [thick] (1,1)--(2,1);
      \draw [thick] (2,2)--(2,1);
      \draw [thick] (0,1)--(1,1);
      \draw [fill] (0,0) circle [radius=0.075];
      \draw [fill] (0,1) circle [radius=0.075];
      \draw [fill] (0,2) circle [radius=0.075];
      \draw [fill] (0.2,2.2) circle [radius=0.075];
      \draw [fill] (1,0) circle [radius=0.075];
      \draw [fill] (1,1) circle [radius=0.075];
      \draw [fill] (1,2) circle [radius=0.075];
      \draw [fill] (2,0) circle [radius=0.075];
      \draw [fill] (2,1) circle [radius=0.075];
      \draw [fill] (2,2) circle [radius=0.075];
      \node at (1,-.5) {$G_1$};
      \node at (-.25,2) {${}_1$};
      \node at (2.25,1) {${}_2$};
      \node at (2.25,0) {${}_3$};
      \node at (1.2,.75) {${}_4$};
    \end{tikzpicture}
    \qquad
    \begin{tikzpicture}[scale=0.8]
      \draw [thick] (0,2)--(0,0)--(2,0)--(2,1);
      \draw [thick] (1,0)--(1,2)--(2,2);
      \draw [thick] (1,2)--(0.2,2.2);
      \draw [thick] (1,1)--(1.8,1.2);
      \draw [thick] (2,2)--(2,1);
      \draw [thick] (0,1)--(1,1);
      \draw [fill] (0,0) circle [radius=0.075];
      \draw [fill] (0,1) circle [radius=0.075];
      \draw [fill] (0,2) circle [radius=0.075];
      \draw [fill] (0.2,2.2) circle [radius=0.075];
      \draw [fill] (1,0) circle [radius=0.075];
      \draw [fill] (1,1) circle [radius=0.075];
      \draw [fill] (1,2) circle [radius=0.075];
      \draw [fill] (2,0) circle [radius=0.075];
      \draw [fill] (2,1) circle [radius=0.075];
      \draw [fill] (1.8,1.2) circle [radius=0.075];
      \draw [fill] (2,2) circle [radius=0.075];
      \node at (1,-.5) {$G_2$};
      \node at (-.25,2) {${}_1$};
      \node at (2.25,1) {${}_2$};
      \node at (2.25,0) {${}_3$};
      \node at (1.2,.75) {${}_4$};
    \end{tikzpicture}
    \qquad\!
    \begin{tikzpicture}[scale=0.8]
      \draw [thick] (0,2)--(0,0)--(2,0)--(2,1);
      \draw [thick] (1,0)--(1,2)--(2,2);
      \draw [thick] (1,2)--(0.2,2.2);
      \draw [thick] (1,1)--(1.8,1.2);
      \draw [thick] (2,2)--(2.2,1.2);
      \draw [thick] (0,1)--(1,1);
      \draw [fill] (0,0) circle [radius=0.075];
      \draw [fill] (0,1) circle [radius=0.075];
      \draw [fill] (0,2) circle [radius=0.075];
      \draw [fill] (0.2,2.2) circle [radius=0.075];
      \draw [fill] (1,0) circle [radius=0.075];
      \draw [fill] (1,1) circle [radius=0.075];
      \draw [fill] (1,2) circle [radius=0.075];
      \draw [fill] (2,0) circle [radius=0.075];
      \draw [fill] (2,1) circle [radius=0.075];
      \draw [fill] (1.8,1.2) circle [radius=0.075];
      \draw [fill] (2.2,1.2) circle [radius=0.075];
      \draw [fill] (2,2) circle [radius=0.075];
      \node at (1,-.5) {$G_3$};
      \node at (-.25,2) {${}_1$};
      \node at (2.25,.9) {${}_2$};
      \node at (2.25,0) {${}_3$};
      \node at (1.2,.75) {${}_4$};
    \end{tikzpicture}
    \qquad\!
    \begin{tikzpicture}[scale=0.8]
      \draw [thick] (0,2)--(0,0)--(2,0)--(2,1);
      \draw [thick] (1,0)--(1,2)--(2,2);
      \draw [thick] (1,2)--(0.2,2.2);
      \draw [thick] (1,1)--(1.8,1.2);
      \draw [thick] (2,2)--(2.2,1.2);
      \draw [thick] (0,1)--(0.8,1.2);
      \draw [fill] (0,0) circle [radius=0.075];
      \draw [fill] (0,1) circle [radius=0.075];
      \draw [fill] (0,2) circle [radius=0.075];
      \draw [fill] (0.2,2.2) circle [radius=0.075];
      \draw [fill] (1,0) circle [radius=0.075];
      \draw [fill] (1,1) circle [radius=0.075];
      \draw [fill] (0.8,1.2) circle [radius=0.075];
      \draw [fill] (1,2) circle [radius=0.075];
      \draw [fill] (2,0) circle [radius=0.075];
      \draw [fill] (2,1) circle [radius=0.075];
      \draw [fill] (1.8,1.2) circle [radius=0.075];
      \draw [fill] (2.2,1.2) circle [radius=0.075];
      \draw [fill] (2,2) circle [radius=0.075];
      \node at (1,-.5) {$G_4=T$};
      \node at (-.25,2) {${}_1$};
      \node at (2.25,.9) {${}_2$};
      \node at (2.25,0) {${}_3$};
      \node at (1.2,.75) {${}_4$};
    \end{tikzpicture}
$$
\caption{A treeification sequence; numbers show the first few vertices in the ordering}\label{figTreeify}
\end{figure}
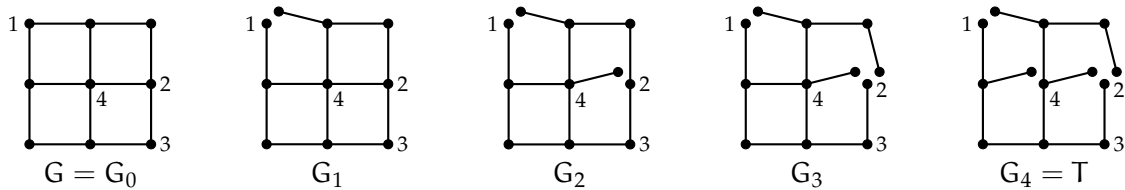

Note that if a vertex $v$ is split
multiple
times when treeifying a graph $G$, these
splits occur contiguously (because of the ordering placed on the vertices of $G$).
Thus, if $v$ is split
$r$ times, there is a contiguous subsequence $G_j,G_{j+1},\ldots,G_{j+r}$
of the treeification sequence
that corresponds to the splitting of $v$.
See Figure~\ref{figTreeify} for an example
of a treeification sequence.

There is a natural way to establish a relationship between tours on different graphs in a treeification sequence $G_0,...,G_t$.
The treeification process induces \emph{graph homomorphisms} (edge preserving maps) between the graphs in such a sequence.
For all $i<j$, there is a \emph{surjective} 
homomorphism from $G_j$ onto $G_i$.
This homomorphism is also \emph{locally injective} since it maps neighbourhoods of $G_j$ injectively into neighbourhoods of $G_i$.
A locally injective map such as this is also known as a \emph{partial cover}.
In particular, for each $j<t$, there is a
partial cover of $G_{j+1}$ onto $G_j$ that maps
the new pendant edge $xv'$ to the edge $xv$ that it replaces.
These homomorphisms
impart
a natural correspondence between families of tours on different graphs in the treeification sequence, which we
employ
later
to determine our bounds.


Although the concept of treeification is a very natural one,
these particular ``partial covering trees'' do not appear to have been studied before;
their
only previous use seems to be by
Yarkony, Fowlkes and Ihler
to address a problem in
computer vision \cite{YFI2010}.
For a general introduction to graph homomorphisms, see
see the monograph by Hell \& Ne\v{s}e\-t\v{r}il~\cite{HN2004}.
For more on partial maps and other locally constrained graph homomorphisms,
see the survey article by Fiala \& Kratochv\'{\i}l~\cite{FK2008}.

If we have a treeification sequence
$G=G_0,\ldots,G_t=T$ for a connected graph $G$,
we can use a three-stage process to establish a lower or upper bound for a family of tours on $G$.
In the first stage (``splitting once''), we relate the number of tours in the family on $G_j$ ($j<t$) to the number of tours in a related family on $G_{j+1}$.
In the second stage (``fully splitting one vertex''), for a vertex $v$, we consider the subsequence $G_j,\ldots,G_{j+r}$
that corresponds to the splitting of $v$
and, iterating
the inequality from
the first stage, relate the number of tours on $G_j$ to the number of tours on $G_{j+r}$.
Finally (``fully splitting all vertices''), iterating
the inequality from
the second stage, we relate the number of tours on $G=G_0$ to the number of tours on $G_t=T$, and employ
the bounds on tours on $T$ from Corollary~\ref{corTreeBounds} to determine the bound for the family of tours on $G$.

In Subsection~\ref{sectLowerBound}, we use this three-stage process to produce a lower bound on the value of
$\WB_G((k_i);u)$.
Then, in Subsection~\ref{sectUpperBound}, we use the same three-stage process to establish an upper bound on
$W_G((h_i);u)$.

\subsection{The lower bound}\label{sectLowerBound}
Our lower bound is
on the number of \emph{balanced} tours. We only
consider the families in which \emph{every} edge is traversed at least once in each direction. On a connected graph, these families are never empty.
\begin{lemma}
\label{lemmaLowerBound}
If $G$ is a connected graph with $m$ edges and $k_1,\ldots,k_m$ are all positive, then for any vertex $u\in V(G)$,
$$
\WB_G(k_1,k_2,\ldots,k_m;u)
\;\geqslant\;
\prod_{v\in V(G)}\binom
{{k^v_1}+{k^v_2}+\ldots+{k^v_{\dg(v)}}-\dg(v)}
{{k^v_1}\!-\!1,\,{k^v_2}\!-\!1,\,\ldots,\,{k^v_{\dg(v)}}\!-\!1}.
$$
\end{lemma}
This lower bound does not hold in general for a \emph{disconnected} graph since there are no tours possible if
there is any positive $k_i$
in a component not containing $u$.

\begin{proof}
Let $T$ be some treeification of $G$ with treeification sequence $G=G_0,\ldots,G_t=T$ in which
vertex $u$ is never split.
(This is possible by positioning $u$ last in the ordering on the vertices.)

By exhibiting a surjection from $\WWWB_G((k_i);u)$ onto $\WWWB_T((k_i);u)$ that is consistent with the homomorphism from $T$ onto $G$ induced by the treeification process, we
determine an inequality relating
the number of tours
in the two families.

\subsubsection*{I. Splitting once}
Our first stage is to associate a number of tours on $G_j$, in $\WWWB_{G_j}((k_i);u)$, to each tour on $G_{j+1}$, in $\WWWB_{G_{j+1}}((k_i);u)$.

To simplify the notation, let $\Hzero=G_j$ and $H=G_{j+1}$ for some $j<t$.
Let $v$ be the vertex of $\Hzero$ that is split in $H$,
and let $v'$ be the leaf vertex in $H$ added when splitting $v$.
Let $e_1$ be the (only) edge incident to $v'$ in $H$; we also use $e_1$ to refer to the corresponding edge (incident to $v$) in $\Hzero$ (see Figure~\ref{figSplittingLB}).

\begin{figure}[ht]
  $$
    \raisebox{.4in}{$H=G_{j+1}\!:\;\:$}
    \begin{tikzpicture}[scale=1.36,rotate=-36]
      \draw [thick] (1.1,0)--(.2,0) node[midway,below]{$e_1$};
      \draw [thick] (1,.2)--(.55,.2) node[very near end,right]{$\;k_1$};
      \draw [thick] (.55,.2) to [out=180,in=90](.5,.15) to [out=-90,in=180](.55,.1);
      \draw [<-,thick] (1,.1)--(.55,.1);
      \draw [thick] (0,0) -- (.31,.95);
      \draw [thick] (-.63,.36)--(0,0)--(-.57,-.57);
      \draw [thick] (0,0)--(.14,-.79);
      \draw [dashed,thick,] (1,0) to [out=0,in=-45] (1.46,1.06) to [out=135,in=72] (.31,.95);
      \draw [fill] (0,0) circle [radius=0.044];
      \draw [fill] (.2,0) circle [radius=0.044];
      \draw [fill] (1.1,0) circle [radius=0.044];
      \draw [fill] (.31,.95) circle [radius=0.044];
      \node at (-.2,0) {$v\:$};
      \node at (.26,.21) {$v'$};
    \end{tikzpicture}
    \qquad
    \raisebox{.45in}
    {\begin{tikzpicture}\draw [->,ultra thick](0,0)--(.6,0);\end{tikzpicture}}
    \qquad
    \raisebox{.4in}{$\Hzero=G_j\!:\;\:$}
    \begin{tikzpicture}[scale=1.36,rotate=-36]
      \draw [thick] (1.1,0)--(0,0) node[midway,below]{$e_1$};
      \draw [->,thick] (.95,.1)--(.45,.1);
      \draw [<-,thick] (.95,.2)--(.45,.2) node[very near end,right]{$\;k_1$};
      \draw [thick] (0,0) -- (.31,.95);
      \draw [thick] (-.63,.36)--(0,0)--(-.57,-.57);
      \draw [thick] (0,0)--(.14,-.79);
      \draw [dashed,thick,] (1,0) to [out=0,in=-45] (1.46,1.06)
        to [out=135,in=72] (.31,.95);
      \draw [fill] (0,0) circle [radius=0.044];
      \draw [fill] (1.1,0) circle [radius=0.044];
      \draw [fill] (.31,.95) circle [radius=0.044];
      \node at (-.2,0) {$v\:$};
    \end{tikzpicture}
    \vspace{-3pt}
  $$
  \caption{Tours on $\Hzero$ corresponding to a tour on $H$}\label{figSplittingLB}
\end{figure}
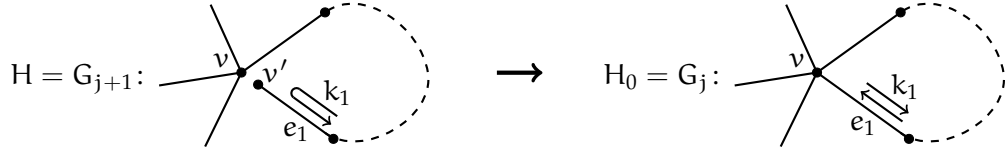
Any tour in $\WWWB_H((k_i);u)$ visits vertex $v$ exactly $\vis{H}{v}$ times and visits vertex $v'$ (along $e_1$)
$k_1$
times. The corresponding tour on $\Hzero$ visits $v$ exactly $\vis{\Hzero}{v}=\vis{H}{v}+k_1$ times. Of these visits there are $k_1$ which arrive along $e_1$ and then depart along $e_1$.

Since $\vis{\Hzero}{v}$ is positive,
separating the visits are $\vis{\Hzero}{v}-1$ excursions from $v$. Depending on whether the final visit to $v$ departs along $e_1$ or not, either $k_1-1$ or $k_1$ of these excursions begin with a traversal of $e_1$; these are interleaved with the other $\vis{H}{v}$ or $\vis{H}{v}-1$ excursions which begin with a traversal of some other edge.

Changing the interleaving of these two sets of excursions (without altering their internal ordering) produces at least
$$
\min\left[\binom{\vis{\Hzero}{v}-1}{k_1-1},\binom{\vis{\Hzero}{v}-1}{k_1}\right]
\;\geqslant\;
\binom{\vis{\Hzero}{v}-2}{k_1-1}
$$
distinct tours in $\WWWB_{\Hzero}((k_i);u)$.

Note that there is only one interleaving of the sets of excursions that corresponds to a valid tour in $\WWWB_H((k_i);u)$: the one in which the excursions beginning with a traversal of $e_1$ away from $v$ are arranged so they occur immediately following a traversal of $e_1$ towards~$v$.

Hence we can deduce that
\begin{equation}\label{eqUnsplitOnce}
\WB_{\Hzero}((k_i);u)
\;\geqslant \;
\binom{\vis{\Hzero}{v}-2}{k_1-1}\WB_H((k_i);u).
\end{equation}

\subsubsection*{II. Fully splitting one vertex}
For a given vertex $v$, let
$H_0,H_1,\ldots,H_r$
be the subsequence of graphs that corresponds to the splitting of $v$.
In our second stage, we
relate the number of tours on $H_0$ to the number of tours on $H_r$.

Note that
$\vis{H_0}{v}=\vis{G}{v}$
and
$\vis{H_r}{v}=\vis{T}{v}$
since the splitting of other vertices cannot affect the number of visits to $v$.

Let $e_1,\ldots,e_r$ be the new pendant edges in $H_r$, and hence also in $T$, added when $v$ is split, and let $e_1,\ldots,e_r$ also denote the corresponding edges in $G$.
Then
$\vis{H_{i-1}}{v}=\vis{H_i}{v}+k_i$
for $1\leqslant i\leqslant r$,
and thus
$\vis{H_{i-1}}{v}
=\vis{T}{v}+k_i+k_{i+1}+\ldots+k_r$,
and in particular
$\vis{G}{v}
=\vis{T}{v}+k_1
+\ldots+k_r$.

Hence, by iterating inequality~\eqref{eqUnsplitOnce}, 
\begin{align}\label{eqUnsplitVertex}
\WB_{H_0}((k_i);u)
\;\geqslant\;\; &
\prod_{i=1}^r\binom{\vis{H_{i-1}}{v}-2}{k_i-1}
\+\WB_{H_r}((k_i);u) \nonumber\\[10pt]
=\;\; &
\prod_{i=1}^r\binom{\vis{T}{v}+\big(\sum_{j=i}^r k_j\big)-2}{k_i-1}
\+\WB_{H_r}((k_i);u) \nonumber\\[10pt]
\geqslant\;\; &
\prod_{i=1}^r\binom{\vis{T}{v}+\big(\sum_{j=i}^r (k_j-1)\big)-1}{k_i-1}
\+\WB_{H_r}((k_i);u) \nonumber\\[10pt]
=\;\; &
\binom{\vis{G}{v}-(r+1)}{\vis{T}{v}-1,\,k_1-1,\,k_2-1,\,\ldots,\,k_r-1}
\+\WB_{H_r}((k_i);u).
\end{align}

\subsubsection*{III. Fully splitting all vertices}
Finally, our third stage is to relate the number of tours on $G$ to the number
of tours on $T$ and then apply the tree bounds to establish the required lower bound.

For each $v\in V(G)$, let $r(v)$ be the number of times $v$ is split.
Note that $r(v)$ is less than the degree of $v$ in $G$ since $\dg_G(v)=\dg_T(v)+r(v)$.

Thus,
with a suitable indexing of the edges around each vertex,
if we
iterate inequality~\eqref{eqUnsplitVertex}
and combine with the lower bound on $\WB_{T}((k_i);u)$ from Corollary~\ref{corTreeBounds}, we get
\begin{align*}
\WB_G((k_i);u)
\;\;\geqslant\;\; &
\prod_{v\in V(G)}
\binom{\vis{G}{v}-(r(v)+1)}{\vis{T}{v}-1,\,{k^v_1}-1,\ldots,{k^v_{r(v)}}-1}
\+\WB_{T}((k_i);u) \\[10pt]
\geqslant\;\; &
\prod_{v\in V(G)}
\binom{\vis{G}{v}-(r(v)+1)}{\vis{T}{v}-1,\,{k^v_1}-1,\ldots,{k^v_{r(v)}}-1}
\binom{\vis{T}{v}-\dg_T(v)}{{k^v_{r(v)+1}}\!-\!1,\ldots,{k^v_{\dg_G(v)}}\!-\!1} \\[10pt]
\;\;\geqslant\;\; &
\prod_{v\in V(G)}
\binom{\vis{G}{v}-\dg_G(v)}{{k^v_1}-1,\,{k^v_2}-1,\ldots,{k^v_{\dg_G(v)}}-1}
\end{align*}
concluding the proof of Lemma~\ref{lemmaLowerBound}.
\end{proof}

\subsection{The upper bound}\label{sectUpperBound}
Our upper bound applies to arbitrary families of tours $\WWW_G((h_i);u)$, without any restriction on the values of the $h_i$.
Subsequently, we apply this result to families of tours of even length.
\begin{lemma}
\label{lemmaUpperBound}
If $G$ is a connected graph with $m$ edges and $u$ is any vertex of $G$,
then
$$
W_G(h_1,h_2,\ldots,h_m;u)
\; \leqslant \;
(h+2m)^m
\prod\limits_{v\in V(G)}\binom
{{k^v_1}+{k^v_2}+\ldots+{k^v_{\dg(v)}}}
{{k^v_1},\,{k^v_2},\,\ldots,\,{k^v_{\dg(v)}}}
$$
for some
$k_i\in[\half h_i, \half h_i + m]$ $(1\leqslant i\leqslant m)$,
where $h=h_1+\ldots+h_m$ is the length of the tours in the family
and
$k^v_1,k^v_2,\ldots,k^v_{\dg(v)}$ are the $k_i$
corresponding to
edges incident to $v$.
\end{lemma}
\begin{proof}
Let $T$ be some treeification of $G$ with treeification sequence $G=G_0,\ldots,G_t=T$ in which
vertex $u$ is never split.
(This is possible by positioning $u$ last in the ordering on
the vertices.)

We relate the number of (arbitrary) tours in $\WWW_G((h_i);u)$ to the number of (balanced) tours in $\WWWB_T((k_i);u)$, for some $k_i$ not much greater than $\half h_i$.
This is achieved by exhibiting a surjection from $\WWWB_T((k_i);u)$ onto $\WWW_G((h_i);u)$
that is consistent with the
homomorphism from $T$ onto $G$ induced by the treeification process.

The proof is broken down into the same three stages as for the proof of the lower bound.
Initially, we restrict ourselves to the case in which all the $\vis{G}{v}$ are positive.
The case of unvisited vertices is addressed in an additional stage at the end. 

\subsubsection*{I. Splitting once}
Our first stage is to associate to each tour on $G_j$ a number of tours on $G_{j+1}$.
However, unlike in the proof of the lower bound, the relationship is not between classes with the same parameterisation.
Rather, we relate tours in $\WWW_{G_j}((h_i);u)$
to slightly longer tours in
$\WWW_{G_{j+1}}((h'_i);u)$, for some $h'_i$ such that,
for each $i$, $h_i \leqslant h'_i \leqslant h_i + 2$.

As we did for the lower bound, let $\Hzero=G_j$ and $H=G_{j+1}$ for some $j<t$.
Let $v$ be the vertex of $\Hzero$ that is split in $H$,
and let $v'$ be the leaf vertex in $H$ added when splitting $v$.

Again, let $e_1$ be the (only) edge incident to $v'$ in $H$; we also use $e_1$ to refer to the corresponding edge (incident to $v$) in $\Hzero$.

Let
$C$ be some cycle in $\Hzero$
containing $e_1$, and let
$e_2$ be the other edge on $C$ that is incident to $v$ (in both $\Hzero$ and $H$).

\begin{figure}[ht]
  $$
    \raisebox{.4in}{$\Hzero=G_j\!:\;\:$}
    \begin{tikzpicture}[scale=1.36,rotate=-36]
      \draw [thick] (1.1,0)--(0,0) node[midway,below]{$e_1$};
      \draw [<->,thick] (.95,.1)--(.45,.1) node[very near end,right]{$\;h_1$};
      \draw [thick] (0,0) -- (.31,.95) node[midway,above]{$e_2\;$};
      \draw [thick] (-.63,.36)--(0,0)--(-.57,-.57);
      \draw [thick] (0,0)--(.14,-.79);
      \draw [dashed,thick,] (1,0) to [out=0,in=-45] (1.46,1.06) node[left]{$C$} to [out=135,in=72] (.31,.95);
      \draw [fill] (0,0) circle [radius=0.044];
      \draw [fill] (1.1,0) circle [radius=0.044];
      \draw [fill] (.31,.95) circle [radius=0.044];
      \node at (-.2,0) {$v\:$};
    \end{tikzpicture}
    \qquad
    \raisebox{.45in}
    {\begin{tikzpicture}\draw [->,ultra thick](0,0)--(.6,0);\end{tikzpicture}}
    \qquad
    \raisebox{.4in}{$H=G_{j+1}\!:\;\:$}
    \begin{tikzpicture}[scale=1.36,rotate=-36]
      \draw [thick] (1.1,0)--(.2,0) node[midway,below]{$e_1$};
      \draw [thick] (1,.2)--(.55,.2) node[very near end,right]{$\;k_1$};
      \draw [thick] (.55,.2) to [out=180,in=90](.5,.15) to [out=-90,in=180](.55,.1);
      \draw [<-,thick] (1,.1)--(.55,.1);
      \draw [thick] (0,0) -- (.31,.95) node[midway,above]{$e_2\;$};
      \draw [thick] (-.63,.36)--(0,0)--(-.57,-.57);
      \draw [thick] (0,0)--(.14,-.79);
      \draw [dashed,thick,] (1,0) to [out=0,in=-45] (1.46,1.06) to [out=135,in=72] (.31,.95);
      \draw [fill] (0,0) circle [radius=0.044];
      \draw [fill] (.2,0) circle [radius=0.044];
      \draw [fill] (1.1,0) circle [radius=0.044];
      \draw [fill] (.31,.95) circle [radius=0.044];
      \node at (-.2,0) {$v\:$};
      \node at (.26,.21) {$v'$};
    \end{tikzpicture}
    \vspace{-3pt}
  $$
  \caption{Tours on $H$ corresponding to a tour on $\Hzero$; $k_1=\floor{\half h_1}+1$}\label{figSplittingUB}
  \end{figure}

Given a tour on $\Hzero$, we want to modify it so that the result is a valid tour on $H$. For a tour on $\Hzero$ to be valid on $H$, each traversal of $e_1$ towards $v$ must be immediately followed by a traversal of $e_1$ from $v$. See Figure~\ref{figSplittingUB}.

\begin{samepage}
To achieve this, we make three kinds of changes to excursions from $v$: \vspace{-12pt}
\begin{enumerate}
  \itemsep0pt
  \item Reverse the direction of some of the excursions.
  \item Add one or two additional excursions (around $C$).
  \item Modify the interleaving of excursions.
\end{enumerate}
\vspace{-6pt}
\end{samepage}

\begin{samepage}
To manage the details, given a tour on $\Hzero$, we consider the
$\vis{\Hzero}{v}-1$
excursions from $v$ to be partitioned into subsets as follows:
\begin{center}
\begin{tabular}{rl}
  {\xx{}}: & $a_0$ excursions that don't traverse $e_1$ at all \\[3pt]
  {\ox{}}: & $a_1$ excursions that begin but don't end with a traversal of $e_1$ \\[3pt]
  {\xo{}}: & $a_2$ excursions that end but don't begin with a traversal of $e_1$ \\[3pt]
  {\oo{}}: & $a_3$ excursions that both begin and end with traversals of $e_1$
\end{tabular}
\end{center}
\end{samepage}
We also refer to \ox{} and \oo{} excursions as \emph{1-initial}, and \xx{} and \xo{} excursions as \emph{\xinit{}}.

We refer to the edge traversed in arriving for the first visit to $v$ as the \emph{arrival edge}
and to the edge traversed in departing from the last visit to $v$ as the \emph{departure edge}.
We call their traversals the \emph{arrival} and the \emph{departure} respectively.
To account for these,
we define $a^+_1$ to be $a_1+1$ if the
departure edge is $e_1$
and to be $a_1$ otherwise,
and define $a^+_2$ to be $a_2+1$ if the
arrival edge is $e_1$
and to be $a_2$ otherwise.

So, to transform a tour on $\Hzero$ into one on $H$, we perform the following three steps: \vspace{-6pt}
\begin{enumerate}
  \item If $a^+_2>a_1+1$, reverse the direction of the last $\floor{\half(a^+_2-a_1)}$ of the \xo{} excursions (making them \ox{}). \\
      On the other hand, if $a^+_2<a_1$, reverse the direction of the last $\ceil{\half(a_1-a^+_2)}$ of the \ox{} excursions (making them \xo{}). \\
      Update the values of
      $a_1$ and $a_2$
      to reflect
      these reversals; we now have $a^+_2=a_1$ or $a^+_2=a_1+1$.
  \item If $a^+_1+a^+_2$ is even ($h_1$ is even) or $a^+_1=a_1$
      (the departure edge \emph{isn't} $e_1$),
      add a new \ox{} excursion consisting of a tour around the cycle $C$ (returning to $v$ along $e_2$); this should be added following all the existing excursions. \\
      Also, if $a^+_1+a^+_2$ is even
      ($h_1$ is even)
      or $a^+_1=a_1+1$
      (the departure edge \emph{is} $e_1$),
      add a new \xo{} excursion consisting of a tour around the cycle $C$ (departing from $v$ along $e_2$); this should be added following all the existing excursions. \\
      Update the values of
      $a_1$ and $a_2$
      to reflect the presence of the new excursion(s); we now have $a^+_2=a^+_1$.
  \item Change the interleaving of
      the 1-initial excursions
      with
      the \xinit{} excursions
      so that
      each visit to $v$ along $e_1$ returns immediately along $e_1$.
      This is always possible (see below) and there is only one way of doing it.
      We now have a valid tour on $H$.
\end{enumerate}
\begin{figure}[ht]
\newcommand{\gp}{@{$\:\,$}}
\newcommand{\mb}[1]{$\:\!\!$\textbf{#1}$\:\!\!$}
\newcommand{\mi}[1]{$\:\!\!$\emph{#1}$\:\!\!$}
\newcommand{\mbi}[1]{$\!\!$\textbf{\emph{#1}}$\!\!$}
\begin{center}
\begin{tabular}
{|l|c\gp|\gp c\gp|\gp c\gp|\gp c\gp|\gp c\gp|\gp c\gp|\gp c\gp|\gp c\gp|\gp c\gp|\gp c\gp|\gp c\gp|\gp c\gp|\gp c\gp|\gp c|}
\hline
      &-1&3-2&3-1&\mb{1-1}&3-1&2-1&2-1&3-2&\mb{1-3}&3-1&2-2&\mb{1-1}&&2- \\ \hline
Step~1&-1&3-2&3-1&\mb{1-1}&3-1&2-1&\mbi{1-2}&3-2&\mb{1-3}&\mbi{1-3}&2-2&\mb{1-1}&&2- \\ \hline
Step~2&-1&3-2&3-1&\mb{1-1}&3-1&2-1&\mb{1-2}&3-2&\mb{1-3}&\mb{1-3}&2-2&\mb{1-1}&\mb{1-2}&2- \\ \hline
Step~3&-1&\mb{1-1}&\mb{1-2}&3-2&3-1&\mb{1-3}&3-1&\mb{1-3}&2-1&\mb{1-1}&\mb{1-2}&3-2&2-2&2- \\ \hline
\end{tabular}
\end{center}
\begin{center}
\begin{tabular}
{|l|c\gp|\gp c\gp|\gp c\gp|\gp c\gp|\gp c\gp|\gp c\gp|\gp c\gp|\gp c\gp|\gp c\gp|\gp c\gp|\gp c\gp|\gp c\gp|\gp c\gp|\gp c\gp|\gp c\gp|\gp c|}
\hline
      &-3&\mb{1-1}&\mb{1-3}&\mb{1-3}&3-1&\mb{1-2}&3-3&\mb{1-2}&2-1&\mb{1-1}&3-1&2-3&\mb{1-1}&&&1- \\ \hline
Step~1&-3&\mb{1-1}&\mb{1-3}&\mb{1-3}&3-1&\mb{1-2}&3-3&\mi{2-1}&2-1&\mb{1-1}&3-1&2-3&\mb{1-1}&&&1- \\ \hline
Step~2&-3&\mb{1-1}&\mb{1-3}&\mb{1-3}&3-1&\mb{1-2}&3-3&2-1&2-1&\mb{1-1}&3-1&2-3&\mb{1-1}&\mb{1-2}&2-1&1- \\ \hline
Step~3&-3&3-1&\mb{1-1}&\mb{1-3}&3-3&2-1&\mb{1-3}&2-1&\mb{1-2}&3-1&\mb{1-1}&\mb{1-1}&\mb{1-2}&2-3&2-1&1- \\ \hline
\end{tabular}
\vspace{-3pt}
\end{center}
\caption{Two examples of transforming tours by modifying excursions}\label{figModifyExcursions}
\end{figure}
Figure~\ref{figModifyExcursions} shows two examples of this process. The two-digit entries in the table represent the initial and final edges traversed by excursions
from $v$;
the single-digit entries give the
arrival and departure edges;
$e_3$ is an additional edge incident to $v$.
1-initial excursions (whose interleaving with the
\xinit{}
excursions is modified by Step~3) are shown in bold.
In Step~1, excursions which are reversed are shown in italics.

\vspace{-9pt}
\subsubsection*{Validation of Step~3}
\vspace{-9pt}
If we consider the
1-initial excursions and the \xinit{} excursions as two separate lists,
with the
\xinit{} excursions
(together with the arrival and departure)
as ``fixed'', then we can insert
1-initial excursions into the list of \xinit{} excursions as follows: \vspace{-12pt}
\begin{quote}
Following
each
\xo{} excursion
(and the arrival if it is along $e_1$),
place the next unused \ox{} excursion
together with any unused 1-1 excursions that precede it.
\end{quote}
\vspace{-12pt}
This procedure is successful, and ensures that
each visit to $v$ along $e_1$ returns immediately along $e_1$ as along as
the number of traversals of $e_1$ towards $v$ equals the number of traversals of $e_1$ away from $v$, unless either
\vspace{-12pt}
\begin{itemize}
  \itemsep0pt
  \item the departure edge is \emph{not} $e_1$ and the last 1-initial excursion is a 1-1 excursion (the minimal example being~\mbox{-2~1-1~2-}, using the notation of Figure~\ref{figModifyExcursions}), or
  \item the departure edge \emph{is} $e_1$ and the last \xinit{} excursion is a \xx{} excursion
   (the minimal example being
   \mbox{-1~2-2~1-}).
\end{itemize}
\vspace{-12pt}
The rules controlling the addition of new final \xo{} and \ox{} excursions in Step~2 guarantee both that the number of traversals of $e_1$
towards $v$ is the same as the number of traversals of $e_1$ away from $v$, and also that neither of the two exceptional cases occur. Thus Step~3 is always valid.

\vspace{-9pt}
\subsubsection*{Counting}
\vspace{-9pt}
Step~2 can add at most two additional excursions from $v$ (around $C$), so
given a tour in $\WWW_{\Hzero}((h_i);u)$, this process produces a tour in $\WWW_H(2k_1,h'_2,\ldots,h'_m;u)$
where $k_1=\floor{\half h_1}+1$,
and for each $i$, $h_i \leqslant h'_i \leqslant h_i + 2$.

After completing Step~1, there are $a_1+a_2+1$ ways in which it could be undone (reverse no more than $a_1$ \ox{} excursions, reverse no more than $a_2$ \xo{} excursions, or do nothing). Since $h_1=a_1+a_2+2a_3$, this does not exceed $h_1+1$.

Also, after Step~3, there are either $k_1$ or $k_1-1$ excursions that begin with a traversal of $e_1$ that could, prior to the step, have been arbitrarily interleaved with those that don't.

Thus we see that there are
no more than
$$
(h_1+1)\max\!\left[\binom{\vis{H}{v}+k_1-1}{k_1},\binom{\vis{H}{v}+k_1-1}{k_1-1}\right]
\;\leqslant\;
2\+k_1\binom{\vis{H}{v}+k_1}{k_1}
$$
distinct tours in
$\WWW_{\Hzero}((h_i);u)$ that generate any specific tour in $\WWW_H(2k_1,h'_2,\ldots,h'_m;u)$.

Hence, 
\begin{equation}\label{eqSplitOnce}
W_{\Hzero}((h_i);u)
\;\leqslant\;
2\+k_1\binom{\vis{H}{v}+k_1}{k_1} W_H(2k_1,h'_2,\ldots,h'_m;u).
\end{equation}

Note also that either $\vis{H_0}{v}=\vis{H}{v}+k_1-2$ or $\vis{H_0}{v}=\vis{H}{v}+k_1-1$ (depending on whether $h_1$ is even or odd), and so
\begin{equation}\label{eqSplitOnceVisits}
\vis{H_0}{v}<\vis{H}{v}+k_1.
\end{equation}
Furthermore, $\vis{H}{v}$ is positive, since
the additional
excursion(s) ensure that $h'_2$ is positive.

\subsubsection*{II. Fully splitting one vertex}
For a given vertex $v$, let
$H_0,H_1,\ldots,H_r$
be the subsequence of graphs that corresponds to the splitting of $v$.
In the second stage of our proof, we
relate the number of tours on $H_0$ to the number of tours on $H_r$.

Note again that
$\vis{H_0}{v}=\vis{G}{v}$
and
$\vis{H_r}{v}=\vis{T}{v}$
since the splitting of other vertices cannot affect the number of visits to $v$.

We assume that $\vis{G}{v}$ is positive, and hence that $\vis{H_0}{v},\ldots,\vis{H_r}{v}=\vis{T}{v}$ are all positive too.

Let $e_1,\ldots,e_r$ be the new pendant edges in $T$ added when $v$ is split, and let $e_1,\ldots,e_r$ also denote the corresponding edges in $G$.
Then, by~\eqref{eqSplitOnceVisits},
for some $k_1,\ldots,k_r$ such that $\half h_i\leqslant k_i\leqslant \half h_i+i$,
we have
$
\vis{H_{i-1}}{v}\;<\;\vis{H_i}{v}+k_i$,
and thus
$$\vis{H_{i-1}}{v}
\;<\;\vis{T}{v}+k_i+
\ldots+k_r.$$
Hence, by iterating inequality~\eqref{eqSplitOnce},
if
$h'_i=2k_i$ for $1\leqslant i\leqslant r$,
then for some $h'_{r+1},\ldots,h'_m$ such that $h_i\leqslant h'_i\leqslant h_i+2\+r$,
\begin{align}\label{eqSplitVertex}
W_{H_0}((h_i);u)
\;\leqslant\;\; &
2^r
\+\Bigg(\prod_{i=1}^r k_i\+ \binom{\vis{H_i}{v}+k_i}{k_i}\!\Bigg)
\+W_{H_r}((h'_i);u)
\nonumber \\[10pt]
< \;\; &
2^r
\+\Bigg(\prod_{i=1}^r k_i\+ \binom{\vis{T}{v}+\sum_{j=i}^r k_j}{k_i}\!\Bigg)
\+W_{H_r}((h'_i);u)
\nonumber \\[10pt]
= \;\; &
2^r
\+\Big(\prod_{i=1}^rk_i\Big)
\+\binom{\vis{T}{v}+
\sum_{i=1}^r k_i
}{\vis{T}{v},\,k_1,\,\ldots,\,k_r}
\+W_{H_r}((h'_i);u).
\end{align}

\subsubsection*{III. Fully splitting all vertices}
In the third stage of the proof, we relate the number of tours on $G$ to the number
of tours on $T$ and then apply the tree bounds to establish the required upper bound
for the case in which all the $\vis{G}{v}$ are positive.

For each $v\in V(G)$, let $r(v)$ be the number of times $v$ is split.
Also, let $h=h_1+\ldots+h_m$ be the length of the tours in $\WWW_G((h_i);u)$.

Thus,
with a suitable indexing of the edges around each vertex,
if we
iterate inequality~\eqref{eqSplitVertex} and combine with the upper bound on $\WB_{T}((k_i);u)$ from Corollary~\ref{corTreeBounds},
we get,
for some $k_1,\ldots,k_m$ such that $\half h_i\leqslant k_i\leqslant \half h_i+m$,
\begin{align*}
W_G((h_i);u)
\;\leqslant\;\; &
2^m
\+\Big(\prod_{i=1}^m k_i\Big)
\prod_{v\in V(G)}
\binom{\vis{T}{v}+
\sum_{i=1}^{r(v)} k^v_i
}{\vis{T}{v},\,k^v_1,\,\ldots,\,k^v_{r(v)}}
\+\WB_T((k_i);u) \\[10pt]
\leqslant\;\; &
(h+2m)^m
\prod_{v\in V(G)}
\binom{\vis{T}{v}+
\sum_{i=1}^{r(v)} k^v_i
}{\vis{T}{v},\,k^v_1,\,\ldots,\,k^v_{r(v)}}
\+\binom{\vis{T}{v}}{{k^v_{r(v)+1}},\,\ldots,\,{k^v_{\dg_G(v)}}} \\[10pt]
=\;\; &
(h+2m)^m
\prod_{v\in V(G)}
\binom
{k^v_1+\ldots+k^v_{\dg(v)}}
{k^v_1,\,\ldots,\,k^v_{\dg(v)}},
\end{align*}
using the fact
that for each $i$, we have $k_i\leqslant \half h+m$.

\subsubsection*{IV. Unvisited vertices}
Thus we have the desired result for the case in which all the $\vis{G}{v}$ are positive.
To complete the proof, we consider families of tours in which some of the vertices are not visited.

If not all the $\vis{G}{v}$ are positive, then
let $G^+$ be the subgraph of $G$ induced by the vertices actually visited by tours in $\WWW_G((h_i);u)$.
Then
$W_G((h_i);u)=W_{G^+}((h_i);u)$. But we know that
\begin{align*}
W_{G^+}((h_i);u)
\;\leqslant\;\; &
(h+2m)^m
\prod_{v\in V(G^+)}
\binom
{k^v_1+\ldots+k^v_{\dg(v)}}
{k^v_1,\,\ldots,\,k^v_{\dg(v)}} \\[10pt]
\;\leqslant\;\; &
(h+2m)^m
\prod_{v\in V(G)}
\binom
{k^v_1+\ldots+k^v_{\dg(v)}}
{k^v_1,\,\ldots,\,k^v_{\dg(v)}}
\end{align*}
because the inclusion of
the unvisited vertices in
$V(G)\setminus V(G^+)$ cannot decrease the value of the product.
So the bound holds for any family $\WWW_G((h_i);u)$.

This concludes the proof of Lemma~\ref{lemmaUpperBound}.
\end{proof}

\subsection{Tours of even length}
In this subsection, we consider the family of all tours of \emph{even length} on a graph and prove that it grows at the same rate as the more restricted family of all balanced tours.

To do this, we make use of the fact
that the growth rate of a collection of objects does not change if we make
``small'' changes to what we are counting:
\begin{obs}\label{obsGRPolySum}
If $\SSS$ is a collection of objects, containing $S_k$ objects of each size $k$, that has a finite growth rate, then for any positive polynomial $P$ and fixed non-negative integers $d_1,d_2$ with $d_1\leqslant d_2$,
$$
\limkinfty \Big(P(k)\sum_{j\,=\,k+d_1}^{k+d_2}\!S_j\Big)^{1/k}
\;=\;
\limkinfty S_k^{1/k}
\;=\;
\gr(\SSS)
.
$$
\end{obs}
This follows directly from the definition of the growth rate and is a generalisation of an observation made in Section~\ref{sectAnalyticCombin}.
We also use this observation when we consider the relationship between permutation grid classes and families of tours on graphs in the next section.

We can employ our upper bound for $W_G((h_i);u)$ to
give us an upper bound for tours of a
specific even length.
We use
$W_G(h;u)$ to denote the number of tours of length $h$
starting and ending at vertex $u$.
\begin{lemma}\label{lemmaUpperBoundNew}
If $G$ is a connected graph with $m$ edges and $u$ is any vertex of $G$, then
the number of tours of length $2k$ on $G$ starting and ending at vertex $u$ is bounded above as follows:
$$
W_G(2k;u)
\;\leqslant\;
(m+1)^m\+
(2k+2m)^m
\displaystyle\sum_{j=k}^{k+m^2}
\sum\limits_{k_1+\ldots+k_m\,=\,j}
\+\+
\prod\limits_{v\in V(G)}\dbinom{{k^v_1}+
\ldots+{k^v_{\dg(v)}}}
{{k^v_1},\,
\ldots,\,{k^v_{\dg(v)}}}.
$$
\end{lemma}
\begin{proof}
From Lemma~\ref{lemmaUpperBound}, for any vertex $u$ of a graph $G$ with $m$ edges, we know that
$$
\begin{array}{rcl}
W_G(2k;u)
& = &
\displaystyle
\sum\limits_{h_1+\ldots+h_m\,=\,2k}
W_G((h_i);u) \\[15pt]
& \leqslant &
(2k+2m)^m
\displaystyle
\sum\limits_{h_1+\ldots+h_m\,=\,2k}
\+\+
\prod\limits_{v\in V(G)}\dbinom{{k^v_1}+{k^v_2}+\ldots+{k^v_{\dg(v)}}}
{{k^v_1},\,{k^v_2},\,\ldots,\,{k^v_{\dg(v)}}}
\end{array}
$$
where
each $k_i$ is dependent on the sequence $(h_i)$ with
$\half h_i\leqslant k_i\leqslant \half h_i+m$.

There are no more than
$(m+1)^m$
different values of the $h_i$ that give rise to any specific set of $k_i$, and we have
$k\leqslant k_1+\ldots+k_m\leqslant k+m^2$,
so
\[
W_G(2k;u)
\;\leqslant\;
(m+1)^m\+
(2k+2m)^m
\displaystyle\sum_{j=k}^{k+m^2}
\sum\limits_{k_1+\ldots+k_m\,=\,j}
\+\+
\prod\limits_{v\in V(G)}\dbinom{{k^v_1}+
\ldots+{k^v_{\dg(v)}}}
{{k^v_1},\,
\ldots,\,{k^v_{\dg(v)}}}.
\qedhere
\]
\end{proof}
Now, drawing together our upper and lower bounds enables us to deduce that
the family of balanced tours on a graph $G$ grows at the same rate as the family of all tours of even length on $G$.
We use $\WWWB_G$ for the family of all balanced tours on $G$ and
$\WWWE_G$ for the family of all tours of even length on $G$, where, in both cases,
we consider the size of a tour to be \emph{half} its length.

\thmbox{
\begin{thm}\label{thmBalancedEqualsEven}
The growth rate of the family of balanced tours ($\WWWB_G$) on a connected graph is the same as growth rate of the family of all tours of even length ($\WWWE_G$) on the graph.
\end{thm}
} 

\begin{proof}
From Lemma~\ref{lemmaLowerBound}, we know that
$$
\prod_{v\in V(G)}\binom
{{k^v_1}+
\ldots+{k^v_{\dg(v)}}}
{{k^v_1},\,
\ldots,\,{k^v_{\dg(v)}}}
\;\leqslant\;
\WB_G(k_1\!+\!1,
\ldots,k_m\!+\!1;u).
$$
Substitution in the inequality in the statement of Lemma~\ref{lemmaUpperBoundNew} then yields the following relationship between families of even-length and balanced tours:
$$
W_G(2k;u)
\;\leqslant\;
(m+1)^m\+
(2k+2m)^m
\displaystyle\sum_{j=k+m}^{k+m+m^2}
\WB_G(j;u)
$$
where
$\WB_G(j;u)$ is the number of balanced tours of length $2j$ on $G$ starting and ending at $u$.
Combining this with Observation~\ref{obsGRPolySum} and the fact that $\WB_G(k;u)\leqslant W_G(2k;u)$
produces the result $\gr(\WWWB_G)=\gr(\WWWE_G)$.
\end{proof}
Finally, before moving on to the relationship with permutation grid classes,
we
determine the value of the growth rate of
the family of even-length tours
$\WWWE_G$.
This requires only elementary algebraic graph theory.
We
recall here
the relevant 
concepts. 
The \emph{adjacency matrix}, $A = A(G)$ of a graph $G$ has rows and columns indexed
by the vertices of $G$, with $A_{i,j}= 1$ or $A_{i,j}= 0$ according to whether vertices $i$ and $j$
are adjacent (joined by an edge) or not.
The \emph{spectral radius} $\rho(G)$ of a graph $G$ is the largest eigenvalue (which is real and positive) of its adjacency matrix.

\begin{lemma}\label{lemmaGREven}
The growth rate of $\WWWE_G$ exists and is equal to the square of the spectral radius of $G$.
\end{lemma}
\begin{proof}
If $G$ has $n$ vertices, then
$$
W_G(2k)
\;=\;
\sum\limits_{u\in V(G)} \!W_G(2k;u)
\;=\;
\tr(A(G)^{2k})
\;=\;
\sum\limits_{i=1}^n \lambda_i^{2k},
$$
where the $\lambda_i$ are the (real) eigenvalues of $A(G)$, the adjacency matrix of $G$, since the diagonal entries of $A(G)^{2k}$ count the number of tours of length $2k$ starting at each vertex. Thus,
$$
\gr(\WWWE_G)
\;=\;
\limkinfty\Big(\sum\limits_{i=1}^n \lambda_i^{2k}\Big)^{1/k}
$$
Now the spectral radius is given by
$
\rho
=
\rho(G)
=
\max\limits_{1\leqslant i\leqslant n} \lambda_i,
$
so we can conclude that
$$
\rho^2
\;=\;
\limkinfty (\rho^{2k})^{1/k}
\;\leqslant\;
\limkinfty\Big(\sum\limits_{i=1}^n \lambda_i^{2k}\Big)^{1/k}
\;\leqslant\;
\limkinfty \big((n\rho)^{2k}\big)^{1/k}
\;=\;
\rho^2.
$$
Thus, $\gr(\WWWE_G)=\rho(G)^2$.
\end{proof}

\section{Grid classes}\label{sectGridClasses}
In this section, we prove our main theorem, that
the growth rate of a monotone grid class of permutations is
equal to the square of the spectral radius of its row-column graph.

The proof is as follows:
First, we present an explicit expression for the number of \emph{gridded permutations} of a given length.
Then, we
use this to show that the class of gridded permutations grows at the same rate as the family of tours of even length on its row-column graph.
Finally, we utilize the fact that the growth rate of a grid class is the same as the growth rate of the corresponding class of gridded permutations.

\subsection{Counting gridded permutations}\label{sectCountingGriddedPerms}
As we saw in the previous chapter, it is possible to give an explicit expression for the number of
gridded permutations of length $k$ in any specified grid class.
This is essentially the same as Lemma~\ref{lemmaCountGriddings1}, but using different notation.
We repeat the proof for completeness.
Observe the similarity to the formulae for numbers of tours.
\begin{lemma}\label{lemmaCountGriddings}
If $G=G(M)$ is the row-column graph of
$\Grid(M)$,
and $G$ has $m$ edges
$e_1,\ldots,e_m$,
then the number of gridded permutations of length $k$ in $\Gridhash(M)$ is
given by
$$
\big|\Gridhash_k(M)\big|
\;=\;
\sum_{k_1+\ldots+k_m\,=\,k}
\+\+
\prod_{v\in V(G)}\binom{{k^v_1}+{k^v_2}+\ldots+{k^v_{\dg(v)}}}
{{k^v_1}\!,\,\,{k^v_2}\!,\,\,\ldots,\,\,{k^v_{\dg(v)}}}
$$
where $k^v_1,k^v_2,\ldots,k^v_{\dg(v)}$ are the $k_i$
corresponding to
edges incident to $v$ in $G$.
\end{lemma}
\begin{proof}
A gridded permutation in $\Gridhash(M)$ consists of a number of points in each of the cells that correspond to a non-zero entry of $M$. For every permutation, the relative ordering of points (increasing or decreasing) within a particular cell is fixed by the value of the corresponding matrix entry.
However, the relative interleaving between points in distinct cells in the same row or column
can be chosen arbitrarily and independently for each row and column.

Now, each vertex in $G$ corresponds to either a row or a column in $M$, with an incident edge for each non-zero entry in that row or column. Thus,
the number of gridded permutations with $k_i$ points in the cell corresponding to edge $e_i$ for each $i$ is given by the following product of multinomial coefficients:
$$
\prod_{v\in V(G)}\binom{{k^v_1}+{k^v_2}+\ldots+{k^v_{\dg(v)}}}
{{k^v_1}\!,\,\,{k^v_2}\!,\,\,\ldots,\,\,{k^v_{\dg(v)}}}.
$$
The result follows by summing over values of $k_i$ that sum to $k$.
\end{proof}
As an immediate consequence, we have the fact that
the enumeration of a class of gridded permutations depends only on its row-column graph:
\begin{cor}\label{corGRHashEqForSameRCGraph}
If $G(M)=G(M')$, then
$\Gridhash_k(M) = \Gridhash_k(M')$ for all $k$.
\end{cor}

\subsection{Gridded permutations and tours}\label{sectPermsAndTours}
We now
use Lemmas~\ref{lemmaLowerBound} and~\ref{lemmaUpperBoundNew} to
relate the number of gridded permutations of length $k$ in $\Gridhash(M)$ to the number of tours of length $2k$ on $G(M)$. We restrict ourselves to permutation classes with connected row-column graphs.

\begin{lemma}\label{lemmaGRsEqual}
If $G(M)$ is connected, the growth rate of $\Gridhash(M)$ exists and is equal to the growth rate of $\WWWE_{G(M)}$.
\end{lemma}
\begin{proof}
If matrix $M$ has $m$ non-zero entries (and thus $G(M)$ has $m$ edges), then
for any vertex $u$ of $G(M)$,
combining
Lemmas~\ref{lemmaCountGriddings} and~\ref{lemmaLowerBound},
gives us
\begin{align}\label{eqLB}
|\Gridhash_k(M)| &
\;\leqslant\;
\displaystyle
\sum_{k_1+\ldots+k_m\,=\,k}
\WB_{G(M)}(k_1+1,k_2+1,\ldots,k_m+1;u) \nonumber\\[4pt]
& \;\leqslant\;
\displaystyle
\sum_{k_1+\ldots+k_m\,=\,k}
W_{G(M)}(2k_1+2,2k_2+2,\ldots,2k_m+2;u) \nonumber\\[4pt]
& \;\leqslant\;
W_{G(M)}(2k+2m).
\end{align}


On the other hand, from Lemma~\ref{lemmaUpperBoundNew}, for any vertex $u$ of a graph $G$ with $m$ edges, 
$$
W_G(2k;u)
\;\leqslant\;
(m+1)^m\+
(2k+2m)^m
\displaystyle\sum_{j=k}^{k+m^2}
\sum\limits_{k_1+\ldots+k_m\,=\,j}
\+\+
\prod\limits_{v\in V(G)}\dbinom{{k^v_1}+
\ldots+{k^v_{\dg(v)}}}
{{k^v_1},\,
\ldots,\,{k^v_{\dg(v)}}}.
$$

Let $W_G(h)$ be the number of tours of length $h$ on
$G$ (starting at any vertex).

Now $W_G(h)=\sum\limits_{u\in V(G))}\!\!W_G(h;u)$, so, using Lemma~\ref{lemmaCountGriddings}, if $G(M)$ has $n$ vertices and $m$ edges, we have
$$
W_{G(M)}(2k)
\;\leqslant\;
n\+
(m+1)^m\+
(2k+2m)^m
\displaystyle\sum_{j=k}^{k+m^2}
|\Gridhash_j(M)|.
$$
The multiplier on the right side of this inequality is a polynomial in $k$.
Hence, using inequality~\eqref{eqLB} and Observation~\ref{obsGRPolySum}, we can conclude that
\[
\gr(\Gridhash(M))\;=\;\gr(\WWWE_{G(M)})
\]
if $G(M)$ is connected.
\end{proof}

\subsection{Counting permutations}\label{sectCountingPerms}
We nearly have the result we want.
The final link is
Lemma~\ref{lemmaGRGriddings}
which tells us that, as far as growth rates are concerned,
classes of gridded permutations are indistinguishable from their grid classes.
Thus, by Corollary~\ref{corGRHashEqForSameRCGraph}:
\begin{cor}\label{corGREqForSameRCGraph}
Monotone grid classes with the same row-column graph have the same growth rate.
\end{cor}

\subsection{The growth rate of grid classes}
\label{sectProofOfTheorem}
We now have all we need for the proof of our main theorem.

\thmbox{
\begin{thm}\label{thmGrowthRate}
The growth rate of a monotone grid class of permutations exists and is equal to the square of the spectral radius of its row-column graph.
\end{thm}
} 

\begin{proof}
For connected grid classes, the result follows immediately from Lemmas~\ref{lemmaGREven},~\ref{lemmaGRsEqual} and~\ref{lemmaGRGriddings}.
A little more work is required to handle the disconnected case.

If $G(M)$ is disconnected,
then the growth rate of $\Grid(M)$ is the maximum of the growth rates of the grid classes corresponding to the connected components of $G(M)$
(see Proposition 2.10 in Vatter~\cite{Vatter2011}).

Similarly, the spectrum of a disconnected graph is the union (with multiplicities) of the spectra of the graph's connected components (see Theorem 2.1.1 in Cvetkovi\'c, Rowlinson \& Simi\'c~\cite{CRS2010}). Thus the spectral radius of a disconnected graph is the maximum of the spectral radii of its components.

Combining these facts with
Lemmas~\ref{lemmaGREven},~\ref{lemmaGRsEqual} and~\ref{lemmaGRGriddings}
yields
$$ 
\gr(\Grid(M))\;=\;\rho(G(M))^2
$$ 
as required.
\end{proof}

\section{Implications}\label{sectConsequences}
As a consequence of Theorem~\ref{thmGrowthRate}, results concerning the spectral radius of graphs can be translated
into facts about the growth rates of permutation grid classes.
So we now present a number of corollaries
that follow from spectral graph theoretic considerations.
The two recent monographs by Cvetkovi\'c, Rowlinson \& Simi\'c~\cite{CRS2010}
and Brouwer \& Hae\-mers~\cite{BH2012}
provide a valuable overview of spectral graph theory, so, where appropriate, we cite the relevant sections of these (along with the original reference for a result).

As a result of
Corollary~\ref{corGREqForSameRCGraph},
changing the sign of non-zero entries in matrix $M$ has no effect on the growth rate of $\Grid(M)$.
For this reason, when considering particular collections of grid classes below, we choose to represent them by \emph{grid diagrams} in which non-zero matrix entries are represented by a $\!\gxone{1}{1}\!$. As with grid classes, we freely apportion properties of a row-column graph to corresponding grid diagrams.

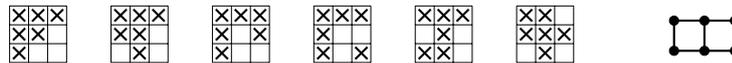
\begin{figure}[ht]
    $$
    \gxthree{3}{1,1,1}{1,1,0}{1,0,0}
    \quad
    \gxthree{3}{1,1,1}{1,1,0}{0,1,0}
    \quad
    \gxthree{3}{1,1,1}{1,0,1}{1,0,0}
    \quad
    \gxthree{3}{1,1,1}{1,0,0}{1,0,1}
    \quad
    \gxthree{3}{1,1,1}{0,1,0}{1,1,0}
    \quad
    \gxthree{3}{1,1,0}{1,1,1}{0,1,0}
    \quad\quad\quad
    \raisebox{-.07in}{\begin{tikzpicture}[scale=0.4]
      \draw [thick] (2,1)--(0,1)--(0,0)--(2,0);
      \draw [thick] (1,0)--(1,1);
      \draw [fill] (0,0) circle [radius=0.15];
      \draw [fill] (1,0) circle [radius=0.15];
      \draw [fill] (2,0) circle [radius=0.15];
      \draw [fill] (0,1) circle [radius=0.15];
      \draw [fill] (1,1) circle [radius=0.15];
      \draw [fill] (2,1) circle [radius=0.15];
    \end{tikzpicture}}
    $$
\caption{Some unicyclic grid diagrams that have the same row-column graph}\label{figGridDiags}
\end{figure}
Since transposing a matrix or
permuting its rows and columns does not change the row-column graph of its grid class, there may be a number of distinct grid diagrams
corresponding to a specific row-column graph (see Figure~\ref{figGridDiags} for an example).

In many cases, we illustrate a result by showing a row-column graph and a corresponding grid diagram.
We display just one of the possible grid diagrams corresponding to the row-column graph.

Our first result is the following elementary observation, which
specifies a limitation on which numbers can be grid class growth rates.
This is a consequence of the fact that the spectral radius of a graph is
a root of the characteristic polynomial of an integer matrix.
\begin{cor}\label{corAlgebraicInteger}
The growth rate of a monotone grid class is an algebraic integer (the root of a monic polynomial).
\end{cor}

\subsection{Slowly growing grid classes}\label{sectSmallGridClasses}
Using results concerning graphs with small spectral radius, we can characterise grid classes with
growth rates no greater
than~$\frac{9}{2}$.
This is
similar to Vatter's characterisation of ``small'' permutation classes (with growth rate less than $\kappa\approx 2.20557$) in~\cite{Vatter2011}.

First, we recall that the growth rate of a disconnected grid class is the maximum of the growth rates of its components 
(see the proof of Theorem~\ref{thmGrowthRate}),
so we only need to consider connected grid classes.

\begin{figure}[ht]
\begin{center}
\renewcommand{\arraystretch}{2.1}
\begin{tabular}{c@{$\qquad$}c@{$\qquad$}c@{$\qquad$}c@{$\qquad$}c}
  \raisebox{0.02in}{\begin{tikzpicture}[scale=0.4]
    \draw [thick] (-.5,.866)--(0,0)--(-.5,-.866)--(-1.5,-.866)--(-2,0)--(-1.5,.866)--(-.5,.866);
    \draw [fill] (0,0) circle [radius=0.15];
    \draw [fill] (-.5,.866) circle [radius=0.15];
    \draw [fill] (-.5,-.866) circle [radius=0.15];
    \draw [fill] (-1.5,.866) circle [radius=0.15];
    \draw [fill] (-1.5,-.866) circle [radius=0.15];
    \draw [fill] (-2,0) circle [radius=0.15];
  \end{tikzpicture}}
&
  \begin{tikzpicture}[scale=0.4]
    \draw [thick] (0,0)--(6,0);
    \draw [thick] (0,1)--(0,0)--(0,-1);
    \draw [thick] (6.0,1)--(6,0)--(6.0,-1);
    \foreach \x in {0,...,6}
      \draw [fill] (\x,0) circle [radius=0.15];
    \draw [fill] (0,1) circle [radius=0.15];
    \draw [fill] (0,-1) circle [radius=0.15];
    \draw [fill] (6,1) circle [radius=0.15];
    \draw [fill] (6,-1) circle [radius=0.15];
  \end{tikzpicture}
&
  \begin{tikzpicture}[scale=0.4]
    \draw [thick] (0,0)--(4,0);
    \draw [thick] (2,0)--(2,2);
    \foreach \x in {0,...,4}
      \draw [fill] (\x,0) circle [radius=0.15];
    \draw [fill] (2,1) circle [radius=0.15];
    \draw [fill] (2,2) circle [radius=0.15];
  \end{tikzpicture}
&
  \raisebox{0.08in}{\begin{tikzpicture}[scale=0.4]
    \draw [thick] (0,0)--(6,0);
    \draw [thick] (3,0)--(3,1);
    \foreach \x in {0,...,6}
      \draw [fill] (\x,0) circle [radius=0.15];
    \draw [fill] (3,1) circle [radius=0.15];
  \end{tikzpicture}}
&
  \raisebox{0.08in}{\begin{tikzpicture}[scale=0.4]
    \draw [thick] (0,0)--(7,0);
    \draw [thick] (5,0)--(5,1);
    \foreach \x in {0,...,7}
      \draw [fill] (\x,0) circle [radius=0.15];
    \draw [fill] (5,1) circle [radius=0.15];
  \end{tikzpicture}}
\\
  \gxthree{3}{0,1,1}{1,0,1}{1,1,0}
&
  \gxfour{7}{0,0,0,0,1,1,1}{0,0,0,1,1,0,0}{0,0,1,1,0,0,0}{1,1,1,0,0,0,0}
&
  \gxthree{4}{1,0,0,1}{1,0,1,0}{1,1,0,0}
&
  \gxthree{5}{0,0,0,1,1}{0,1,1,1,0}{1,1,0,0,0}
&
  \gxfour{5}{0,0,0,0,1}{0,0,1,1,1}{0,1,1,0,0}{1,1,0,0,0}
\end{tabular}
\vspace{-6pt}
\end{center}
\caption{
A cycle, an $H$ graph and the three
other Smith graphs,
with corresponding grid diagrams
}\label{figSmith}
\end{figure}
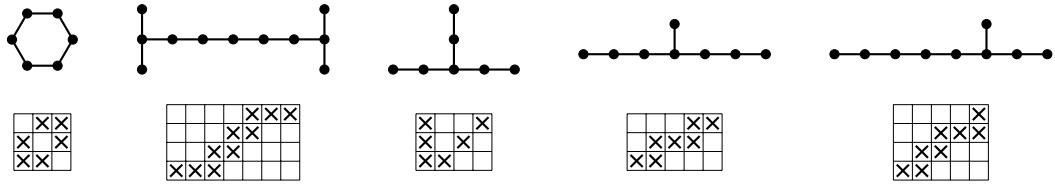
The connected graphs with spectral radius 2 are known as the \emph{Smith} graphs. These are precisely the \emph{cycle} graphs, the $H$ graphs (paths with two pendant edges
attached to both endvertices, including the star graph $K_{1,4}$), and the three other graphs shown in Figure~\ref{figSmith}.

\begin{figure}[ht]
\begin{center}
\renewcommand{\arraystretch}{2.1}
\begin{tabular}{c@{$\qquad$}c@{$\qquad$}c@{$\qquad$}c@{$\qquad$}c}
  \raisebox{0.14in}{\begin{tikzpicture}[scale=0.4]
    \draw [thick] (0,0)--(6,0);
    \foreach \x in {0,...,6}
      \draw [fill] (\x,0) circle [radius=0.15];
  \end{tikzpicture}}
&
  \begin{tikzpicture}[scale=0.4]
    \draw [thick] (0,0)--(5,0);
    \draw [thick] (-.5,.866)--(0,0)--(-.5,-.866);
    \foreach \x in {0,...,5}
      \draw [fill] (\x,0) circle [radius=0.15];
    \draw [fill] (-.5,.866) circle [radius=0.15];
    \draw [fill] (-.5,-.866) circle [radius=0.15];
  \end{tikzpicture}
&
  \raisebox{0.06in}{\begin{tikzpicture}[scale=0.4]
    \draw [thick] (0,0)--(4,0);
    \draw [thick] (2,0)--(2,1);
    \foreach \x in {0,...,4}
      \draw [fill] (\x,0) circle [radius=0.15];
    \draw [fill] (2,1) circle [radius=0.15];
  \end{tikzpicture}}
&
  \raisebox{0.06in}{\begin{tikzpicture}[scale=0.4]
    \draw [thick] (0,0)--(5,0);
    \draw [thick] (3,0)--(3,1);
    \foreach \x in {0,...,5}
      \draw [fill] (\x,0) circle [radius=0.15];
    \draw [fill] (3,1) circle [radius=0.15];
  \end{tikzpicture}}
&
  \raisebox{0.06in}{\begin{tikzpicture}[scale=0.4]
    \draw [thick] (0,0)--(6,0);
    \draw [thick] (4,0)--(4,1);
    \foreach \x in {0,...,6}
      \draw [fill] (\x,0) circle [radius=0.15];
    \draw [fill] (4,1) circle [radius=0.15];
  \end{tikzpicture}}
\\
  \gxthree{4}{0,0,1,1}{0,1,1,0}{1,1,0,0}
&
  \gxfour{5}{0,0,0,0,1}{0,0,0,1,1}{0,0,1,1,0}{1,1,1,0,0}
&
  \gxthree{3}{0,1,1}{0,1,0}{1,1,0}
&
  \gxthree{4}{0,0,0,1}{0,1,1,1}{1,1,0,0}
&
  \gxfour{4}{0,0,1,1}{0,0,1,0}{0,1,1,0}{1,1,0,0}
\end{tabular}
\vspace{-6pt}
\end{center}
\caption{
A path, a $Y$ graph and the three
other connected proper subgraphs of Smith graphs,
with corresponding grid diagrams
}\label{figSubSmith}
\end{figure}
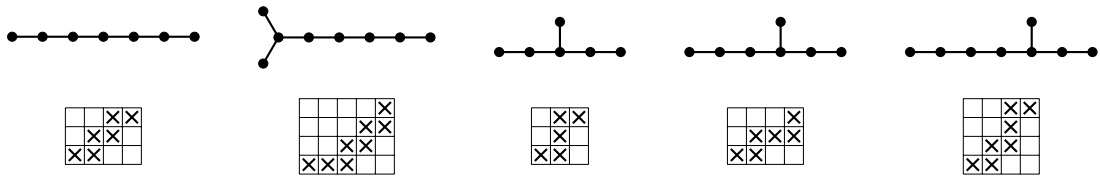
Similarly, the connected proper
subgraphs of the Smith graphs are precisely the \emph{path} graphs, the $Y$ graphs (paths with two pendant edges attached to one endvertex)
and the three other graphs in Figure~\ref{figSubSmith}.
For details, see
Smith~\cite{Smith1970} and Lemmens \& Seidel~\cite{LS1973}; 
also see~\cite[Theorem 3.11.1]{CRS2010} and~\cite[Theorem 3.1.3]{BH2012}.

\Needspace*{2\baselineskip}
With these, we can characterise all grid classes with growth rate no greater than 4:
\begin{cor}\label{corSmith}
If the growth rate of a connected monotone grid class equals 4, then its row-column graph is a Smith graph.
If the growth rate of a connected monotone grid class is less than 4, then its row-column graph is a connected proper
subgraph of a Smith graph.
\end{cor}
In particular, we have the following:
\begin{cor}\label{corCycle}
A monotone grid class
of any size
whose row-column graph is a cycle
or
an $H$ graph
has growth rate
4.
\end{cor}

In Appendix~A of~\cite{Vatter2011}, Vatter considers \emph{staircase} grid classes, whose row-column graphs are paths (see the leftmost grid diagram in Figure~\ref{figSubSmith}).
The 
spectral radius of a path graph has long been known (Lov\'asz \& Pelik\'an~\cite{LP1973}; also see~\cite[Theorem~8.1.17]{CRS2010} and \cite[1.4.4]{BH2012}), from which we can conclude:
\begin{cor}\label{corPath}
A monotone grid class
of size $m$
(having $m$ non-zero cells)
whose row-column graph is a path
has growth rate $4\cos^2\!\left(\frac{\pi}{m+2}\right)$.
This is minimal for any connected grid class of size $m$.
\end{cor}
A $Y$ graph
of size $m$
has spectral radius
$2\cos\!\left(\frac{\pi}{2m}\right)$,
and the spectral radii of
the three other graphs at the right of Figure~\ref{figSubSmith} are $2\cos\!\left(\frac{\pi}{12}\right)$,
$2\cos\!\left(\frac{\pi}{18}\right)$,
and
$2\cos\!\left(\frac{\pi}{30}\right)$,
from left to right (see~\cite[3.1.1]{BH2012}).
Thus
we have the following characterisation of
growth rates less than~4:
\begin{cor}\label{corSmallGrowthRates}
If the growth rate of a monotone grid class is less than 4, it is equal to $4\cos^2\!\left(\frac{\pi}{k}\right)$ for some integer $k\geqslant 3$.
\end{cor}
The only grid class growth rates
no greater than 3 are $1$, $2$, $\half(3+\sqrt{5})\approx2.618$, and $3$.

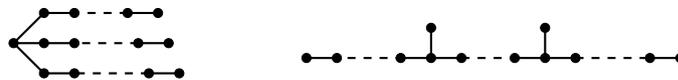
\begin{figure}[ht]
  $$
  \raisebox{-.08in}
    {\begin{tikzpicture}[scale=0.4]
    \draw [thick] (0,0)--(2,0);
    \draw [thick,dashed] (4.1,0)--(2,0);
    \draw [thick] (4.1,0)--(5.1,0);
    \draw [thick] (2,1)--(1,1)--(0,0)--(1,-1)--(2,-1);
    \draw [thick] (3.75,1)--(4.75,1);
    \draw [thick] (4.45,-1)--(5.45,-1);
    \draw [thick,dashed] (3.75,1)--(2,1);
    \draw [thick,dashed] (4.45,-1)--(2,-1);
    \draw [fill] (0,0) circle [radius=0.15];
    \foreach \x in {1,2,3.75,4.75} \draw [fill] (\x,1) circle [radius=0.15];
    \foreach \x in {1,2,4.1,5.1} \draw [fill] (\x,0) circle [radius=0.15];
    \foreach \x in {1,2,4.45,5.45} \draw [fill] (\x,-1) circle [radius=0.15];
  \end{tikzpicture}}
  \quad\quad\quad\quad
  \begin{tikzpicture}[scale=0.4]
    \draw [thick] (-0.35,0)--(0.65,0);
    \draw [thick,dashed] (1,0)--(2.75,0);
    \draw [thick] (2.75,0)--(4.75,0);
    \draw [thick] (3.75,0)--(3.75,1);
    \draw [thick,dashed] (4.75,0)--(6.5,0);
    \draw [thick] (6.5,0)--(8.5,0);
    \draw [thick] (7.5,0)--(7.5,1);
    \draw [thick,dashed] (8.5,0)--(10.95,0);
    \draw [thick] (10.95,0)--(11.95,0);
    \foreach \x in {-0.35,0.65,2.75,3.75,4.75,6.5,7.5,8.5,10.95,11.95} \draw [fill] (\x,0) circle [radius=0.15];
    \draw [fill] (3.75,1) circle [radius=0.15];
    \draw [fill] (7.5,1) circle [radius=0.15];
  \end{tikzpicture}
  $$
\caption{$E$ and $F$ graphs}\label{figEFGraphs}
\end{figure}
In order to characterise grid classes with growth rates slightly greater than 4,
let
an $E$ graph be a tree consisting of three paths having one endvertex in common, and an $F$ graph be a tree consisting of a path with a pendant edge attached to each of two distinct internal vertices (see Figure~\ref{figEFGraphs}).
Then,
results of
Brouwer \& Neumaier~\cite{BN1989} and Cvetkovi\'{c}, Doob \& Gutman~\cite{CDG1982}
imply the following (also see~\cite[Theorem 3.11.2]{CRS2010}):
\begin{cor}
If a connected monotone grid class has growth rate
between 4 and $2+\sqrt{5}$,
then its row-column graph is an $E$ or $F$ graph.
\end{cor}

Thus, since $\sqrt{2+\sqrt{5}}$ cannot be an eigenvalue of any graph (see~\cite[p.~93]{CRS2010}), we can deduce the following:
\begin{cor}\label{corProperCycle}
If a monotone grid class properly contains a cycle then its growth rate exceeds $2+\sqrt{5}$.
\end{cor}
More recently, Woo \& Neumaier~\cite{WN2007}
have investigated the structure of graphs with spectral radius no greater than $\frac{3}{2}\sqrt{2}$
(also see~\cite[Theorem 3.11.3]{CRS2010}).
As a consequence,
we have the following:
\begin{cor}
If the growth rate of
a connected monotone grid class is no greater than $\frac{9}{2}$, then its row-column graph is one of the following: \vspace{-12pt}
\begin{enumerate}[~~~~~(a)]
  \itemsep0pt
  \item a tree of maximum degree 3 such that all vertices of degree 3 lie on a path,
  \item a unicyclic graph of maximum degree 3 such that all vertices of degree 3 lie on the cycle, or
  \item a tree consisting of a path with three pendant edges attached to one endvertex.
\end{enumerate}
\end{cor}

\subsection{Accumulation points of grid class growth rates}\label{sectAccumulationPoints}
Using graph theoretic results of Hoffman and Shearer, it is possible to characterise \emph{all} accumulation points of grid class growth rates.

As we have seen, the growth rates of grid classes whose row-column graphs are paths and $Y$ graphs grow to 4 from below; 4 is the least accumulation point of growth rates. The following characterises all
accumulation points
below $2+\sqrt{5}$ (see Hoffman~\cite{Hoffman1972}):
\begin{cor}
For $k=1,2,\ldots$, let $\beta_k$ be the positive root of
$$P_k(x)=x^{k+1}-(1+x+x^2+\ldots+x^{k-1})$$
and let
$\gamma_k=2+\beta_k+\beta_k^{-1}$.
Then $4=\gamma_1<\gamma_2<\ldots$ are all the accumulation points of growth rates of
monotone grid classes smaller than $2+\sqrt{5}$.
\end{cor}
The approximate values of
the first eight accumulation points
are:
4, 4.07960, 4.14790, 4.18598, 4.20703, 4.21893, 4.22582, 4.22988.

At $2+\sqrt{5}$, things change dramatically; from this value upwards grid class growth rates are dense (see Shearer~\cite{Shearer1989}):
\begin{cor}\label{corAccumulationPoints}
Every $\gamma\geqslant2+\sqrt{5}$ is an accumulation point of growth rates of
monotone grid classes.
\end{cor}

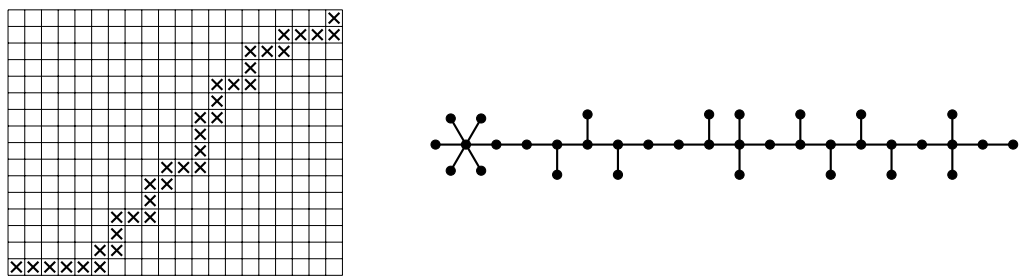
\begin{figure}[ht]
$$
\gclass[0.22]{20}{16}{
\gxrow{15}{0,0,0,0,0,0,0,0,0,0,0,0,0,0,0,0,0,0,0,1}
\gxrow{14}{0,0,0,0,0,0,0,0,0,0,0,0,0,0,0,0,1,1,1,1}
\gxrow{13}{0,0,0,0,0,0,0,0,0,0,0,0,0,0,1,1,1,0,0,0}
\gxrow{12}{0,0,0,0,0,0,0,0,0,0,0,0,0,0,1,0,0,0,0,0}
\gxrow{11}{0,0,0,0,0,0,0,0,0,0,0,0,1,1,1,0,0,0,0,0}
\gxrow{10}{0,0,0,0,0,0,0,0,0,0,0,0,1,0,0,0,0,0,0,0}
\gxrow{ 9}{0,0,0,0,0,0,0,0,0,0,0,1,1,0,0,0,0,0,0,0}
\gxrow{ 8}{0,0,0,0,0,0,0,0,0,0,0,1,0,0,0,0,0,0,0,0}
\gxrow{ 7}{0,0,0,0,0,0,0,0,0,0,0,1,0,0,0,0,0,0,0,0}
\gxrow{ 6}{0,0,0,0,0,0,0,0,0,1,1,1,0,0,0,0,0,0,0,0}
\gxrow{ 5}{0,0,0,0,0,0,0,0,1,1,0,0,0,0,0,0,0,0,0,0}
\gxrow{ 4}{0,0,0,0,0,0,0,0,1,0,0,0,0,0,0,0,0,0,0,0}
\gxrow{ 3}{0,0,0,0,0,0,1,1,1,0,0,0,0,0,0,0,0,0,0,0}
\gxrow{ 2}{0,0,0,0,0,0,1,0,0,0,0,0,0,0,0,0,0,0,0,0}
\gxrow{ 1}{0,0,0,0,0,1,1,0,0,0,0,0,0,0,0,0,0,0,0,0}
\gxrow{ 0}{1,1,1,1,1,1,0,0,0,0,0,0,0,0,0,0,0,0,0,0}
}
\quad\quad\quad
\raisebox{.5in}{\begin{tikzpicture}[scale=0.4]
  \draw [thick] (0,0)--(19,0);
  \foreach \x in {0,...,19} \draw [fill] (\x,0) circle [radius=0.15];
  \foreach \x in {4,6,10,13,15,17}
  { \draw [fill] (\x,-1) circle [radius=0.15]; \draw [thick] (\x,0)--(\x,-1); }
  \foreach \x in {10,17,5,9,12,14}
  { \draw [fill] (\x,1) circle [radius=0.15]; \draw [thick] (\x,0)--(\x,1); }
  \foreach \x in {0.5,1.5}
  {
    \draw [fill] (\x,.866) circle [radius=0.15];
    \draw [fill] (\x,-.866) circle [radius=0.15];
    \draw [thick] (1,0)--(\x,.866);
    \draw [thick] (1,0)--(\x,-.866);
  }
\end{tikzpicture}}
$$
\caption{A grid diagram
whose growth rate differs from $2\pi$ by less than $10^{-7}$,
and its caterpillar row-column graph}\label{figCaterpillar}
\end{figure}
Thus, for every $\gamma\geqslant2+\sqrt{5} \approx4.236068$,
there is a
grid class with growth rate arbitrarily close to $\gamma$.
Indeed, for $\gamma>2+\sqrt{5}$, Shearer's proof provides an iterative process for generating a sequence of grid classes,
each with a row-column graph that
is a \emph{caterpillar}
(a tree such that all vertices of degree 2 or more lie on a path), with growth rates converging to $\gamma$ from below.
An example is shown in Figure~\ref{figCaterpillar}.

\subsection{Increasing the size of a grid class}
We now consider
the effect on the growth rate of making
small changes to a grid class.

Firstly, growth rates of connected grid classes satisfy a \emph{strict} monotonicity condition (see~\cite[Proposition 1.3.10]{CRS2010}):
\begin{cor}\label{corAddCell}
Adding a non-zero cell to a connected monotone grid class
while preserving connectivity
increases its growth rate.
\end{cor}
On the other hand,
particularly surprising
is the fact that grid classes with \emph{longer internal} paths or cycles have \emph{lower} growth rates.

An edge $e$ of $G$ is said to lie on an \emph{endpath} of $G$ if $G-e$ is disconnected and one of its components is a (possibly trivial) path. An edge that does not lie on an endpath is said to be \emph{internal}.
Note that a graph has an internal edge if and only if it contains either a cycle or non-star $H$ graph.

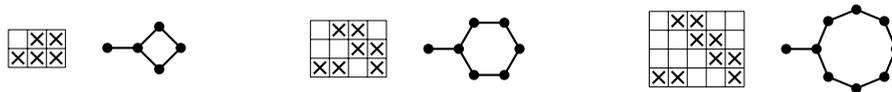
\begin{figure}[ht]
    $$
    \gxtwo{3}{0,1,1}{1,1,1}
    \quad
    \raisebox{-0.09in}{\begin{tikzpicture}[scale=0.4]
    \draw [thick] (0,0)--(1,0)--(1.707,0.707)--(2.414,0)--(1.707,-0.707)--(1,0);
    \draw [fill] (0,0) circle [radius=0.15];
    \draw [fill] (1,0) circle [radius=0.15];
    \draw [fill] (2.414,0) circle [radius=0.15];
    \draw [fill] (1.707,0.707) circle [radius=0.15];
    \draw [fill] (1.707,-0.707) circle [radius=0.15];
    \end{tikzpicture}}
    \quad\quad\quad\quad
    \gxthree{4}{0,1,1}{0,0,1,1}{1,1,0,1}
    \quad
    \raisebox{-0.12in}{\begin{tikzpicture}[scale=0.4]
    \draw [thick] (0,0)--(1,0)--(1.5,0.866)--(2.5,0.866)--(3,0)--(2.5,-0.866)--(1.5,-0.866)--(1,0);
    \draw [fill] (0,0) circle [radius=0.15];
    \draw [fill] (1,0) circle [radius=0.15];
    \draw [fill] (3,0) circle [radius=0.15];
    \draw [fill] (1.5,0.866) circle [radius=0.15];
    \draw [fill] (2.5,0.866) circle [radius=0.15];
    \draw [fill] (1.5,-0.866) circle [radius=0.15];
    \draw [fill] (2.5,-0.866) circle [radius=0.15];
    \end{tikzpicture}}
    \quad\quad\quad\quad
    \gxfour{5}{0,1,1}{0,0,1,1}{0,0,0,1,1}{1,1,0,0,1}
    \quad
    \raisebox{-0.19in}{\begin{tikzpicture}[scale=0.4]
    \draw [thick] (0,0)--(1,0)--(1.383,0.924)--(2.307,1.307)--(3.23,0.924)--(3.613,0)--(3.23,-0.924)--(2.307,-1.307)--(1.383,-0.924)--(1,0);
    \draw [fill] (0,0) circle [radius=0.15];
    \draw [fill] (1,0) circle [radius=0.15];
    \draw [fill] (3.613,0) circle [radius=0.15];
    \draw [fill] (1.383,0.924) circle [radius=0.15];
    \draw [fill] (1.383,-0.924) circle [radius=0.15];
    \draw [fill] (3.23,0.924) circle [radius=0.15];
    \draw [fill] (3.23,-0.924) circle [radius=0.15];
    \draw [fill] (2.307,1.307) circle [radius=0.15];
    \draw [fill] (2.307,-1.307) circle [radius=0.15];
    \end{tikzpicture}}
    \vspace{-3pt}
    $$
\caption{Three unicyclic grid diagrams, of increasing size but decreasing growth rate from left to right, and their row-column graphs}\label{figSubdivision}
\end{figure}
An early result of Hoffman \& Smith~\cite{HS1975}
shows that the
subdivision of an internal edge \emph{reduces} the spectral radius (also see~\cite[Proposition 3.1.4]{BH2012} and~\cite[Theorem 8.1.12]{CRS2010}). Hence, we can deduce the following unexpected consequence for grid classes:
\begin{cor}
\label{corSubdivision}
If $\Grid(M)$ is connected,
and $G(M')$ is obtained from $G(M)$ by subdividing an internal edge, then $\gr(\Grid(M'))<\gr(\Grid(M))$
unless $G(M)$ is a cycle or an $H$ graph.
\end{cor}
\vspace{-7pt}
For an example, see Figure~\ref{figSubdivision}.

\subsection{Grid classes with extremal growth rates}
Finally, we briefly consider grid classes with maximal or minimal growth rates for their size.

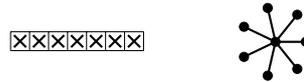
\begin{figure}[ht]
  $$
  \gxone{7}{1,1,1,1,1,1,1}
  \quad\quad\quad
  \raisebox{-.16in}{\begin{tikzpicture}[scale=0.40]
    \draw [thick] (.716,-.898)--(0,0)--(.716,.898);
    \draw [thick] (-.256,-1.121)--(0,0)--(-.256,1.121);
    \draw [thick] (-1.036,-.499)--(0,0)--(-1.036,.499);
    \draw [thick] (0,0)--(1,0);
    \draw [fill] (0,0) circle [radius=0.15];
    \draw [fill] (1,0) circle [radius=0.15];
    \draw [fill] (.716,-.898) circle [radius=0.15];
    \draw [fill] (-.256,-1.121) circle [radius=0.15];
    \draw [fill] (-1.036,-.499) circle [radius=0.15];
    \draw [fill] (.716,.898) circle [radius=0.15];
    \draw [fill] (-.256,1.121) circle [radius=0.15];
    \draw [fill] (-1.036,.499) circle [radius=0.15];
  \end{tikzpicture}}
  \vspace{-3pt}
  $$
\caption{A skinny grid diagram
and its row-column star graph}\label{figSkinny}
\end{figure}
The row-column graph of a skinny grid class is a \emph{star} (see Figure~\ref{figSkinny}).
Stars have maximal spectral radius among trees (see~\cite[Theorem 8.1.17]{CRS2010}). This yields:
\begin{cor}\label{corSkinny}
Among all connected acyclic monotone grid classes of size $m$,
the skinny grid classes
have the largest growth rate
(equal to $m$).
\end{cor}

We have already seen (Corollary~\ref{corPath}) that the
connected grid classes with \emph{least} growth rates are those whose row-column graph is a path. For unicyclic grid classes, we have the following (see~\cite[Theorem 8.1.18]{CRS2010}):
\begin{cor}\label{corCycleMin}
Among all connected unicyclic monotone grid classes of size $m$, those
whose row-column graph is a single cycle of length $m$
have the smallest growth rate
(equal to 4).
\end{cor}
There are many additional results known concerning graphs with extremal values for their spectral radii, especially for graphs with a small number of cycles.
For an example, see the two papers by Simi\'c~\cite{Simic1987a,Simic1989}
on the largest eigenvalues of unicyclic and bicyclic graphs.
Results like these
can be translated into further facts concerning the growth rates of grid classes.

\cleardoublepage


\newcommand{\gammaM}{\gamma_{\!M}^{\phantom{n}}}
\newcommand{\gammaMn}{{\gammaM}^{\!\!n}}

\chapter{Grid class limit shapes}\label{chap06}

In this chapter, we briefly investigate the shape of a ``typical'' large permutation in
a monotone grid class. We show that almost all large permutations in a grid class do, in fact, have the same shape,
and present a way of discovering this shape.

As we saw above (Lemma~\ref{lemmaCountGriddings1}),
if $M$ has dimensions $r\times s$, then
the number of gridded permutations of length $n$ in $\Gridhash(M)$ is
given by a sum of products of multinomial coefficients, one for each column and one for each row:
\begin{equation*}
\big|\Gridhash_n(M)\big|
\;=\;
\raisebox{-4.5pt}{\fontsize{24.88pt}{0pt}\selectfont $\sum$}
\prod_{i=1}^r \binom{a_{i,1}+a_{i,2}+\ldots+a_{i,s}}{a_{i,1},a_{i,2},\ldots,a_{i,s}}
\prod_{j=1}^s \binom{a_{1,j}+a_{2,j}+\ldots+a_{r,j}}{a_{1,j},a_{2,j},\ldots,a_{r,j}} ,
\end{equation*}
where the sum is over all combinations of non-negative $a_{i,j}$ such that $\sum_{i,j}a_{i,j}=n$ and $a_{i,j}=0$ if $M_{i,j}=0$.
Each $a_{i,j}$ represents the number of points in the $(i,j)$ cell.

To determine grid class limit shapes, we focus on the \emph{proportion} $\alpha_{i,j}=a_{i,j}/n$ of points in each cell.
(Vatter~\cite[Appendix A.4]{Vatter2011} takes a similar approach to yield a continued fraction expression for the growth rate of monotone grid classes whose row-column graph is a path.)

Given a $0/\pm1$ matrix $M=(m_{i,j})$ with dimensions $r\times s$, let us use $\mathbf{A}(M)$ to denote the 
set of real $r\times s$ 
matrices $A=(\alpha_{i,j})$ with nonnegative entries, such that $\alpha_{i,j}=0$ if $m_{i,j}=0$ and $\sum\limits_{i,j}^{\phantom{.}}\alpha_{i,j}=1$.
For example,
$
\begin{smallmx}0.2&0.7\\ 0.1&0\end{smallmx}
\:\in\:
\mathbf{A}\begin{smallmx}-1&1\\ -1&0\end{smallmx}.
$

Given a matrix $A=(\alpha_{i,j})\in\mathbf{A}(M)$, let $\Gridhash(A)$ be the set of gridded permutations in $\Gridhash(M)$ for which the proportion of points in each cell is given by the values of the $\alpha_{i,j}$. Thus,
$$
\Gridhash(M) \;=\; \biguplus_{A\in\mathbf{A}(M)} \!\Gridhash(A).
$$
Note that, if $A=(\alpha_{i,j})$, then $\Gridhash_n(A)$ is nonempty only if, for each $i,j$, the number $\alpha_{i,j}n$ is an integer.

We claim that, asymptotically, the proportion of points in each cell is the same for
almost all large gridded permutations in $\Gridhash(M)$.
To state this formally, we require the following definition:
For a real matrix $A$, let $\mathbf{A}\raisebox{-2pt}{$\varepsilon$}(A)$ be the set of matrices that are \emph{$\varepsilon$-close} to $A$:
$$
\mathbf{A}\raisebox{-2pt}{$\varepsilon$}(A) \;=\; \big\{B : \big|\!\big|B-A\big|\!\big|_{\max}\leqslant\varepsilon \big\},
$$
where $\big|\!\big|(a_{i,j})\big|\!\big|_{\max}=\max\limits_{i,j}|a_{i,j}|$ is the ``entrywise'' max norm.

\thmbox{
\begin{prop}\label{propGriddedPermsLimitShape}
  For any connected class of gridded permutations, $\Gridhash(M)$, there is a unique $A_M\in\mathbf{A}(M)$, such that for all $\varepsilon>0$,
  $$
  %
  \big|\Gridhash_n(M)\big|
  \;\sim\;
  \sum_{B\in\mathbf{A}\raisebox{0pt}{${}_\varepsilon$}(A_M)} \big|\Gridhash_n(B)\big|
  $$
\end{prop}
} 

Rather than giving a formal proof, we just demonstrate the result by considering the simplest nontrivial grid class, $\CCC=\Gridhash(\!\raisebox{-1pt}{\setgcscale{0.32}\gcone{2}{1,1}}\!)$.

If $\alpha n$ is an integer, then the number of gridded permutations in $\CCC$ of length $n$ that have $\alpha n$ points in the left cell and
$(1-\alpha)n$ points in the right cell is $\binom{n}{\alpha n}$.
Applying Stirling's approximation, asymptotically we have
$$
\big|\Gridhash_n(\alpha,1-\alpha)\big|
\;\sim\;
\frac{1}{\sqrt{2\pi\alpha(1-\alpha)n}} \left(\alpha^{-\alpha}(1-\alpha)^{-(1-\alpha)}\right)^{\!n}
\;=\;
c_\alpha \+ n^{-\nfrac{1}{2}} \+ \gamma_{\!\alpha}^{\+n}.
$$
It can easily be confirmed using elementary calculus that the exponential term, $\gamma_{\!\alpha}$, is unimodal and takes a maximum value (of 2) when $\alpha=\half$. Let $A_M=(\half,\half)$.

Partitioning $\CCC$ between, firstly, those
gridded permutations whose
distribution of points is $\varepsilon$-close to $A_M$ and, secondly, those
whose
distribution of points is \emph{not} $\varepsilon$-close to $A_M$ yields
\begin{equation}\label{eqGriddedPermPartition}
\big|\CCC_n\big|
\;=\;
\sum_{B\in\mathbf{A}\raisebox{0pt}{${}_\varepsilon$}(\half,\half)} \!\big|\Gridhash_n(B)\big|
\:+\:
\sum_{B\in\mathbf{A}(1,1) \setminus \mathbf{A}\raisebox{0pt}{${}_\varepsilon$}(\half,\half)} \!\big|\Gridhash_n(B)\big|.
\end{equation}
Now, there is always some $\alpha\in\big[\half-\tfrac{1}{2n},\half\big]$
for which 
$\Gridhash_n(\alpha,1-\alpha)$ is nonempty.
Hence, for large enough $n$, the first summand in~\eqref{eqGriddedPermPartition} is at least
$c_{\alpha_1} \+ n^{-\nfrac{1}{2}} \+ \gamma_{\!\alpha_1}^{\+n}$,
where
$\alpha_1=\half-\tfrac{\varepsilon}{2}$.
On the other hand,
the second summand is no more than
$n \+ c_{\alpha_2} \+ n^{-\nfrac{1}{2}} \+ \gamma_{\!\alpha_2}^{\+n}$,
where
$\alpha_2=\half-\varepsilon$.
So, since $\gamma_{\!\alpha_1}>\gamma_{\!\alpha_2}$,
the set of gridded permutations whose asymptotic distribution of points is $\varepsilon$-close to $A_M$ grows exponentially faster than the set of those whose distribution of points is not $\varepsilon$-close to $A_M$.
Thus, asymptotically, almost all of the gridded permutations in $\CCC$
have a distribution of points $\varepsilon$-close to $A_M$.

We claim that this is true of any class of gridded permutations as long as its row-column graph is connected.

Now, we know (see Lemma~\ref{lemmaGRGriddings})
that the maximum number of griddings of permutations in a monotone grid class grows only polynomially with length.
Hence, asymptotically almost all permutations in $\Grid(M)$ have griddings whose distribution is $\varepsilon$-close to $A_M$. Moreover, it can be shown that, of all the ways of interleaving points in cells in the same row or column, asymptotically almost all of them distribute the points evenly. Determining $A_M$ is thus sufficient to give us
the \emph{limit shape} of permutations in $\Grid(M)$, showing what 
a typical
large gridded permutation in the class looks like.
\begin{figure}[ht]
  \begin{center}
    \includegraphics[scale=0.33]{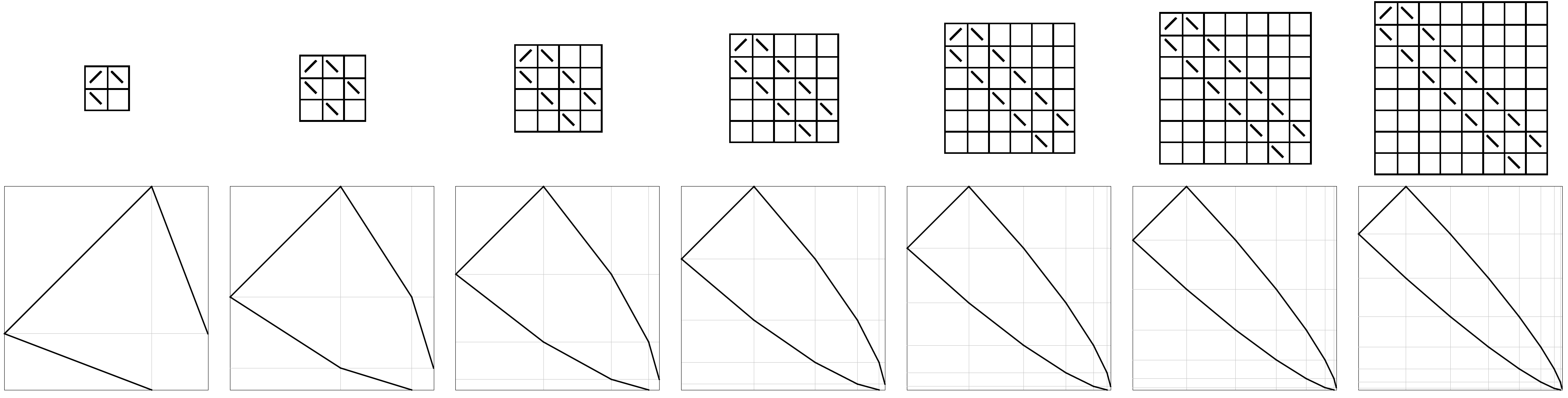}
  \end{center}
  \caption{Limit shapes of monotone grid classes whose row-column graphs are paths}\label{figLimitShapes}
\end{figure}

Let us now consider how to ascertain the limiting distribution of points.
When applied to a multinomial coefficient, Stirling's approximation gives the following asymptotic form:
$$
\binom{\tau\+n}{\beta_1 n,\beta_2 n,\ldots,\beta_k n}
\;\sim\;
\sqrt{\frac{\tau}{(2\+\pi)^{k-1}\+\beta_1\+\beta_2\ldots\beta_k}}
\+
n^{-(k-1)/2}
\+
\left(\tau^\tau\+\beta_1^{-\beta_1}\+\beta_2^{-\beta_2}\ldots\+\beta_k^{-\beta_k}\right)^{\!n},
$$
where $\tau=\beta_1+\beta_2+\ldots+\beta_k$ and each $\beta_i$ is positive.

Thus if
$A=(\alpha_{i,j})$ has dimensions $r\times s$ and all the $\alpha_{i,j}$ are rational,
$$
\grup(\Gridhash(A))
\;=\;
\prod_{i=1}^r 
\frac {{\kappa_i}^{\kappa_i}} {\prod_j{\alpha_{i,j}}^{\alpha_{i,j}}}
\,\times\:
\prod_{j=1}^s 
\frac {{\rho_j}^{\rho_j}} {\prod_i{\alpha_{i,j}}^{\alpha_{i,j}}}
,
$$
where $\kappa_i=\sum_j\alpha_{i,j}$ is the proportion of points in column $i$
and $\rho_j=\sum_i\alpha_{i,j}$ is the proportion of points in row $j$,
and we take sums and products over the nonzero $\alpha_{i,j}$ only.

Hence, by an argument analogous to that used above in analysing the asymptotic structure of $\Gridhash(\!\raisebox{-1pt}{\setgcscale{0.32}\gcone{2}{1,1}}\!)$,
and the fact that $\gr(\Grid(M))=\gr(\Gridhash(M))$ (see Lemma~\ref{lemmaGRGriddings}),
we have the following consequence of Proposition~\ref{propGriddedPermsLimitShape}:

\newpage  
\thmbox{
\begin{lemma}\label{lemmaAsymptGridGR}
The growth rate of a connected monotone grid class, $\Grid(M)$, is given by
$$
  \gr(\Grid(M))
  \;=\;
  \max_{(\alpha_{i,j})\in\AAA(M)} \prod_{i=1}^r
  \;
  \frac {{\kappa_i}^{\kappa_i}} {\prod_j{\alpha_{i,j}}^{\alpha_{i,j}}}
  \,\times\:
  \prod_{j=1}^s
  \frac {{\rho_j}^{\rho_j}} {\prod_i{\alpha_{i,j}}^{\alpha_{i,j}}}
  ,
$$
where $\kappa_i=\sum_j\alpha_{i,j}$ and $\rho_j=\sum_i\alpha_{i,j}$.


Moreover, the maximum is achieved uniquely by the matrix $A_M$, specified in Proposition~\ref{propGriddedPermsLimitShape}, which indicates the limiting distribution of points in each cell.
\end{lemma}
}

To determine the limiting distribution, we can make use of the following result:

\thmbox{
\begin{prop}\label{propAsymptLagrange}
  Let $A=({\alpha}_{i,j})$ be the distribution of points in a gridding of a typical large permutation in a connected monotone grid class.
  Then there is a constant $C$ such that,
  for every nonzero ${\alpha}_{i,j}$, we have ${{\alpha}_{i,j}}^2/{\kappa}_i{\rho}_j=C$, where ${\kappa}_i=\sum_j{\alpha}_{i,j}$ and ${\rho}_j=\sum_i{\alpha}_{i,j}$.
\end{prop}
} 

\begin{proof}
We use the method of Lagrange multipliers.
Let $f(A)$ represent the expression from Lemma~\ref{lemmaAsymptGridGR} that we need to maximise. We
introduce the auxiliary function
$$
\Lambda(A,\lambda) \;=\; \log f(A)-\lambda\big(1-\sum_{i,j}\alpha_{i,j}\big).
$$
The limiting distribution is then given by the solution to the set of equations
$$
\frac{\partial\Lambda(A,\lambda)}{\partial\alpha_{i,j}} \;=\; 0,
$$
one for each nonzero $\alpha_{i,j}$, together with the original constraint $\sum_{i,j}\alpha_{i,j}=1$.

Now, for each $\alpha_{i,j}$,
$$
\frac{\partial\Lambda(A,\lambda)}{\partial\alpha_{i,j}}
\;=\;
\lambda \:-\: 2\log\alpha_{i,j}
\:+\:
\log \kappa_i 
\:+\:
\log \rho_j 
.
$$
Rearrangement and exponentiation then yields the fact that, for every $\alpha_{i,j}$, the expression
$
{{\alpha_{i,j}}^{2}}
/
{
\kappa_i 
\rho_j 
} 
$
has the same value ($e^\lambda$).
\end{proof}

Proposition~\ref{propAsymptLagrange} can be used to generate a set of polynomial equations for $A_M$ which can, in theory, be solved.
By applying Lemma~\ref{lemmaAsymptGridGR}, this gives a method for determining the growth rate of a grid class.
In practice, this is rather harder than
finding the growth rate by
evaluating the spectral radius of the row-column graph.
However it is useful if we desire to know the limit shape.

Let's work through an example: $\Gridhash\big(\!\gctwo{3}{-1,-1,-1}{-1,0,0}\!\big)$.

For this class, the equations for the entries in 
$A_M=\begin{smallmx}\alpha_{1,2}&\alpha_{2,2}&\alpha_{3,2}\\ \alpha_{1,1}&0&0\end{smallmx}$ are
$$
\tfrac{{\alpha_{1,1}}}{{\alpha_{1,1}}+{\alpha_{1,2}}}
\;=\;
\tfrac{{\alpha_{1,2}}^2}{({\alpha_{1,1}}+{\alpha_{1,2}}) ({\alpha_{1,2}}+{\alpha_{2,2}}+{\alpha_{3,2}})}
\;=\;
\tfrac{{\alpha_{2,2}}}{{\alpha_{1,2}}+{\alpha_{2,2}}+{\alpha_{3,2}}}
\;=\;
\tfrac{{\alpha_{3,2}}}{{\alpha_{1,2}}+{\alpha_{2,2}}+{\alpha_{3,2}}},
$$
together with
$$
{\alpha_{1,1}}+{\alpha_{1,2}}+{\alpha_{2,2}}+{\alpha_{3,2}} \;=\; 1 .
$$
These equations have the unique positive solution
$$
  \alpha_{1,1}  \;=\;  \tfrac{1}{4}(2-\sqrt{2}) , \qquad\quad
  \alpha_{1,2}  \;=\;  \tfrac{1}{2\sqrt{2}} , \qquad\quad
  \alpha_{2,2}  \;=\;  
  \alpha_{3,2}  \;=\;  \tfrac{1}{4} .
$$
Thus, we discover that, in $\Gridhash\big(\!\gctwo{3}{-1,-1,-1}{-1,0,0}\!\big)$, 
almost all large permutations look like this:
\vspace{-6pt}
\begin{center}
\includegraphics[scale=0.85]{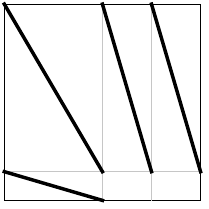}
\end{center}

The limit shapes of some other grid classes are illustrated in Figure~\ref{figLimitShapes}. We encounter similar pictures when considering the structure of permutations in the class $\av(\pdiamond)$ in the final chapter of this thesis.

One interpretation of the limit shape of a monotone grid class is as the
measure-theoretic limit of convergent sequences of permutations. 
We conclude this chapter with a very brief sketch of this concept.

One of the most exciting recent developments
in graph theory is the idea of the
limit of a convergent sequence of finite graphs, introduced by
Lov\'asz \&
Szegedy~\cite{LS2006,Lovasz2012},
and
known as a \emph{graphon}. A graphon is a Lebesgue measurable
function $W$ from $[0,1]^2$ to $[0,1]$ that is symmetric in its two arguments, i.e., $W(x,y)=W(y,x)$ for every $x,y\in[0,1]$.

A similar analytic
limit object,
called a \emph{permuton}~\cite{GGKK2013},
has also been defined recently for sequences of permutations
by
Hoppen, Kohayakawa, Sampaio and their collaborators~\cite{BHKS2011,HKMRS2013,HKMS2011b,HKMS2011,HKS2012}.
Formally, a permuton is a probability measure $\mu$ on the $\sigma$-algebra 
of Borel sets of the
unit square $[0, 1]^2$ such that $\mu$ has uniform marginals, i.e.,
$$
\mu([\alpha,\beta] \times [0,1]) \;=\; \mu([0,1] \times [\alpha,\beta]) \;=\; \beta-\alpha
$$
for every $0 \leqslant \alpha \leqslant \beta \leqslant 1$.

\begin{figure}[t] 
  \begin{center}
    \includegraphics[scale=0.58]{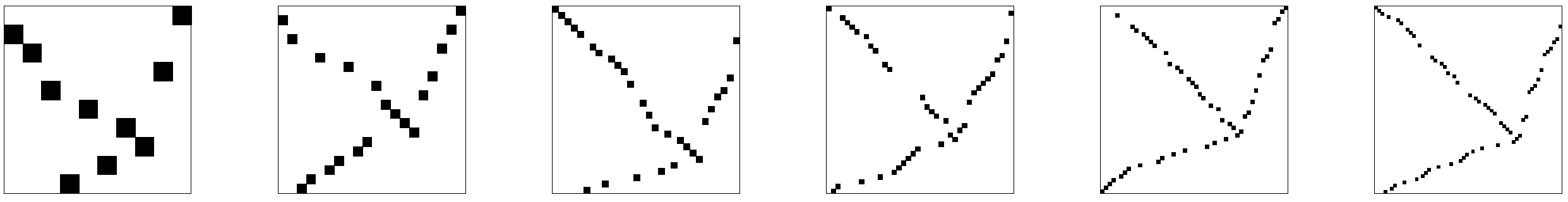}
  \end{center}
  \caption{Permutons associated with permutations in \protect\gctwo{2}{-1,1}{1,0}}
  \label{figPermutonExample}
\end{figure}
One way of understanding the idea of the limit of a sequence of permutations is to associate a permuton $\mu_\sigma$ with each permutation $\sigma$. If $|\sigma|=n$, then for each $i=1,\ldots,n$,
$$
\mu_\sigma([(i-1)/n,i/n] \times [(\sigma(i)-1)/n,\sigma(i)/n]) \;=\; 1 .
$$
Over regions that don't intersect these $n$ squares, $\mu_\sigma$ is zero.
See Figure~\ref{figPermutonExample} for some examples.
A sequence $(\sigma_k)$ of permutations of increasing length is then said to converge to a permuton $\mu$
if,
for every measurable region $R \subseteq [0,1]^2$, we have $\mu_{\sigma_k}(R) \rightarrow \mu(R)$ as $k$ tends to infinity.

In the graph-theoretical
context,
Janson~\cite{Janson2011,HJS2013}
has
initiated a study of the
graphons that occur as limits of
sequences of graphs in a graph class.
He discovered an intriguing relationship between the growth rate of such a class and the ``entropy'' of its graphons.
Sadly, this theorem does not transfer naturally to the world of permutations.
However, an investigation of
the permutons that occur as limits
of
sequences of permutations in a permutation class
may yield interesting results.
One initial elementary observation is that the \emph{support} of any such permuton must have zero measure.

\cleardoublepage


\setgcgap{0.08}  

\newcommand{\opmo}{$0$/$\pm1$}
\newcommand{\drm}{M^{\!\times2}}
\newcommand{\grm}{G^{\times\!}}
\newcommand{\tmm}{\mathbb{M}(M)}

\chapter{Geometric grid classes}\label{chap07}

\section{Introduction}

Closely related to monotone grid classes are \emph{geometric} grid classes, as investigated
by Albert, Atkinson, Bouvel, Ru\v{s}kuc \& Vatter~\cite{AABRV2011} and Albert, Ru\v{s}kuc \& Vatter~\cite{ARV2012}.
The geometric grid class $\Geom(M)$ is a subclass of $\Grid(M)$, permutations in $\Geom(M)$ satisfying an additional ``geometric'' constraint.


\begin{figure}[ht]
$$
{
\setgcscale{0.8}
\setgcgap{0.037}
\setgcptgridscale{5}
\setgcptsize{0.095}
\setgcptcolor{blue!50!black}
\gctwo[-1, 1,7,-1,-1, -1, -1,-1,2,-1, -1, 9,8,3,6]{3}{1,0,-1}{1,-1,1}
\qquad
\setgcptgridscale{6}
\gctwo[-1, 1,-1,-1,-1,5, -1, -1,-1,-1,2,-1, -1, 11,10,3,4]{3}{1,0,-1}{1,-1,1}
}
\qquad\qquad\quad
 \raisebox{-.15in}{
  \begin{tikzpicture}[scale=0.8]
      \draw [very thick] (0,0)--(1,0)--(1,1)--(0,1);
      \draw [very thick,dashed] (0,0)--(0,1);
      \draw [very thick,dashed] (1,0)--(2,0);
      \fill[radius=0.12,black!50!blue] (0,0) circle ;
      \fill[radius=0.12,black!50!red] (1,0) circle ;
      \fill[radius=0.12,black!50!blue] (2,0) circle;
      \fill[radius=0.12,black!50!red] (0,1) circle ;
      \fill[radius=0.12,black!50!blue] (1,1) circle ;
      \node[below]at(0,-.1){$c_3$};
      \node[below]at(1,-.1){$r_1$};
      \node[below]at(2,-.1){$c_2$};
      \node[above]at(0,1.1){$r_2$};
      \node[above]at(1,1.1){$c_1$};
    \end{tikzpicture}}
$$
\caption[figGeomClass]{
At left: The standard figure for
$\begin{geommx}1&\pos0&-1\\ 1&-1&\pos1 \end{geommx}$,
showing two 
plots
of the permutation $1527634$
with distinct griddings.
At right: Its row-column graph;
positive edges are shown as solid lines, negative edges are dashed.
}\label{figGeomClass}
\end{figure}
Like a monotone grid class, a geometric grid class is specified by a \opmo{} matrix which represents the shape of plots of permutations in the class.
As before, to match the Cartesian coordinate system, we index these matrices from the lower left, by column and then by row.
If $M$ is such a matrix, then we say that the \emph{standard figure} of $M$, denoted  $\Lambda_M$, is the subset of $\mathbb{R}^2$ consisting of the
union of
oblique
open line segments $L_{i,j}$ with slope $M_{i,j}$ for each $i,j$ for which
$M_{i,j}$ is nonzero,
where $L_{i,j}$ extends
from $(i-1,j-1)$ to $(i,j)$
if $M_{i,j}=1$, and
from $(i-1,j)$ to $(i,j-1)$
if $M_{i,j}=-1$.

The geometric grid class $\Geom(M)$ is
then defined to be
the set of
those
permutations $\sigma_1\sigma_2\ldots\sigma_n$ that can be plotted as a subset of the standard figure, i.e.~for which there exists a sequence of points $(x_1,y_1),\ldots,(x_n,y_n)\in\Lambda_M$ 
such that $x_1<x_2<\ldots<x_n$ and the sequence $y_1,\ldots,y_n$ is order-isomorphic to $\sigma_1,\ldots,\sigma_n$.
See Figure~\ref{figGeomClass} for an example.

We define the row-column graph $G(M)$ of geometric grid class $\Geom(M)$ in the same way as for a monotone grid class (see Section~\ref{sectRowCol}), except that we label each edge $r_ic_j$ with the value of $M_{i,j}$.
Edges labelled $+1$ are called \emph{positive}; edges labelled $-1$ are called \emph{negative}.
See Figure~\ref{figGeomClass} for an illustration.

Clearly $\Geom(M)$ is a subset of $\Grid(M)$. Indeed,
$\Grid(M)$ consists of those permutations that can be plotted as a subset of some figure consisting of the union of \emph{any monotonic curves} $\Gamma_{i,j}$ with the same endpoints as the $L_{i,j}$ in $\Lambda_M$.
This  permits greater flexibility in
the positioning of points in the cells, so one would expect that for some matrices, $\Grid(M)$ would contain permutations that cannot be plotted on $\Lambda_M$.
This is the case. For example, it can easily be shown that $\mathbf{2413}$ and $\mathbf{3142}$ lie in $\Grid\big(\!\gctwo{2}{-1,1}{1,-1}\!\big)$ but not in $\Geom\big(\!\gctwo{2}{-1,1}{1,-1}\!\big)$. 
In fact, the geometric grid class
$\Geom(M)$ and the monotone grid class $\Grid(M)$
are identical
if
and only if
$G(M)$
is acyclic~\cite[Theorem~3.2]{AABRV2011}.

\begin{figure}[ht]
$$
{
\setgcscale{0.4}
\setgcgap{0.04}
\setgcarrowmode
\setgcarrowtip{stealth}
\gcfour{6}{0,3,0,0,-2}{2,0,0,0,0,-3}{0,3,-2,0,0,3}{2,0,0,-3,2}
}
\qquad\quad\qquad
 \raisebox{-.3625in}{
  \begin{tikzpicture}[scale=0.8]
      \draw [very thick] (1.707,-0.707)--(1,0)--(1.707,0.707)--(2.707,0.707);
      \draw [very thick] (2.707,-0.707)--(3.707,-0.707)--(4.414,0)--(3.707,0.707);
      \draw [very thick,dashed] (0,0)--(1,0);
      \draw [very thick,dashed] (1.707,-0.707)--(2.707,-0.707);
      \draw [very thick,dashed] (2.707,0.707)--(3.707,0.707);
      \draw [very thick,dashed] (4.414,0)--(5.414,0);
      \fill[radius=0.12,black!50!blue] (0,0) circle ;
      \fill[radius=0.12,black!50!red] (1,0) circle ;
      \fill[radius=0.12,black!50!blue] (1.707,-0.707) circle;
      \fill[radius=0.12,black!50!blue] (1.707,0.707) circle;
      \fill[radius=0.12,black!50!red] (2.707,-0.707) circle ;
      \fill[radius=0.12,black!50!red] (2.707,0.707) circle ;
      \fill[radius=0.12,black!50!blue] (3.707,-0.707) circle;
      \fill[radius=0.12,black!50!blue] (3.707,0.707) circle;
      \fill[radius=0.12,black!50!red] (4.414,0) circle ;
      \fill[radius=0.12,black!50!blue] (5.414,0) circle ;
      \node[above]at(0,.1){$c'_2$};
      \node[above]at(1,.1){$r_1\,$};
      \node[above]at(1.707,0.807){$c_1$};
      \node[above]at(2.707,0.807){$r_2$};
      \node[above]at(3.707,0.807){$c'_3$};
      \node[above]at(4.414,.1){$\,r'_1$};
      \node[above]at(5.414,.1){$c_2$};
      \node[below]at(1.707,-.807){$c_3$};
      \node[below]at(2.707,-.807){$r'_2$};
      \node[below]at(3.707,-.807){$c'_1$};
    \end{tikzpicture}}
$$
\caption[figDoubleRefinement]{At left: The standard figure of
$\begin{smallmx}1&\pos0&-1\\ 1&-1&\pos1 \end{smallmx}^{\!\times2}$,
with a consistent orientation marked. At right:
Its row-column graph.
}\label{figDoubleRefinement}
\end{figure}
To state our result, we need one final definition related to geometric grid classes.
If $M$ is a \opmo{} matrix of dimensions $t\ttimes u$, we define the \emph{double refinement} $\drm$ of $M$ to be the \opmo{} matrix of dimensions $2\+t\ttimes 2\+u$ obtained from $M$ by replacing each $0$ with $\begin{smallmx}0&0\\0&0\end{smallmx}$, each $1$ with $\begin{smallmx}0&1\\1&0\end{smallmx}$, and each $-1$ with $\begin{smallmx}-1&\pos0\\ \pos0&-1\end{smallmx}$.
See Figure~\ref{figDoubleRefinement} for an example.
Note that the standard figure of $\drm$ is essentially a scaled copy of the standard figure of $M$, so we have:
\begin{obs}\label{obsGeomDouble}
$\Geom(\drm)=\Geom(M)$ for any \opmo{} matrix $M$.
\end{obs}
We demonstrate a connection between the growth rate of
$\Geom(M)$ and the \emph{matching polynomial} of the graph $G(\drm)$, the row-column graph of the double refinement of $M$.
A~$k$-\emph{matching} of a graph is a set of $k$ edges, no
pair of which have a vertex in common.
For example, the negative (dashed) edges in the graph in Figure~\ref{figDoubleRefinement} constitute a $4$-matching.
If, for each~$k$, $m_k(G)$ denotes the number of distinct $k$-matchings of a graph $G$
with $n$ vertices,
then the \emph{matching polynomial}
$\mu_G(z)$
of $G$ is defined to be
\begin{equation}\label{eqMatchingPolynomialDefn}
\mu_G(z)
\;=\; \sum_{k=0}^{\floor{n/2}} (-1)^k\+ m_k(G)\+z^{n-2k}.
\end{equation}
Observe that the exponents of the variable $z$ enumerate \emph{defects} in $k$-matchings: the number of vertices which are \emph{not} endvertices of an edge in such a matching.
If $n$ is even, $\mu_G(z)$ is an even function;
if $n$ is odd, $\mu_G(z)$ is an odd function.

With the relevant definitions complete, we can now state our theorem:

\thmbox{
\begin{thm}\label{thmGeomClassGrowthRate}
The growth rate of
geometric grid class
$\Geom(M)$
exists and
is equal to
the square
of the largest root of
the matching polynomial
$\mu_{G(\drm)}(z)$,
where
$G(\drm)$ is
the row-column graph of the double refinement of $M$.
\end{thm}
} 

In the next section, we prove this theorem by utilizing the link between geometric grid classes and trace monoids, and their connection to rook numbers and the matching polynomial.
Then, in Section~\ref{sectGeomImplications}
we investigate a number of implications of this result by utilizing properties of the matching polynomial, especially the fact that the moments of $\mu_G(z)$
enumerate certain closed walks on $G$.
Firstly, we characterise the growth rates of geometric grid classes in terms of the spectral radii of trees.
Then, we explore the influence of cycle parity on growth rates and relate
the growth rates of geometric grid classes to those of monotone grid classes.
Finally, we consider the effect
of subdividing edges in the row-column graph, proving some new results regarding how edge subdivision affects the largest root of the matching polynomial.

\section{Proof of growth rate theorem} 

In order to prove our result, we make use of the
connection between geometric grid classes and \emph{trace monoids}. This relationship was first used by Vatter \& Waton~\cite{VW2011} to establish certain structural properties of
grid classes, and was developed further in~\cite{AABRV2011} from where we use a number of results.
To begin with, we need to consider griddings of permutations.

We define $M$-griddings and $M$-gridded permutations as for monotone grid classes (see Section~\ref{defGridding}).
A permutation may have multiple distinct griddings in a given geometric grid class; see Figure~\ref{figGeomClass} for an example.
We use $\Geomhash(M)$ to denote the set of all $M$-gridded permutations in $\Geom(M)$.
As is the case for monotone grid classes (Lemma~\ref{lemmaGRGriddings}), the growth rate of
$\Geom(M)$ is equal to the growth rate of the corresponding class of $M$-gridded permutations $\Geomhash(M)$.
So we can restrict our considerations to $M$-gridded permutations.

To determine the growth rate of
$\Geomhash(M)$,
we relate $M$-gridded permutations to words in a trace monoid.
To achieve this, one additional concept is required, that of a \emph{consistent orientation} of a standard figure.
If $\Lambda_M=\bigcup\,\{L_{i,j}:M_{i,j}\neq0\}$ is the standard figure of a \opmo{} matrix $M$, then a \emph{consistent orientation} of $\Lambda_M$ consists of an orientation of each oblique line $L_{i,j}$ such that in each column either all the lines are oriented leftwards or all are oriented rightwards, and in each row either all the lines are oriented downwards or all are oriented upwards.\footnote{For
ease
of exposition, we
use
the concept of a consistent orientation
rather than the
approach used previously involving
partial multiplication matrices; results from \cite{AABRV2011} follow \emph{mutatis mutandis}.}
See Figures~\ref{figDoubleRefinement} and \ref{figPermWord} for examples.

It is not always possible to consistently orient a standard figure.
The ability to do so depends on the \emph{cycles} in the row-column graph.
We say that the \emph{parity} of a cycle in $G(M)$ is
the product of the labels of its edges,
a \emph{positive cycle} is one which has parity $+1$, and
a \emph{negative cycle} is one with parity $-1$.
The following result relates
cycle parity to
consistent orientations:
\begin{lemma}[{see Vatter \& Waton~\cite[Proposition~2.1]{VW2011}}]\label{lemmaNegativeCycles}
The standard figure
$\Lambda_M$
has a consistent orientation if and
only if its row-column graph $G(M)$ contains no negative cycles.
\end{lemma}
For example,
$G\big(\!\gctwo{3}{1,0,-1}{1,-1,1}\!\big)$
contains a negative cycle so its standard figure
has no
consistent orientation (see Figure~\ref{figGeomClass}), whereas
$G\big(\!\gctwo{3}{-1,0,-1}{1,-1,1}\!\big)$
has no negative cycles so its standard figure has a consistent orientation
(see Figure~\ref{figPermWord}).

On the other hand, we can always consistently orient the standard figure of the double refinement of a matrix
by orienting each oblique line towards the centre of its $2\times2$ block (as in Figure~\ref{figDoubleRefinement}).
So
we have the following:
\begin{lemma}[see {\cite[Proposition~4.1]{AABRV2011}}]\label{lemmaDoubleRefinementConsistentOrientation}
If $M$ is any \opmo{} matrix, then $\Lambda_{\drm}$ has a consistent orientation.
\end{lemma}
Thus, by Lemma~\ref{lemmaNegativeCycles}, the row-column graph of the double refinement of a matrix never contains a negative cycle.
Figure~\ref{figDoubleRefinement} shows a consistent orientation of the standard figure of the double refinement of a matrix whose standard figure (shown in Figure~\ref{figGeomClass}) doesn't itself have a consistent orientation.

\begin{figure}[ht]
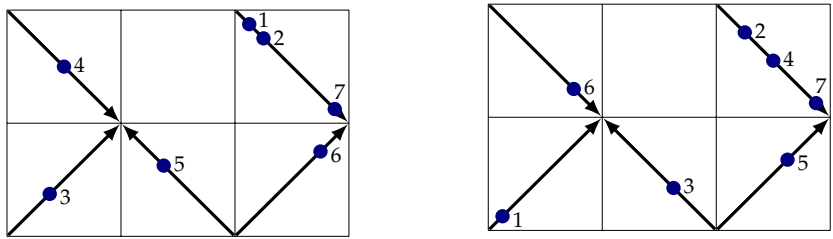

$$
{
\setgcscale{1.5}
\setgcgap{0.01}
\setgcptgridscale{8}
\setgcptsize{0.063}
\setgcptcolor{blue!50!black}
\setgcarrowmode
\setgcextra{
\node[right]at(2.115,1.865){${}^1$};
\node[right]at(2.25,1.71){${}^2$};
\node[right]at(0.375,0.345){${}_3$};
\node[right]at(0.5,1.47){${}^4$};
\node[right]at(1.375,0.595){${}^5$};
\node[right]at(2.75,0.72){${}_6$};
\node[above]at(2.925,1.125){${}_7$};
}
\gctwo[-1, -1,-1,3,12,-1,-1,-1, -1, -1,-1,5,-1,-1,-1,-1, -1, 15,14,-1,-1,-1,6,9]{3}{-2,0,-2}{2,-3,2}
\qquad\qquad
\setgcextra{
\node[right]at(0.125,0.095){${}_1$};
\node[right]at(2.25,1.72){${}^2$};
\node[right]at(1.625,0.345){${}^3$};
\node[right]at(2.5,1.47){${}^4$};
\node[right]at(2.625,0.595){${}_5$};
\node[right]at(0.75,1.22){${}^6$};
\node[above]at(2.925,1.125){${}_7$};
}
\gctwo[-1, 1,-1,-1,-1,-1,10,-1, -1, -1,-1,-1,-1,3,-1,-1, -1, -1,14,-1,12,5,-1,9]{3}{-2,0,-2}{2,-3,2}
}
$$
\caption[figPermWord]{The plots of permutation $1527634$
in $\begin{geommx}-1&\pos0&-1\\ \pos1&-1&\pos1 \end{geommx}$
associated with the
words $a_{32}a_{32}a_{11}a_{12}a_{21}a_{31}a_{32}$ and $a_{11}a_{32}a_{21}a_{32}a_{31}a_{12}a_{32}$.
Both plots correspond to the same gridding.
}\label{figPermWord}
\end{figure}
We are now in a position to describe the association between words and $M$-gridded permutations.
If $M$ is a \opmo{} matrix, then
we
let $\Sigma_M=\{a_{ij}:M_{i,j}\neq0\}$ be an alphabet of symbols, one for each nonzero cell in $M$.
If
we have a consistent orientation for $\Lambda_M$,
then
we can associate to each finite word $w_1\ldots w_n$ over $\Sigma_M$
a specific plot of
a permutation
in $\Geom(M)$
as follows: If $w_k=a_{ij}$, include the point at distance
$k\+\sqrt{2}/(n+1)$
along line segment 
$L_{i,j}$
according to its orientation.
See Figure~\ref{figPermWord} for two examples.
Clearly, this induces a mapping 
from
the set of all finite words over $\Sigma_M$ to $\Geomhash(M)$.
In fact, it can readily be shown that this map is surjective, every $M$-gridded permutation corresponding to some word over $\Sigma_M$ (\cite[Proposition~5.3]{AABRV2011}).

As can be seen in Figure~\ref{figPermWord}, distinct words may be mapped to the same gridded permutation.
This occurs because the order in which two consecutive points are included is immaterial if they occur in cells that
are neither in the same column nor in the same row.
From the perspective of the words, adjacent symbols corresponding to such cells may be interchanged without changing the
gridded permutation.
This corresponds to a structure known as a 
{trace monoid}. 

If we have a consistent orientation for standard figure $\Lambda_M$,
then
we define the \emph{trace monoid of~$M$},
which we denote by $\tmm$,
to be the set of equivalence classes of words over $\Sigma_M$ in which $a_{ij}$ and $a_{k\ell}$ commute (i.e.~$a_{ij}a_{k\ell}=a_{k\ell}a_{ij}$) whenever $i\neq k$ and $j\neq \ell$.
It is then relatively straightforward to show equivalence between gridded permutations and elements of the trace monoid:
\begin{lemma}[see {\cite[Proposition~7.1]{AABRV2011}}]\label{lemmaTraceMonoid}
If the standard figure $\Lambda_M$ has a consistent orientation, then gridded $n$-permutations in $\Geomhash(M)$ are in bijection with equivalence classes of words of length $n$ in
$\tmm$.
\end{lemma}
\vspace{-4pt}
Hence, by combining Lemmas~\ref{lemmaDoubleRefinementConsistentOrientation} and \ref{lemmaTraceMonoid} with Observation~\ref{obsGeomDouble}, we know that the growth rate of $\Geom(M)$ is equal to the growth rate of $\mathbb{M}(\drm)$ if it exists.
All that remains 
is to determine the growth rate of the trace monoid of a matrix.


Trace monoids
were first studied by Cartier \& Foata~\cite{CF1969}.
Using extended M\"obius inversion, they
determined the general form of the generating function, as follows:
\begin{lemma}[\cite{CF1969};
see also Flajolet \& {Sedgewick~\cite[Note~V.10]{FS2009}}] 
The ordinary generating function for $\tmm$ is given by
$$
f_M(z) \;=\;
\frac{1}{
\sum_{k\geqslant0} (-1)^k \+ r_k(M) \+ z^k
}
$$
where $r_k(M)$ is the number of $k$-subsets of $\Sigma_M$ whose elements commute pairwise.
\end{lemma}
Since symbols in $\tmm$ commute if and only if they correspond to cells that are neither in the same column nor in the same row,
it is easy to see that
$r_k(M)$ is the number of distinct ways of placing $k$ chess rooks on the nonzero entries of $M$ in such a way that no two rooks attack each other by being in the same column or row.
The numbers $r_k(M)$ are known as the \emph{rook numbers} for $M$ (see Riordan~\cite{Riordan2002}).
Moreover, a matching
in the row-column graph $G(M)$ also corresponds to a set of cells 
no pair of which share a column or row.
So
the rook numbers for $M$ are
the same as
the numbers of matchings in $G(M)$: 
\begin{obs}\label{obsRookNumIsMatchingNum}
For all $k\geqslant0$, $r_k(M)=m_k(G(M))$.
\end{obs}
Now, by elementary analytic combinatorics (see Section~\ref{sectAnalyticCombin}), we know that the growth rate of $\tmm$ is given by the reciprocal of the root
of the denominator 
of $f_M(z)$
that has least magnitude (Lemma~\ref{lemmaExponentialGrowthFormula}). 
The fact that this polynomial has a \emph{unique} root of smallest modulus
was proved by Goldwurm \& Santini in~\cite{GS2000}. It is real and positive by Pringsheim's Theorem (Lemma~\ref{lemmaPringsheim}).

But the \emph{reciprocal} of the \emph{smallest} root of a polynomial is the same as the \emph{largest} root of the \emph{reciprocal} polynomial (obtained by reversing the order of the coefficients). Hence,
if $M$ has dimensions $t\ttimes u$
and $n=t+u$,
then the growth rate of $\tmm$ is the largest (positive real) root of the polynomial
\begin{equation}\label{eqGMDefn}
g_M(z)
\;=\;
\frac{1}{
z^{\floor{n/2}} \+ f_M \!\left(\frac{1}{z}\right)\!
}
\;=\;
\sum_{k=0}^{\floor{n/2}} (-1)^k \+ r_k(M) \+ z^{\floor{n/2}-k}.
\end{equation}
Here,
$g_M(z)$ is the reciprocal polynomial of
$\left(f_M(z)\right)^{-1}$
multiplied by some nonnegative power of $z$,
since $r_k(M)=0$ for all $k>\floor{n/2}$.
Note also that $n$ is the number of vertices in $G(M)$.

If we now compare the definition of $g_M(z)$ in~\eqref{eqGMDefn}
with that of the matching polynomial $\mu_G(z)$ in~\eqref{eqMatchingPolynomialDefn}
and use Observation~\ref{obsRookNumIsMatchingNum}, then we see that:
$$
g_M(z^2) \;=\;
\begin{cases}
\phantom{z^{-1}\+}\mu_{G(M)}(z), & \text{if $n$ is even;} \\[3pt]
         z^{-1}\+ \mu_{G(M)}(z), & \text{if $n$ is odd.}
\end{cases}
$$
Hence,
the largest root of $g_M(z)$ is the square of the largest root of $\mu_{G(M)}(z)$.

We now have
all we need to prove Theorem~\ref{thmGeomClassGrowthRate}: The growth rate of $\Geom(M)$ is equal to the growth rate of $\mathbb{M}(\drm)$ which equals the square of the largest root of $\mu_{G(\drm)}(z)$.

In the above argument, we only employ the double
refinement $\drm$ to ensure that a consistent orientation is possible.
By
Lemma~\ref{lemmaNegativeCycles}, we know that if
$G(M)$
is free of negative cycles then $\Lambda_M$
can be consistently oriented.
Thus, we have the following special case of Theorem~\ref{thmGeomClassGrowthRate}:

\thmbox{
\begin{cor}\label{corNoNegCycles}
If
$G(M)$
contains no negative cycles,
then
the growth rate of
$\Geom(M)$ is
equal to
the square
of the largest root of
$\mu_{G(M)}(z)$.
\end{cor}
} 

\section{Consequences}
\label{sectGeomImplications}
In this final section, we investigate some of the implications of Theorem~\ref{thmGeomClassGrowthRate}.
By considering properties of the matching polynomial, we characterise the growth rates of geometric grid classes in terms of the spectral radii of trees, prove a monotonicity result, and explore the influence of cycle parity on growth rates.
We then compare the growth rates of geometric grid classes with those of monotone grid classes.
Finally, we consider the effect
of subdividing edges in the row-column graph.

\newcommand{\deledge}{\!\setminus\!}
Let's begin by introducing some notation.
We use $G+H$ to denote the graph composed of two disjoint subgraphs $G$ and $H$.
The graph resulting from deleting
the vertex $v$
(and all edges incident to $v$)
from a graph $G$
is denoted $G-v$.
Generalising this,
if $H$ is a subgraph of~$G$, then $G-H$ is the graph obtained by deleting the vertices of $H$ from $G$.
In contrast, we use
$G\deledge e$ to denote
the graph that results from deleting the edge $e$ from 
$G$.
The number of connected components of $G$ is represented by $\mathrm{comp}(G)$.
The characteristic polynomial of a graph $G$ is denoted $\Phi_G(z)$. We use $\rho(G)$ to denote the spectral radius of $G$, the largest root of $\Phi_G(z)$.
Finally, we use $\lambda(G)$ for the largest root of the matching polynomial $\mu_G(z)$.

The matching polynomial was independently discovered a number of times, beginning with
Heilmann \& Lieb~\cite{HL1970} when investigating monomer-dimer systems in statistical physics.
It was first studied from a combinatorial perspective by Farrell~\cite{Farrell1979} and Gutman~\cite{Gutman1977}.
The theory was then further developed by Godsil \& Gutman~\cite{GG1981} and Godsil~\cite{Godsil1981}.
An introduction can be found in the books by Godsil~\cite{Godsil1993} and Lov\'asz \& Plummer~\cite{LP2009}.

The facts concerning the matching polynomial that we use are covered by three lemmas.
As a consequence of the first, we only need to consider connected graphs:
\begin{lemma}[Farrell~\cite{Farrell1979}; Gutman~\cite{Gutman1977}]
\label{lemmaMuProductComponents}
The matching polynomial of a graph is the product of the matching polynomials of its connected components.
\end{lemma}
Thus, in particular:
\begin{cor}\label{corLambdaMaxComponents}
For any graphs $G$ and $H$, we have
  $\lambda(G+H)=\max(\lambda(G),\lambda(H))$.
\end{cor}

The second lemma relates the matching polynomial to the characteristic polynomial.
\begin{lemma}[Godsil \& Gutman~\cite{GG1981}] 
\label{lemmaMuCycleDef}
If $\CCC_G$ consists of all nontrivial subgraphs of $G$ which are unions of vertex-disjoint cycles (i.e., all subgraphs of $G$ which are regular of degree 2), then
  $$
  \mu_G(z) \;=\; \Phi_G(z) \:+\: \sum_{C\in\CCC_G} \!\!2^{\+\mathrm{comp}(C)} \+ \Phi_{G-C}(z),
  $$
where $\Phi_{G-C}(z)=1$ if $C=G$.
\end{lemma}
As an immediate consequence, we have the following:
\begin{cor}[Sachs~\cite{Sachs1964}; Mowshowitz~\cite{Mowshowitz1972}; Lov\'asz \& Pelik\'an~\cite{LP1973}]
\label{corMuPhiTree}
The matching polynomial of a graph 
is identical to its characteristic polynomial 
if and only if the graph is acyclic.
\end{cor}
In particular, their largest roots are identical:
\begin{cor}\label{corLambdaRhoTree}
  If $G$ is a forest, 
  then
$\lambda(G)=\rho(G)$.
\end{cor}
Thus, using Corollaries~\ref{corNoNegCycles} and~\ref{corLambdaMaxComponents}, we have the following alternative characterisation for the growth rates of acyclic geometric grid classes:

\thmbox{
\begin{cor}\label{corForestRho}
  If $G(M)$ is a forest, then $\gr(\Geom(M))=\rho(G(M))^2$.
\end{cor}
} 

This result is to be expected, since we know that $\Geom(M)=\Grid(M)$ if $G(M)$ is acyclic, and have shown (Theorem~\ref{thmGrowthRate}) that $\gr(\Grid(M))=\rho(G(M))^2$.

The last, and most important, of the three lemmas allows us to determine the largest root of the matching polynomial of a graph from
the spectral radius of a related tree.
It
is a consequence of the fact,
determined by Godsil in~\cite{Godsil1981}, 
that the moments (sums of the powers of the roots) of
$\mu_G(z)$ enumerate certain closed walks on~$G$, which he calls \emph{tree-like}.\footnote{The author is grateful to Brendan McKay, whose helpful response to a question on \emph{MathOverflow}~\cite{McKay2014Thin} alerted him to the significance of this result.}
This is analogous to the fact that the moments of
$\Phi_G(z)$
count \emph{all} closed walks on~$G$.
On a tree, all closed walks are tree-like.

\begin{figure}[ht]
$$
\raisebox{12pt}{
\begin{tikzpicture}[scale=0.36]
    \draw [thick] (0,.3)--(1.7,1.5);
    \draw [thick] (0,.3)--(2.1,1.1);
    \draw [thick] (0,.3)--(3,0.75);
    \draw [fill=gray!30!white,thick] (3.4,1.4) circle [x radius=1.7, y radius=1];
    \draw [fill] (0,.3) circle [radius=0.15];
    \node[left]at(0,.3){$u$};
    \node[left]at(-1.5,.3){$G$:};
    \draw [fill] (1.7,1.5) circle [radius=0.15];
    \node[left]at(1.7,1.6){$v$};
    \node[]at(3.4,1.4){$H$};
    \draw [thick] (0,.3)--(3,-0.75);
    \draw [thick] (0,.3)--(2.33,-1);
    \draw [thick] (0,.3)--(2.17,-1.5);
    \draw [thick] (0,.3)--(2,-2);
    \draw [fill=gray!30!white,thick] (3.4,-1.35) circle [x radius=1.75, y radius=1.15];
    \node[]at(3.4,-1.35){$J$};
\end{tikzpicture}
}
\qquad\qquad\qquad
\begin{tikzpicture}[scale=0.36]
    \draw [fill=gray!30!white,thick] (3.4,3.8) circle [x radius=1.7, y radius=1];
    \draw [thick] (0,.3)--(1.7,3.9);
    \draw [thick] (0,.3)--(2,1);
    \draw [thick] (0,.3)--(3,0.75);
    \draw [fill=gray!30!white,thick] (3.4,1.4) circle [x radius=1.7, y radius=1];
    \draw [fill] (0,.3) circle [radius=0.15];
    \node[left]at(0,.3){$u$};
    \node[left]at(-1.5,1.4){$K$:};
    \draw [fill] (1.7,3.9) circle [radius=0.15];
    \node[left]at(1.7,4){$v$};
    \node[]at(3.4,1.4){$H$};
    \node[]at(3.4,3.8){$H$};
    \draw [thick] (0,.3)--(3,-0.75);
    \draw [thick] (0,.3)--(2.33,-1);
    \draw [thick] (0,.3)--(2.17,-1.5);
    \draw [thick] (0,.3)--(2,-2);
    \draw [fill=gray!30!white,thick] (3.4,-1.35) circle [x radius=1.75, y radius=1.15];
    \node[]at(3.4,-1.35){$J$};
\end{tikzpicture}
\qquad
$$
\caption{Expanding $G$ at $u$ along $uv$; $H$ is the component of $G-u$ that contains $v$}
\label{figExpandingVertex}
\end{figure}
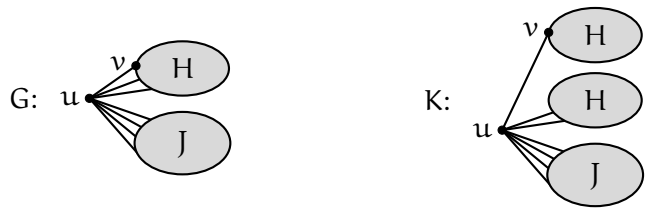
\begin{lemma}[Godsil~\cite{Godsil1981}; see also~\cite{Godsil1993} and~\cite{LP2009}]
\label{lemmaExpandVertex}
  Let $G$ be a graph and let $u$ and $v$ be adjacent vertices in a cycle of $G$.
  Let $H$ be the component of $G-u$ that contains $v$.
  Now let $K$ be the graph constructed by taking a copy of $G\deledge uv$ and a copy of $H$ and joining the occurrence of $u$ in the copy of $G\deledge uv$ to the occurrence of $v$ in the copy of $H$ 
  (see Figure~\ref{figExpandingVertex}).
  Then $\lambda(G)=\lambda(K)$.
\end{lemma}
The process that is described in Lemma~\ref{lemmaExpandVertex}
we call
``\emph{expanding} $G$ \emph{at} $u$ \emph{along} $uv$''.
Each such expansion of a graph $G$ produces a graph with fewer cycles than $G$.
Repeated application of this process thus eventually results in a forest $F$ such that $\lambda(F)=\lambda(G)$.
We shall say that $F$ results from \emph{fully expanding}~$G$.
Hence, by
Corollaries~\ref{corLambdaMaxComponents} and~\ref{corLambdaRhoTree},
the largest root of the matching polynomial of a graph equals the spectral radius of some tree:
for any graph~$G$, there is a tree $T$ such that $\lambda(G)=\rho(T)$.

It is readily observed that
every tree is the row-column graph of some geometric grid class. Thus
we have the following characterisation of geometric grid class growth rates.

\thmbox{
\begin{cor}\label{corGeomSqRhoTrees}
  The set of growth rates of geometric grid classes consists of the squares of the spectral radii of trees.
\end{cor}
} 

\HIDE{
In~\cite{Shearer1989}, Shearer proved that, for any $\rho\geqslant\sqrt{2+\sqrt{5}}$, there exists a sequence of trees $T_1,T_2,\ldots$ such that $\liminfty[k]\rho(T_k)=\rho$. Thus we have the following density result for the set of geometric grid class growth rates:

\thmbox{
\begin{cor}
  Every value at least $2+\sqrt{5}$ is a limit point of growth rates of
geometric grid classes.
\end{cor}
} 
} 

The spectral radii of connected graphs satisfy the following strict monotonicity condition:
\begin{lemma}[{\cite[Proposition~1.3.10]{CRS2010}}]\label{lemmaStrictMono}
  If $G$ is connected and $H$ is a proper subgraph of $G$, then we have $\rho(H)<\rho(G)$.
\end{lemma}
Lemma~\ref{lemmaExpandVertex} enables us to prove
the analogous fact for the largest roots of matching polynomials, from which we can deduce a monotonicity result for geometric grid classes:
\begin{cor}\label{corStrictMono}
If $G$ is connected and $H$ is a proper subgraph of $G$, then $\lambda(H)<\lambda(G)$.
\end{cor}
\begin{proof}
  Suppose we fully expand $H$
  (at vertices $u_1,\ldots,u_k$, say),
  then 
  the result is a forest $F$ such that $\lambda(H)=\rho(F)$.
  Now suppose that we repeatedly expand $G$ analogously at $u_1,\ldots,u_k$, and then continue to fully expand the resulting graph. The outcome is a tree $T$ (since $G$ is connected)
  such that $F$ is a proper subgraph of $T$ and
  $\lambda(G)=\rho(T)$.
  The result follows from Lemma~\ref{lemmaStrictMono}.
\end{proof}
Adding a non-zero cell to a \opmo{} matrix $M$ adds an edge to $G(M)$.
Thus, geometric grid classes satisfy the following monotonicity condition:

\thmbox{
\begin{cor}
If $G(M)$ is connected and $M'$ results from adding a non-zero cell to $M$
in such a way that $G(M')$ is also connected,
then $\gr(\Geom(M'))>\gr(\Geom(M))$.
\end{cor}
} 

\subsection{Cycle parity}

The growth rate of a geometric grid class
depends on the parity of its cycles.
Consider the case of $G(M)$ being a cycle graph $C_n$.
If $G(M)$ is a negative cycle, then $G(\drm)=C_{2n}$.
Now, by Lemma~\ref{lemmaExpandVertex}, we have $\lambda(C_n)=\rho(P_{2n-1})$, where $P_n$ is the path graph on $n$ vertices.
The spectral radius of a path graph on $n$ vertices is $2\cos\frac{\pi}{n+1}$.
So,
\begin{equation}\label{eqGRGeomCycle}
\gr(\Geom(M)) \;=\;
\begin{cases}
4\cos^2\frac{\pi}{2n}, & \text{if $G(M)$ is a positive cycle;} \\[5pt]
4\cos^2\frac{\pi}{4n}, & \text{if $G(M)$ is a negative cycle.}
\end{cases}
\end{equation}
Thus the geometric grid class whose row-column graph is a
negative cycle has a greater growth rate than the class
whose row-column graph is a positive cycle.
As another example,
\begin{equation}\label{eqGRGeomExNegCyc}
  \gr\!\left(\begin{geommx}1&\pos0&-1\\ 1&-1&\pos1 \end{geommx}\right) \;=\; 3+\sqrt{2} \;\approx\; 4.41421,
\end{equation}
whereas
\begin{equation}\label{eqGRGeomExPosCyc}
  \gr\!\left(\begin{geommx}-1&\pos0&-1\\ \pos1&-1&\pos1 \end{geommx}\right) \;=\; 4.
\end{equation}
The former, containing a negative cycle,
has a greater growth rate than the latter, whose cycle is positive.
This is typical;
we prove the following result:

\newcommand{\Msub}{M_1}
\newcommand{\Gsub}{G_1}
\thmbox{
\begin{cor}\label{corNegateCell}
If $G(M)$ is connected and contains no negative cycles, and $\Msub$ 
results from 
changing the sign of
a single entry of $M$ that is in a cycle (thus making one or more cycles in $G(\Msub)$ negative),
then
$\gr(\Geom(\Msub)) > \gr(\Geom(M))$.
\end{cor}
} 

In order to do this, we need to consider the structure of $G(\drm)$.
The graph $G(\drm)$ can be constructed from $G(M)$ as follows: 
If $G(M)$ has vertex set $\{v_1,\ldots,v_n\}$, then we let $G(\drm)$ have vertices $v_1,\ldots,v_n$ and $v'_1,\ldots,v'_n$.
If $v_iv_j$ is a positive edge in $G(M)$, then in $G(\drm)$ we add an edge between $v_i$ and $v_j$ and also between $v'_i$ and $v'_j$.
On the other hand, if $v_iv_j$ is a negative edge in $G(M)$, then in $G(\drm)$ we join $v_i$ to $v'_j$ and $v'_i$ to $v_j$.
The correctness of this construction follows directly from the definitions of double refinement and of the row-column graph of a matrix.
For an illustration, compare the graph in Figure~\ref{figDoubleRefinement} against that in Figure~\ref{figGeomClass}.

Note that if $v_1,\ldots,v_k$ is a positive
$k$-cycle in $G(M)$, then $G(\drm)$ contains
two vertex-disjoint positive $k$-cycles, the union of whose vertices is $\{v_1,\ldots,v_k,v'_1,\ldots,v'_k\}$.
In contrast, if $v_1,\ldots,v_\ell$ is a negative
$\ell$-cycle in $G(M)$, then $G(\drm)$ contains a (positive) $2\+\ell$-cycle on the vertices $\{v_1,\ldots,v_\ell,v'_1,\ldots,v'_\ell\}$ in which $v_i$ is opposite $v'_i$ (i.e. $v'_i$ is at distance $\ell$ from $v_i$ around the cycle) for each $i$, $1\leqslant i\leqslant\ell$.
We make the following additional observations:
\begin{obs}\label{obsPosx2} 
  If $G(M)$ has no negative cycles, then $G(\drm)=G(M)+ G(M)$.
\end{obs}
\begin{obs} 
  If $G(M)$ is connected and has a negative cycle, then $G(\drm)$ is connected.
\end{obs}
We now have all we require to prove our cycle parity result.

\begin{figure}[ht] 
$$
\begin{tikzpicture}[scale=0.36]
    \draw [thick] (0,3)--(-2.5,0);
    \draw [thick] (0,3)--(-1.6,0);
    \draw [thick] (0,3)--(-0.7,0);
    \draw [thick] (0,3)--(-3,4);
    \draw [thick] (0,3)--(-3,3);
    \draw [thick] (0,3)--(-3,2);
    \draw [fill=gray!30!white,thick] (-1.7,0) circle [x radius=1.7, y radius=1];
    \node[]at(-1.7,0){$H$};
    \draw [fill=gray!30!white,thick] (-3,3) circle [x radius=1.7, y radius=1.15];
    \node[]at(-3,3){$J$};
    \draw [thick] (0,0)--(0,3);
    \draw [fill] (0,3) circle [radius=0.15];
    \draw [fill] (0,0) circle [radius=0.15];
    \node[above right]at(0,3){$\!\!u$};
    \node[below right]at(0,0){$\!v$};
    \draw [thick] (3,0)--(3.7,3);
    \draw [thick] (3,0)--(4.6,3);
    \draw [thick] (3,0)--(5.5,3);
    \draw [thick] (3,0)--(6,1);
    \draw [thick] (3,0)--(6,0);
    \draw [thick] (3,0)--(6,-1);
    \draw [fill=gray!30!white,thick] (4.7,3) circle [x radius=1.7, y radius=1];
    \node[]at(4.7,3){$H$};
    \draw [fill=gray!30!white,thick] (6,0) circle [x radius=1.7, y radius=1.15];
    \node[]at(6,0){$J$};
    \draw [thick] (3,0)--(3,3);
    \draw [fill] (3,3) circle [radius=0.15];
    \draw [fill] (3,0) circle [radius=0.15];
    \node[above left]at(3,3){$v'\!\!\!$};
    \node[below left]at(3,.27){$u\!'$};
    \node[]at(1.5,-3.1){$G(\drm)\+=\+G+ G$};
\end{tikzpicture}
\qquad\quad
\begin{tikzpicture}[scale=0.36]
    \draw [thick] (0,3)--(-2.5,0);
    \draw [thick] (0,3)--(-1.6,0);
    \draw [thick] (0,3)--(-0.7,0);
    \draw [thick] (0,3)--(-3,4);
    \draw [thick] (0,3)--(-3,3);
    \draw [thick] (0,3)--(-3,2);
    \draw [fill=gray!30!white,thick] (-1.7,0) circle [x radius=1.7, y radius=1];
    \node[]at(-1.7,0){$H$};
    \draw [fill=gray!30!white,thick] (3,0) circle [x radius=1.7, y radius=1];
    \node[]at(3,0){$H$};
    \draw [fill=gray!30!white,thick] (-3,3) circle [x radius=1.7, y radius=1.15];
    \node[]at(-3,3){$J$};
    \draw [thick] (0,3)--(1.3,0);
    \draw [fill] (0,3) circle [radius=0.15];
    \draw [fill] (1.3,0) circle [radius=0.15];
    \node[above right]at(0,3){$\!\!u$};
    \node[below left]at(1.3,0){$v\!$};
    \node[]at(0,-3.1){$\phantom{G_1}K\phantom{G_1}$};
\end{tikzpicture}
\qquad\quad
\begin{tikzpicture}[scale=0.36]
    \draw [fill=gray!10!white] (-5.2,4.65)--(6.9,4.65)--(6.9,1.5)--(1.25,1.5)--(1.25,-1.5)--(-5.2,-1.5)--(-5.2,4.65);
    \draw [thick] (0,3)--(-2.5,0);
    \draw [thick] (0,3)--(-1.6,0);
    \draw [thick] (0,3)--(-0.7,0);
    \draw [thick] (0,3)--(-3,4);
    \draw [thick] (0,3)--(-3,3);
    \draw [thick] (0,3)--(-3,2);
    \draw [fill=gray!30!white,thick] (-1.7,0) circle [x radius=1.7, y radius=1];
    \node[]at(-1.7,0){$H$};
    \draw [fill=gray!30!white,thick] (-3,3) circle [x radius=1.7, y radius=1.15];
    \node[]at(-3,3){$J$};
    \draw [thick] (0,0)--(3,0);
    \draw [fill] (0,3) circle [radius=0.15];
    \draw [fill] (0,0) circle [radius=0.15];
    \node[above right]at(0,3){$\!\!u$};
    \node[below right]at(0,-0.1){$\!v$};
    \draw [thick] (3,0)--(3.7,3);
    \draw [thick] (3,0)--(4.6,3);
    \draw [thick] (3,0)--(5.5,3);
    \draw [thick] (3,0)--(6,1);
    \draw [thick] (3,0)--(6,0);
    \draw [thick] (3,0)--(6,-1);
    \draw [fill=gray!30!white,thick] (4.7,3) circle [x radius=1.7, y radius=1];
    \node[]at(4.7,3){$H$};
    \draw [fill=gray!30!white,thick] (6,0) circle [x radius=1.7, y radius=1.15];
    \node[]at(6,0){$J$};
    \draw [thick] (0,3)--(3,3);
    \draw [fill] (3,3) circle [radius=0.15];
    \draw [fill] (3,0) circle [radius=0.15];
    \node[above left]at(3,3){$v'\!\!\!$};
    \node[below left]at(3,.17){$u\!'\!$};
    \node[]at(1.5,-3.1){$\Gsub\+=\+G({\Msub}^{\!\times2})$};
\end{tikzpicture}
\vspace{-10.5pt}
$$
\caption{Graphs used in the proof of Corollary~\ref{corNegateCell}}
\label{figNegCycle}
\end{figure}
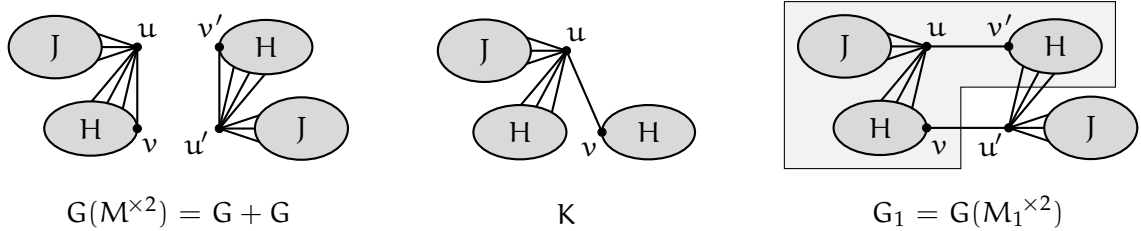
\begin{proof}[Proof of Corollary~\ref{corNegateCell}]
Let $G=G(M)$ and $\Gsub=G({\Msub}^{\!\times2})$, and let
$uv$ be the edge in $G$ corresponding to the entry in $M$ that is negated to create $\Msub$.
Since $G$ 
contains no negative cycles, by Observation~\ref{obsPosx2},
$G(\drm)=G+G$.
Thus,
since $G$ is connected, it has the form at the left of Figure~\ref{figNegCycle}, in which $H$ is the component of $G-u$ containing $v$.
Moreover,
we have
$\gr(\Geom(M))=\lambda(G)$.
(This also follows from Corollary~\ref{corNoNegCycles}.)
Now, if we expand $G$ at $u$ along $uv$, by Lemma~\ref{lemmaExpandVertex}, 
$\lambda(G)=\lambda(K)$, where $K$ is the graph in the centre of Figure~\ref{figNegCycle}.

On the other hand, $\Gsub$ is obtained from $G(\drm)$ by removing the edges $uv$ and $u'v'$, and adding $uv'$ and $u'v$, as shown at the
right of Figure~\ref{figNegCycle}.
It is readily observed that $K$ is a proper subgraph of $\Gsub$ (see the shaded box in Figure~\ref{figNegCycle}), and hence, by Corollary~\ref{corStrictMono}, $\lambda(K)<\lambda(\Gsub)$. Since $\gr(\Geom(\Msub))=\lambda(\Gsub)$, the result follows.
\end{proof}
Thus, making the first negative cycle increases the growth rate.
We suspect, in fact,
that the following stronger statement is also true:

\thmbox{
\begin{conj}\label{conjNegateCell}
If $G(M)$ is connected and $\Msub$ results from negating a single entry of $M$
that is in one or more positive cycles but in no negative cycle,
then
\vspace{-6pt}
$$\gr(\Geom(\Msub)) \;>\; \gr(\Geom(M)).$$
\end{conj}
} 

To prove this more general result seems to require some new ideas.
If $G(M)$ already contains a negative cycle, then $G(\drm)$ is connected, and, when this is the case,
there appears to be no obvious way to generate a subgraph of $G({\Msub}^{\!\times2})$
by expanding $G(\drm)$.

\subsection{Monotone grid classes}

Typically,
the growth rate of a monotone grid class is greater than
that
of the corresponding geometric grid class.
For example,
if $G(M)$ is a cycle then $\gr(\Grid(M))=4$, whereas from~\eqref{eqGRGeomCycle} we know that $\gr(\Geom(M))<4$.
As another example,
$$
\gr\!\left(\begin{gridmx}1&\pos0&-1\\ 1&-1&\pos1 \end{gridmx}\right)
\;=\;
\gr\!\left(\begin{gridmx}-1&\pos0&-1\\ \pos1&-1&\pos1 \end{gridmx}\right)
\;=\;
\thalf(5+\sqrt{17})
\;\approx\;
4.56155,
$$
which should be compared with~\eqref{eqGRGeomExNegCyc} and~\eqref{eqGRGeomExPosCyc}.

The fact that the growth rate of the monotone grid class is strictly greater
is a consequence of the fact that, if $G$ is connected and not acyclic, then 
$\lambda(G)$ and $\rho(G)$ are distinct:
\begin{lemma}[Godsil \& Gutman~\cite{GG1981}] 
\label{lemmaLambdaLessRhoCyclic}
If 
$G$ is connected and contains a cycle, then
$\lambda(G)<\rho(G)$.
\end{lemma}
\begin{proof} 
By Lemma~\ref{lemmaStrictMono},
if $C$ is a nonempty subgraph of $G$, then $\rho(G-C)<\rho(G)$.
So we have $\Phi_{G-C}(z)>0$ for all $z\geqslant\rho(G)$.
Moreover, $\Phi_G(z)\geqslant0$ for $z\geqslant\rho(G)$. So,
since $G$ contains a cycle, from Lemma~\ref{lemmaMuCycleDef} we can deduce that $\mu_G(z)>0$ if $z\geqslant\rho(G)$, and thus
$\lambda(G)<\rho(G)$.
\end{proof}
Note that,
analogously to Observation~\ref{obsGeomDouble},
$\Grid(\drm)=\Grid(M)$.
Hence it must be the case that $\rho(G(\drm))=\rho(G(M))$, the growth rate of a monotone grid class thus being independent of the parity of its cycles.
As a consequence, from Lemma~\ref{lemmaLambdaLessRhoCyclic} we can deduce that in the non-acyclic case
there is a strict inequality between the growth rate of a geometric grid class and the growth rate of the corresponding monotone grid class:

\thmbox{
\begin{cor}\label{corGeomGridGRIneq}
  If $G(M)$ is connected,
  then $\gr(\Geom(M))<\gr(\Grid(M))$ 
  if and only if $G(M)$ contains a cycle.
\end{cor}
} 

\subsection{Subdivision of edges}

One particularly surprising result in Chapter~\ref{chap05} concerning the growth rates of monotone grid classes is the fact that classes
whose row-column graphs
have longer internal paths or cycles exhibit \emph{lower} growth rates (Corollary~\ref{corSubdivision}).
Recall that an edge $e$ of a graph $G$ is said to lie on an \emph{endpath} of $G$ if $G\deledge e$ is disconnected and one of its components is a (possibly trivial) path. An edge that does not lie on an endpath is said to be \emph{internal}.
The following result of Hoffman \& Smith states that the subdivision of an edge increases or decreases the spectral radius of the graph depending on whether the edge lies on an endpath or is internal:
\begin{lemma}[Hoffman \& Smith~\cite{HS1975}]\label{lemmaRhoSubdivision}
  Let $G$ be a connected graph and $G'$ be obtained from $G$ by subdividing an edge $e$. If $e$ lies on an endpath, then $\rho(G')>\rho(G)$. Otherwise (if $e$ is an internal edge), $\rho(G')\leqslant\rho(G)$, with equality if and only if $G$ is a cycle or has the following form (which we call an ``$H$~graph''):
$$
  \begin{tikzpicture}[scale=0.4]
    \draw [thick,dashed] (0,0)--(4,0);
    \draw [thick] (-.5,.866)--(0,0)--(-.5,-.866);
    \draw [thick] (4.5,.866)--(4,0)--(4.5,-.866);
    \draw [fill] (0,0) circle [radius=0.15];
    \draw [fill] (4,0) circle [radius=0.15];
    \draw [fill] (-.5,.866) circle [radius=0.15];
    \draw [fill] (-.5,-.866) circle [radius=0.15];
    \draw [fill] (4.5,.866) circle [radius=0.15];
    \draw [fill] (4.5,-.866) circle [radius=0.15];
  \end{tikzpicture}
$$
\end{lemma}
Thus for monotone grid classes, if $G(M)$ is connected,
and $G(M')$ is obtained from $G(M)$ by the subdivision of one or more internal edges, then $\gr(\Grid(M'))\leqslant\gr(\Grid(M))$.

As we see shortly, the situation is not as simple for geometric grid classes.
The effect of edge subdivision on the largest root of the matching polynomial does not seem to have been addressed previously.
In fact, the subdivision of an edge that is in a cycle may
cause $\lambda(G)$ 
to increase or decrease, or may leave it unchanged.
See Figures~\ref{figGeomSubdivisionIncr}--\ref{figGeomSubdivisionDecr} for illustrations of the three cases.
We investigate this further below.
However, if the edge being subdivided is not on a cycle in $G$, then the behaviour of $\lambda(G)$ mirrors that of $\rho(G)$, as we now demonstrate:
\begin{lemma}\label{lemmaSubdividing1}
  Let $G$ be a connected graph and $G'$ be obtained from $G$ by subdividing an edge $e$.
  If $e$ lies on an endpath, then $\lambda(G')>\lambda(G)$.
  However, if $e$ is an internal edge and not on a cycle, then $\lambda(G')\leqslant\lambda(G)$, with equality if and only if $G$ is an $H$ graph.
\end{lemma}

\begin{figure}[ht]
$$
\raisebox{12pt}{
\begin{tikzpicture}[scale=0.36]
    \draw [thick] (0,.3)--(1.7,1.5);
    \draw [thick] (0,.3)--(2.1,1.1);
    \draw [thick] (0,.3)--(3,0.75);
    \draw [thick,fill=gray!30!white] (3,1.4) circle [x radius=1.3, y radius=1];
    \draw [thick,fill=gray!30!white] (7.2,1.4) circle [x radius=1.2, y radius=.9];
    \draw [fill] (0,.3) circle [radius=0.15];
    \node[left]at(0,.3){$u$};
    \node[left]at(-1.5,.3){$G$:};
    \draw [fill] (1.7,1.5) circle [radius=0.15];
    \draw [fill] (4.3,1.4) circle [radius=0.15];
    \draw [fill] (6,1.4) circle [radius=0.15];
    \draw [ultra thick] (4.3,1.4)--(6,1.4);
    \node[above]at(5.15,1.4){$e$};
    \node[left]at(1.7,1.6){$v$};
    \node[]at(3,1.4){$\,H_1$};
    \node[]at(7.2,1.4){$H_2$};
    \draw [thick] (0,.3)--(3,-0.75);
    \draw [thick] (0,.3)--(2.33,-1);
    \draw [thick] (0,.3)--(2.17,-1.5);
    \draw [thick] (0,.3)--(2,-2);
    \draw [thick,fill=gray!30!white] (3.4,-1.35) circle [x radius=1.75, y radius=1.15];
    \node[]at(3.4,-1.35){$J$};
\end{tikzpicture}
}
\qquad\qquad\qquad
\begin{tikzpicture}[scale=0.36]
    \draw [thick,fill=gray!30!white] (3,3.8) circle [x radius=1.3, y radius=1];
    \draw [thick] (0,.3)--(1.7,3.9);
    \draw [thick] (0,.3)--(2.1,1.1);
    \draw [thick] (0,.3)--(3,0.75);
    \draw [thick,fill=gray!30!white] (3,1.4) circle [x radius=1.3, y radius=1];
    \draw [fill] (0,.3) circle [radius=0.15];
    \node[left]at(0,.3){$u$};
    \node[left]at(-1.5,1.4){$K$:};
    \draw [fill] (1.7,3.9) circle [radius=0.15];
    %
    \draw [thick,fill=gray!30!white] (7.2,3.8) circle [x radius=1.2, y radius=.9];
    \draw [fill] (4.3,3.8) circle [radius=0.15];
    \draw [fill] (6,3.8) circle [radius=0.15];
    \draw [ultra thick] (4.3,3.8)--(6,3.8);
    \node[above]at(5.15,3.8){$e_1$};
    %
    \draw [thick,fill=gray!30!white] (7.2,1.4) circle [x radius=1.2, y radius=.9];
    \draw [fill] (4.3,1.4) circle [radius=0.15];
    \draw [fill] (6,1.4) circle [radius=0.15];
    \draw [ultra thick] (4.3,1.4)--(6,1.4);
    \node[above]at(5.15,1.4){$e_2$};
    \node[left]at(1.7,4){$v$};
    \node[]at(3,1.4){$\,H_1$};
    \node[]at(3,3.8){$\,H_1$};
    \node[]at(7.2,1.4){$H_2$};
    \node[]at(7.2,3.8){$H_2$};
    \draw [thick] (0,.3)--(3,-0.75);
    \draw [thick] (0,.3)--(2.33,-1);
    \draw [thick] (0,.3)--(2.17,-1.5);
    \draw [thick] (0,.3)--(2,-2);
    \draw [thick,fill=gray!30!white] (3.4,-1.35) circle [x radius=1.75, y radius=1.15];
    \node[]at(3.4,-1.35){$J$};
\end{tikzpicture}
\qquad
$$
\caption{Graphs used in the proof of Lemma~\ref{lemmaSubdividing1}} 
\label{figExpandingVertexEInternal}
\end{figure}
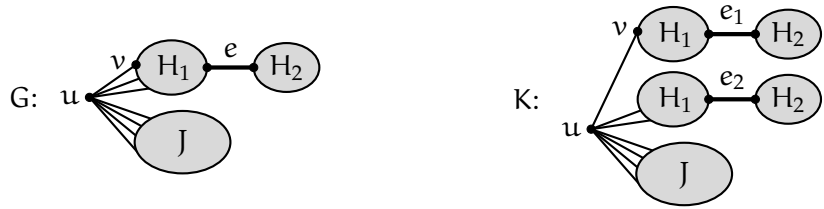
\begin{proof}
If $e$ lies on an endpath, then $G$ is a proper subgraph of $G'$ and so
the result follows from Corollary~\ref{corStrictMono}.
On the other hand, if $e$ is internal and $G$ is acyclic, the conclusion is a consequence of Corollary~\ref{corLambdaRhoTree} and Lemma~\ref{lemmaRhoSubdivision}.
Thus, we need only consider the situation in which
$e$ is internal and $G$ contains a cycle. We proceed by induction on the number of cycles in $G$, acyclic graphs constituting the base case.
Let $uv$ be an edge in a cycle of $G$ such that $u$ is not an endvertex of~$e$.
Now, let $K$ be the result of expanding $G$ at $u$ along $uv$, and let $K'$, analogously, be the result of expanding $G'$ at $u$ along~$uv$.

We consider the effect of the expansion of $G$ upon $e$ and the effect of the expansion of $G'$ upon the two edges resulting from the subdivision of $e$.
If $e$ is in the component of $G-u$ containing~$v$, then $e$ is duplicated in $K$, both copies of $e$ remaining internal (see Figure~\ref{figExpandingVertexEInternal}).
Moreover, $K'$ results from subdividing both copies of $e$ in $K$.
Conversely, if $e$ is in a component of $G-u$ not containing~$v$, then $e$ is not duplicated in $K$ (and remains internal).
In this case, $K'$ results from subdividing $e$ in $K$.
In either case, $K'$ is the result of subdividing internal edges of $K$ (a graph with fewer cycles than $G$), and so
the result follows from the induction hypothesis.
\end{proof}
Now,
the subdivision of an edge
of a row-column graph
that is not on a cycle has no effect on the parity of the cycles.
Hence,
we
have the following conclusion
for the growth rates of geometric grid classes:

\thmbox{
\begin{cor}
If $G(M)$ is connected,
and $G(M')$ is obtained from $G(M)$ by the subdivision of one or more internal edges not on a cycle, then $\gr(\Geom(M'))\leqslant\gr(\Geom(M))$, with equality if and only if $G(M)$ is an $H$ graph.
\end{cor}
} 

\begin{figure}[ht] 
\vspace{6pt}
$$
\raisebox{34.2pt}{
\begin{tikzpicture}[scale=0.36]
    \draw [thick] (0,3)--(-3,4);
    \draw [thick] (0,3)--(-3,3);
    \draw [thick] (0,3)--(-3,2);
    \draw [thick] (0,3)--(1.3,4.7);
    \draw [thick] (0,3)--(1.662,4.126);
    \draw [thick] (0,3)--(2.3,3.834);
    \draw [thick] (0,3)--(2.4,2.1);
    \draw [thick] (0,3)--(2.1,1.7);
    \draw [thick] (0,3)--(1.8,1.3);
    \draw [thick] (0,3)--(1.5,0.9);
    \draw [ultra thick] (5.3,1.3)--(5.3,4.7);
    \draw [fill=gray!30!white,thick] (3.3,4.7) circle [x radius=2, y radius=1];
    \node[]at(3.3,4.7){$\phantom{{}_.}H_1$};
    \draw [fill=gray!30!white,thick] (3.3,1.3) circle [x radius=2, y radius=1];
    \node[]at(3.3,1.3){$\phantom{{}_.}H_2$};
    \draw [fill=gray!30!white,thick] (-3,3) circle [x radius=1.7, y radius=1.1];
    \node[]at(-3,3){$J$};
    \draw [fill] (0,3) circle [radius=0.15];
    \draw [fill] (1.3,4.7) circle [radius=0.15];
    \draw [fill] (1.662,4.126) circle [radius=0.15];
    \draw [fill] (2.3,3.834) circle [radius=0.15];
    \draw [fill] (5.3,4.7) circle [radius=0.15];
    \draw [fill] (5.3,1.3) circle [radius=0.15];
    \node[above left]at(0,3.1){$u\!\!\!$};
    \node[right]at(5.3,3){$e$};
    \node[right]at(5.3,4.8){$x_1$};
    \node[right]at(5.3,1.2){$x_2$};
    \node[left]at(-5.4,3){$G$:};
    \node[below left]at(0,1.3){$\phantom{x_1}$};
\end{tikzpicture}
}
\qquad\qquad\quad
\begin{tikzpicture}[scale=0.36]
    \draw [thick] (0,3)--(1.3,9.4);
    \draw [ultra thick] (5.3,9.4)--(7.7,9.4);
    \draw [fill=gray!30!white,thick] (3.3,9.4) circle [x radius=2, y radius=1];
    \node[]at(3.3,9.4){$\phantom{{}_.}H_1$};
    \draw [fill=gray!30!white,thick] (9.7,9.4) circle [x radius=2, y radius=1];
    \node[]at(9.7,9.4){$\phantom{{}_.}H_2$};
    \draw [fill] (1.3,9.4) circle [radius=0.15];
    \draw [fill] (5.3,9.4) circle [radius=0.15];
    \draw [fill] (7.7,9.4) circle [radius=0.15];
    \node[below left]at(8.2,9.4){$x_2$};
    \draw [thick] (0,3)--(1.662,6.126);
    \draw [ultra thick] (5.3,6.7)--(7.7,6.7);
    \draw [fill=gray!30!white,thick] (3.3,6.7) circle [x radius=2, y radius=1];
    \node[]at(3.3,6.7){$\phantom{{}_.}H_1$};
    \draw [fill=gray!30!white,thick] (9.7,6.7) circle [x radius=2, y radius=1];
    \node[]at(9.7,6.7){$\phantom{{}_.}H_2$};
    \draw [fill] (1.662,6.126) circle [radius=0.15];
    \draw [fill] (5.3,6.7) circle [radius=0.15];
    \draw [fill] (7.7,6.7) circle [radius=0.15];
    \node[below left]at(8.2,6.7){$x_2$};
    \draw [thick] (0,3)--(2.3,3.134);
    \draw [ultra thick] (5.3,4)--(7.7,4);
    \draw [fill=gray!30!white,thick] (3.3,4) circle [x radius=2, y radius=1];
    \node[]at(3.3,4){$\phantom{{}_.}H_1$};
    \draw [fill=gray!30!white,thick] (9.7,4) circle [x radius=2, y radius=1];
    \node[]at(9.7,4){$\phantom{{}_.}H_2$};
    \draw [fill] (2.3,3.134) circle [radius=0.15];
    \draw [fill] (5.3,4) circle [radius=0.15];
    \draw [fill] (7.7,4) circle [radius=0.15];
    \node[below left]at(8.2,4){$x_2$};
    \draw [thick] (0,3)--(-3,4);
    \draw [thick] (0,3)--(-3,3);
    \draw [thick] (0,3)--(-3,2);
    \draw [thick] (0,3)--(2.4,2.1);
    \draw [thick] (0,3)--(2.1,1.7);
    \draw [thick] (0,3)--(1.8,1.3);
    \draw [thick] (0,3)--(1.5,0.9);
    \draw [ultra thick] (5.3,1.3)--(7.7,1.3);
    \draw [fill=gray!30!white,thick] (3.3,1.3) circle [x radius=2, y radius=1];
    \node[]at(3.3,1.3){$\phantom{{}_.}H_2$};
    \draw [fill=gray!30!white,thick] (9.7,1.3) circle [x radius=2, y radius=1];
    \node[]at(9.7,1.3){$\phantom{{}_.}H_1$};
    \draw [fill=gray!30!white,thick] (-3,3) circle [x radius=1.7, y radius=1.1];
    \node[]at(-3,3){$J$};
    \draw [fill] (0,3) circle [radius=0.15];
    \draw [fill] (5.3,1.3) circle [radius=0.15];
    \draw [fill] (7.7,1.3) circle [radius=0.15];
    \node[above left]at(0,3.1){$u\!\:\!\!$};
    \node[below left]at(8.2,1.3){$x_1$};
    \node[left]at(-5.4,6.4){$K$:};
\end{tikzpicture}
$$
\caption{Graphs used in Lemma~\ref{lemmaSubdividing2}}
\label{figExpandingCycle}
\end{figure}
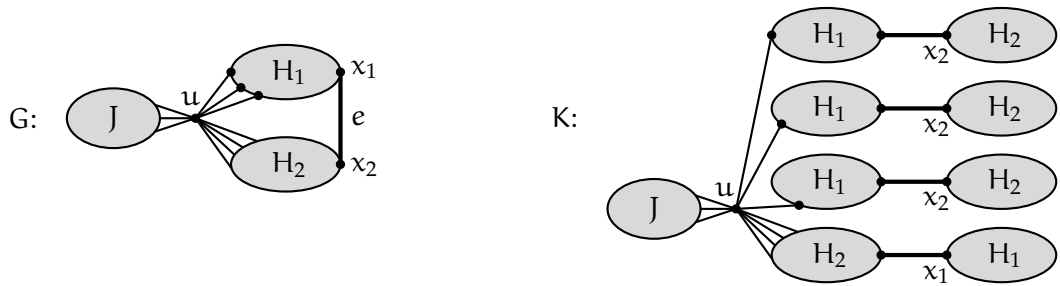
Let us now investigate the effect of subdividing an edge $e$ that lies on a cycle.
We restrict our attention to graphs
in which there is a vertex $u$ such that
the two endvertices of $e$ are in distinct
components of $(G\deledge e) - u$.
See the graph at the left of Figure~\ref{figExpandingCycle} for an illustration.
We leave the consideration of multiply-connected graphs that fail to satisfy this condition for future study.
\begin{lemma}\label{lemmaSubdividing2}
Let $G$ be a connected graph and $e=x_1x_2$ an edge on a cycle $C$ of $G$.
Let $u$ be a vertex on $C$, and let $H_1$ and $H_2$ be the distinct
components of $(G\deledge e) - u$ that contain $x_1$ and $x_2$ respectively.
Finally, let $G'$ be the graph obtained from $G$ by subdividing $e$.
  \vspace{-9pt}
  \begin{enumerate}[(a)]
    \item If, for $i\in\{1,2\}$, $H_i$ is a (possibly trivial) path of which $x_i$ is an endvertex, then $\lambda(G')>\lambda(G)$.
    \item If, for $i\in\{1,2\}$, $H_i$ is not a path or is a path of which $x_i$ is not an endvertex, then $\lambda(G')<\lambda(G)$.
  \end{enumerate}
  \vspace{-9pt}
\end{lemma}
\begin{proof}
Let $K$ be the result of repeatedly expanding $G$ at $u$ along every edge joining $u$ to $H_1$.
$K$~has the form shown at the right of Figure~\ref{figExpandingCycle}.
Also let $K'$ be the result of repeatedly expanding $G'$ ($G$~with edge $e$ subdivided) in an analogous way at $u$.
Clearly $K'$ is the same as the
graph that results from subdividing the
copies of $e$
in $K$.

Now, for part (a), since $H_1$ is a path with an end at $x_1$, and also $H_2$ is a path with an end at $x_2$, we see that $K'$ is the result of subdividing edges of $K$ that are on endpaths.
Hence, by the first part of Lemma~\ref{lemmaSubdividing1}, we have $\lambda(G')>\lambda(G)$ as required.

For part (b), since $H_1$ is not a path with an end at $x_1$, and nor is $H_2$ a path with an end at $x_2$,
we see that $K'$ is the result of subdividing internal edges of $K$.
Since $K$ is not an $H$ graph, by Lemma~\ref{lemmaSubdividing1}, we have $\lambda(G')<\lambda(G)$ as required.
\end{proof}
If the conditions for parts (a) and (b) of this lemma both fail to be satisfied (i.e. $H_1$ is a suitable path and $H_2$ isn't, or \emph{vice versa}), then
the proof fails.
This is due to the fact that
expansion leads to at least one copy of $e$ in $K$ being internal
and to another copy of $e$ in $K$ being on an endpath.
Subdivision of the former decreases $\lambda(G)$ whereas subdivision of the latter causes it to increase.
Sometimes, as in Figure~\ref{figGeomSubdivisionEq}, these effects balance exactly; on other occasions one or the other dominates.
We leave a detailed analysis of such cases for later study.

\begin{figure}[t] 
    $$
    \gctwo{4}{1,1,1}{0,1,1,1}
    \quad
    \raisebox{-0.09in}{\begin{tikzpicture}[scale=0.4]
    \draw [thick] (0,0)--(1,0)--(1.707,0.707)--(2.414,0)--(1.707,-0.707)--(1,0);
    \draw [thick] (2.414,0)--(3.414,0);
    \draw [fill] (0,0) circle [radius=0.15];
    \draw [fill] (1,0) circle [radius=0.15];
    \draw [fill] (2.414,0) circle [radius=0.15];
    \draw [fill] (3.414,0) circle [radius=0.15];
    \draw [fill] (1.707,0.707) circle [radius=0.15];
    \draw [fill] (1.707,-0.707) circle [radius=0.15];
    \end{tikzpicture}}
    \quad\quad\quad\quad
    \gcthree{5}{1,1,0,1}{0,1,1}{0,0,1,1,1}
    \quad
    \raisebox{-0.13in}{\begin{tikzpicture}[scale=0.4,rotate=30]
    \draw [thick] (0,0)--(1,0)--(1.5,0.866)--(2.5,0.866)--(3,0)--(2.5,-0.866)--(1.5,-0.866)--(1,0);
    \draw [thick] (2.5,-0.866)--(3,-1.732);
    \draw [fill] (0,0) circle [radius=0.15];
    \draw [fill] (1,0) circle [radius=0.15];
    \draw [fill] (3,0) circle [radius=0.15];
    \draw [fill] (3,-1.732) circle [radius=0.15];
    \draw [fill] (1.5,0.866) circle [radius=0.15];
    \draw [fill] (2.5,0.866) circle [radius=0.15];
    \draw [fill] (1.5,-0.866) circle [radius=0.15];
    \draw [fill] (2.5,-0.866) circle [radius=0.15];
    \end{tikzpicture}}
    \quad\quad\quad\quad
    \gcfour{6}{1,1,0,0,1}{0,1,1}{0,0,1,1}{0,0,0,1,1,1}
    \quad
    \raisebox{-0.22in}{\begin{tikzpicture}[scale=0.4,rotate=45]
    \draw [thick] (0,0)--(1,0)--(1.383,0.924)--(2.307,1.307)--(3.23,0.924)--(3.613,0)--(3.23,-0.924)--(2.307,-1.307)--(1.383,-0.924)--(1,0);
    \draw [thick] (2.307,-1.307)--(2.307,-2.307);
    \draw [fill] (0,0) circle [radius=0.15];
    \draw [fill] (1,0) circle [radius=0.15];
    \draw [fill] (3.613,0) circle [radius=0.15];
    \draw [fill] (1.383,0.924) circle [radius=0.15];
    \draw [fill] (1.383,-0.924) circle [radius=0.15];
    \draw [fill] (3.23,0.924) circle [radius=0.15];
    \draw [fill] (3.23,-0.924) circle [radius=0.15];
    \draw [fill] (2.307,1.307) circle [radius=0.15];
    \draw [fill] (2.307,-1.307) circle [radius=0.15];
    \draw [fill] (2.307,-2.307) circle [radius=0.15];
    \end{tikzpicture}}
    \vspace{-9pt}
    $$
\caption{Standard figures and row-column graphs of geometric grid classes whose growth rates increase
from left to right}
\label{figGeomSubdivisionIncr}
    $$
    \gctwo{5}{1,1,1,1}{0,0,1,1,1}
    \quad
    \raisebox{-0.09in}{\begin{tikzpicture}[scale=0.4]
    \draw [thick] (0.293,0.707)--(1,0)--(1.707,0.707)--(2.414,0)--(1.707,-0.707)--(1,0);
    \draw [thick] (0.293,-0.707)--(1,0);
    \draw [thick] (2.414,0)--(3.414,0);
    \draw [fill] (0.293,0.707) circle [radius=0.15];
    \draw [fill] (0.293,-0.707) circle [radius=0.15];
    \draw [fill] (1,0) circle [radius=0.15];
    \draw [fill] (2.414,0) circle [radius=0.15];
    \draw [fill] (3.414,0) circle [radius=0.15];
    \draw [fill] (1.707,0.707) circle [radius=0.15];
    \draw [fill] (1.707,-0.707) circle [radius=0.15];
    \end{tikzpicture}}
    \quad\quad\quad\:\;
    \gcthree{6}{1,1,1,0,1}{0,0,1,1}{0,0,0,1,1,1}
    \quad
    \raisebox{-0.16in}{\begin{tikzpicture}[scale=0.4,rotate=30]
    \draw [thick] (0.234,0.643)--(1,0)--(1.5,0.866)--(2.5,0.866)--(3,0)--(2.5,-0.866)--(1.5,-0.866)--(1,0);
    \draw [thick] (2.5,-0.866)--(3,-1.732);
    \draw [thick] (0.234,-0.643)--(1,0);
    \draw [fill] (0.234,0.643) circle [radius=0.15];
    \draw [fill] (0.234,-0.643) circle [radius=0.15];
    \draw [fill] (1,0) circle [radius=0.15];
    \draw [fill] (3,0) circle [radius=0.15];
    \draw [fill] (1.5,0.866) circle [radius=0.15];
    \draw [fill] (2.5,0.866) circle [radius=0.15];
    \draw [fill] (3,-1.732) circle [radius=0.15];
    \draw [fill] (1.5,-0.866) circle [radius=0.15];
    \draw [fill] (2.5,-0.866) circle [radius=0.15];
    \end{tikzpicture}}
    \quad\quad\quad\:\;
    \gcfour{7}{1,1,1,0,0,1}{0,0,1,1}{0,0,0,1,1}{0,0,0,0,1,1,1}
    \quad
    \raisebox{-0.23in}{\begin{tikzpicture}[scale=0.4,rotate=45]
    \draw [thick] (0.207,0.609)--(1,0)--(1.383,0.924)--(2.307,1.307)--(3.23,0.924)--(3.613,0) --(3.23,-0.924)--(2.307,-1.307)--(1.383,-0.924)--(1,0);
    \draw [thick] (2.307,-1.307)--(2.307,-2.307);
    \draw [thick] (0.207,-0.609)--(1,0);
    \draw [fill] (0.207,0.609) circle [radius=0.15];
    \draw [fill] (0.207,-0.609) circle [radius=0.15];
    \draw [fill] (1,0) circle [radius=0.15];
    \draw [fill] (3.613,0) circle [radius=0.15];
    \draw [fill] (1.383,0.924) circle [radius=0.15];
    \draw [fill] (1.383,-0.924) circle [radius=0.15];
    \draw [fill] (3.23,0.924) circle [radius=0.15];
    \draw [fill] (3.23,-0.924) circle [radius=0.15];
    \draw [fill] (2.307,1.307) circle [radius=0.15];
    \draw [fill] (2.307,-1.307) circle [radius=0.15];
    \draw [fill] (2.307,-2.307) circle [radius=0.15];
    \end{tikzpicture}}
    \vspace{-9pt}
    $$
\caption{Standard figures and row-column graphs of geometric grid classes whose growth rates are all the same
(equal to 5)
}
\label{figGeomSubdivisionEq}
\vspace{3pt}
    $$
    \gctwo{6}{1,1,1,1}{0,0,1,1,1,1}
    \quad
    \raisebox{-0.09in}{\begin{tikzpicture}[scale=0.4]
    \draw [thick] (0.293,0.707)--(1,0)--(1.707,0.707)--(2.414,0)--(1.707,-0.707)--(1,0);
    \draw [thick] (0.293,-0.707)--(1,0);
    \draw [thick] (3.121,0.707)--(2.414,0)--(3.121,-0.707);
    \draw [fill] (0.293,0.707) circle [radius=0.15];
    \draw [fill] (0.293,-0.707) circle [radius=0.15];
    \draw [fill] (1,0) circle [radius=0.15];
    \draw [fill] (2.414,0) circle [radius=0.15];
    \draw [fill] (3.121,0.707) circle [radius=0.15];
    \draw [fill] (3.121,-0.707) circle [radius=0.15];
    \draw [fill] (1.707,0.707) circle [radius=0.15];
    \draw [fill] (1.707,-0.707) circle [radius=0.15];
    \end{tikzpicture}}
    \quad\quad\quad
    \gcthree{7}{1,1,1,0,1}{0,0,1,1}{0,0,0,1,1,1,1}
    \quad
    \raisebox{-0.16in}{\begin{tikzpicture}[scale=0.4,rotate=30]
    \draw [thick] (0.234,0.643)--(1,0)--(1.5,0.866)--(2.5,0.866)--(3,0)--(2.5,-0.866)--(1.5,-0.866)--(1,0);
    \draw [thick] (2.326,-1.851)--(2.5,-0.866)--(3.44,-1.208);
    \draw [thick] (0.234,-0.643)--(1,0);
    \draw [fill] (0.234,0.643) circle [radius=0.15];
    \draw [fill] (0.234,-0.643) circle [radius=0.15];
    \draw [fill] (1,0) circle [radius=0.15];
    \draw [fill] (3,0) circle [radius=0.15];
    \draw [fill] (1.5,0.866) circle [radius=0.15];
    \draw [fill] (2.5,0.866) circle [radius=0.15];
    \draw [fill] (2.326,-1.851) circle [radius=0.15];
    \draw [fill] (3.44,-1.208) circle [radius=0.15];
    \draw [fill] (1.5,-0.866) circle [radius=0.15];
    \draw [fill] (2.5,-0.866) circle [radius=0.15];
    \end{tikzpicture}}
    \quad\quad\quad
    \gcfour{8}{1,1,1,0,0,1}{0,0,1,1}{0,0,0,1,1}{0,0,0,0,1,1,1,1}
    \quad
    \raisebox{-0.23in}{\begin{tikzpicture}[scale=0.4,rotate=45]
    \draw [thick] (0.207,0.609)--(1,0)--(1.383,0.924)--(2.307,1.307)--(3.23,0.924)--(3.613,0) --(3.23,-0.924)--(2.307,-1.307)--(1.383,-0.924)--(1,0);
    \draw [thick] (1.698,-2.1)--(2.307,-1.307)--(2.916,-2.1);
    \draw [thick] (0.207,-0.609)--(1,0);
    \draw [fill] (0.207,0.609) circle [radius=0.15];
    \draw [fill] (0.207,-0.609) circle [radius=0.15];
    \draw [fill] (1,0) circle [radius=0.15];
    \draw [fill] (3.613,0) circle [radius=0.15];
    \draw [fill] (1.383,0.924) circle [radius=0.15];
    \draw [fill] (1.383,-0.924) circle [radius=0.15];
    \draw [fill] (3.23,0.924) circle [radius=0.15];
    \draw [fill] (3.23,-0.924) circle [radius=0.15];
    \draw [fill] (2.307,1.307) circle [radius=0.15];
    \draw [fill] (2.307,-1.307) circle [radius=0.15];
    \draw [fill] (1.698,-2.1) circle [radius=0.15];
    \draw [fill] (2.916,-2.1) circle [radius=0.15];
    \end{tikzpicture}}
    \vspace{-9pt}
    $$
\caption{Standard figures and row-column graphs of geometric grid classes whose growth rates decrease
from left to right}
\label{figGeomSubdivisionDecr}
\end{figure}
To conclude, we state the consequent result for the growth rates of geometric grid classes.
To simplify its statement and avoid having to concern ourselves directly with cycle parities, we
define $\grm(M)$ to be $G(M)$ when $G(M)$ has no negative cycles and $\grm(M)$ to be $G(\drm)$ otherwise.

\newpage  
\thmbox{
\begin{cor}
  Suppose $\grm(M)$ is connected.
  \vspace{-3pt}
  \begin{enumerate}[(a)]
    \item If $\grm(M')$ is obtained from $\grm(M)$ by subdividing one or more edges that satisfy the conditions of part (a) of Lemma~\ref{lemmaSubdividing2}, then $\gr(\Geom(M'))>\gr(\Geom(M))$.
    \item If $\grm(M')$ is obtained from $\grm(M)$ by subdividing one or more edges that satisfy the conditions of part (b) of Lemma~\ref{lemmaSubdividing2}, then $\gr(\Geom(M'))<\gr(\Geom(M))$.
  \end{enumerate}
  \vspace{-9pt}
\end{cor}
} 

Figure~\ref{figGeomSubdivisionIncr} provides an illustration of part (a) and Figure~\ref{figGeomSubdivisionDecr} an illustration of part (b).

\HIDE{
\textbf{Small growth rates}

only others less than $2+\sqrt5$: cycles, pans: cycle + pendant edge, or $4$-cycle + path with 2 edges

--- Ma~\cite{Ma2001,Ma2005} (in Chinese), see Qiao \& Zhan~\cite{QZ2011}
} 

\HIDE{
\textbf{Bounds}
\begin{bullets}
  \item If $\Delta>1$, $\gamma<4\+\Delta-4$ (Heilmann \& Lieb~\cite{HL1972})
  \item $\gamma\geqslant\Delta$ (Lov\'asz \& Pelik\'an~\cite{LP1973})?
  \item $\gamma\geqslant2\overline{d}-1$ (Fisher \& Ryan~\cite{FR1992})
\end{bullets}
Note: $\drm$ has the same maximum and average degree as $M$.
} 

\setgcgap{0.18} 

\cleardoublepage


\part{\textsc{The Set of Growth Rates}}\label{partII} 
\setcounter{chapter}{7}

\chapter{Intervals of growth rates}\label{chap08}

\section{Introduction}\label{sectLambdaPaperIntro}

What does the set of possible permutation class growth rates look like?
The last few years have seen substantial progress on answering this question, with particular focus on 
significant phase transition values.
Kaiser \& Klazar~\cite{KK2003} characterised all growth rates up to 2, showing that
the golden ratio $\varphi\approx1.61803$ (the unique positive root of $x^2-x-1$) is the least growth rate greater than 1.
Indeed, they prove a stronger statement, known as the \emph{Fibonacci dichotomy}, that if $|\CCC_n|$ is ever less
than the $n$th Fibonacci number, then $|\CCC_n|$ is eventually polynomial.
Kaiser \& Klazar also determine that
2 is the least limit point in the set of growth rates.
Huczynska \& Vatter~\cite{HV2006} later gave a shorter proof of the Fibonacci dichotomy, by characterising  polynomial permutation classes in terms of monotone grid classes (see the brief discussion above on page~\mypageref{defPolynomialClass}).

Klazar~\cite{Klazar2004} considered the least growth rate admitting uncountably many permutation classes (which he denoted $\kappa$) and proved that $\kappa$ is at least 2 and is no greater than
approximately $2.33529$
(a root of a quintic). 
Vatter~\cite{Vatter2011} determined the exact value of $\kappa$ to be the unique real root of $x^3-2\+x^2-1$ (approximately $2.20557$) and completed the
characterisation of all growth rates up to $\kappa$, proving that there are uncountably many permutation classes with growth rate $\kappa$ and only countably many with growth rate less than $\kappa$.
Central to Vatter's work is the use of (generalised) grid classes.
The phase
transition at $\kappa$ also has enumerative ramifications: Albert, Ru\v{s}kuc \& Vatter~\cite{ARV2012} have shown that every permutation class with growth rate less than
$\kappa$ has a rational generating function.

Balogh, Bollob\'as \& Morris~\cite{BBM2006} extended Kaiser \& Klazar's work to the more
general setting of ordered graphs.
They conjectured that the set of growth
rates of hereditary classes of ordered graphs contains no limit points from above, and also that all
such growth rates are integers or algebraic irrationals. These conjectures were
disproved by Albert \& Linton~\cite{AL2009} who exhibited an uncountable \emph{perfect set} (a closed set every member of which is a limit point) of growth rates of permutation classes and in turn conjectured that the set of growth rates includes some interval $(\lambda,\infty)$.

Their conjecture was established by Vatter~\cite{Vatter2010b} who proved that there are permutation classes of every growth rate at least $\lambda_A\approx2.48187$ (the unique real root of $x^5-2\+x^4-2\+x^2-2\+x-1$). In addition, he conjectured 
that this value was optimal, the set of growth rates below $\lambda_A$ being nowhere dense.
By generalising Vatter's constructions, we now refute his conjecture by proving the following two results:

\thmbox{
\begin{thm}\label{thmTheta}
Let $\theta_B\approx2.35526$ be the unique real root of $x^7-2\+x^6-x^4-x^3-2\+x^2-2\+x-1$.
For any $\varepsilon>0$, there exist $\delta_1$ and $\delta_2$ with $0<\delta_1<\delta_2<\varepsilon$ such that every value in the interval $[\theta_B+\delta_1,\theta_B+\delta_2]$ is the growth rate of a permutation class.
\end{thm}
}

\thmbox{
\begin{thm}\label{thmLambda}
Let $\lambda_B\approx2.35698$ be the unique positive root of $x^8-2\+x^7-x^5-x^4-2\+x^3-2\+x^2-x-1$.
Every value at least $\lambda_B$ is the growth rate of a permutation class.
\end{thm}
}

The proofs of these theorems
are based upon an analysis of expansions of real numbers in non-integer bases, 
the study of which was initiated by R\'enyi in the 1950s.
In particular,
we prove two generalisations of a result of Pedicini concerning %
expansions in which the digits are drawn from sets of allowed values.
These results have been submitted for
publication (see~\cite{Bevan2014b}).

\newcommand{\grset}{{\boldsymbol\Gamma}}
In order to complete the characterisation of all permutation class growth rates, further investigation is required of the values between $\kappa$ and $\lambda_B$.
This interval contains the first limit point from above (which we denote $\xi$) and the first perfect set (whose infimum we denote $\eta$), as well as the first interval (whose infimum we denote $\theta$).
In~\cite{Vatter2010b},
Vatter showed that $\xi_A\approx2.30522$ (the unique real root of $x^5-2\+x^4-x^2-x-1$) is the infimum of a perfect set in the set of growth rates and hence a limit point from above, and conjectured that this is the least such limit point (in which case $\eta=\xi$).
If $\grset$ denotes the set of all permutation class growth rates, then the current state of our knowledge can be summarised by the following table:
\begin{center}
\renewcommand{\arraystretch}{1.25}
\begin{tabular}{|c|l|r|}
  \hline
  $\kappa$  & $\inf\,\{\+\gamma:|\grset\cap(1,\gamma)|>\aleph_0\+\}$ & $\kappa\approx2.20557$ \\ \hline
  $\xi$     & $\inf\,\{\+\gamma:\gamma\text{~is a limit point from above in~}\grset\+\}$ & $\xi\leqslant\xi_A\approx2.30522$ \\ \hline
  $\eta$    & $\inf\,\{\+\gamma:\gamma\in P, P\subset\grset, P\text{~is~a~perfect~set}\+\}$ & $\eta\leqslant\xi_A\approx2.30522$ \\ \hline
  $\theta$  & $\inf\,\{\+\gamma:(\gamma,\delta)\subset\grset\+\}$ & $\theta\leqslant\theta_B\approx2.35526$ \\ \hline
  $\lambda$ & $\inf\,\{\+\gamma:(\gamma,\infty)\subset\grset\+\}$ & $\lambda\leqslant\lambda_B\approx2.35698$ \\
  \hline
\end{tabular}
\end{center}
Further research is required to determine whether the upper bounds on $\xi$, $\eta$, $\theta$ and $\lambda$ are, in fact, equalities.
Currently, the only proven lower bound we have for any of these is $\kappa$.

Clearly, there are analogous phase transition values in the set of growth rates of hereditary classes of ordered graphs, and the permutation class upper bounds also bound the corresponding ordered graph values from above.
The techniques used here
could readily be applied directly to ordered graphs to yield smaller upper bounds
to the ordered graph phase transitions.
We leave such a study for the future.

In the next section, we investigate expansions of real numbers in non-integer bases.
In~\cite{Pedicini2005}, Pedi\-cini considered the case in which the digits are drawn from some allowable digit set
and determined the conditions on the base under which every real number in an interval can be expressed by such an expansion.
We prove two generalisations of Pedicini's result, firstly permitting each
digit to be drawn from a different allowable digit set, and secondly allowing the use of what we call \emph{generalised digits}.

In Section~\ref{sectSumClosed}, we consider the growth rates of sum-closed permutation classes (whose members are constructed from a sequence of indecomposable permutations).
Building on our results in the previous section concerning expansions of real numbers in non-integer bases, we
determine sufficient conditions for the set of growth rates of a family
of sum-closed permutation classes to include an interval of the real line.
These conditions relate to the enumeration of the indecomposable permutations in the classes. 

In the final section, we prove our two theorems by constructing families of permutation classes whose indecomposable permutations
are enumerated by sequences satisfying the required conditions.
The key permutations in our constructions are formed
from permutations known as \emph{increasing oscillations} by \emph{inflating} their ends.
Our constructions are similar to, but considerably more general than, those used by Vatter to prove that the set of permutation class growth rates includes the interval $[\lambda_A,\infty)$.

\section{Non-integer bases and generalised digits}

Suppose we want a way of representing any value in an interval of the real line.
As is well known, infinite sequences of small non-negative integers suffice.
Given some integer $\beta>1$ and
any number
$x\in[0,1]$
there is some sequence $(a_n)$ where each
$a_n\in\{0,1,\ldots,\beta-1\}$
such that
$$
x \;=\;
\sum\limits_{n=1}^{\infty}a_n \+ \beta^{-n}
.
$$
This is simply the familiar ``base $\beta$'' expansion 
of $x$,
normally written $x=0.a_1a_2a_3\ldots$,
the $a_n$ being ``$\beta$-ary digits''.
For simplicity, we use $(a_n)_\beta$ to denote the sum above. 
For example,
if $a_n=1$ for each $n\geqslant1$
then $(a_n)_3=\thalf$.

In a seminal paper, R\'enyi~\cite{Renyi1957} generalized these expansions to arbitrary real
bases $\beta > 1$, and since then many papers have been devoted to the various connections with other branches of mathematics,
such as measure theory and dynamical systems. For further information, see the recent survey of Komornik~\cite{Komornik2011}.
Of particular relevance to us, Pedi\-cini~\cite{Pedicini2005} explored the further generalisation in which the digits are drawn from some allowable digit set
$A = \{a_0,\ldots,a_m\}$,
with $a_0 < \ldots < a_m$.
He proved that every point in the interval
$\big[\frac{a_0}{\beta-1},\frac{a_m}{\beta-1}\big]$
has an expansion of the above form
with $a_n\in A$ for all $n\geqslant1$ if and only if
$
\max\limits_{1\leqslant j\leqslant m}(a_j-a_{j-1})
\leqslant
\frac{a_m-a_0}{\beta-1}
.
$

We establish two additional generalisations of Pedicini's result.
Firstly, we permit each digit to be drawn from a different
allowable digit set.
Specifically, 
for each~$n$
we allow the $n^{\text{th}}$
digit $a_n$ to be drawn from some finite set $A_n$ of permitted values.
For instance, we could have $A_n=\{1,4\}$ for odd $n$ and $A_n=\{1,3,5,7,9\}$ for even $n$.
We return to this example below.

Let $(A_n)_\beta=\{(a_n)_\beta:a_n\in A_n\}$ be the set of values for which there is an expansion in base~$\beta$.
The following lemma, generalising the result of Pedicini, gives sufficient conditions on $\beta$ for $(A_n)_\beta$ to be an interval of the real line.

\begin{lemma}\label{lemmaBaseNotation1}
Given a sequence $(A_n)$ of non-empty finite sets of non-negative real numbers, let $\ell_n$ and $u_n$
denote the least and greatest (`lower' and `upper') elements of $A_n$ respectively,
and let $\Delta_n=u_n-\ell_n$.
Also,
let
$\delta_n$
be the largest gap between consecutive elements of $A_n$ (with $\delta_n=0$ if $|A_n|=1$).

If $(u_n)$ is bounded, 
$\beta>1$,
and
for each $n$, $\beta$ satisfies the inequality
$$
\delta_n
 \;\leqslant\;
\sum\limits_{i=1}^{\infty}\Delta_{n+i}\+\beta^{-i},
$$
then
$(A_n)_\beta=\big[(\ell_n)_\beta,\,(u_n)_\beta\big]$.
\end{lemma}
\begin{proof}
If we let $A'_n=\{a-\ell_n:a\in A_n\}$ for each $n$, then the set of values that can be represented by digits from the $A_n$ is
the same as the set of values representable using digits from the $A'_n$ shifted 
by $(\ell_n)_\beta$:
$$
(A_n)_\beta \;=\; \{(a_n)_\beta:a_n\in A_n\} \;=\; \{(a'_n)_\beta+(\ell_n)_\beta:a'_n\in A'_n\}.
$$
So we need only consider cases in which each $\ell_n$ is zero, whence we have $\Delta_n=u_n$ for all $n$.

Given some $x$ in the specified interval,
we choose each digit maximally such that no partial sum exceeds $x$,
i.e. for each $n$, we select $a_n$ to be the greatest element of $A_n$ such that
$$
S(n) \;=\;
\sum\limits_{i=1}^{n}a_i\+ \beta^{-i} \;\leqslant\; x.
$$
It follows that
$(a_n)_\beta
=\liminfty S(n)
\leqslant x
$.
We claim that $(a_n)_\beta=x$.

If, for some $n$, $a_n<u_n$ then there is some element of $A_n$ greater than $a_n$ but no larger than $a_n+\delta_n$, so by our choice of the $a_n$, $S(n)+\delta_n\+\beta^{-n}>x$.
Hence, $|x-S(n)|<\delta_n\+\beta^{-n}$.
Thus, if $a_n<u_n$ for infinitely many values of $n$, $(a_n)_\beta=\liminfty S(n)=x$, since the $\delta_n$ are bounded and $\beta>1$.

Suppose now that $a_n=u_n$ for all but finitely many $n$.
If, in fact, $a_n=u_n$ for all $n$, then we have $x=(u_n)_\beta=(a_n)_\beta$ as required.
Indeed, this is the only possibility given the way we have chosen the $a_n$.\footnote{For example, our approach would select $0.50000\ldots$ rather than $0.49999\ldots$ as the decimal representation of~$\half$, but $0.99999\ldots$ as the only possible representation for $1$.}
Let us assume that $N$ is the greatest index for which $a_N<u_N$, so that $S(N)+\delta_N\+\beta^{-N}>x$.

Now, by the constraints on $\beta$ and the fact that $\Delta_i=u_i$ for all $i$, we have
$$
\delta_N\+ \beta^{-N}
\;\leqslant\;
\sum\limits_{i=N+1}^{\infty}u_i\+\beta^{-i},
$$
and thus
$$
(a_n)_\beta \;=\; S(N) \:+\: \sum\limits_{i=N+1}^{\infty}u_i\+\beta^{-i} \;>\;x,
$$
which produces a contradiction.
\end{proof}
We refer to the conditions relating $\beta$ to the values of the $\delta_n$ and $\Delta_n$ as the \emph{gap inequalities}. These reoccur in subsequent lemmas.

Let us return to our example. If $A_n=\{1,4\}$ for odd $n$ and $A_n=\{1,3,5,7,9\}$ for even $n$, then to produce an interval of values it is sufficient for $\beta$ to satisfy the 
two gap inequalities:
$$
\begin{array}{rcl}
3 & \!\leqslant\! & 8\+\beta^{\negsup1} + 3\+\beta^{\negsup2} + 8\+\beta^{\negsup3} + 3\+\beta^{\negsup4} + \ldots \;=\; \tfrac{8\+\beta+3}{\beta^2-1},
\qquad\!\text{so~~} \beta\:\leqslant\: \frac{1}{3}(4+\sqrt{34}) \:\approx\: 3.27698 , \\[12pt]
2 & \!\leqslant\! & 3\+\beta^{\negsup1} + 8\+\beta^{\negsup2} + 3\+\beta^{\negsup3} + 8\+\beta^{\negsup4} + \ldots \;=\; \tfrac{3\+\beta+8}{\beta^2-1},
\qquad\!\text{so~~} \beta\:\leqslant\: \frac{1}{4}(3+\sqrt{89}) \:\approx\: 3.10850 .
\end{array}
$$
Thus, if $1<\beta\leqslant\frac{1}{4}\+(3+\sqrt{89})$ then $(A_n)_\beta$ 
is an interval.
For instance, with $\beta=3$, we have $(A_n)_3=\big[\frac{1}{8},\frac{9}{4}\big]$.

We now generalise our concept of a digit 
by permitting members of the $A_n$ to have the form $c_0.c_1\ldots c_k$ for non-negative integers $c_i$, where this is to be interpreted to mean
$$
c_0.c_1\ldots c_k \;=\; \sum\limits_{i=0}^{k}c_i\+ \beta^{-i} 
$$
as would be expected.
We call $c_0.c_1\ldots c_k$ a \emph{generalised digit} or a \emph{generalised $\beta$-digit} of \emph{length} $k+1$ with \emph{subdigits} $c_0,c_1,\ldots,c_k$.
Note that the value of a generalised $\beta$-digit depends on the value of~$\beta$.

As an example of sets of generalised digits,
consider
$A_n=\{1.1,1.11,1.12,1.2,1.21,1.22\}$ for odd~$n$ and $A_n=\{0\}$ for even $n$.
This is similar to what we use below to create intervals of permutation class growth rates. We return to this example below.

As long as the sets of generalised digits are suitably bounded, the same gap constraints on $\beta$ suffice as before to ensure that $(A_n)_\beta$ 
is an interval. Thus we have the following generalisation of Lemma~\ref{lemmaBaseNotation1}.

\begin{lemma}\label{lemmaBaseNotation2}
Given a sequence $(A_n)$ of non-empty finite sets of generalised $\beta$-digits, for a fixed $\beta$ let $\ell_n$ and $u_n$
denote the least and greatest elements of $A_n$ respectively,
and let $\Delta_n=u_n-\ell_n$.
Also,
let
$\delta_n$
be the largest gap between consecutive elements of $A_n$ (with $\delta_n=0$ if $|A_n|=1$).

If both the length and the subdigits of all the generalised digits in the $A_n$ are
bounded, 
$\beta>1$,
and
for each~$n$, $\beta$~satisfies the inequality
$$
\delta_n
 \;\leqslant\;
\sum\limits_{i=1}^{\infty}\Delta_{n+i}\+\beta^{-i},
$$
then
$(A_n)_\beta=\big[(\ell_n)_\beta,\,(u_n)_\beta\big]$.
\end{lemma}
We omit the proof since it is essentially the same as that for
Lemma~\ref{lemmaBaseNotation1}.

Let us analyse our generalised digit example.
If $A_n=\{1.1,1.11,1.12,1.2,1.21,1.22\}$ for odd~$n$ and $A_n=\{0\}$ for even $n$,
then for odd $n$ we have $\Delta_n=0.12
=\beta^{\negsup1}+2\+\beta^{\negsup2}$.
However, the determination of $\delta_n$ depends on whether $1.2$ exceeds $1.12$ or not.
If it does (which is the case when $\beta$ exceeds~2), then $\delta_n$ is the greater of
$\beta^{\negsup2}$
and
$\beta^{\negsup1}-2\+\beta^{\negsup2}$.
If $\beta<2$,
then $\delta_n$
is the greatest of $\beta^{\negsup2}$, $\beta^{\negsup1}-\beta^{\negsup2}$ and $-\beta^{\negsup1}+2\+\beta^{\negsup2}$.
Careful solving of all the resulting inequalities reveals that
if $1<\beta\leqslant\half\+(1+\sqrt{13})\approx2.30278$,
then $(A_n)_\beta$ 
is an interval.\footnote
{When the $A_n$ contain generalised digits,
the solution set need not consist of a single interval.
For example, if
$A_n=\{0.5,0.501,0.502,0.51,0.511,0.512,0.52,0.521,1.3\}$ for odd~$n$
and
$A_n=\{0\}$ for even~$n$, then
$(A_n)_\beta$ 
is an interval if $1<\beta\leqslant2.732...$ and also if $2.879...\leqslant\beta\leqslant2.923...$. 
}

Let us now begin to relate this to the construction of permutation classes.

\section{Sum-closed permutation classes}\label{sectSumClosed}

Recall, from Section~\ref{defDirectSum}, the definition
of the direct sum, $\sigma\oplus\tau$, of two permutations and the concept of an indecomposable, 
and also that
of a sum-closed class, and of the sum closure $\sumclosed\SSS$ of a downward closed set of indecomposables. 
All the permutation classes that we construct to prove our results are sum-closed, and we define them by specifying the (downward closed)
sets of indecomposables 
from which they are composed.
The remainder of this section is devoted to study the growth rates of sum-closed classes, and our main results, Lemmas~\ref{lemmaGRInterval1} and~\ref{lemmaGRInterval2}, provide conditions on when the growth rates of a family of sum-closed permutation classes form an interval.
In Section~\ref{sectConstructions}, we construct families of sum-closed classes satisfying these conditions, thereby proving Theorems~\ref{thmTheta} and~ \ref{thmLambda}.

We say that a set of permutations $\SSS$ is \emph{positive} and \emph{bounded} (by a constant~$c$) if $1\leqslant|\SSS_n|\leqslant c$ for all $n$.
As long as its set of indecomposables is positive and bounded, the growth rate of a sum-closed permutation class is easy to determine.
Variants of the following lemma can be found in the work of Albert \& Linton~\cite{AL2009} and of Vatter~\cite{Vatter2010b}.
\begin{lemma}
\label{lemmaGRSumClosed}
If $\SSS$ is a downward closed set of indecomposables that is positive and bounded, then $\gr(\sumclosed\SSS)$ is the unique $\gamma$ such that $\gamma>1$ and $S(\gamma^{-1})=\sum\limits_{n=1}^\infty|\SSS_n|\+\gamma^{-n}=1$.
\end{lemma}
This follows from the following standard analytic combinatorial result.
\begin{lemma}
[{\cite[Theorem~V.1]{FS2009}}] 
\label{lemmaSupercritical}
Suppose 
$\FFF=\seq{\GGG}$ 
and the corresponding generating functions, $F$ and $G$, are analytic at $0$, with $G(0) = 0$ and
$$
\lim\limits_{x\rightarrow\rho^-}G(x) \;>\; 1,
$$
where $\rho$ is the radius of convergence
of $G$. If, in addition, $G$ is aperiodic, i.e. there
does not exist an integer $d \geqslant 2$ such that $G(z) = H(z^d)$ for some $H$ analytic at $0$, then
$\gr(\FFF)=\sigma^{-1}$,
where $\sigma$ is the only root in $(0, \rho)$ of $G(x) = 1$.
\end{lemma}

\begin{proof}[Proof of Lemma~\ref{lemmaGRSumClosed}]
The 
permutation class
$\sumclosed\SSS$ is in bijection with $\seq{\SSS}$.
Since $|\SSS_n|$ is positive and bounded, $S(z)$ is aperiodic, its radius of convergence is $1$,
and $\!\!\!\lim\limits_{\:\:\:x\rightarrow1^-}\!\!S(x)=+\infty$.
Thus
the requirements of Proposition~\ref{lemmaSupercritical} are satisfied, from which the desired result follows immediately.
\end{proof}


Building on our earlier results concerning expansions of real numbers
in non-integer bases,
we are now in a position to describe conditions under which
the set of growth rates of
a family of sum-closed permutation classes
includes an interval of the real line.
To state the lemma,
we introduce a notational convenience.
If $(a_n)$ is a bounded sequence of positive integers that enumerates a downward closed set of indecomposables $\SSS$ (i.e. $|\SSS_n|=a_n$ for all $n$), let us use
$\gr(\sumclosed (a_n))$ to denote the growth rate of
the sum-closed permutation class $\sumclosed\SSS$. 



\begin{lemma}\label{lemmaGRInterval1}
Given a sequence $(A_n)$ of non-empty finite sets of positive integers, let $\ell_n$ and $u_n$
denote the least and greatest elements of $A_n$ respectively,
and let $\Delta_n=u_n-\ell_n$.
Also,
let~$\delta_n$
be the largest gap between consecutive elements of $A_n$ (with $\delta_n=0$ if $|A_n|=1$).

If $(u_n)$ is bounded, then for every real number $\gamma$ such that $\gr(\sumclosed(\ell_n))\leqslant\gamma\leqslant\gr(\sumclosed(u_n))$ and
$$
\delta_n
 \;\leqslant\;
\sum\limits_{i=1}^{\infty}\Delta_{n+i}\+\gamma^{-i}
$$
for each~$n$,
there is some sequence $(a_n)$ with each
$a_n\in A_n$
such that
$\gamma=\gr(\sumclosed(a_n))$.
\end{lemma}
\begin{proof}
Let $\gamma_1=\gr(\sumclosed(\ell_n))$ and $\gamma_2=\gr(\sumclosed(u_n))$, so $\gamma_1\leqslant\gamma\leqslant\gamma_2$.
Now, by Lemma~\ref{lemmaGRSumClosed},
$$
(\ell_n)_\gamma
\;=\;
\sum\limits_{n=1}^{\infty}\ell_n\+\gamma^{-n}
\;\leqslant\;
\sum\limits_{n=1}^{\infty}\ell_n\+\gamma_1^{-n}
\;=\;
1
\;=\;
\sum\limits_{n=1}^{\infty}u_n\+\gamma_2^{-n}
\;\leqslant\;
\sum\limits_{n=1}^{\infty}u_n\+\gamma^{-n}
\;=\;
(u_n)_\gamma .
$$
Hence,
$
(\ell_n)_\gamma
\leqslant
1
\leqslant
(u_n)_\gamma
$
and so, since $\gamma$ satisfies the
gap inequalities,
by Lemma~\ref{lemmaBaseNotation1}
there is some sequence $(a_n)$ with each
$a_n\in A_n$
such that
$(a_n)_\gamma = 1$ and thus, by Lemma~\ref{lemmaGRSumClosed}, $\gamma=\gr(\sumclosed(a_n))$.
\end{proof}
This result generalises Proposition~2.4 in Vatter~\cite{Vatter2010b}, which treats only the case in which all the $A_n$ are identical and consist of an interval of integers $\{\ell,\ell+1,\ldots,u\}$.

As a consequence of this lemma, given a suitable sequence $(A_n)$ of sets of integers, if we could construct a family of sum-closed permutation classes such that
every sequence $(a_n)$ with each $a_n\in A_n$ enumerates
the indecomposables of
some member of the family, then the set of growth rates of the classes in the family would include an interval of the real line.
Returning to our original example in which $A_n=\{1,4\}$ for odd $n$ and $A_n=\{1,3,5,7,9\}$ for even $n$, suppose it were possible to construct
sum-closed classes whose indecomposables were enumerated by each sequence $\{(a_n):a_n\in A_n\}$.
Then the set of growth rates would include the interval from
$\gr(\sumclosed(1,1,\ldots))=2$ to the lesser of
$\gr(\sumclosed(4,9,4,9,\ldots))=2+\sqrt{14}$ 
and $\frac{1}{4}\+(3+\sqrt{89})$, the greatest value permitted by the gap inequalities. 
(Clearly, the fact that there are only countably many permutation classes with growth rate less than $\kappa\approx2.20557$ means that such a construction is, in fact, impossible.)

We now broaden this result to handle the case in which the $A_n$ contain generalised digits.
In the integer case, a term $a_n$ corresponds to a set consisting of $a_n$ indecomposables of length $n$.
In the generalised digit case, $a_n=c_0.c_1\ldots c_k$ corresponds to a set consisting of $c_0$ indecomposables of length $n$, $c_1$ indecomposables of length $n+1$, and so on, up to $c_k$ indecomposables of length $n+k$.

Before stating the generalised digit version of our lemma, we
generalise our terminology and notation.
Let us say that a sequence of generalised digits $(a_n)$ \emph{enumerates} a set~$\SSS$ of permutations if
$\SSS$ is the union $\biguplus \SSS^{\!(\!n\!)}$
of {disjoint} sets
$\SSS^{\!(\!n\!)}$,
such that, for each $n$, if we have $a_n=c_0.c_1\ldots c_k$ then $\SSS^{\!(\!n\!)}$~consists of exactly $a_n+c_i$ permutations of length $n+i$ for $i=0,\ldots,k$.
If $(a_n)$ is a sequence of generalised digits that enumerates a downward closed set of indecomposables $\SSS$,
let us use
$\gr(\sumclosed (a_n))$ to denote the growth rate
of the sum-closed permutation class $\sumclosed\SSS$.

In the proofs of our theorems, we make use of the following elementary fact.
Given a sequence of generalised digits that enumerates a set $\SSS$, there is a unique sequence of integers that also enumerates $\SSS$.
For example, if $a_n=1.12$ for odd $n$ and $a_n=3.1$ for even $n$, then
$(a_n)$ and the sequence
$(t_n)=(1,4,4,4,\ldots)$
are equivalent in this way.
We write such equivalences $(a_n)\equiv(t_n)$.

\begin{lemma}\label{lemmaGRInterval2}
Given a sequence $(A_n)$ of non-empty finite sets of generalised $\gamma$-digits, for a fixed $\gamma$ let $\ell_n$ and $u_n$
denote the least and greatest elements of $A_n$ respectively,
and let $\Delta_n=u_n-\ell_n$.
Also,
let
$\delta_n$
be the largest gap between consecutive elements of $A_n$ (with $\delta_n=0$ if $|A_n|=1$).

If the length and subdigits of all the generalised $\gamma$-digits in the $A_n$ are
bounded,
then for every real number $\gamma$ such that $\gr(\sumclosed(\ell_n))\leqslant\gamma\leqslant\gr(\sumclosed(u_n))$ and
$$
\delta_n
 \;\leqslant\;
\sum\limits_{i=1}^{\infty}\Delta_{n+i}\+\gamma^{-i}
$$
for each~$n$,
there is some sequence $(a_n)$ with each
$a_n\in A_n$
such that
$\gamma=\gr(\sumclosed(a_n))$.
\end{lemma}
We omit the proof since it is essentially the same as that for
Lemma~\ref{lemmaGRInterval1}, but making use of Lemma~\ref{lemmaBaseNotation2} rather than Lemma~\ref{lemmaBaseNotation1}.

Let us revisit our generalised digit example in which
$A_n=\{1.1,1.11,1.12,1.2,1.21,1.22\}$ for odd~$n$ and $A_n=\{0\}$ for even $n$.
Suppose it were
possible to construct
sum-closed classes whose indecomposables were enumerated by each sequence $\{(a_n):a_n\in A_n\}$.
Then the set of growth rates of the classes
would include the interval extending from
$\gr(\sumclosed(1,1,\ldots))=2$ to the lesser of
$\gr(\sumclosed(1,2,3,2,3,\ldots))\approx2.51155$ (the real root of a cubic)
and $\half\+(1+\sqrt{13})$, the maximum value permitted by the gap inequalities.
(As with our previous example, such a construction is, in fact, impossible.)

We also need to use the following elementary analytic result, also found in Vatter~\cite{Vatter2010b}.
\begin{lemma}[{\cite[Proposition 2.3]{Vatter2010b}}] 
\label{lemmaGRSumClosedClose}
Given $\varepsilon>0$ and $c>0$, there is a positive integer $m$, dependent only on $\varepsilon$ and $c$, such that if
$(r_n)$ and $(s_n)$ are two sequences of positive integers bounded by $c$, and
$r_n=s_n$ for all $n\leqslant m$, then
$\gr(\sumclosed(r_n))$ and $\gr(\sumclosed(s_n))$ differ by no more than $\varepsilon$.
\end{lemma}
\begin{proof}
Let $\gamma_1=\gr(\sumclosed(r_n))$ and $\gamma_2=\gr(\sumclosed(s_n))$ and suppose that $\gamma_1<\gamma_2$.
Note that $\gamma_1\geqslant2$ because $(r_n)$ is positive, and $\gamma_2\leqslant c+1$ because $(s_n)$ is bounded by $c$.

From Lemma~\ref{lemmaGRSumClosed}, we have
$
\sum\limits_{n=1}^\infty r_n\+\gamma_1^{-n}
=
\sum\limits_{n=1}^\infty s_n\+\gamma_2^{-n}
$.
Hence, since
$r_n=s_n$ for $n\leqslant m$,
$$
\frac{\gamma_2-\gamma_1}{(c+1)^2}
\;\leqslant\;
\gamma_1^{-1}-\gamma_2^{-1}
\;\leqslant\;
\sum\limits_{n\leqslant m}
r_n\+(\gamma_1^{-n}-\gamma_2^{-n})
\;=\;
\sum\limits_{n>m} s_n\+\gamma_2^{-n} - \sum\limits_{n>m} r_n\+\gamma_1^{-n}
\;\leqslant\;
\frac{c}{2^m} .
$$
Thus, $m=\ceil{\log_2 
\frac{c\+(c+1)^2}{\varepsilon}
}$ suffices.
\end{proof}

\section{Constructions that yield intervals of growth rates}\label{sectConstructions}

To prove our theorems, we construct families of permutation classes whose indecomposables are enumerated by sequences
satisfying the requirements of Lemma~\ref{lemmaGRInterval2}.
In doing this, it helps to take a graphical perspective.
Recall from Section~\ref{defGraphOfPerm} 
that the (ordered) graph, $G_\sigma$, of a permutation $\sigma$ of length~$n$ has
vertex set $\{1,\ldots,n\}$
with an edge between vertices $i$ and $j$ 
if $i<j$ and $\sigma(i)>\sigma(j)$.
Recall also that
$\tau\leqslant\sigma$
if and only if
$G_\tau$ is an induced ordered subgraph of $G_\sigma$.
Observe too that $\sigma$ is indecomposable if and only if $G_\sigma$ is \emph{connected}.
It tends to be 
advantageous to think of indecomposables as those permutations whose graphs are connected.

\begin{figure}[ht]
  \mybox{
  \vspace{3pt}
  $$
  \begin{tikzpicture}[scale=0.225]
    \plotpermnobox{6}{3, 1, 5, 2, 6, 4}
    \draw [thin] (2,1)--(1,3)--(4,2)--(3,5)--(6,4)--(5,6);
    \node at(3.5,-1){$\omega_6 = \mathbf{315264}$};
  \end{tikzpicture}
  \quad\qquad
  \begin{tikzpicture}[scale=0.225]
    \plotpermnobox{6}{2, 4, 1, 6, 3, 5}
    \draw [thin] (1,2)--(3,1)--(2,4)--(5,3)--(4,6)--(6,5);
    \node at(3.5,-1){$\overline\omega_6 = \mathbf{241635}$};
  \end{tikzpicture}
  \quad\qquad
  \begin{tikzpicture}[scale=0.225]
    \plotpermnobox{6}{3, 1, 5, 2, 7, 4, 6}
    \draw [thin] (2,1)--(1,3)--(4,2)--(3,5)--(6,4)--(5,7)--(7,6);
    \node at(3.5,-1){$\omega_7 = \mathbf{3152746}$};
  \end{tikzpicture}
  \quad\qquad
  \begin{tikzpicture}[scale=0.225]
    \plotpermnobox{6}{2, 4, 1, 6, 3, 7, 5}
    \draw [thin] (1,2)--(3,1)--(2,4)--(5,3)--(4,6)--(7,5)--(6,7);
    \node at(3.5,-1){$\overline\omega_7 = \mathbf{2416375}$};
  \end{tikzpicture}
  \vspace{-6pt}
  $$
  }
  \caption{Primary and secondary oscillations}\label{figOscillating}
\end{figure}
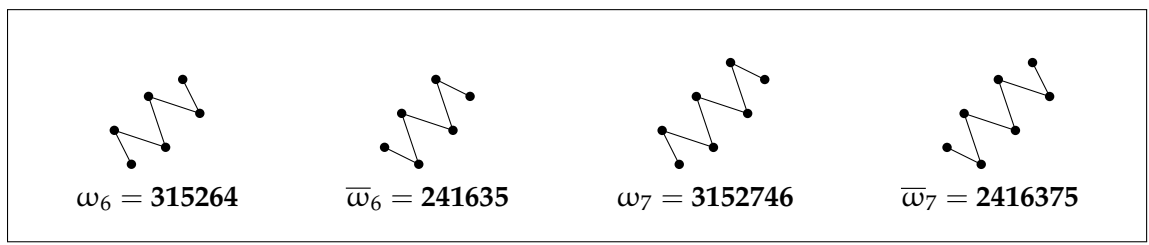
All our constructions are based on permutations whose graphs are \emph{paths}.
Such permutations are called \emph{increasing oscillations}. We distinguish two cases,
the \emph{primary oscillation} of length $n$, which we denote $\omega_n$,
whose \emph{least} entry corresponds to an end of the path,
and the \emph{secondary oscillation} of length $n$, denoted~$\overline\omega_n$
whose \emph{first} entry corresponds to an end of the path.
See Figure~\ref{figOscillating} for an illustration.
Trivially, we have $\omega_1=\overline\omega_1=\mathbf{1}$ and $\omega_2=\overline\omega_2=\mathbf{21}$.
The (lower and upper) \emph{ends} of an increasing oscillation are the
entries corresponding to the (lower and upper)
ends
of its path graph.

\HIDE{
Formally, for $n\geqslant 4$, the primary oscillation of length $n$, $\omega_n$, and the secondary oscillation of length $n$, $\overline\omega_n$, can be defined as follows:
$$
\begin{array}{rclr}
\omega_n          &=&          3,    1,\,\,\, 5,2,\,\,\, 7,4,\,\,\, 9,6,\,\,\, \ldots,\,\,\, n-3,n-6,\,\,\, n-1,            n-4,\,\,\,
         n,   n-2 & \text{~for even~}n\geqslant4, \\[3pt]
\omega_n          &=&          3,    1,\,\,\, 5,2,\,\,\, 7,4,\,\,\, 9,6,\,\,\, \ldots,\,\,\, n-2,n-5,\,\,\, n,\phantom{{}-0}n-3,\,\,\, \phantom{n,{}}n-1 & \text{~for odd~}n\geqslant5,  \\[3pt]
\overline\omega_n &=& \phantom{0,{}} 2,\,\,\, 4,1,\,\,\, 6,3,\,\,\, 8,5,\,\,\, \ldots,\,\,\, n-2,n-5,\,\,\, n,\phantom{{}-0}n-3,\,\,\, \phantom{n,{}}n-1 & \text{~for even~}n\geqslant4, \\[3pt]
\overline\omega_n &=& \phantom{0,{}} 2,\,\,\, 4,1,\,\,\, 6,3,\,\,\, 8,5,\,\,\, \ldots,\,\,\, n-3,n-6,\,\,\, n-1,            n-4,\,\,\,
         n,   n-2 & \text{~for odd~}n\geqslant5.
\end{array}
$$
} 

\begin{figure}[ht]
  \mybox{
  \vspace{-6pt}
  $$
  \begin{tikzpicture}[scale=0.2]
    \plotpermnobox{9}{4, 1, 2, 6, 3, 9, 5, 7, 8}
    \draw [thin] (2,1)--(1,4);
    \draw [thin] (3,2)--(1,4)--(5,3)--(4,6)--(7,5)--(6,9)--(8,7);
    \draw [thin] (6,9)--(9,8);
    \node at(5,-1){$\;\;\omega_7^{2,2}$};
  \end{tikzpicture}
  \quad
  \begin{tikzpicture}[scale=0.2]
    \plotpermnobox{10}{4, 1, 2, 6, 3, 10, 5, 7, 8, 9}
    \draw [thin] (2,1)--(1,4);
    \draw [thin] (3,2)--(1,4)--(5,3)--(4,6)--(7,5)--(6,10)--(8,7);
    \draw [thin] (10,9)--(6,10)--(9,8);
    \node at(5.5,-1){$\;\;\omega_7^{2,3}$};
  \end{tikzpicture}
  \quad
  \begin{tikzpicture}[scale=0.2]
    \plotpermnobox{10}{5, 1, 2, 3, 7, 4, 10, 6, 8, 9}
    \draw [thin] (2,1)--(1,5)--(4,3);
    \draw [thin] (3,2)--(1,5)--(6,4)--(5,7)--(8,6)--(7,10)--(9,8);
    \draw [thin] (7,10)--(10,9);
    \node at(5.5,-1){$\;\;\omega_7^{3,2}$};
  \end{tikzpicture}
  \quad
  \begin{tikzpicture}[scale=0.2]
    \plotpermnobox{11}{5, 1, 2, 3, 7, 4, 11, 6, 8, 9, 10}
    \draw [thin] (2,1)--(1,5)--(4,3);
    \draw [thin] (3,2)--(1,5)--(6,4)--(5,7)--(8,6)--(7,11)--(9,8);
    \draw [thin] (11,10)--(7,11)--(10,9);
    \node at(6,-1){$\;\;\omega_7^{3,3}$};
  \end{tikzpicture}
  \quad
  \begin{tikzpicture}[scale=0.2]
    \plotpermnobox{11}{6, 1, 2, 3, 4, 8, 5, 11, 7, 9, 10}
    \draw [thin] (2,1)--(1,6)--(4,3);
    \draw [thin] (5,4)--(1,6);
    \draw [thin] (3,2)--(1,6)--(7,5)--(6,8)--(9,7)--(8,11)--(11,10);
    \draw [thin] (8,11)--(10,9);
    \node at(6,-1){$\;\;\omega_7^{4,2}$};
  \end{tikzpicture}
  \quad
  \begin{tikzpicture}[scale=0.2]
    \plotpermnobox{12}{6, 1, 2, 3, 4, 8, 5, 12, 7, 9, 10, 11}
    \draw [thin] (2,1)--(1,6)--(4,3);
    \draw [thin] (5,4)--(1,6);
    \draw [thin] (3,2)--(1,6)--(7,5)--(6,8)--(9,7)--(8,12)--(11,10);
    \draw [thin] (12,11)--(8,12)--(10,9);
    \node at(6.5,-1){$\;\;\omega_7^{4,3}$};
  \end{tikzpicture}
  \vspace{-9pt}
  $$
  }
  \caption{The set of permutations $R_7^{4,3}$}\label{figR743}
\end{figure}
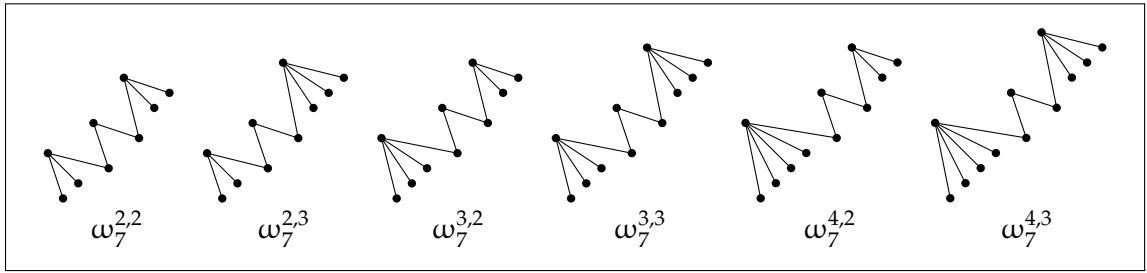
The key permutations in our constructions are 
formed by \emph{inflating} (i.e. replacing)
each end of a primary oscillation of odd length with an increasing permutation.
For $n\geqslant4$ and $r,s\geqslant1$,
let~$\omega_n^{r,s}$ denote the
permutation of length $n-2+r+s$
formed by replacing 
the lower and upper ends of $\omega_n$ with the increasing permutations of length $r$ and~$s$ respectively.
The graph of $\omega_n^{r,s}$ thus consists of a path on $n-2$ vertices with $r$ pendant edges attached to its lower end and
$s$ pendant edges attached to its upper end.
See Figure~\ref{figR743} for some examples.
We also occasionally make use of $\overline\omega_n^{r,s}$, which we define analogously.

The building blocks for our constructions are the following sets of indecomposables.
Given a primary oscillation with both ends inflated, i.e. some $\omega_n^{r,s}$ with $r,s\geqslant2$, we define~$R_n^{r,s}$ to be the set of primary oscillations with both ends inflated that are
subpermutations of~$\omega_n^{r,s}$:
$$
R_n^{r,s} \;=\; \{\omega_n^{u,v} \,:\, 2\leqslant u\leqslant r,\,2\leqslant v\leqslant s\}.
$$
We only make use of $R_n^{r,s}$ for odd $n$. 
See Figure~\ref{figR743} for an example.

Let us investigate the properties of these sets.
Firstly, if $n\neq m$, then no element of $R_n^{r,s}$ is a subpermutation of 
an element of~$R_m^{r,s}$.
This is a consequence of 
the following elementary observation,
which follows directly from consideration of the (ordered) graphs of the permutations. 
\begin{obs} 
If $n,m\geqslant4$, $n\neq m$ and $r,s,u,v\geqslant2$, then $\omega_n^{r,s}$ and $\omega_m^{u,v}$ are incomparable under the subpermutation 
order.
\end{obs}

Thus for fixed $r$ and $s$, and varying $n$, the $R_n^{r,s}$ form a collection of sets of permutations, no member of one being a subpermutation of a member of another. This 
should be compared and contrasted with
the 
concept of an \emph{antichain}, which is a set of permutations, none of which is a subpermutation of another.

We build permutation classes by specifying that, for some fixed $r$ and $s$, their indecomposables must
include some subset of $R_n^{r,s}$ for each (odd) $n$.
Because, if $n\neq m$, any element of $R_n^{r,s}$ is incomparable with any element of $R_m^{r,s}$,
the choice of subset to include can be made independently for each~$n$.
This provides the flexibility we need to construct families of classes whose growth rates include an interval.
(Vatter's simpler constructions in~\cite{Vatter2010b},
which are also based on inflating the ends of increasing oscillations,
rely on being able to choose any subset of an infinite antichain to include in the indecomposables of a class.)

Clearly, for each $n$, the subset of $R_n^{r,s}$ included in the indecomposables must be downward closed.
It has been observed heuristically that
subsets containing $\omega_n^{3,2}$ 
are better for generating intervals of growth rates.
In light of this, we define $\FFF_n^{r,s}$ to be the family of downward closed subsets of $R_n^{r,s}$ that contain $\omega_n^{3,2}$ (and hence also contain $\omega_n^{2,2}$).

For example, $\FFF_n^{4,3}$ consists of the seven downsets whose 
sets of maximal elements
are
$$
\{\omega_n^{3,2}\},\;\; \{\omega_n^{3,2},\omega_n^{2,3}\},\;\; \{\omega_n^{3,3}\},\;\; \{\omega_n^{4,2}\},\;\; \{\omega_n^{4,2},\omega_n^{2,3}\},\;\; \{\omega_n^{4,2},\omega_n^{3,3}\}, \;\;\text{and}\;\; \{\omega_n^{4,3}\}.
$$
To each downset $S$ in $\FFF_n^{r,s}$, we can associate a generalised digit $c_0.c_1\ldots c_k$ 
such that $c_i$ records 
the number of elements in $S$ of length $n+2+i$ for each $i\geqslant0$.
For example, the generalised digits associated with 
the elements of $\FFF_n^{4,3}$ are (in the same order as above)
$$
1.1,\;\; 1.2,\;\; 1.21,\;\; 1.11,\;\; 1.21,\;\; 1.22 \;\;\text{and}\;\; 1.221.
$$
We also need to take into account the indecomposables that are subpermutations of 
$\omega_n^{r,s}$ but are \emph{not} elements of~$R_n^{r,s}$.
For odd $n$, these are of the following types:
\begin{bullets}
  \item Primary oscillations $\omega_m$ 
  and secondary oscillations $\overline\omega_m$.
  \item Permutations whose graphs are stars 
  $K_{1,u}$ 
  and whose first (and greatest) entry corresponds to the internal vertex;
  we use~$\psi_u$ to denote these star permutations.
  \item Increasing oscillations with just one end inflated: $\omega_m^{u,1}$, $\omega_m^{1,v}$ and $\overline\omega_m^{1,v}$. 
\end{bullets}
\vspace{6pt}
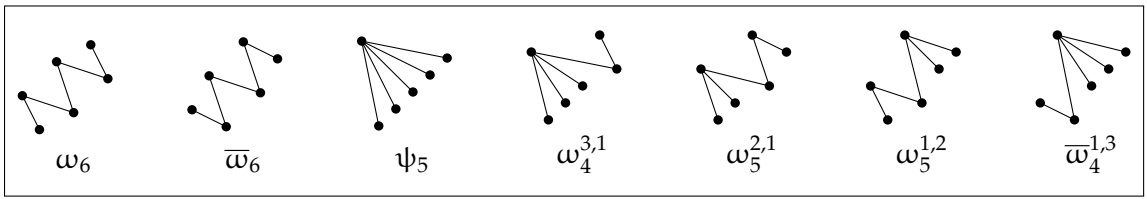
\begin{figure}[ht]
  \mybox{
  \vspace{-6pt}
  $$
  \begin{tikzpicture}[scale=0.225]
    \plotpermnobox{6}{3, 1, 5, 2, 6, 4}
    \draw [thin] (2,1)--(1,3)--(4,2)--(3,5)--(6,4)--(5,6);
    \node at(3.5,-1){$\;\;\omega_6$};
  \end{tikzpicture}
  \;\qquad\;
  \begin{tikzpicture}[scale=0.225]
    \plotpermnobox{6}{2, 4, 1, 6, 3, 5}
    \draw [thin] (1,2)--(3,1)--(2,4)--(5,3)--(4,6)--(6,5);
    \node at(3.5,-1){$\;\;\overline\omega_6$};
  \end{tikzpicture}
  \;\qquad\;
  \begin{tikzpicture}[scale=0.225]
    \plotpermnobox{6}{6, 1, 2, 3, 4, 5}
    \draw [thin] (2,1)--(1,6)--(3,2);
    \draw [thin] (4,3)--(1,6)--(5,4);
    \draw [thin] (6,5)--(1,6);
    \node at(3.5,-1){$\;\;\psi_5$}; 
  \end{tikzpicture}
  \;\qquad\;
  \begin{tikzpicture}[scale=0.225]
    \plotpermnobox{6}{5, 1, 2, 3, 6, 4}
    \draw [thin] (2,1)--(1,5)--(3,2);
    \draw [thin] (4,3)--(1,5)--(6,4)--(5,6);
    \node at(3.5,-1){$\;\;\omega_4^{3,1}$};
  \end{tikzpicture}
  \;\qquad\;
  \begin{tikzpicture}[scale=0.225]
    \plotpermnobox{6}{4, 1, 2, 6, 3, 5}
    \draw [thin] (2,1)--(1,4)--(3,2);
    \draw [thin] (1,4)--(5,3)--(4,6)--(6,5);
    \node at(3.5,-1){$\;\;\omega_5^{2,1}$};
  \end{tikzpicture}
  \;\qquad\;
  \begin{tikzpicture}[scale=0.225]
    \plotpermnobox{6}{3, 1, 6, 2, 4, 5}
    \draw [thin] (2,1)--(1,3)--(4,2)--(3,6)--(5,4);
    \draw [thin] (3,6)--(6,5);
    \node at(3.5,-1){$\;\;\omega_5^{1,2}$};
  \end{tikzpicture}
  \;\qquad\;
  \begin{tikzpicture}[scale=0.225]
    \plotpermnobox{6}{2,6,1,3,4,5}
    \draw [thin] (1,2)--(3,1)--(2,6)--(4,3);
    \draw [thin] (6,5)--(2,6)--(5,4);
    \node at(3.5,-1){$\;\;\overline\omega_4^{1,3}$};
  \end{tikzpicture}
  \vspace{-9pt}
  $$
  }
  \caption{The elements of $Q^{4,3}$ of size 6}\label{figS426}
\end{figure}
Given $r,s\geqslant2$, let $Q^{r,s}$ be the (infinite) set of 
permutations that are subpermutations of 
$\omega_n^{r,s}$ for some odd $n\geqslant5$, but are not elements of~$R_n^{r,s}$.
See Figure~\ref{figS426} for an illustration.


We now have the building blocks we need.
Given $r\geqslant3$, $s\geqslant2$, and odd $k\geqslant5$, we define $\Phi_{r,s,k}$ to be the family of those permutation classes whose indecomposables are the union of $Q^{r,s}$ together with 
an element of $\FFF_n^{r,s}$ for each odd $n\geqslant k$:
$$
\Phi_{r,s,k} \;=\;
\Big\{
\sumclosed \big( Q^{r,s} \:\cup\: S_k \:\cup\: S_{k+2} \:\cup\: S_{k+4} \:\cup\: \ldots
\big) \::\:
S_n\in \FFF_n^{r,s}
, \, n=k,k+2,\ldots
\Big\} .
$$
The sequence of families $(\Phi_{5,3,k})_{k=5,7,\ldots}$ is what we need to prove our first theorem, which we restate below.
The set of growth rates of permutation classes in
each family $\Phi_{5,3,k}$ %
consists of an interval, $I_k$, such that the sequence $(I_k)$ of these intervals approaches $\theta_B$ from above.

We make use of one additional notational convenience. A sequence of integers $(a_n)$ whose entries have the same value for all $n\geqslant k$ is denoted $(a_1,a_2,\ldots,a_{k-1},\overline{a_k})$. For example,
$(1,\overline4)=(1,4,4,4,\ldots)$.
We also use $(0^k,a_{k+1},\ldots)$ for a sequence whose first $k$ terms are zero.
For example, $(0^3,1)=(0,0,0,1)$.

\begin{repthm}{thmTheta}
Let $\theta_B\approx2.35526$ be the unique real root of $x^7-2\+x^6-x^4-x^3-2\+x^2-2\+x-1$.
For any $\varepsilon>0$, there exist $\delta_1$ and $\delta_2$ with $0<\delta_1<\delta_2<\varepsilon$ such that every value in the interval $[\theta_B+\delta_1,\theta_B+\delta_2]$ is the growth rate of a permutation class.
\end{repthm}
\begin{proof}
Let $(q_n)$ be the sequence that enumerates 
$Q^{5,3}$.
It can readily be checked that we have
$(q_n)=(1,1,2,3,5,7,\overline8)$, and that
$\gr(\sumclosed(q_n))=\theta_B$.

Let $k\geqslant5$ be odd.
For each odd $n\geqslant k+2$, let
$$
F_n \;=\; \{1.1,1.11,1.111,1.2,1.21,1.211,1.22,1.221,1.222,1.2221\}
$$
be the set of generalised digits
associated with 
the sets of indecomposables in
$\FFF_{n\negsub2}^{5,3}$.
Otherwise (if $n$ is even or $n\leqslant k$), let $F_n=\{\+0\+\}$.
Now, for each $n$, let $A_n=\{q_n+f_n:f_n\in F_n\}$.

So, by construction, for every permutation class $\sumclosed\SSS\in\Phi_{5,3,k}$ there is a corresponding sequence $(a_n)$, with each
$a_n\in A_n$, that enumerates $\SSS$.

Let $\ell_n=q_n+1.1$ for odd $n\geqslant k+2$ and $\ell_n=q_n$ otherwise.
Similarly, let $u_n=q_n+1.2221$ for odd $n\geqslant k+2$ and $u_n=q_n$ otherwise.
We have the following equivalences:
$$
\begin{array}{rcl}
(\ell_n) &\!\equiv\!& (1,1,2,3,5,7,\overbrace{8,8,\ldots,8}^{k-5},\overline9) \\[8pt]
(u_n)    &\!\equiv\!& (1,1,2,3,5,7,\overbrace{8,8,\ldots,8}^{k-5},9,10,11,\overline{12})
\end{array}
$$
We now apply Lemma~\ref{lemmaGRInterval2}.
It can be checked that the gap inequalities 
necessitate only that 
the growth rate does not exceed $\gamma_{\max}\approx2.470979$, the unique positive root of the quartic $x^4-2\+x^3-x^2-1$. This is independent of the value of $k$, and
is
greater than $\gr(\sumclosed(u_n))$ for all odd $k\geqslant5$
since $\gr(\sumclosed(u_n))\approx2.362008$ if $k=5$.

So, for each $k$ we have an interval of growth rates:
If $\gamma$ is such that it is the case that $\gr(\sumclosed(\ell_n))\leqslant\gamma\leqslant\gr(\sumclosed(u_n))$, then there is some permutation class in $\Phi_{5,3,k}$ whose growth rate is $\gamma$.

Moreover, by Lemma~\ref{lemmaGRSumClosedClose},
\[
\liminfty[k] \gr(\sumclosed(\ell_n)) \;=\;
\liminfty[k] \gr(\sumclosed(u_n)) \;=\;
\gr(\sumclosed(q_n)) \;=\; \theta_B, 
\]
so these intervals can be found arbitrarily close to $\theta_B$.
\end{proof}

For our second theorem, we need to add extra sets of indecomposables to our constructions.
As before, we start with $r\geqslant3$, $s\geqslant2$, and odd $k\geqslant5$.
A suitable collection, $\HHH$, of extra sets of indecomposables satisfies the following two conditions:
\begin{bullets}
  \item Each set in $\HHH$ is disjoint from $Q^{r,s}$ and also disjoint from each $R_n^{r,s}$ for odd $n\geqslant k$.
  \item For each set $H\in\HHH$, the union 
  $H\cup Q^{r,s}$ is a downward closed set of indecomposables.
\end{bullets}
Given these conditions,
we define $\Phi_{r,s,k,\HHH}$ to be the family of those permutation classes whose indecomposables are the union of $Q^{r,s}$
together with an element of $\HHH$
and an element of $\FFF_n^{r,s}$ for each odd $n\geqslant k$:
$$
\Phi_{r,s,k,\HHH} \;=\;
\Big\{
\sumclosed \big( Q^{r,s} \:\cup\: H \:\cup\: S_k \:\cup\: S_{k+2} \:\cup\: S_{k+4} \:\cup\: \ldots
\big) \::\:
H \in \HHH \:\text{~and~}\:
S_n\in \FFF_n^{r,s}
, \, n=k,k+2,\ldots\!
\Big\} .
$$

We define 
our extra sets of indecomposables
by specifying 
an upper set of \emph{maximal} indecomposables, $U$, and a lower set of \emph{required} indecomposables, $L$.
If $S$ is a set of indecomposables, let
${\downarrow}S$ denote the downset consisting of those indecomposables that are subpermutations of elements of $S$.
Then, given $r,s\geqslant2$ and suitable sets of indecomposables $U$ and $L$, we use
$U\,{\Downarrow}_{r,s}\+L$
to denote the collection of those downward closed subsets of ${\downarrow}U\setminus Q^{r,s}$
that include the set of required indecomposables~$L$.
For instance, we have
$\FFF^{r,s}_n=\{\omega_n^{r,s}\}\,{\Downarrow}_{r,s}\+\{\omega_n^{3,2}\}$.


Let's consider as an example the family $\Phi_{5,3,5,\HHH}$ where $\HHH 
= \{\omega_5^{7,1}\}\,{\Downarrow}_{5,3}\+\{\psi_7\}$;
see $\pi_0$ and $\mu_1$ in Figure~\ref{figColABPerms} below.
The sets in $\HHH$ consist of indecomposables that are subpermutations of $\omega_5^{7,1}$ but are not in $Q^{5,3}$.
There are six such indecomposables: $\psi_7$, $\psi_8$, $\omega_4^{6,1}$, $\omega_4^{7,1}$, $\omega_5^{6,1}$, and $\omega_5^{7,1}$.
The collection $\HHH$ consists of the nine downward closed subsets of these six that contain $\psi_7$.

The set of indecomposables in a permutation class in our family consists of $Q^{5,3}$ together with an element of $\FFF_n^{5,3}$ for each odd $n\geqslant5$ and an extra set from $\HHH$.
As in the proof of Theorem~\ref{thmTheta},
$Q^{5,3}$ contributes $(q_n)=(1,1,2,3,5,7,\overline8)$ to the enumeration of the indecomposables, and
for each odd $n\geqslant5$, there are ten distinct generalised digits associated with sets of indecomposables in $\FFF_n^{5,3}$, ranging between 1.1 and 1.2221.
Let $F_n$ consist of this set of generalised digits for odd $n\geqslant7$ and otherwise contain only 0.
The extra sets in $\HHH$
have seven distinct enumerations. These can be represented by the set of generalised digits $H_8=\{1,1.1,1.11,1.2,1.21,1.22,1.221\}$; let $H_n=\{0\}$ for $n\neq8$.
Now, for each $n$, let $A_n=\{q_n+f_n+h_n:f_n\in F_n,h_n\in H_n\}$.

So, by construction, for every permutation class $\sumclosed\SSS\in\Phi_{5,3,5,\HHH}$ there is a corresponding sequence $(a_n)$, with each
$a_n\in A_n$, that enumerates $\SSS$.
The minimal enumeration sequence is $(\ell_n)\equiv(1, 1, 2, 3, 5, 7, 9, 10, \overline9)$ for which the growth rate is $\gr(\sumclosed(\ell_n)) \approx 2.36028$. Similarly,
the maximal enumeration sequence is $(u_n)\equiv(1, 1, 2, 3, 5, 7, 9, 11, 13, 14, 13, \overline{12})$ for which the growth rate is $\gr(\sumclosed(u_n)) \approx 2.36420$.
The gap inequalities only necessitate that the growth rate not exceed $\gamma_{\max}\approx2.47098$. Thus the growth rates of permutation classes in our example family $\Phi_{5,3,5,\HHH}$ form an interval.


The proof of our second theorem follows similar lines to this example.


\begin{repthm}{thmLambda}
Let $\lambda_B\approx2.35698$ be the unique positive root of $x^8-2\+x^7-x^5-x^4-2\+x^3-2\+x^2-x-1$.
Every value at least $\lambda_B$ is the growth rate of a permutation class.
\end{repthm}
\begin{proof}
In~\cite{Vatter2010b}, Vatter has shown that there are permutation classes of every growth rate at least $\lambda_A\approx2.48187$ (the unique real root of $x^5-2\+x^4-2\+x^2-2\+x-1$).
It thus suffices to exhibit permutation classes whose growth rates cover the interval $[\lambda_B,\lambda_A]$.
With the $\pi_i$ and $\mu_j$ as in Figures~\ref{figColABPerms}--\ref{figColEPerms} below,
we claim that the permutation classes in the following five families meet our needs:

\begin{tabular}{lccllccl}
\textbf{Family} &$\!\!\!\!\!$\textbf{A} &$\!\!\!\!\!\!$\textbf{:}&
$\!\Phi_{5,3,7,\AAA}$ & $\!\!$where$\!\!$ &
$\AAA$ & $\!\!\!=\!\!\!$ & $\{\pi_1\}\,{\Downarrow}_{5,3}\+\{\mu_1\}$ \\[3pt]
\textbf{Family} &$\!\!\!\!\!$\textbf{B} &$\!\!\!\!\!\!$\textbf{:}&
$\!\Phi_{5,3,5,\BBB}$ & $\!\!$where$\!\!$ &
$\BBB$ & $\!\!\!=\!\!\!$ & $\{\pi_2\}\,{\Downarrow}_{5,3}\+\varempty$ \\[3pt]
\textbf{Family} &$\!\!\!\!\!$\textbf{C} &$\!\!\!\!\!\!$\textbf{:}&
$\!\Phi_{9,8,5,\CCC}$ & $\!\!$where$\!\!$ &
$\CCC$ & $\!\!\!=\!\!\!$ & $\{\pi_3\}\,{\Downarrow}_{9,8}\+\{\mu_2\}$ \\[3pt]
\textbf{Family} &$\!\!\!\!\!$\textbf{D} &$\!\!\!\!\!\!$\textbf{:}&
$\!\Phi_{5,3,5,\DDD}$ & $\!\!$where$\!\!$ &
$\DDD$ & $\!\!\!=\!\!\!$ & $\{\pi_4,\pi_5\}\,{\Downarrow}_{5,3}\+\{\mu_3\}$ \\[3pt]
\textbf{Family} &$\!\!\!\!\!$\textbf{E} &$\!\!\!\!\!\!$\textbf{:}&
$\!\Phi_{5,5,5,\EEE}$ & $\!\!$where$\!\!$ &
$\EEE$ & $\!\!\!=\!\!\!$ & $\{\pi_6,\pi_7,\pi_8\}\,{\Downarrow}_{5,5}\+\{\mu_3,\mu_6\} \,\cup\, \{\pi_6,\pi_7,\pi_8\}\,{\Downarrow}_{5,5}\+\{\mu_2,\mu_4,\mu_5\}$
\end{tabular}

\vspace{6pt}

\begin{figure}[ht]
  \mybox{
  \vspace{3pt}
  $$
  \begin{tikzpicture}[scale=0.2]
    \plotpermnobox{11}{9, 1, 2, 3, 4, 5, 6, 7, 11, 8, 10}
    \draw [thin] (2,1)--(1,9)--(3,2);
    \draw [thin] (4,3)--(1,9)--(5,4);
    \draw [thin] (6,5)--(1,9)--(7,6);
    \draw [thin] (8,7)--(1,9)--(10,8);
    \draw [thin] (10,8)--(9,11)--(11,10);
    \node at(6,-1){$\;\;\pi_0=\omega_5^{7,1}$};
  \end{tikzpicture}
  \qquad\quad
  \begin{tikzpicture}[scale=0.2]
    \plotpermnobox{15}{11, 1, 2, 3, 4, 5, 6, 7, 8, 9, 13, 10, 15, 12, 14}
    \draw [thin] (2,1)--(1,11)--(3,2);
    \draw [thin] (4,3)--(1,11)--(5,4);
    \draw [thin] (6,5)--(1,11)--(7,6);
    \draw [thin] (8,7)--(1,11)--(9,8);
    \draw [thin] (10,9)--(1,11)--(12,10)--(11,13)--(14,12)--(13,15)--(15,14);
    \node at(8,-1){$\;\;\pi_1=\omega_7^{9,1}$};
  \end{tikzpicture}
  \qquad\quad
  \begin{tikzpicture}[scale=0.2]
    \plotpermnobox{8}{8, 1, 2, 3, 4, 5, 6, 7}
    \draw [thin] (2,1)--(1,8)--(3,2);
    \draw [thin] (4,3)--(1,8)--(5,4);
    \draw [thin] (6,5)--(1,8)--(7,6);
    \draw [thin] (1,8)--(8,7);
    \node at(4.5,-1){$\;\;\mu_1=\psi_7$};
  \end{tikzpicture}
  \qquad\quad
  \begin{tikzpicture}[scale=0.2]
    \plotpermnobox{10}{2,9,1,3,4,5,6,7,10,8}
    \draw [thin] (4,3)--(2,9)--(5,4);
    \draw [thin] (6,5)--(2,9)--(7,6);
    \draw [thin] (8,7)--(2,9);
    \draw [thin] (1,2)--(3,1)--(2,9)--(10,8)--(9,10);
    \node at(5.5,-1){$\;\pi_2=\mathbf{2\;\!9\;\!1\;\!3\;\!4\;\!5\;\!6\;\!7\;\!10\;\!8}$};
  \end{tikzpicture}
  \vspace{-6pt}
  $$
  }
  \caption{Permutations used to define Families~A and~B}\label{figColABPerms}
\end{figure}
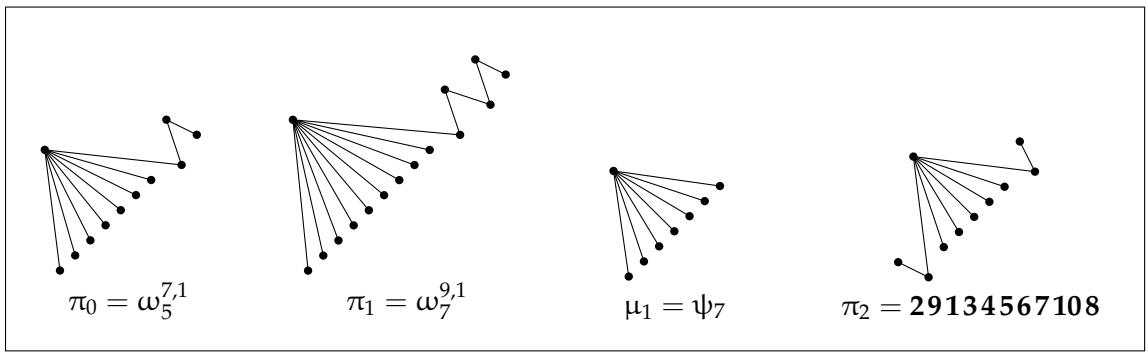
We briefly outline the calculations concerning each of these families. The details can be fleshed out in the same way as the example above.

\vspace{6pt}
\newpage 
\textbf{Family A:} $\,\Phi_{5,3,7,\AAA}$
\begin{bullets}
\raggedright
\item $Q^{5,3}$ is enumerated by $(1,1,2,3,5,7,\overline8)$.
\item For each odd $n\geqslant7$, there are 10 distinct generalised digits associated with sets of indecomposables in $\FFF_n^{5,3}$, ranging between 1.1 and 1.2221.
\item There are 47 distinct enumerations of sets of indecomposables in $\AAA=\{\pi_1\}\,{\Downarrow}_{5,3}\+\{\mu_1\}$, ranging between $(0^7,1)$ and $(0^7,1,2,3,4,4,3,2,1)$.
\item The indecomposables in the smallest permutation class in {Family A} are enumerated by $(\ell_n) \equiv (1,1,2,3,5,7,8,\overline9)$.
\item The indecomposables in the largest permutation class in {Family A} are enumerated by $(u_n) \equiv (1,1,2,3,5,7,8,9,11,13,15,16,15,14,13,\overline{12})$.
\item The gap inequalities require the growth rate not to exceed $\gamma_{\max}\approx2.470979$.
\item $\gr(\sumclosed(\ell_n)) = \lambda_B \approx 2.356983$; $\gr(\sumclosed(u_n)) \approx 2.359320$.
\end{bullets}
\vspace{6pt}
\textbf{Family B:} $\,\Phi_{5,3,5,\BBB}$
\begin{bullets}
\raggedright
\item $Q^{5,3}$ is enumerated by $(1,1,2,3,5,7,\overline8)$.
\item For each odd $n\geqslant5$, there are 10 distinct generalised digits associated with sets of indecomposables in $\FFF_n^{5,3}$, ranging between 1.1 and 1.2221.
\item There are 29 distinct enumerations of sets of indecomposables in $\BBB=\{\pi_2\}\,{\Downarrow}_{5,3}\+\varempty$, ranging between $(0)$ and
$(0^5, 1, 2, 3, 3, 1)$.
\item The indecomposables in the smallest permutation class in {Family B} are enumerated by $(\ell_n) \equiv (1, 1, 2, 3, 5, 7,\overline9)$.
\item The indecomposables in the largest permutation class in {Family B} are enumerated by $(u_n) \equiv (1, 1, 2, 3, 5, 8, 11, 13, 14, 13,\overline{12})$.
\item The gap inequalities require the growth rate not to exceed $\gamma_{\max}\approx2.470979$.
\item $\gr(\sumclosed(\ell_n)) \approx 2.359304$; $\gr(\sumclosed(u_n)) \approx 2.375872$.
\end{bullets}
\vspace{6pt}
\begin{figure}[ht]
  \mybox{
  \vspace{-6pt}
  $$
  \begin{tikzpicture}[scale=0.2]
    \plotpermnobox{10}{3,1,8,2,4,5,6,10,7,9}
    \draw [thin] (5,4)--(3,8)--(7,6);
    \draw [thin] (6,5)--(3,8);
    \draw [thin] (2,1)--(1,3)--(4,2)--(3,8)--(9,7)--(8,10)--(10,9);
    \node at(5.5,-1){$\;\pi_3=\mathbf{3\;\!1\;\!8\;\!2\;\!4\;\!5\;\!6\;\!10\;\!7\;\!9}$};
  \end{tikzpicture}
  \;\;
  \begin{tikzpicture}[scale=0.2]
    \plotpermnobox{6}{2,5,1,3,6,4}
    \draw [thin] (1,2)--(3,1)--(2,5)--(6,4)--(5,6);
    \draw [thin] (2,5)--(4,3);
    \node at(3.5,-1){$\;\mu_2=\mathbf{251364}$};
  \end{tikzpicture}
  \;\;
  \begin{tikzpicture}[scale=0.2]
    \plotpermnobox{6}{3, 1, 4, 5, 6, 2}
    \draw [thin] (2,1)--(1,3)--(6,2)--(5,6);
    \draw [thin] (4,5)--(6,2)--(3,4);
    \node at(3.5,-1){$\;\pi_4=\mathbf{314562}$};
  \end{tikzpicture}
  \;\;
  \begin{tikzpicture}[scale=0.2]
    \plotpermnobox{9}{2, 8, 1, 3, 4, 5, 6, 9, 7}
    \draw [thin] (4,3)--(2,8)--(5,4);
    \draw [thin] (6,5)--(2,8)--(7,6);
    \draw [thin] (1,2)--(3,1)--(2,8)--(9,7)--(8,9);
    \node at(5,-1){$\;\pi_5=\mathbf{281345697}$};
  \end{tikzpicture}
  \;\;
  \begin{tikzpicture}[scale=0.2]
    \plotpermnobox{4}{2,3,4,1}
    \draw [thin] (1,2)--(4,1)--(3,4);
    \draw [thin] (2,3)--(4,1);
    \node at(2.5,-1){$\;\mu_3=\mathbf{2341}$};
  \end{tikzpicture}
  \vspace{-9pt}
  $$
  }
  \caption{Permutations used to define Families~C,~D and~E}\label{figColCDEPerms}
\end{figure}
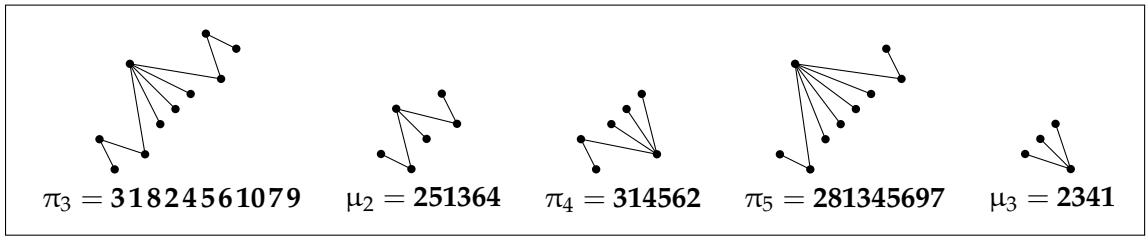
\textbf{Family C:} $\,\Phi_{9,8,5,\CCC}$
\begin{bullets}
\raggedright
\item $Q^{9,8}$ is enumerated by $(1, 1, 2, 3, 5, 7, 9, 11, 13, 15,\overline{17})$.
\item For each odd $n\geqslant5$, there are 574 distinct generalised digits associated with sets of indecomposables in $\FFF_n^{9,8}$, ranging between 1.1 and 1.2345677654321.
\item There are 19 distinct enumerations of sets of indecomposables in $\CCC=\{\pi_3\}\,{\Downarrow}_{9,8}\+\{\mu_2\}$, ranging between $(0^5, 1)$ and
$(0^5, 1, 3, 4, 3, 1)$.
\item The indecomposables in the smallest permutation class in {Family C} are enumerated by $(\ell_n) \equiv (1, 1, 2, 3, 5, 8, 10, 12, 14, 16,\overline{18})$.
\item The indecomposables in the largest permutation class in {Family C} are enumerated by $(u_n) \equiv (1, 1, 2, 3, 5, 8, 13, 17, 20, 22, 26, 29, 33, 36, 39, 41, 43, 44,\overline{45})$.
\Needspace*{2\baselineskip}
\item The gap inequalities require the growth rate not to exceed $\gamma_{\max}\approx2.786389$.
\item $\gr(\sumclosed(\ell_n)) \approx 2.373983$; $\gr(\sumclosed(u_n)) \approx 2.389043$.
\end{bullets}
\vspace{6pt}
\textbf{Family D:} $\,\Phi_{5,3,5,\DDD}$
\begin{bullets}
\raggedright
\item $Q^{5,3}$ is enumerated by $(1,1,2,3,5,7,\overline8)$.
\item For each odd $n\geqslant5$, there are 10 distinct generalised digits associated with sets of indecomposables in $\FFF_n^{5,3}$, ranging between 1.1 and 1.2221.
\item There are 37 distinct enumerations of sets of indecomposables in $\DDD=\{\pi_4,\pi_5\}\,{\Downarrow}_{5,3}\+\{\mu_3\}$, ranging between $(0^3, 1)$ and
$(0^3, 1, 2, 2, 2, 2, 1)$.
\item The indecomposables in the smallest permutation class in {Family D} are enumerated by $(\ell_n) \equiv (1, 1, 2, 4, 5, 7,\overline9)$.
\item The indecomposables in the largest permutation class in {Family D} are enumerated by $(u_n) \equiv (1, 1, 2, 4, 7, 9, 11,\overline{12})$.
\item The gap inequalities require the growth rate not to exceed $\gamma_{\max}\approx 2.470979$.
\item $\gr(\sumclosed(\ell_n)) \approx 2.389038$; $\gr(\sumclosed(u_n)) \approx  2.430059$.
\end{bullets}
\vspace{6pt}
\begin{figure}[ht]
  \mybox{
  \vspace{-6pt}
  $$
  \begin{tikzpicture}[scale=0.2]
    \plotpermnobox{4}{3,4,1,2}
    \draw [thin] (1,3)--(3,1)--(2,4)--(4,2)--(1,3);
    \node at(2.5,-1){$\;\pi_6=\mathbf{3412}$};
  \end{tikzpicture}
  \quad\;\,
  \begin{tikzpicture}[scale=0.2]
    \plotpermnobox{7}{2, 6, 1, 3, 4, 7, 5}
    \draw [thin] (4,3)--(2,6)--(5,4);
    \draw [thin] (1,2)--(3,1)--(2,6)--(7,5)--(6,7);
    \node at(4,-1){$\;\pi_7=\mathbf{2613475}$};
  \end{tikzpicture}
  \quad\;\,
  \begin{tikzpicture}[scale=0.2]
    \plotpermnobox{8}{3, 1, 4, 5, 6, 8, 2, 7}
    \draw [thin] (3,4)--(7,2)--(4,5);
    \draw [thin] (7,2)--(5,6);
    \draw [thin] (2,1)--(1,3)--(7,2)--(6,8)--(8,7);
    \node at(4.5,-1){$\;\pi_8=\mathbf{31456827}$};
  \end{tikzpicture}
  \quad\;\,
  \begin{tikzpicture}[scale=0.2]
    \plotpermnobox{5}{2, 3, 4, 5, 1}
    \draw [thin] (1,2)--(5,1)--(3,4);
    \draw [thin] (2,3)--(5,1)--(4,5);
    \node at(3,-1){$\;\mu_4=\mathbf{23451}$};
  \end{tikzpicture}
  \quad\;\,
  \begin{tikzpicture}[scale=0.2]
    \plotpermnobox{5}{2, 3, 5, 1, 4}
    \draw [thin] (1,2)--(4,1)--(3,5)--(5,4);
    \draw [thin] (2,3)--(4,1);
    \node at(3,-1){$\;\mu_5=\mathbf{23514}$};
  \end{tikzpicture}
  \vspace{-9pt}
  $$
  }
  \caption{Permutations used to define Family~E}\label{figColEPerms}
\end{figure}
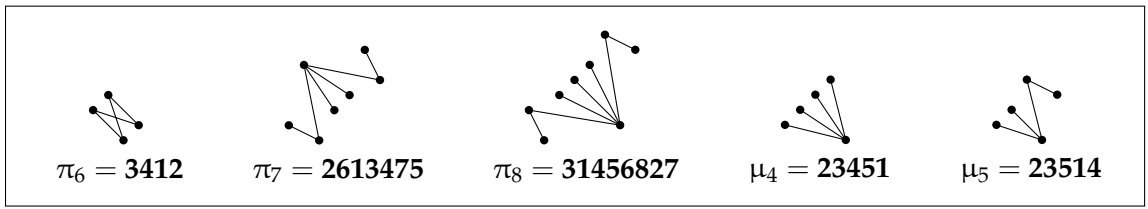
\textbf{Family E:} $\,\Phi_{5,5,5,\EEE}$
\begin{bullets}
\raggedright
\item $Q^{5,5}$ is enumerated by $(1, 1, 2, 3, 5, 7, 9,\overline{10})$.
\item For each odd $n\geqslant5$, there are 26 distinct generalised digits associated with sets of indecomposables in $\FFF_n^{5,5}$, ranging between 1.1 and 1.234321.
\item There are 61 distinct enumerations of sets of indecomposables in $\EEE=\{\pi_6,\pi_7,\pi_8\}\,{\Downarrow}_{5,5}\+\{\mu_3,\mu_6\}$ $\,\cup\,$ $\{\pi_6,\pi_7,\pi_8\}\,{\Downarrow}_{5,5}\+\{\mu_2,\mu_4,\mu_5\}$, ranging between $(0^3, 1, 2, 1)$ and
$(0^3, 2, 3, 5, 4, 1)$.
\item The indecomposables in the smallest permutation class in {Family E} are enumerated by $(\ell_n) \equiv (1, 1, 2, 4, 7, 8, 10,\overline{11})$.
\item The indecomposables in the largest permutation class in {Family E} are enumerated by $(u_n) \equiv (1, 1, 2, 5, 8, 12, 14, 13, 14, 16, 17,\overline{18})$.
\item The gap inequalities require the growth rate to be at least $\gamma_{\min}\approx 2.363728$, but not to exceed $\gamma_{\max}\approx 2.489043$.
\item $\gr(\sumclosed(\ell_n)) \approx 2.422247$; $\gr(\sumclosed(u_n)) \approx  2.485938 > \lambda_A$.
\end{bullets}
\vspace{6pt}
\Needspace*{2\baselineskip}
Here is a summary:
\begin{center}
\renewcommand{\arraystretch}{1.25}
\begin{tabular}{|c|l|l|}
  \hline
  & 
    \emph{Enumeration of smallest and largest sets of indecomposables}
  & \emph{Interval covered} \\
  \hline
  {~A~} &
    $\hspace{-\arraycolsep}\begin{array}{l}
      (1,1,2,3,5,7,8,\overline9) \\
      (1,1,2,3,5,7,8,9,11,13,15,16,15,14,13,\overline{12})
    \end{array}\hspace{-\arraycolsep}$ &
    $2.356983$ -- $2.359320$ \\
  \hline
  {B} &
    $\hspace{-\arraycolsep}\begin{array}{l}
      (1,1,2,3,5,7,\overline9) \\
      (1,1,2,3,5,8,11,13,14,13,\overline{12})
    \end{array}\hspace{-\arraycolsep}$ &
    $2.359304$ -- $2.375872$ \\
  \hline
  {C} &
    $\hspace{-\arraycolsep}\begin{array}{l}
      (1,1,2,3,5,8,10,12,14,16,\overline{18}) \\
      (1,1,2,3,5,8,13,17,20,22,26,29,33,36,39,41,43,44,\overline{45})
    \end{array}\hspace{-\arraycolsep}$ &
    $2.373983$ -- $2.389043$ \\
  \hline
  {D} &
    $\hspace{-\arraycolsep}\begin{array}{l}
      (1,1,2,4,5,7,\overline9) \\
      (1,1,2,4,7,9,11,\overline{12})
    \end{array}\hspace{-\arraycolsep}$ &
    $2.389038$ -- $2.430059$ \\
  \hline
  {E} &
    $\hspace{-\arraycolsep}\begin{array}{l}
      (1,1,2,4,7,8,10,\overline{11}) \\
      (1,1,2,5,8,12,14,13,14,16,17,\overline{18})
    \end{array}\hspace{-\arraycolsep}$ &
    $2.422247$ -- $2.485938$ \\
  \hline
\end{tabular}
\end{center}
Thus we have five intervals of growth rates that cover $[\lambda_B,\lambda_A]$.
\end{proof}

\cleardoublepage


\part{\textsc{Hasse Graphs}}\label{partIII} 
\setcounter{chapter}{8}

\newcommand{\prefix}[1]{#1\mbox{-}\nobreak\hspace{0pt}}

\chapter{Introducing Hasse graphs}\label{chap09}

We now turn to the investigation of the structure of various classes of permutations, by considering how the \emph{Hasse graphs} of permutations in these classes can be built from a sequence of rooted source graphs.
For most of the classes, the analysis leads to functional equations for their generating functions.
In some cases, we are able to use the kernel method to solve these equations and determine the corresponding algebraic generating function.
Our approach is similar to that of ``adding a slice'', used previously to enumerate various classes of polyominoes and other combinatorial structures.

\section{Building Hasse graphs}

Corresponding to each permutation $\sigma$, we define an ordered plane graph $H_\sigma$, which we call its \emph{Hasse graph}.
To create the Hasse graph for a permutation $\sigma=\sigma_1\ldots\sigma_n$, let
vertex $i$ be the point $(i,\sigma_i)$ in the Euclidean plane.
Now, for each pair $i,j$ such that $i<j$, add an edge between vertices $i$ and $j$, if and only if $\sigma(i)<\sigma(j)$ and there is no vertex $k$ such that $i<k<j$ and $\sigma(i)<\sigma(k)<\sigma(j)$.
Note that the edges of $H_\sigma$ correspond to the edges of the Hasse diagram of the sub-poset, $P_\sigma$, of $\mathbb{N}^2$ consisting of the points $(i,\sigma_i)$.
See Figure~\ref{figHasseGraph} for an illustration.
Observe, also, that the Hasse graph of a permutation $\sigma$ is the transitive reduction of the reversal of the ordered (inversion) graph of the reversal of $\sigma$.
In practice, we tend not to distinguish between a permutation and its Hasse graph.
The minimal elements of the poset $P_\sigma$ are the left-to-right minima of the permutation $\sigma$.
Similarly, maximal elements of $P_\sigma$ are the right-to-left maxima of $\sigma$.

\begin{figure}[ht]
  $$
  \begin{tikzpicture}[scale=0.275,line join=round]
    \draw [red,thin] (3,11)--(5,16);
    \draw [red,thin] (5,16)--(4,4)--(7,14);
    \draw [red,thin] (12,13)--(8,8)--(13,12)--(9,6)--(12,13);
    \draw [red,thin] (6,1)--(7,14);
    \draw [red,thin] (8,8)--(6,1)--(9,6);
    \draw [red,thin] (10,3)--(12,13)--(11,2)--(13,12)--(10,3)--(14,10)--(11,2)--(15,9)--(10,3)--(16,7)--(11,2)--(17,5)--(10,3);
    \draw [blue,very thick] (2,17)--(1,15)--(5,16);
    \draw [blue,very thick] (7,14)--(3,11);
    \draw [blue,very thick] (12,13)--(3,11)--(13,12);
    \draw [blue,very thick] (16,7)--(9,6)--(4,4)--(8,8)--(14,10)--(9,6)--(15,9)--(8,8);
    \draw [blue,very thick] (4,4)--(17,5);
    \draw [blue,very thick] (10,3)--(6,1)--(11,2);
    \plotpermnobox{}{15,17,11, 4,16, 1,14, 8, 6, 3, 2,13,12,10, 9, 7, 5}
    \draw [thin] (1,15)  circle [radius=0.4];
    \draw [thin] (3,11) circle [radius=0.4];
    \draw [thin] (4,4) circle [radius=0.4];
    \draw [thin] (6,1) circle [radius=0.4];
  \end{tikzpicture}
  $$
  \caption{The Hasse graph of a permutation, showing its four source graphs, rooted at the circled vertices}
  \label{figHasseGraph}
\end{figure}
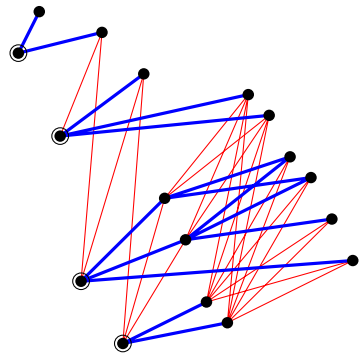
Hasse graphs of permutations were previously considered by Bousquet-M\'elou \& Butler~\cite{BMB2007}, who
determined the algebraic generating function of the class of \emph{forest-like} permutations
whose Hasse graphs are acyclic.
We revisit forest-like permutations in Chapter~\myref{chapForest} below.

Given a permutation $\sigma$, we partition the vertices of $H_\sigma$ by spanning it with a sequence of graphs, which we call the \emph{source graphs} of $\sigma$.
There is one source graph for each left-to-right minimum of $\sigma$.
Suppose $u_1,\ldots,u_m$ are the vertices of $H_\sigma$ corresponding to the left-to-right minima of $\sigma$, listed from left to right.
Then the $k$th source graph $G_k$ is the graph induced by $u_k$ and those vertices of $H_\sigma$ lying above and to the right of $u_k$ that are not in $G_1,\ldots,G_{k-1}$. We refer to $u_k$ as the \emph{root} of source graph $G_k$.
See Figure~\ref{figHasseGraph} for an illustration.
In the figures, edges of source graphs are shown as thick lines.

The structure of the source graphs of permutations in a specific permutation class is constrained by the need to avoid the patterns in the basis of the class.
If the source graphs for a class are acyclic, we refer to them as \emph{source trees}.

The \emph{bottom subgraph}
of a Hasse graph
is
the graph induced by its lowest vertex (the
least entry in the permutation) and all the vertices lying above and to its right.
Observe that the bottom subgraph may contain vertices from more than one source graph.
For example, the bottom subgraph of the Hasse graph in Figure~\ref{figHasseGraph} contains vertices from three source graphs.
Bottom subgraphs of permutations in a specific permutation class satisfy the same structural restrictions as do the source graphs.
We refer to an acyclic bottom subgraph as a \emph{bottom subtree}.

We build the Hasse graph of a permutation by starting with a single source graph and then repeatedly adding another source graph to the lower right, rooted at a new left-to-right minimum.
The technique is similar to that of ``adding a slice'', which has been used successfully to enumerate constrained compositions and other classes of polyominoes, a topic of interest in statistical mechanics
(see, for example, Bousquet-M\'elou's review paper~\cite{Bousquet-Melou1996}, the books of van Rensburg~\cite{vanRensburg2000} and Guttmann~\cite{Guttmann2009}, and~\cite[Examples~III.22 and~V.20]{FS2009}).
When a source graph is added, its vertices are interleaved horizontally with the non-root vertices of the bottom subgraph of the graph built from the previous source graphs.
Typically, the positioning of the vertices of the new source graph
relative to those of the bottom subgraph
is constrained by the need to avoid forbidden patterns.

When a new source graph is added, additional edges may be added, each joining a vertex in the source graph to a vertex in the bottom subgraph.
We call these edges \emph{tendrils} and say that the vertices at the upper ends of the tendrils are \emph{tethered} by the tendrils.
Tendrils are represented by thin edges in Figure~\ref{figHasseGraph} and in other figures below.
Clearly, the positioning of the vertices of the source graph
relative to those of the bottom subgraph determines which tendrils are added, and \emph{vice versa}.
Rather than analysing the possible positions for the vertices of the source graph,
it is often easier to consider the permissable choices for the tendrils instead.

\section{Avoiding 1324}\label{sect1324Tethering}
A number of the classes that we analyse in the following chapters
are subclasses of $\av(\pdiamond)$.
What can we say about the Hasse graphs of permutations in such classes?

\begin{figure}[ht]
  $$
  \begin{tikzpicture}[scale=0.25]
    \draw [blue,ultra thick] (1,15)--(4,20);
    \draw [blue,ultra thick] (2,9)--(3,10)--(5,13)--(15,14);
    \draw [blue,ultra thick] (7,5)--(12,6)--(13,8);
    \draw [blue,ultra thick] (14,1)--(19,3);
    \draw [blue,very thick] (1,15)--(6,16)--(8,19);
    \draw [blue,very thick] (10,18)--(6,16)--(11,17);
    \draw [blue,very thick] (3,10)--(16,11)--(17,12);
    \draw [blue,very thick] (12,6)--(18,7);
    \draw [blue,very thick] (14,1)--(20,2);
    \draw [red,thin] (3,10)--(4,20);
    \draw [red,thin] (5,13)--(6,16);
    \draw [red,thin] (8,19)--(7,5)--(10,18)--(9,4)--(11,17)--(7,5);
    \draw [red,thin] (9,4)--(12,6);
    \draw [red,thin] (18,7)--(14,1)--(16,11)--(13,8)--(15,14)--(14,1);
    \plotpermnobox{20}{15, 9, 10, 20, 13, 16, 5, 19, 4, 18, 17, 6, 8, 1, 14, 11, 12, 7, 3, 2}
    \draw [thin] (1,15) circle [radius=0.4];
    \draw [thin] (2,9)  circle [radius=0.4];
    \draw [thin] (7,5)  circle [radius=0.4];
    \draw [thin] (9,4)  circle [radius=0.4];
    \draw [thin] (14,1) circle [radius=0.4];
  \end{tikzpicture}
  $$
  \caption{The graph of a permutation
  avoiding $\pdiamond$,
  showing its construction from five source trees}
  \label{fig1324HasseGraphExample}
\end{figure}
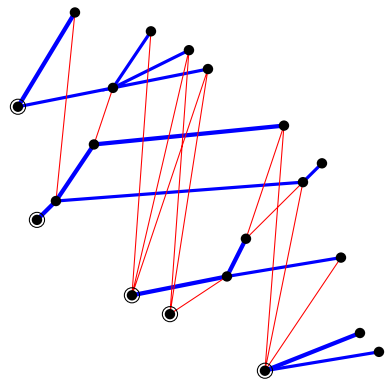
If a permutation avoids 
$\pdiamond$, then its
Hasse graph does not have the diamond graph
$H_\pdiamond =
\raisebox{-2.5pt}{\begin{tikzpicture}[scale=0.12,line join=round]
  \draw[] (1,1)--(2,3)--(4,4)--(3,2)--(1,1);
  \plotpermnobox{}{1,3,2,4}
\end{tikzpicture}}
$
as a minor.
In particular,
each source graph of $H_\sigma$ is a \emph{tree}, and so is
its bottom subgraph.
We call the uppermost (leftmost) branch (path towards the upper right from the root) of a source tree its \emph{trunk}.
The topmost vertex of the trunk is its \emph{tip}.
When adding a new source tree to the lower right of a Hasse graph $H$,
to avoid creating a $\pdiamond$, any vertices of the source tree not in the trunk must be positioned to the right of all
vertices of $H$.
Thus the lower ends of any new tendrils created are on the trunk of the new source tree.
Moreover, when a new source tree is added, no vertex in (the bottom subtree of) $H$ is tethered by more than one tendril.
See Figure~\ref{fig1324HasseGraphExample} for an illustration.

\section{Outline of Part III}

In the remaining chapters of this thesis, we study seven different permutation classes.
In four cases, analysis of the Hasse graph structure enables us to ascertain the algebraic generating function for the class. Two of these results are new.
In two others, we are only able to establish functional equations.
In the final case, the $\pdiamond$-avoiders, we achieve neither of these, but are able to determine a new lower bound for the growth rate of the class.

In Chapters~\myref{chapR}--\myref{chapForest}, we consider three classes that avoid $\pdiamond$.
In Chapter~\myref{chapR}, we look at $\av(\pdiamond,\mathbf{2314})$,
which is counted by the (large) Schr\"oder numbers,
and present a new derivation of its algebraic generating function.
This class was first enumerated by Kremer~\cite{Kremer2000,Kremer2003}.

In Chapter~\myref{chapG}, we consider $\av(\pdiamond,\mathbf{1432})$.
The source graphs in this class have a very simple form.
We derive a (very complicated) functional equation for
the generating function of
this class, which, unfortunately, we are unable to solve.

In Chapter~\myref{chapForest}, we use our methods to give a simplified derivation of the algebraic generating function for
\emph{forest-like} permutations, originally enumerated by Bousquet-M\'elou \& Butler~\cite{BMB2007}.
This is the class whose Hasse graphs are acyclic.
It is not a classical permutation class: 
the Hasse graph of
$\mathbf{2143}$ is a cycle, so $\mathbf{2143}$ is not forest-like, but $\mathbf{2143}$ is
contained in $\mathbf{21354}$, the Hasse graph of which is a tree.
The class of forest-like permutations can be defined as
$\av(\pdiamond,\mathbf{21\bar{3}54})$, where $\mathbf{21\bar{3}54}$ is a {barred pattern}
(see Section~\myref{defBarredPattern}).

In Chapter~\myref{chapPlane}, we investigate the class of \emph{plane permutations}, those permutations whose Hasse graph is plane (i.e. non-crossing).
This is the barred pattern class $\av(\mathbf{21\bar{3}54})$.
We establish a nice functional equation for the generating function of this class, which we are not able to solve.

In Chapters~\myref{chapF} and~\myref{chapE}, we enumerate two more
permutation classes that each avoid two patterns of length 4.
We establish the generating function for
$\av(\mathbf{1234},\mathbf{2341})$ in Chapter~\myref{chapF}.
In doing so, the kernel method is used six times to solve
the relevant functional equations.
In Chapter~\myref{chapE}, we derive functional equations for $\av(\mathbf{1243},\mathbf{2314})$, and solve them to determine its algebraic generating function. This requires an
unusual simultaneous double application of the kernel method.
These results have been accepted for publication (see~\cite{Bevan2014a}).

Finally, in Chapter~\myref{chap1324}, we use our Hasse graph techniques in an investigation of the structure of the notorious class $\av(\pdiamond)$. As a result, we establish a new lower bound of 9.81 for its growth rate.
As a consequence of our examination of the asymptotic distribution of certain substructures in the Hasse graphs of $\pdiamond$-avoiders, we are able to prove that, in the limit, patterns in
{\L}uka\-sie\-wicz paths exhibit a concentrated Gaussian distribution.
This work has been accepted for publication (see~\cite{Bevan2014Thin}).

\HIDE{
\section{Terminology and notation}

In the following chapters, the word ``class'' does not necessarily refer to a (downward-closed) permutation class, but is used more generally for any combinatorial class.
For a given class of permutations, the class of its source graphs is denoted $\SSS$. This class is (clearly) different in different chapters.
The connected (skew-indecomposable) elements of a class $\AAA$ are denoted $\AAA_\CC$.
Clearly, $\AAA=\seqplus{\AAA_\CC}$.
Finally, we use $\TTT$ to denote the class of rooted plane trees, and $t(z)$ to denote its generating function $\half(1 - \sqrt{1 - 4\+ z})$.
} 

\cleardoublepage


\chapter{Avoiding 1324 and 2314}
\label{chapR}

We begin our investigation of Hasse graphs of permutations by studying
$\av(\pdiamond,\mathbf{2314})$.
This class was first enumerated by Kremer~\cite{Kremer2000,Kremer2003} using generating trees.
It is counted by \href{http://oeis.org/A006318}{A006318} in OEIS~\cite{OEIS}.
In this chapter, we determine its generating function by analys\-ing its Hasse graphs, which have a relatively simple structure.

Let us use $\RRR$ to denote $\av(\pdiamond,\mathbf{2314})$.
Since permutations in $\RRR$ avoid $\pdiamond$, the source graphs and bottom subgraphs are trees.
The avoidance of $\mathbf{2314}$ places no further restrictions on the structure of source trees because
$\mathbf{2314}$ is not contained in any permutation whose Hasse graph is a tree rooted at the lower left.

In order to avoid $\mathbf{2314}$, a permutation can only contain occurrences of $\mathbf{213}$ for which
there is no point to the lower left of the $\mathbf{2}$.
Thus the $\mathbf{2}$ in any $\mathbf{213}$ must be a left-to-right minimum, the root of a source tree.
In any $\mathbf{213}$, the edge joining the $\mathbf{1}$ to the $\mathbf{3}$ must be a tendril.
Hence, when adding a new source tree,
the upper end of each tendril must be a child of the root of the bottom subtree.
Specifically, if the root of the bottom subtree has $k$ children, $v_1, \ldots, v_k$, from left to right, then the upper ends of the tendrils are $v_i, \ldots, v_k$ for some $i\in\{1,\ldots,k\}$.
See Figure~\ref{figRExample} for an illustration.
Clearly, the key parameter to keep account of is the root degree of the bottom subtree.

\begin{figure}[ht]
  $$
  \begin{tikzpicture}[scale=0.275]
    \draw [blue,ultra thick] (1,12)--(3,16);
    \draw [blue,ultra thick] (2,7)--(7,10)--(9,11);
    \draw [blue,ultra thick] (4,4)--(13,6);
    \draw [blue,ultra thick] (10,1)--(14,2)--(15,3);
    \draw [blue,very thick] (1,12)--(5,14)--(6,15);
    \draw [blue,very thick] (1,12)--(8,13);
    \draw [blue,very thick] (2,7)--(11,8)--(12,9);
    \draw [blue,very thick] (4,4)--(16,5);
    \draw [red,thin] (13,6)--(10,1)--(11,8)--(4,4)--(5,14)--(2,7)--(3,16);
    \draw [red,thin] (4,4)--(7,10)--(8,13);
    \draw [red,thin] (15,3)--(16,5);
    \plotpermnobox{16}{12,7,16,4,14,15,10,13,11,1,8,9,6,2,3,5}
    \draw [thin] (1,12) circle [radius=0.4];
    \draw [thin] (2,7)  circle [radius=0.4];
    \draw [thin] (4,4)  circle [radius=0.4];
    \draw [thin] (10,1) circle [radius=0.4];
  \end{tikzpicture}
  $$
  \caption{The Hasse graph of a permutation
  in class $\RRR$,
  showing its construction from four source trees}
  \label{figRExample}
\end{figure}
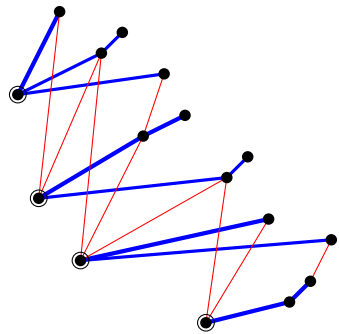

Let us imagine that source trees come with tendrils already attached to some of their trunk vertices.
We use $\SSS$ to denote the class of extended source trees.
As we have noted, the set of tethered vertices in the bottom subtree can be determined from the number of tendrils.
(Indeed,
if there are $m$ tendrils, they are the rightmost $m$ children of the root of the bottom subtree.)
Thus, an element of $\SSS$ specifies both a specific source tree and also the positioning of its vertices relative to those of the bottom subtree.

Let us determine the structure of elements of $\SSS$.
We use $z$ to mark the number of vertices in the source tree. We do not count the (tethered) vertices at the upper ends of the tendrils.
In addition, we use $y$ to mark the root degree (including counting any tendrils attached to the root).
After a new source tree has been attached, this value is the root degree of the bottom subtree in the resulting permutation.
Finally, we use $x$ to mark the total number of tendrils.
For example, the last source tree in Figure~\ref{figRExample} would be represented by $z^3\+y^3\+x^3$.

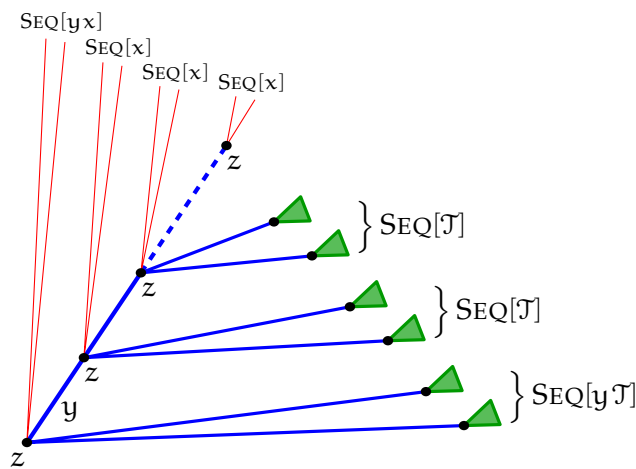
\begin{figure}[ht]
  $$
  \begin{tikzpicture}[xscale=0.25,yscale=0.225,line join=round]
    \draw [blue,ultra thick] (1,1)--(4,6)--(7,11);
    \draw [blue,ultra thick,dashed] (7,11)--(11.5,18.5);
    \draw [blue,very thick] (24,2)--(1,1)--(22,4);
    \draw [blue,very thick] (20,7)--(4,6)--(18,9);
    \draw [blue,very thick] (16,12)--(7,11)--(14,14);
    \draw [red,thin] (2,24.8)--(1,1)--(3,24.6);
    \draw [red,thin] (5,23.4)--(4,6)--(6,23.2);
    \draw [red,thin] (8,22)--(7,11)--(9,21.8);
    \draw [red,thin] (12,21.4)--(11.5,18.5)--(13,21.2);
    \draw [very thick,black!40!green,fill=black!40!green!65!white,rotate around={25:(24,2)}] (24,2)--(25.75,1.25)--(25.9,2.75)--(24,2);
    \draw [very thick,black!40!green,fill=black!40!green!65!white,rotate around={25:(22,4)}] (22,4)--(23.75,3.25)--(23.9,4.75)--(22,4);
    \draw [very thick,black!40!green,fill=black!40!green!65!white,rotate around={25:(20,7)}] (20,7)--(21.75,6.25)--(21.9,7.75)--(20,7);
    \draw [very thick,black!40!green,fill=black!40!green!65!white,rotate around={25:(18,9)}] (18,9)--(19.75,8.25)--(19.9,9.75)--(18,9);
    \draw [very thick,black!40!green,fill=black!40!green!65!white,rotate around={25:(16,12)}] (16,12)--(17.75,11.25)--(17.9,12.75)--(16,12);
    \draw [very thick,black!40!green,fill=black!40!green!65!white,rotate around={25:(14,14)}] (14,14)--(15.75,13.25)--(15.9,14.75)--(14,14);
    \plotpermnobox{24}{1,0,0,6,0,0,11,0,0,0,0,0,0,14,0,12,0,9,0,7,0,4,0,2}
    \fill[radius=0.275] (11.5,18.5) circle;
    \node[right] at (26,3.7) {$\!\Big\}\,\seq{y\+\TTT}$};
    \node[right] at (22,8.7) {$\!\Big\}\,\seq{\TTT}$};
    \node[right] at (18,13.7) {$\!\Big\}\,\seq{\TTT}$};
    \node[below] at (.5,1) {$z$};
    \node[below right] at (2.3,4) {$y$};
    \node[below right] at (3.5,6) {$z$};
    \node[below right] at (6.5,11) {$z$};
    \node[below right] at (11,18.5) {$z$};
    \node[above] at (2.5,24.5) {$\,\,\,{}_{\seq{yx}}$};
    \node[above] at (5.5,23.0) {$\,\,\,{}_{\seq{x}}$};
    \node[above] at (8.5,21.5) {$\,\,\,{}_{\seq{x}}$};
    \node[above] at (12.5,20.9) {$\,\,\,{}_{\seq{x}}$};
  \end{tikzpicture}
  $$
  \caption{The structure of extended source trees for class $\RRR$}
  \label{figRSStruct}
\end{figure}
The class, $\SSS$, of source trees with tendrils attached
satisfies the following structural equation:
\begin{equation}\label{eqnS}
\SSS \;=\;
\ZZZ\+\seq{y\+ x} \:+\:
\big(
\ZZZ\+\seq{y\+ x}\+ y\+\seq{y\+\TTT}
\times\seq{\ZZZ\+\seq{x}\+\seq{\TTT}}
\times \ZZZ\+\seq{x}
\big),
\end{equation}
where $\TTT$ represents plane trees.

The first summand corresponds to single vertex source trees, and the second to source trees with more than one vertex (and hence more than one trunk vertex).
As illustrated in Figure~\ref{figRSStruct},
the first term in the parenthesized product 
corresponds to
the root vertex and structures attached to it,
the second term corresponds to intermediate vertices in the trunk and structures attached to them,
and the third term corresponds to the tip of the trunk and tendrils attached to it. 

To account for child vertices of the root of the bottom subtree which are not tethered (vertices $v_1,\ldots v_{i-1}$ above), we can consider these vertices as being the upper ends of ``virtual'' 
tendrils which are not attached at the lower end to any vertex of the source tree.
Let $\SSS^\star$ denotes the extension of class $\SSS$ to include these virtual tendrils.
Thus, an element of $\SSS^\star$ consists of a source tree with tendrils attached to its trunk along with zero or more virtual tendrils.
We have $\SSS^\star = \seq{x}\+\SSS$. Expanding~\eqref{eqnS} then gives us the following trivariate generating function for $\SSS^\star$:
\begin{equation*}
S^\star(z,y,x) \;=\; \frac
{z\+ (x\+ (1 - x)\+ (1 - y) - z\+ (1 - y + y\+ (1 - y\+ x) \+(t(z) -  x ) ))}
{(1 - y + z\+ y^2)\+(1 - x)\+ (1 - y\+ x)\+ (x - x^2 - z)},
\end{equation*}
where $t(z)=\half(1 - \sqrt{1 - 4\+ z})$ is the generating function for $\TTT$, the class of rooted plane trees.

Let $R(y)=R(z,y)$ be the bivariate generating function for $\RRR$, in which $y$ marks the root degree of the bottom subtree.
Since a vertex in the bottom subtree can only be tethered by one tendril,
adding a new source tree corresponds to the linear operator on $R(y)$ that maps $y^d$ to $[x^d]S^\star(z,y,x)$.

The result now follows by algebraic manipulation and the kernel method. It can be checked (using a computer algebra system, or otherwise) that
\begin{equation*}
[x^d]S^\star(z,y,x) \;=\;
\frac{z}{1-y+z\+y^2}\+\big(1+y\+(t(z)-1-y^d+z^{-d}\+t(z)^d)\big) .
\end{equation*}
The class of initial source trees is defined by $\ZZZ\times\seq{y\+\TTT}$, so $\RRR$ satisfies the functional equation
\begin{equation*}
R(y) \;=\;
\frac{z}{1-y\+t(z)} \:+\:
\frac{z}{1-y+z\+y^2}\+\big(R(1)+y\+\big(t(z)-1-R(y)+R(t(z)/z)\big)\big).
\end{equation*}
This can be solved for $R(1)$ by using the kernel method. Rearranging yields
\begin{equation}\label{eqnRFunc}
(1 -y +z\+y +z\+y^2)\+R(y) \;=\;
z\+ \big(1 + R(1) + y\+(R(t(z)/z)-1-\sqrt{1-4\+z})\big) .
\end{equation}
The kernel can be cancelled by setting $y=(1-z-\sqrt{z^2-6\+ z+1})/2\+ z$, which results in an equation that can be solved to give $R(t(z)/z)$ in terms of $R(1)$. Finally, substituting for $R(t(z)/z)$ in \eqref{eqnRFunc}, setting $y$ to 1 and solving for $R(1)$ yields the following:

\newpage  
\thmbox{
\begin{thm}\label{thmR}
The class of permutations avoiding $\mathbf{\pdiamond}$ and $\mathbf{2314}$ has the algebraic generating function
$$
\half\+\Big(1-z-\sqrt{1-6\+z+z^2}\Big).
$$
\end{thm}
}

The corresponding numbers are known as the (large) Schr\"oder numbers. 
These numbers also count the
number of paths from the southwest corner $(0,0)$ of an $n\times n$ grid to the northeast corner $(n,n)$, using only single steps east, northeast, or north, that do not rise above the SW–-NE diagonal.
The following question may be worth consideration:
\begin{question}
Is there a nice bijection between these lattice paths and the Hasse graphs of permutations in $\av(\mathbf{\pdiamond},\mathbf{2314})$?
\end{question}

\cleardoublepage


\chapter{Avoiding 1324 and 1432}
\label{chapG}

In this chapter, we investigate $\av(\pdiamond,\mathbf{1432})$. This is one of the classes, listed in~\cite{WikiEnumPermClassesThin}, with
a basis consisting of two permutations of length 4, that remains to be enumerated.
The source graphs in this class have a very simple form. Despite this, the functional equation we derive for its generating function is rather complicated, and we are unable to solve it.

Let us use $\GGG$ to denote $\av(\pdiamond,\mathbf{1432})$.
As in the previous chapter, since $\GGG$ avoids $\pdiamond$, source graphs and bottom subgraphs are trees.
The further requirement to avoid $\mathbf{1432}$ severely restricts the structure of both source trees and bottom subtrees: they must either be paths or else have a single fork. Let's call these \emph{spindly trees}.

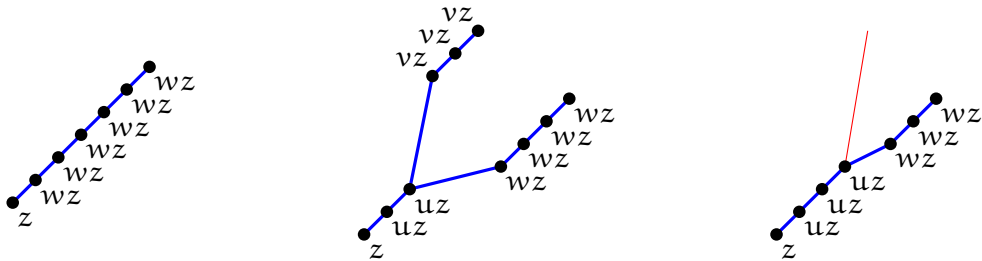
\begin{figure}[ht]
  $$
  \raisebox{12pt}{
  \begin{tikzpicture}[scale=0.3,line join=round]
    \draw [blue,very thick] (1,1)--(7,7);
    \plotpermnobox{6}{1,2,3,4,5,6,7}
    \node [below right] at (1,1) {$\!z$};
    \node [below right] at (2,2) {$\!w\+z$};
    \node [below right] at (3,3) {$\!w\+z$};
    \node [below right] at (4,4) {$\!w\+z$};
    \node [below right] at (5,5) {$\!w\+z$};
    \node [below right] at (6,6) {$\!w\+z$};
    \node [below right] at (7,7) {$\!w\+z$};
  \end{tikzpicture}
  }
  \qquad\qquad\quad
  \begin{tikzpicture}[scale=0.3,line join=round]
    \draw [blue,very thick] (1,1)--(3,3);
    \draw [blue,very thick] (6,10)--(4,8)--(3,3)--(7,4)--(10,7);
    \plotpermnobox{10}{1,2,3,8,9,10,4,5,6,7}
    \node [below right] at (1,1) {$\!z$};
    \node [below right] at (2,2) {$\!u\+z$};
    \node [below right] at (3,3) {$\!u\+z$};
    \node [above left] at (4,8) {$v\+z\!$};
    \node [above left] at (5,9) {$v\+z\!$};
    \node [above left] at (6,10) {$v\+z\!$};
    \node [below right] at (7,4) {$\!w\+z$};
    \node [below right] at (8,5) {$\!w\+z$};
    \node [below right] at (9,6) {$\!w\+z$};
    \node [below right] at (10,7) {$\!w\+z$};
  \end{tikzpicture}
  \qquad\qquad\quad
  \begin{tikzpicture}[scale=0.3,line join=round]
    \draw [blue,very thick] (1,1)--(4,4)--(6,5)--(8,7);
    \draw [red,thin] (4,4)--(5,10);
    \plotpermnobox{6}{1,2,3,4,0,5,6,7}
    \node [below right] at (1,1) {$\!z$};
    \node [below right] at (2,2) {$\!u\+z$};
    \node [below right] at (3,3) {$\!u\+z$};
    \node [below right] at (4,4) {$\!u\+z$};
    \node [below right] at (6,5) {$\!w\+z$};
    \node [below right] at (7,6) {$\!w\+z$};
    \node [below right] at (8,7) {$\!w\+z$};
  \end{tikzpicture}
  $$
  \caption{Spindly trees: a path, a forked tree, and a path with a tendril attached}
  \label{figSpindly}
\end{figure}

Let $\SSS$ denote the class of spindly trees. Let us use
\vspace{-9pt}
\begin{itemize}
  \itemsep0pt
  \item $w$ to mark vertices in path trees (except the root) and in the \emph{right} (lower) branch of forked trees,
  \item $v$ to mark vertices in the \emph{left} (upper) branch of forked trees, and
  \item $u$ to mark vertices in the \emph{stem} of forked trees: the forked vertex and vertices below the fork (except the root).
\end{itemize}
\vspace{-9pt}

There are thus two cases, source trees that are paths, the class of which we denote $\SSS_\PP$, and forked source trees, for which we use $\SSS_\FF$. These can be defined by the following
two structural equations.

Firstly,
\begin{equation*}
  \SSS_\PP \;=\; \ZZZ\times\seq{w\+\ZZZ}.
\end{equation*}
Secondly,
\begin{equation*}
  \SSS_\FF \;=\; \ZZZ\times\seq{u\+\ZZZ}\times\seqplus{v\+\ZZZ}\times\seqplus{w\+\ZZZ}.
\end{equation*}

The corresponding generating functions are thus
\begin{equation*}
  S_\PP(w) \;=\; \frac{z}{1-z\+w}
\end{equation*}
and
\begin{equation*}
  S_\FF(u,v,w) \;=\; \frac{z^3\+v\+w }{(1-z\+u) \+ (1-z\+v) \+(1-z\+w)}.
\end{equation*}
The sum of these give the following generating function for $\SSS$:
\begin{equation*}
  S(u,v,w) \;=\; \frac{z\+(1-z\+u\+(1-z\+v)-z\+v\+(1-z\+u))}{(1-z\+u)\+(1-z\+v)\+(1-z\+w)}.
\end{equation*}

It is also helpful to consider the class consisting of path source trees with a tendril attached to any one of the vertices. We denote this class $\widehat{\SSS}_\PP$.
We use $w$ to mark the vertices on the path above the tendril and $u$ to mark the rest except the root.
Thus,
\begin{equation*}
  \widehat{\SSS}_\PP \;=\; \ZZZ\times\seq{u\+\ZZZ}\times\seq{w\+\ZZZ},
\end{equation*}
for which the generating function is
\begin{equation*}
  \widehat{S}_\PP(u,w) \;=\; \frac{z}{(1-z\+u)\+(1-z\+w)}.
\end{equation*}

Let $G_\CC(u,v,w) = G_\CC(z,u,v,w)$
be the multivariate generating function for connected permutations in $\GGG$, in which $u$, $v$ and $w$ mark the non-root vertices in the spindly bottom subtree as described above.

When adding a new spindly source tree, no more than two tendrils can be added.
We count \emph{connected} (skew-indecomposable) permutations.
The addition of a new spindly source tree can be broken down into four cases. We denote the corresponding linear operators on the generating function by $\oper_1$, $\oper_2$, $\oper_3$ and $\oper_4$.

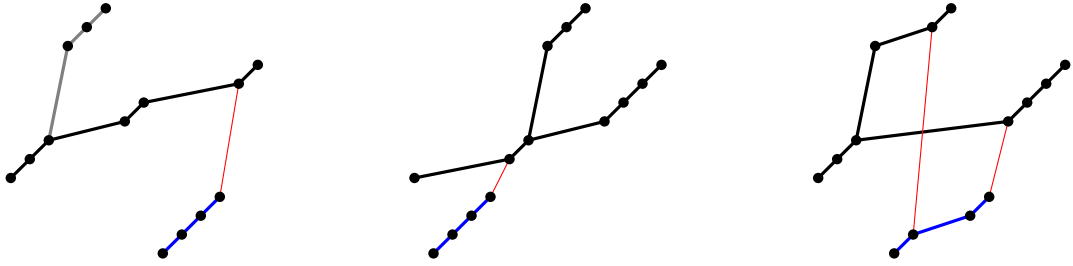
\begin{figure}[ht]
  $$
  \begin{tikzpicture}[scale=0.25]
    \draw [very thick] (1,5)--(3,7)--(7,8)--(8,9)--(13,10)--(14,11);
    \draw [gray,very thick] (3,7)--(4,12)--(6,14);
    \draw [blue,very thick] (9,1)--(12,4);
    \draw [red,thin] (12,4)--(13,10);
    \plotpermnobox{14}{5,6,7,12,13,14,8,9,1,2,3,4,10,11}
  \end{tikzpicture}
  \qquad\qquad\quad
  \begin{tikzpicture}[scale=0.25]
    \draw [very thick] (1,5)--(6,6)--(7,7)--(8,12)--(10,14);
    \draw [very thick] (7,7)--(11,8)--(14,11);
    \draw [blue,very thick] (2,1)--(5,4);
    \draw [red,thin] (5,4)--(6,6);
    \plotpermnobox{14}{5,1,2,3,4,6,7,12,13,14,8,9,10,11}
  \end{tikzpicture}
  \qquad\qquad\quad
  \begin{tikzpicture}[scale=0.25]
    \draw [very thick] (1,5)--(3,7)--(4,12)--(7,13)--(8,14);
    \draw [very thick] (3,7)--(11,8)--(14,11);
    \draw [blue,very thick] (5,1)--(6,2)--(9,3)--(10,4);
    \draw [red,thin] (6,2)--(7,13);
    \draw [red,thin] (10,4)--(11,8);
    \plotpermnobox{14}{5,6,7,12,1,2,13,14,3,4,8,9,10,11}
  \end{tikzpicture}
  $$
  \caption{Cases I, II and III: Three methods for adding a 
  path}
  \label{figGCases123}
\end{figure}
Cases I, II and III are only applicable when the source tree is a path. These are
illustrated in Figure~\ref{figGCases123}.
\subsubsection*{Case I}
In case I, a path tree is inserted below a path bottom subtree or the right branch of a forked bottom subtree so that the resultant bottom subtree is a path, a tendril being attached to the tip of the path.
Thus,
\begin{equation*}
  \oper_1[u^r\+v^s\+w^t] \;=\; \sum_{i=1}^t S_\PP(w)\+w^i \;=\; \frac{z\+w\+(1-w^t)}{(1-w)\+(1-z\+w)},
\end{equation*}
where $i$ counts the tethered vertex in the right branch of the bottom subtree, starting from the tip.
Hence,
\begin{equation*}
  \oper_1[f(u,v,w)] \;=\; \frac{z\+w}{(1-w)\+(1-z\+w)}\+\big(f(1,1,1)-f(1,1,w)\big).
\end{equation*}

\subsubsection*{Case II}
In case II, a path tree is inserted to the left of the fork in a forked bottom subtree, a tendril being attached to the tip of the path.
The resultant bottom subtree is forked.
Thus,
\begin{equation*}
  \oper_2[u^r\+v^s\+w^t] \;=\; \sum_{i=1}^r S_\PP(u)\+u^i\+v^s\+w^t \;=\; \frac{z\+u\+(1-u^r)\+v^s\+w^t}{(1-u)\+(1-z\+u)},
\end{equation*}
where $i$ counts the tethered vertex in the stem of the bottom subtree, starting from the forked vertex.
Hence,
\begin{equation*}
  \oper_2[f(u,v,w)] \;=\; \frac{z\+u}{(1-u)\+(1-z\+u)}\+\big(f(1,v,w)-f(u,v,w)\big).
\end{equation*}

\subsubsection*{Case III}
In case III, a path tree is inserted below the left branch of a forked bottom subtree.
The leftmost vertex in the right branch is always tethered to the tip of the path tree.
The resultant bottom subtree is forked.
Thus,
\begin{equation*}
  \oper_3[u^r\+v^s\+w^t] \;=\; \sum_{i=1}^s \widehat{S}_\PP(u,w)\+v^i\+w^t \;=\; \frac{z\+v\+(1-v^s)\+w^t}{(1-z\+u)\+(1-v)\+(1-z\+w)},
\end{equation*}
where $i$ counts the tethered vertex in the left branch of the bottom subtree, starting from the tip.
Hence,
\begin{equation*}
  \oper_3[f(u,v,w)] \;=\; \frac{z\+v}{(1-z\+u)\+(1-v)\+(1-z\+w)}\+\big(f(1,1,w)-f(1,v,w)\big).
\end{equation*}

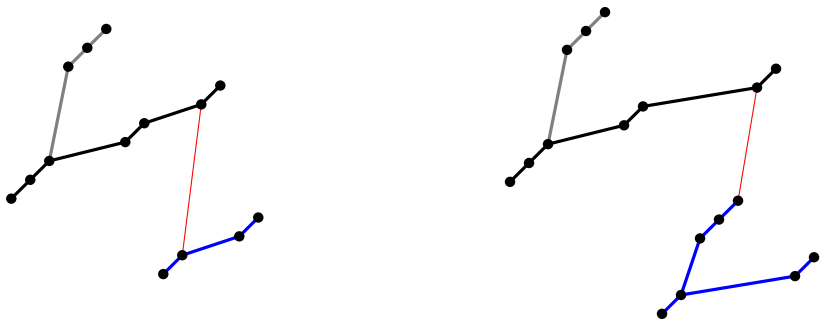
\begin{figure}[ht]
  $$
  \raisebox{15pt}{
  \begin{tikzpicture}[scale=0.25]
    \draw [very thick] (1,5)--(3,7)--(7,8)--(8,9)--(11,10)--(12,11);
    \draw [gray,very thick] (3,7)--(4,12)--(6,14);
    \draw [blue,very thick] (9,1)--(10,2)--(13,3)--(14,4);
    \draw [red,thin] (10,2)--(11,10);
    \plotpermnobox{14}{5,6,7,12,13,14,8,9,1,2,10,11,3,4}
  \end{tikzpicture}
  }
  \qquad\qquad\qquad\qquad
  \begin{tikzpicture}[scale=0.25]
    \draw [very thick] (1,8)--(3,10)--(7,11)--(8,12)--(14,13)--(15,14);
    \draw [gray,very thick] (3,10)--(4,15)--(6,17);
    \draw [blue,very thick] (9,1)--(10,2)--(16,3)--(17,4);
    \draw [blue,very thick] (10,2)--(11,5)--(13,7);
    \draw [red,thin] (13,7)--(14,13);
    \plotpermnobox{17}{8,9,10,15,16,17,11,12,1,2,5,6,7,13,14,3,4}
  \end{tikzpicture}
  $$
  \caption{Case IV: The addition of a path and of a forked tree}
  \label{figGCase4}
\end{figure}
\subsubsection*{Case IV}
Case IV is illustrated in Figure~\ref{figGCase4}.
In this case, a path or forked tree is inserted below a path bottom subtree or the right branch of a forked bottom subtree so that
the resultant bottom subtree is forked.
In this case, the tendril may not be attached to the tip of a path source tree.
Thus,
\begin{equation*}
  \oper_4[u^rv^sw^t] =
  \sum_{i=1}^t \big(\widehat{S}_\PP(u,w)-S_\PP(u)+S_\FF(u,v,w)\big) v^i
  = \frac{z^2vw(1-v^t)}{(1-zu)(1-v)(1-zv)(1-zw)},
\end{equation*}
and hence
\begin{equation*}
  \oper_4[f(u,v,w)] \;=\; \frac{z^2\+v\+w}{(1-z\+u)\+(1-v)\+(1-z\+v)\+(1-z\+w)}\+\big(f(1,1,1)-f(1,1,v)\big).
\end{equation*}

\newpage  
Thus we have the following:

\thmbox{
\begin{thm}\label{thmGCrec}
The generating function $G_\CC(u,v,w)$ for connected permutations in the class $\av(\pdiamond,\mathbf{1432})$ satisfies the functional equation
$$
  G_\CC(u,v,w) \;=\; S(u,v,w) \:+\: (\oper_1+ \oper_2 + \oper_3 + \oper_4)[G_\CC(u,v,w)].
$$
\end{thm}
} 

This is a very complex algebraic identity relating $G_\CC(u,v,w)$ to $G_\CC(1,v,w)$, $G_\CC(1,1,w)$, $G_\CC(1,1,v)$ and $G_\CC(1,1,1)$.

The kernel method can be applied to this to yield $G_\CC(1,v,w)$ as a function of $G_\CC(1,1,w)$, $G_\CC(1,1,v)$ and $G_\CC(1,1,1)$, and then used again to give $G_\CC(1,1,w)$ in terms of $G_\CC(1,1,1)$ and $G_\CC(1,1,\beta(w))$ for some algebraic function $\beta(w)$.
This does not appear to help us to get an explicit form for $G_\CC(1,1,1)$.
Despite the simplicity of the structure of the source trees of $\GGG$, its exact enumeration remains an open question.

However, the
recurrence in Theorem~\ref{thmGCrec}
can be iterated to extract the coefficients of $z^n$ from $G_\CC(1,1,1)$ and hence
enumerate $\GGG=\av(\pdiamond,\mathbf{1432})$ for small $n$.
Values up to $n=40$ (determined in seven hours by \emph{Mathematica}~\cite{Mathematica}) can be found at \href{http://oeis.org/A165542}{A165542} in OEIS~\cite{OEIS}.

\emph{Postscript}: Following submission of this thesis, at the AMS meeting in Washington, DC, in March 2015, Nathaniel Shar presented a much more efficient recurrence for class $\GGG$; his work is currently unpublished.


\cleardoublepage


\chapter{Forest-like permutations}
\label{chapForest}

In this chapter, we consider \emph{forest-like} permutations, $\av(\pdiamond,\mathbf{21\bar{3}54})$, permutations whose Hasse graph is a plane forest. 
These are $\pdiamond$-avoiders whose Hasse graph is plane.
This class is counted by \href{http://oeis.org/A111053}{A111053} in OEIS~\cite{OEIS}.
It was first enumerated by Bousquet-M\'elou \& Butler~\cite{BMB2007}.
Their approach makes use of a rather complicated analysis of the Hasse graphs of permutations.
We present here a somewhat simpler derivation of the algebraic generating function for this class.

Since permutations in the class avoid $\pdiamond$, the ``tethering rules'' for $\pdiamond$-avoiders presented in Section~\ref{sect1324Tethering} apply.
Source and bottom subgraphs are trees.
The planarity of the Hasse graphs implies that
the only vertices in the bottom subtree that can be tethered are those
in its
lowermost (rightmost) branch.
As a consequence, to avoid creating a $\pdiamond$, at most one tendril can be used when adding a new source tree.
We count \emph{connected} (skew-indecomposable) permutations in the class, those constructed by using exactly one tendril when attaching each new source tree.
Clearly, we must keep track of the length of the lowermost branch of the bottom subtree.

We begin by presenting structural equations for source trees.
We use $u$ to mark vertices in the uppermost branch (the trunk) and $v$ to mark vertices in the lowermost branch, excluding the root vertex.
There are two cases, source trees that are paths, the class of which we denote $\SSS_\PP$, and source trees with at least one fork, for which we use $\SSS_\FF$. These can be defined by the following
two structural equations.

Firstly,
\begin{equation*}
  \SSS_\PP \;=\; u\+\ZZZ \times \seq{u\+v\+\ZZZ},
\end{equation*}
the first term corresponding to the root and the second to the remaining vertices.

\begin{figure}[ht]
  $$
  \begin{tikzpicture}[scale=0.225,line join=round]
    \draw [blue,very thick] (0,0)--(3,3);
    \draw [blue,very thick,dashed] (3,3)--(5,5);
    \draw [blue,very thick] (5,5)--(6,18)--(7,21);
    \draw [blue,very thick,dashed] (7,21)--(8.5,25.5);
    \draw [blue,very thick] (5,5)--(18,6)--(21,7);
    \draw [blue,very thick,dashed] (21,7)--(25.5,8.5);
    \draw [blue,very thick] (14,16)--(5,5)--(16,14);
    \draw [very thick,black!40!green,fill=black!40!green!65!white,rotate around={45:(14,16)}] (14,16)--(15.75,15.25)--(15.75,16.75)--(14,16);
    \draw [very thick,black!40!green,fill=black!40!green!65!white,rotate around={45:(16,14)}] (16,14)--(17.75,13.25)--(17.75,14.75)--(16,14);
    \draw [very thick,black!40!green,fill=black!40!green!65!white] (7,21)--(10.5,23.5)--(11.5,22.5)--(7,21);
    \draw [very thick,black!40!green,fill=black!40!green!65!white] (6,18)--(12.5,20.5)--(13.5,19.5)--(6,18);
    \draw [very thick,black!40!green,fill=black!40!green!65!white] (21,7)--(23.5,10.5)--(22.5,11.5)--(21,7);
    \draw [very thick,black!40!green,fill=black!40!green!65!white] (18,6)--(20.5,12.5)--(19.5,13.5)--(18,6);
    \plotpermnobox{24}{1,2,3,0,5,18,21,0,0,0,0,0,0,16,0,14,0,6,0,0,7}
    \fill[radius=0.275] (0,0) circle;
    \fill[radius=0.275] (8.5,25.5) circle;
    \fill[radius=0.275] (25.5,8.5) circle;
    \node [rotate=-45] at (2.2,3.8) {$\Bigg\{$};
    \node [rotate=45] at (1,5) {$\seq{u\+v\+z}$};
    \node [rotate=45] at (16.9,16.9) {$\bigg\}$};
    \node [rotate=-45] at (18,18) {$\,\seq{\TTT}$};
    \node [rotate=-71.565] at (21,5.5) {$\Bigg\}$};
    \node [rotate=18.435] at (21.6,3.7) {$\seq{v\+\TTT}$};
    \node [rotate=-18.435] at (5.5,21) {$\Bigg\{$};
    \node [rotate=71.565] at (3.7,21.6) {$\seq{u\+\TTT}$};
    \node[below] at (-.5,0) {$u\+z$};
    \node [above right] at (8.4,25.4) {$u\+z$};
    \node [above right] at (25.4,8.4) {$v\+z$};
  \end{tikzpicture}
  $$
  \caption{The structure of source trees for forest-like permutations}
  \label{figLSFStruct}
\end{figure}
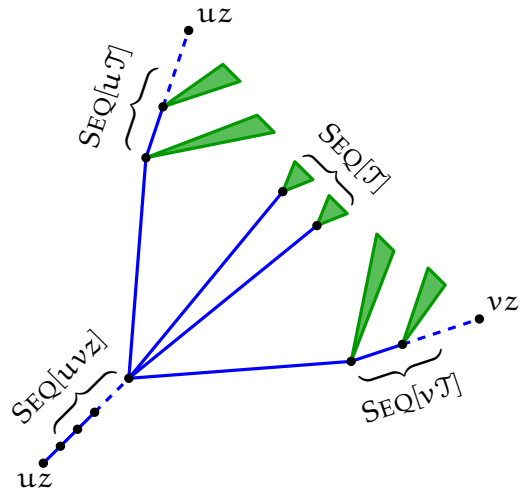
Secondly,
\begin{equation*}
  \SSS_\FF \;=\; u\+\ZZZ \times \seq{u\+v\+\ZZZ} \times (\seq{u\+\TTT}\times u\+\ZZZ) \times (\seq{v\+\TTT}\times v\+\ZZZ) \times \seq{\TTT},
\end{equation*}
where $\TTT$ represents plane trees.

This is illustrated in Figure~\ref{figLSFStruct}.
The first term in the product corresponds to the root.
The second corresponds to the \emph{stem} of the tree (vertices up to and including the first forked vertex).
The third and fourth terms correspond to vertices above the stem lying on the uppermost and lowermost branches respectively, along with subtrees rooted at these vertices, the subterms $u\+\ZZZ$ and $v\+\ZZZ$ corresponding to the tips of the relevant branches.
The final term corresponds to the remaining vertices, which lie on subtrees whose roots are children of the first forked vertex that are not on the uppermost or lowermost branch.

The corresponding trivariate generating functions are thus:
\begin{equation*}
  S_\PP(u,v) \;=\; \frac{z\+u}{1-z\+u\+v},
\end{equation*}
\begin{equation*}
  S_\FF(u,v) \;=\; \frac{z^2 \+ u^2 \+ v \+ t(z)}{(1-z\+u\+v) \+ (1-u\+t(z)) \+ (1-v\+t(z)) },
\end{equation*}
where $t(z)$ is the generating function for plane trees.
Let $S(u,v)=S_\PP(u,v)+S_\FF(u,v)$ be the generating function for all source trees.

We now designate one vertex in the trunk to which a tendril is to be attached.
Once we have done this, we no longer need to keep track of the number of vertices in the trunk.
We use the pointing construction
(see Section~\ref{defPointing})
to derive the appropriate generating functions.

For source trees that are paths, we exclude the case in which the tendril is attached to the tip (because we need to handle this case separately). If $\widehat{S}_\PP(v)$ is the corresponding bivariate generating function for paths with a tendril attached to a non-tip vertex, then
\begin{equation*}
  \widehat{S}_\PP(v) \;=\; \partial_u S_\PP(u,v)\big|_{u=1} \:-\: S_\PP(1,v) \;=\; \frac{z^2\+v}{(1-z\+v)^2}.
\end{equation*}
Analogously, if we use $\widehat{S}_\FF(v)$ for forked source trees with a tendril attached to a trunk vertex, then
\begin{equation*}
  \widehat{S}_\FF(v) \;=\; \partial_u S_\FF(u,v)\big|_{u=1} \;=\; \frac{ z^2\+v\+t(z)\+ (2-z\+v-t(z))}{(1-z\+v)^2\+(1-t(z))^2\+  (1-v\+t(z))}.
\end{equation*}
Let $\widehat{S}(v)=\widehat{S}_\PP(v)+\widehat{S}_\FF(v)$.

Let $L_\CC(v)=L_\CC(z,v)$ be the bivariate generating function for connected forest-like permutations, in which $v$ marks the length (number of edges in) the lowermost branch of the bottom subtree.

\begin{figure}[ht]
  $$
  \begin{tikzpicture}[scale=0.25]
    \draw [very thick] (1,6)--(2,7)--(6,8)--(12,9)--(15,10)--(16,11);
    \draw [very thick] (2,7)--(3,14)--(5,15);
    \draw [very thick] (3,14)--(4,16);
    \draw [very thick] (13,13)--(12,9)--(14,12);
    \draw [blue,very thick] (7,1)--(11,5);
    \draw [red,thin] (11,5)--(12,9);
    \plotpermnobox{16}{6,7,14,16,15,8,1,2,3,4,5,9,13,12,10,11}
  \end{tikzpicture}
  \qquad\qquad\qquad
  \begin{tikzpicture}[scale=0.25]
    \draw [very thick] (1,9)--(2,10)--(6,11)--(10,12)--(13,13)--(14,14);
    \draw [very thick] (2,10)--(3,17)--(5,18);
    \draw [very thick] (3,17)--(4,19);
    \draw [very thick] (11,16)--(10,12)--(12,15);
    \draw [blue,very thick] (7,1)--(8,2)--(9,4)--(15,5)--(16,8);
    \draw [blue,very thick] (8,2)--(19,3);
    \draw [blue,very thick] (17,7)--(15,5)--(18,6);
    \draw [red,thin] (9,4)--(10,12);
    \plotpermnobox{19}{9,10,17,19,18,11,1,2,4,12,16,15,13,14,5,8,7,6,3}
  \end{tikzpicture}
  $$
  \caption{The addition of a path and of a tree}
  \label{figL}
\end{figure}
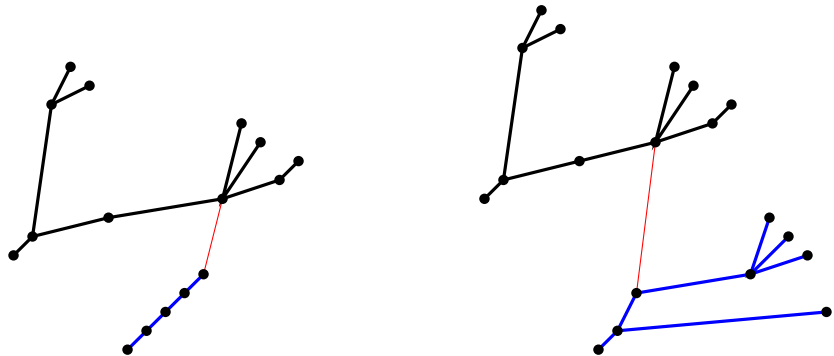
There are two methods by which a new source tree can be added to a bottom subtree, as shown in Figure~\ref{figL}.
In the first, a path source tree is added with a tendril attached to its tip.
In this case,
the lowermost branch of the
result may contain vertices from the bottom subtree.
The action of adding a path source tree in this way is thus seen to be reflected by the linear operator
$\oper_1$, acting on $L_\CC(v)$, that satisfies
\begin{equation*}
  \oper_1[v^k] \;=\; \sum_{i=1}^k S_\PP(1,v) \+v^i \;=\; \frac{z\+v\+(1-v^k)}{(1-z\+v)\+(1-v)},
\end{equation*}
where $i$ counts the tethered vertex in the lowermost branch of the bottom subtree, starting from the tip. Hence,
\begin{equation*}
  \oper_1[f(v)] \;=\; \frac{z\+v\+\big(f(1)-f(v)\big)}{(1-z\+v)\+(1-v)}.
\end{equation*}

In the second method, a source tree (which may be a path) is added in such a way that some of its vertices are positioned to the right of the bottom subtree. This action is reflected by the linear operator
$\oper_2$ that satisfies
\begin{equation*}
  \oper_2[v^k] \;=\; k \+ \widehat{S}(v),
\end{equation*}
since there are $k$ choices for the upper end of the tendril.
Hence,
\begin{equation*}
  \oper_2[f(v)] \;=\; \widehat{S}(v) \+ f'(1).
\end{equation*}

The class of initial source trees is enumerated by $S(1,v)$.
Hence, the generating function for connected forest-like permutations, $L_\CC(v)$, satisfies the functional equation
\begin{equation*}
  L_\CC(v) \;=\; S(1,v) \:+\: (\oper_1+\oper_2)[L_\CC(v)],
\end{equation*}
which expands to give
\begin{equation}\label{eqLCv}
  L_\CC(v) \;=\; S(1,v)
    \:+\: \frac{z\+v\+\big(L_\CC(1)-L_\CC(v)\big)}{(1-z\+v)\+(1-v)}
    \:+\: \widehat{S}(v) \+ L_\CC'(1).
\end{equation}
Taking the limit as $v$ tends to 1, we get
\begin{equation}\label{eqLC1}
  L_\CC(1) \;=\; S(1,1)
    \:+\: \frac{z}{(1-z)}\+L_\CC'(1)
    \:+\: \widehat{S}(1) \+ L_\CC'(1).
\end{equation}
If we eliminate $L_\CC'(1)$ from~\eqref{eqLCv} and~\eqref{eqLC1}, solve for $L_\CC(v)$, expand and simplify, then
\begin{equation*}
  L_\CC(v) \;=\; z \+ \frac
  {(1 - v)\+\big(1 - 2\+ z\+ v - (1 - z\+ v)\+ t(z)\big) \:-\: v\+ \big(z\+ v - (2 - v)\+ (1 - t(z))\big)\+ L_\CC(1)}
  {\big(1-z\+v-t(z)\big)\+\big(1-v+z\+v^2\big)}.
\end{equation*}
The kernel method now applies. Cancelling the kernel, by letting $v=t(z)/z$, and solving for $L_\CC(1)$ yields
\begin{equation*}
  L_\CC(1) \;=\; \frac   {z-3\+z^2 +(1-5\+z)\+t(z)}   {2-9\+z},
\end{equation*}
and so the generating function for all forest-like permutations is given by
\begin{equation*}
  \frac{L_\CC(1)}{1-L_\CC(1)}\;=\; \frac   {z^2+z^3 +(1-5\+z)\+t(z)}   {1-5\+z+2\+z^2-z^3}.
\end{equation*}
Hence, expanding $t(z)$, we have:

\thmbox{
\begin{thm}\label{thmL}
The class of forest-like permutations has the algebraic generating function
$$
\frac
{
2\+ z^2\+ (1 + z) + (1 - 5 z)\+ (1 - \sqrt{1 - 4\+ z})
}{
2\+(1 - 5\+z + 2\+z^2 - z^3)
} .
$$
\end{thm}
}

\cleardoublepage


\chapter{Plane permutations}
\label{chapPlane}

In~\cite{BMB2007}, Bousquet-M\'elou \& Butler ask the question, ``What is the number of plane permutations?''. Plane permutations are those whose Hasse graph is plane (i.e. non-crossing).
See Figure~\ref{figM} for an example. (The permutations in the illustration on page~\pageref{figFrontispiece} are also plane permutations.)
Plane permutations are characterised by
avoiding the barred pattern $\mathbf{21\bar{3}54}$.
In this chapter, we investigate this class of permutations.

We determine a functional equation for this class.
Our first observation is that
plane permutations
of size $n$
are in bijection with
plane
source graphs with $n+1$ vertices, by the addition/removal of a (single) root vertex to the lower left.
So we consider plane
source graphs, and then discount the root vertex.

\begin{figure}[ht]
  $$
  \begin{tikzpicture}[scale=0.25,line join=round]
    \draw [very thick] (1,1)--(2,5)--(3,6)--(4,20)--(8,21)--(9,24) node [right] {${}_1$};
    \draw [very thick] (8,21)--(19,22)--(21,23) node [right] {${}_2$};
    \draw [very thick] (3,6)--(5,16)--(6,18)--(7,19) node [right] {${}_3$};
    \draw [red,thin] (7,19)--(8,21);
    \draw [very thick] (5,16)--(14,17) node [right] {${}_4$};
    \draw [red,thin] (14,17)--(19,22);
    \draw [very thick] (3,6)--(11,7)--(12,9)--(13,10)--(18,11)--(20,13)--(23,14)--(24,15) node [right] {${}_5$};
    \draw [red,thin] (13,10)--(14,17);
    \draw [red,thin] (18,11)--(19,22);
    \draw [red,thin] (20,13)--(21,23);
    \draw [very thick] (18,11)--(22,12) node [right] {${}_6$};
    \draw [red,thin] (22,12)--(23,14);
    \draw [very thick] (11,7)--(15,8) node [right] {${}_7$};
    \draw [red,thin] (15,8)--(18,11);
    \draw [very thick] (1,1)--(10,2)--(16,3)--(17,4) node [right] {${}_8$};
    \draw [red,thin] (10,2)--(11,7);
    \draw [red,thin] (17,4)--(18,11);
     \plotpermnobox{24}{1,5,6,20,16,18,19,21,24,2,7,9,10,17,8,3,4,11,22,13,23,12,14,15}
  \end{tikzpicture}
  $$
  \caption{A plane source graph with its spanning tree built from eight branches; the tips of the branches are numbered in the order they are added}
  \label{figM}
\end{figure}
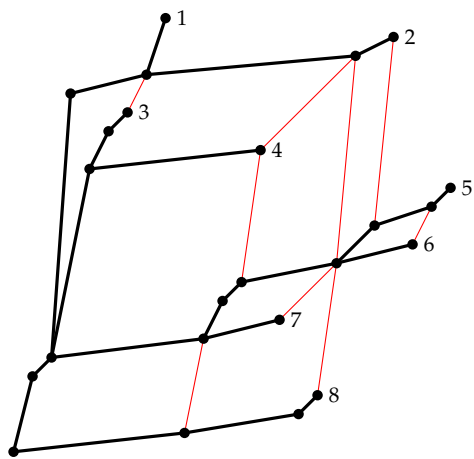
To build a plane source graph, we repeatedly add \emph{branches} (maximal increasing paths) towards the lower right.
The initial branch consists of the left-to-right maxima of the permutation. 
These branches build a rooted spanning tree.
See Figure~\ref{figM} for an illustration.

Specifically, suppose we are given a plane source graph $G$
and have already constructed a
partial spanning tree, $T$, of $G$ from a sequence of branches.
Let $x$ be the maximal (most north-easterly) vertex in $T$
for which there is some vertex in $G\setminus T$ above it
(to its north-east); $x$ is among the
right-left minima of $T$.
Then
the next branch is joined to $T$ at $x$ and consists of those
left-to-right maxima of $G\setminus T$ that are above $x$.
See Figure~\ref{figM} for an illustration of this process.
We call $x$ the \emph{joining} vertex.

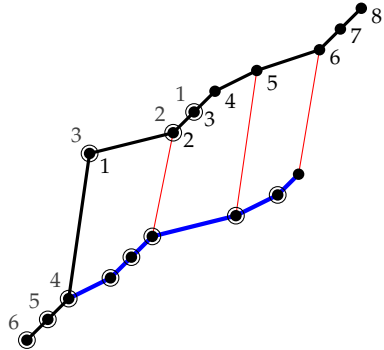
\begin{figure}[ht]
  $$
  \begin{tikzpicture}[scale=0.275,line join=round]
    \draw [very thick]
      (1,1)node [gray!50!black,above left] {${}_6\!$}--
      (2,2)node [gray!50!black,above left] {${}_5\!$}--
      (3,3)node [gray!50!black,above left] {${}_4\!$} 
      --
      (4,10)node [gray!50!black,above left] {${}_3\!$}--
      (8,11)node [gray!50!black,above left] {${}_2\!$}--
      (9,12)node [gray!50!black,above left] {${}_1\!$}--
      (10,13) -- (12,14) -- (15,15) -- (16,16) -- (17,17);
    \node [below right] at (4,10.25) {${}^1$};
    \node [below right] at (8,11.5) {${}^2$};
    \node [below right] at (9,12.5) {${}^3$};
    \node [below right] at (10,13.4) {${}^4$};
    \node [below right] at (12,14.25) {${}^5$};
    \node [below right] at (15,15.5) {${}^6$};
    \node [below right] at (16,16.5) {${}^7$};
    \node [below right] at (17,17.5) {${}^8$};
    \draw [blue,ultra thick] (3,3)--(5,4)--(7,6)--(11,7)--(13,8)--(14,9);
    \draw [red,thin] (7,6)--(8,11);
    \draw [red,thin] (11,7)--(12,14);
    \draw [red,thin] (14,9)--(15,15);
     \plotpermnobox{17}{1,2,3,10,4,5,6,11,12,13,7,14,8,9,15,16,17}
    \draw [thin] (1,1) circle [radius=0.4];
    \draw [thin] (2,2) circle [radius=0.4];
    \draw [thin] (3,3) circle [radius=0.4];
    \draw [thin] (4,10)  circle [radius=0.4];
    \draw [thin] (8,11)  circle [radius=0.4];
    \draw [thin] (9,12)  circle [radius=0.4];
    \draw [thin] (5,4)  circle [radius=0.4];
    \draw [thin] (6,5)  circle [radius=0.4];
    \draw [thin] (7,6)  circle [radius=0.4];
    \draw [thin] (11,7) circle [radius=0.4];
    \draw [thin] (13,8) circle [radius=0.4];
  \end{tikzpicture}
  $$
  \caption{Adding a new branch; open vertices, both before and after adding the branch, are circled} 
  \label{figMAdd}
\end{figure}
To keep track of the possibilities, at each step, we \emph{mark} each of the current right-to-left minima 
as either \emph{open} or \emph{closed}.
Open vertices are those to which the next branch may be joined.
In the initial branch, all vertices except its tip are open.
After a branch is added, the open vertices are those to the lower left of its tip; the tip and any vertices to its upper right are closed.
See Figure~\ref{figMAdd} for an example of adding a branch.
In this example, prior to the new branch being added, there were 6 open vertices and 5 closed vertices.
Afterwards, there are 8 open vertices and 4 closed vertices.

Note that
vertices in a new branch are positioned to the right of the marked (open or closed)
vertex immediately above the joining vertex.
They may be arbitrarily interleaved horizontally with any further marked vertices above the joining vertex, occurring immediately above the joining vertex to avoid creating a $\mathbf{2143}$.

Given this construction process, it is now straightforward to specify a functional equation for the generating function.
Let us use $u$ for open vertices and $v$ for closed vertices.
Then the process of adding a new branch to a partial spanning tree $T$ corresponds to the linear operator $\oper$ satisfying the equation
\begin{equation}\label{eqnMOper}
  \oper[u^r\+v^s] \;=\;
   \sum_{i=1}^r
   \sum_{j=1}^{s+i-1}
   u^{r+1-i} \+
   \seq{z\+u}^j \+
   z \+
   v^{s+i-j}.
\end{equation}
Here, $i$ indexes (from the upper right) the choice of joining vertex from among the open vertices of $T$,
and $j$ indexes (from the lower left) the choice for the relative position of the tip of the branch
from among the marked vertices of $T$ above the joining vertex.
For example, in Figure~\ref{figMAdd}, the choices are $i=4$ and $j=5$.
Each of the $j$ occurrences of $\seq{z\+u}$ corresponds to a sequence of new \emph{closed} vertices inserted to the right of a marked vertex of $T$. In Figure~\ref{figMAdd}, these sequences have lengths $3, 0, 0, 1, 1$ from left to right.

We can now take the expansion and linear extension of \eqref{eqnMOper} and combine it with an expression that enumerates the initial branch to yield a functional equation for the generating function of plane permutations. 
After discounting the root vertex,
this gives us the following result:

\thmbox{
\begin{thm}\label{thmM}
The generating function for plane permutations is given by $P(1,1)$, where
$$
  P(u,v) \;=\;
    \frac{z\+u\+v}{1-z\+u}
    \:+\:
    \frac{z\+u\+v}{1-v-z\+u\+v}
    \left(
    \frac{P(u,q)-P(q,q)}{u-q}
    \:-\:
    \frac{P(u,v)-P(v,v)}{u-v}
    \right) ,
$$
in which $q = \frac{1}{1-z\+u}$. 
\end{thm}
}

\HIDE{
\thmbox{
\begin{thm}\label{thmM}
The generating function for plane permutations is given by $P(1,1)$, where
$$
  P(u,v) \;=\;
    \frac{z\+u\+v}{1-z\+u}
    \:+\:
    \frac{z\+u}{(1-z\+u)\+(1-v)}
    \left(
    \frac{v\+(1+u-v)\+(P(1,1)-P(u,1))}{1-u}
    \:+\:
    \frac{u\+(P(v,v)-P(u,v))}{v-u}
    \right) .
$$
\end{thm}
}
} 

It is not known how to solve this equation.
However, this recurrence can be iterated to extract the coefficients of $z^n$ from $M(1,1)$ and hence
enumerate $\av(\mathbf{21\bar{3}54})$ for small $n$.
Values up to $n=37$ (determined in twelve hours by \emph{Mathematica}~\cite{Mathematica}) can be found at \href{http://oeis.org/A117106}{A117106} in OEIS~\cite{OEIS}.

\section*{Conjectures}

Two intriguing 
conjectures have been made concerning the enumeration of the class of plane permutations.
Firstly,
in the OEIS entry,
Mark van Hoeij posits a connection between plane permutations and the sequence of \emph{Ap\'ery numbers}
$$
a_n \;=\; \sum_{k=0}^n \binom{n}{k}^{\!2} \binom{n+k}{k}.
$$
These are known as Ap\'ery numbers because, along with other sequences, they were used in
Ap\'ery's celebrated proof of the irrationality of $\zeta(2)$ and $\zeta(3)$~\cite{Apery1979,Apery1981} (see~\cite{vdPoorten1979}).
Van Hoeij's conjecture is as follows:
\begin{conj}[van Hoeij; \href{http://oeis.org/A117106}{A117106} in OEIS~\cite{OEIS}]\label{conjVanHoeij}
The number of plane permutations, $p_n$, of length $n\geqslant2$ is given by
$$
p_n \;=\; \frac{24\+\big((5\+n^3-5\+n+6)\+a_{n+1} \:-\: (5\+n^2+15\+n+18)\+a_n\big)}{5\+(n-1)\+n^2\+ (n+2)^2\+ (n+3)^2\+ (n+4)}
,
$$
where $a_n = \sum_{k=0}^n \binom{n}{k}^{\!2} \binom{n+k}{k}$.
\end{conj}
We have confirmed that this conjecture is consistent with the values calculated using the recurrence in Theorem~\ref{thmM}.

\HIDE{
\begin{conj}[van Hoeij; see~\href{http://oeis.org/A117106}{A117106} in OEIS~\cite{OEIS}]
The generating function for plane permutations is
$$
\frac{(1+z^2) \left(\alpha(z)  \, _2F_1(\frac{5}{12},\frac{13}{12};1;q(z))- \beta(z)\+  p(z)\, _2F_1(\frac{1}{12},\frac{5}{12};1;q(z))\right)}{720\+ z^4\+ p(z)^{5/4}} \:-\: \frac{1+8\+ z-6\+ z^2+7\+ z^3}{5\+ z^3} ,
$$
where
\begin{eqnarray*}
  p(z) &=& 1 - 12\+z + 14\+z^2 + 12\+z^3 + z^4 ,\\[3pt]
  q(z) &=& \frac{1728\+z^5 \+(1-11 z-z^2)}{p(z)^3} ,\\[3pt]
  \alpha(z) &=& (1-18\+ z+74\+ z^2+18\+ z^3+ z^4)\+(1 +228\+ z +494\+ z^2-228\+ z^3+ z^4) , \\[3pt] 
  \beta(z) &=& 1+78\+ z-1606\+ z^2-78\+ z^3+  z^4 .
\end{eqnarray*}
\end{conj}
} 

Ap\'ery found that $a_n$ satisfies the following recurrence:
$$
a_{n} \;=\; \frac{1}{n^2}\+\big((11\+n^2-11\+n+3)\+a_{n-1} \:+\: (n-1)^2\+a_{n-2}\big)
.
$$
This can be combined with Conjecture~\ref{conjVanHoeij}, enabling us to conjecture the following: 
\begin{conj}
The number of plane permutations, $p_n$, of length $n\geqslant3$ satisfies the recurrence
$$
p_{n} \;=\; \frac{(11\+n^2+11\+n-6)\+p_{n-1} \:+\: (n-2)\+(n-3)\+p_{n-2}}{(n+3)\+(n+4)} .
$$
\end{conj}

If Conjecture~\ref{conjVanHoeij} holds, then
the growth rate of 
plane permutations, $\liminfty \sqrt[n]{p_n}$, is the same as that for the Ap\'ery numbers.
From the definition of the $a_n$, it is clear that this is given by
$$
\max_{0\leqslant\lambda\leqslant1} \,\liminfty \sqrt[n]{ \tbinom{n}{\lambda n}^{2} \+ \tbinom{(1+\lambda) n}{\lambda n} },
$$
which can easily be computed using Stirling's approximation and elementary calculus:
\begin{conj}
The growth rate of the class of plane permutations is $\half\+(11+5\+\sqrt{5}) \approx 11.09017$.
\end{conj}

A rather different conjecture was made by Carla Savage during her plenary talk at
the Permutation Patterns 2014 conference in Tennessee.
It relates plane permutations to certain \emph{inversion sequences}\footnote{Also known as Lehmer codes.}, which are sequences of nonnegative integers $e_1 e_2 \ldots e_n$ such that $e_i<i$ for each $i$.
These sequences are known as inversion sequences because there is a map, $\varphi$, which maps each permutation $\sigma$ to an inversion sequence
that counts the inversions of $\sigma$:
$$
\varphi(\sigma) \;=\; e_1\ldots e_n, \qquad \text{where~~} e_j \:=\: \big|\{ i \::\: i<j \text{~and~} \sigma(i)>\sigma(j) \}\big| .
$$
This map is a length-preserving bijection (see Knuth~\cite{Knuth1973}).

In an analogous manner to permutations, 
we say
that an inversion sequence $\eta=e_1\ldots e_n$ \emph{contains} a pattern $\sigma=s_1 \ldots s_k$ if there is a subsequence of $\eta$ that is order isomorphic to $\sigma$, and say that $\eta$ \emph{avoids} $\sigma$ if there is no such subsequence of $\eta$. In the context of inversion sequences,
a pattern may be an arbitrary sequence of nonnegative integers.
For example, $\mathbf{01031}$ contains two occurrences of $\mathbf{010}$, but avoids $\mathbf{201}$.

Savage's conjecture proposes a bijection between plane permutations and a particular class of inversion sequences:

\begin{conj}[Savage, 2014]\label{conjSavage}
The number of plane permutations of length $n$ is the same as the number of inversion sequences of length $n$ avoiding $\mathbf{101}$ and $\mathbf{201}$.
\end{conj}

We have confirmed that the classes are indeed equinumerous for $n\leqslant36$.

The conjectures of van Hoeij and Savage appear to be unrelated, but there is in fact a possible connection.
It turns out that the Ap\'ery numbers enumerate sets of lattice points in a certain sequence of lattices, and
inversion sequences can also be interpreted from the perspective of lattice point enumeration.

For each $d\geqslant0$, the $A_d$ lattice consists of the set of those points in $\mathbb{Z}^{d+1}$ whose coordinate sum is zero.
Given a lattice $L$,
the 
\emph{crystal ball} numbers, $G_L(r)$, for $L$, denote 
the number of points in $L$ that are within $r$ steps of the origin.
Crystal ball numbers for the $A_d$ lattices were first determined, independently, by Conway \& Sloane~\cite{CS1997} and Baake \& Grimm~\cite{BG1997}.
The Ap\'ery numbers $a_n$ are the ``diagonal'' crystal ball numbers for the $A_d$ lattices: for each $n$, they
enumerate the number of points in $A_n$ that are within $n$ steps of the origin. In symbols,
$$
a_n \;=\;  G_{A_n}(n) .
$$

On the other hand,
inversion sequences of length $n$ can be viewed as lattice points in a (half-open) $1 \times 2 \times \ldots \times n$ box in the cubic lattice $\mathbb{Z}^n$.
Indeed, Savage and her co-authors make much use of this lattice point enumeration perspective in their work (for example, see~\cite{BS2010}). 
Perhaps there is some way of leveraging these connections to lattices to prove both conjectures.

\cleardoublepage


\chapter{Avoiding 1234 and 2341}
\label{chapF}

In this chapter and the next, we enumerate two classes whose bases consist of two permutations of length 4.
The first of these is $\av(\mathbf{1234},\mathbf{2341})$.
We use our Hasse graph approach to prove the following theorem:

\thmbox{
\begin{thm}\label{thmF}
The class of permutations avoiding $\mathbf{1234}$ and $\mathbf{2341}$ has the algebraic generating function
$$
\frac
{2-10\+z+9\+z^2+7\+z^3-4\+z^4 \:-\: (2-8\+z+9\+z^2-3\+z^3)\+\sqrt{1-4\+z}}
{(1-3\+z+z^2)\+\big((1-5\+z+4\+z^2) \:+\: (1-3\+z)\+\sqrt{1-4\+z}\big)} .
$$
Its growth rate is equal to 4.
\end{thm}
}

In doing so, we use the kernel method six times to solve the relevant
functional equations. 

Let us use $\FFF$ to denote $\av(\mathbf{1234},\mathbf{2341})$.
The structure of class $\FFF$ depends critically on the presence or absence of occurrences of the pattern $\mathbf{123}$.
In light of this,
to enumerate this class, we partition it into three sets $\AAA$, $\BBB$ and $\CCC$ as follows:
\begin{bullets}
  \item $\AAA=\av(\mathbf{123})$.
  \item $\BBB=\av(\mathbf{1234},\mathbf{2341},\mathbf{13524},\mathbf{14523})\setminus\AAA$. Every permutation in $\BBB$ contains at least one occurrence of a $\mathbf{123}$, but avoids $\mathbf{13524}$ and $\mathbf{14523}$.
  \item $\CCC=\av(\mathbf{1234},\mathbf{2341})\setminus(\AAA\cup\BBB)$. Every permutation in $\CCC$ contains a $\mathbf{13524}$ or a $\mathbf{14523}$.
\end{bullets}
We refer to a permutation in $\AAA$ as an \prefix{$\AAA$}permutation, and similarly for $\BBB$ and $\CCC$.

\begin{figure}[t] 
  $$
  \begin{tikzpicture}[scale=0.225,line join=round]
    \draw [red,thin] (3,11)--(5,16);
    \draw [red,thin] (5,16)--(4,4)--(7,14);
    \draw [red,thin] (12,13)--(8,8)--(13,12)--(9,6)--(12,13);
    \draw [red,thin] (6,1)--(7,14);
    \draw [red,thin] (8,8)--(6,1)--(9,6);
    \draw [red,thin] (10,3)--(12,13)--(11,2)--(13,12)--(10,3)--(14,10)--(11,2)--(15,9)--(10,3)--(16,7)--(11,2)--(17,5)--(10,3);
    \draw [blue,very thick] (2,17)--(1,15)--(5,16);
    \draw [blue,very thick] (7,14)--(3,11);
    \draw [blue,very thick] (12,13)--(3,11)--(13,12);
    \draw [blue,very thick] (16,7)--(9,6)--(4,4)--(8,8)--(14,10)--(9,6)--(15,9)--(8,8);
    \draw [blue,very thick] (4,4)--(17,5);
    \draw [blue,very thick] (10,3)--(6,1)--(11,2);
    \plotpermnobox{}{15,17,11, 4,16, 1,14, 8, 6, 3, 2,13,12,10, 9, 7, 5}
    \draw [thin] (1,15)  circle [radius=0.4];
    \draw [thin] (3,11) circle [radius=0.4];
    \draw [thin] (4,4) circle [radius=0.4];
    \draw [thin] (6,1) circle [radius=0.4];
  \end{tikzpicture}
  $$
  \caption{
           A permutation in class $\FFF$, spanned by four source graphs} 
  \label{figFPerm}
\end{figure}
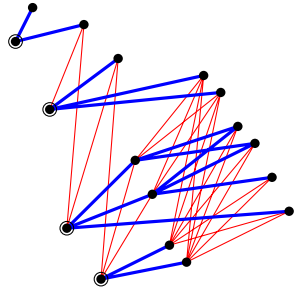
The addition of a source graph to a \prefix{$\CCC$}permutation can only yield another \prefix{$\CCC$}permutation (since it can't cause the removal a $\mathbf{13524}$ or $\mathbf{14523}$ pattern).
Similarly, the addition of a source graph to a \prefix{$\BBB$}permutation can't result in an \prefix{$\AAA$}permutation.
Hence, we can enumerate $\AAA$ without first considering $\BBB$ or $\CCC$, and can enumerate $\BBB$ before considering~$\CCC$.

Before investigating the structure of permutations in $\AAA$, $\BBB$ and $\CCC$, let us briefly examine what a typical source graph in $\FFF$ looks like. Firstly, the avoidance of $\mathbf{1234}$ means that the non-root vertices of any source graph form a \prefix{$\mathbf{123}$}avoider. 
Secondly,
the avoidance of $\mathbf{2341}$ presents no additional restriction on the structure of a source graph,
because the presence of a $\mathbf{2341}$ would
imply the presence of a $\mathbf{123}$ in the non-root vertices.
Thus a source graph in $\FFF$ consists of a root together with a \prefix{$\mathbf{123}$}avoider formed from the non-root vertices.

\subsubsection*{The structure of set $\AAA$}
\vspace{-6pt}

We begin by looking at $\AAA=\av(\mathbf{123})$.
As is very well known, this class is enumerated by the Catalan numbers.
However, we need to keep track of the structure of the bottom subgraph. So we must determine the appropriate
bivariate generating function.

Let $\AAA_\ssS$ denote the set of source graphs in set $\AAA$.
Now, each member of $\AAA_\ssS$ is simply a \emph{fan}, a root vertex connected to a (possibly empty) sequence of pendant edges.
Bottom subgraphs are also fans.
Thus source graphs and bottom subgraphs of $\AAA$ are acyclic.

When enumerating $\AAA$, we use the variable $u$ to mark the \emph{number of leaves} (non-root vertices) in the bottom subtree.
The generating function for $\AAA_\ssS$ is thus given by
$$
A_\ssS(u) \;=\; z+z^2\+u+z^3\+u^2+\ldots \;=\; \frac{z}{1-z\+u} .
$$

We now consider the process of building an \prefix{$\AAA$}permutation from a sequence of source trees.
When a source tree is added to an \prefix{$\AAA$}permutation,
the root vertex of the source tree may be inserted to the left of zero or more of the leaves of the bottom subtree.
See Figure~\ref{figABuild} for an illustration.
Note that, in this and other similar figures, the original bottom subgraph is displayed to the upper left, with the new source graph to the lower right.

\newpage 
The action of adding a source tree is thus seen to be reflected by the linear operator $\oper_{\ssA}$ whose effect on $u^k$ is given by
$$
\oper_{\ssA}\big[u^k\big] \;=\; A_\ssS(u)\+(1+u+\ldots+u^k) \;=\; A_\ssS(u)\+\frac{1-u^{k+1}}{1-u}.
$$
Hence,
the bivariate generating function $A(u)$ for $\AAA$ is defined by the following recursive functional equation:
\begin{equation*}
   A(u)  \;=\;  A_\ssS(u) \:+\: A_\ssS(u)\+\frac{A(1)-u\+A(u)}{1-u}.
\end{equation*}
\begin{figure}[t] 
  $$
  \begin{tikzpicture}[scale=0.225,line join=round]
    \draw [very thick] (2,11)--(1,6)--(3,10);
    \draw [very thick] (5,9)--(1,6)--(6,8);
    \draw [very thick] (7,7)--(1,6);
    \draw [blue,very thick] (8,5)--(4,1)--(9,4);
    \draw [blue,very thick] (10,3)--(4,1)--(11,2);
    \draw [red,thin] (5,9)--(4,1)--(6,8);
    \draw [red,thin] (7,7)--(4,1);
    \plotpermnobox{}{6,11,10, 1, 9, 8, 7, 5, 4, 3, 2}
    \draw [thin] (1,6)  circle [radius=0.4];
    \draw [thin] (4,1) circle [radius=0.4];
  \end{tikzpicture}
  $$
  \caption{Adding a source tree to the bottom subtree of an \prefix{$\AAA$}permutation}
  \label{figABuild}
\end{figure}
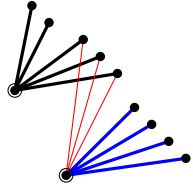This equation can be solved using the {kernel method}.
To start, we express $A(u)$ in terms of $A(1)$,
by expanding and rearranging to give
\begin{equation}\label{eqAKernel}
A(u)  \;=\;   \frac{z \+\big(1-u+A(1)\big)}{1-u+z\+u^2}.
\end{equation}
Now, cancelling the kernel by setting $u=(1-\sqrt{1-4\+z})/2\+z$ yields
the univariate generating function for $\AAA$,
$$
A(1) \;=\; \frac{1-\sqrt{1-4\+ z}}{2\+ z}-1.
$$
This is the generating function for the Catalan numbers as expected.

Finally, by substituting for $A(1)$ back into~\eqref{eqAKernel} we
get the following explicit algebraic expression for $A(u)$:
$$
A(u)  \;=\; \frac{1-2\+z\+u-\sqrt{1-4\+z}}{2\+(1-u+z\+u^2)} .
$$

\subsubsection*{The structure of set $\BBB$}
\vspace{-6pt}

We now consider set $\BBB$.
Recall that sets $\BBB$ and $\CCC$ consist of those permutations in class $\FFF$ that contain at least one occurrence of a $\mathbf{123}$.
We need to keep track of the position of the
leftmost occurrence of a $\mathbf{3}$ in such a pattern.
Given a permutation in $\BBB$ or $\CCC$, let us call the vertex corresponding to the leftmost $\mathbf{3}$ in a $\mathbf{123}$ the \emph{spike}.
In the figures, the spike is marked with a star.

We now make a key observation.\label{obsKey}
When adding a source graph to a permutation containing a $\mathbf{123}$, no vertex of the source graph may be positioned to the right of the spike, or else a $\mathbf{2341}$ would be created.
Hence, the spike in any permutation in classes $\BBB$ or $\CCC$ occurs in its bottom subgraph.
When enumerating sets $\BBB$ and $\CCC$, we use the variable $u$ to mark \emph{the
number of vertices to the left of the spike} in its bottom subgraph.

\begin{figure}[ht]
  $$
  \begin{tikzpicture}[scale=0.225,line join=round]
    \draw [blue,very thick] (2,15)--(1,1)--(3,14);
    \draw [blue,very thick] (4,13)--(1,1)--(5,10);
    \draw [blue,very thick] (6,9)--(1,1)--(7,7);
    \draw [blue,very thick] (8,4)--(1,1);
    \draw [blue,very thick] (14,3)--(1,1)--(15,2);
    \draw [blue,very thick] (9,12)--(5,10)--(10,11)--(6,9)--(9,12)--(7,7)--(10,11)--(8,4)--(9,12);
    \draw [blue,very thick] (7,7)--(11,8)--(8,4);
    \draw [blue,very thick] (12,6)--(8,4)--(13,5);
    \plotpermnobox{}{1,15,14,13,10,9,7,4,12,11, 8, 6, 5, 3, 2}
    \draw [thin] (1,1)  circle [radius=0.4];
    \node[above] at (9,11.6) {${}^\bigstar$};
  \end{tikzpicture}
  $$
  \caption{A source graph in set $\BBB$} 
  \label{figBSource}
\end{figure}
Let $\BBB_\ssS$ be the set of source graphs in set $\BBB$.
Since \prefix{$\BBB$}permutations contain a $\mathbf{123}$ but
avoid $\mathbf{13524}$ and $\mathbf{14523}$,
the non-root vertices of a permutation in $\BBB_\ssS$ consist of two descending sequences, the second sequence beginning (with the spike) above the last vertex in the first sequence.
See Figure~\ref{figBSource} for an illustration.
If we consider the non-root vertices in order from top to bottom, then it can be seen that
$\BBB_\ssS$ is defined by the structural equation
$$
\BBB_\ssS \;=\; u\+\ZZZ \:\times\: \seq{u\+\ZZZ} \:\times\: \ZZZ \:\times\: \seq{u\+\ZZZ+\ZZZ} \:\times\: u\+\ZZZ \:\times\: \seq{\ZZZ}.
$$
The first term on the right corresponds to the root and the remaining terms deal with the non-root vertices in order from top to bottom, vertices to the left of the spike being marked with $u$.
The third term corresponds to the spike and the fifth represents the lowest point to the left of the spike (the rightmost $\mathbf{2}$ of a $\mathbf{123}$).
Hence, the generating function for $\BBB_\ssS$ is
$$
B_\ssS(u) \;=\; \frac{z^3\+u^2}{(1-z)\+ (1-z\+u)\+ (1-z-z\+u)} .
$$

We now study the process of building a \prefix{$\BBB$}permutation from a sequence of source graphs.
There are two cases. A permutation in $\BBB$ may result either from the addition of a source graph to an \prefix{$\AAA$}permutation, or else
from adding a source graph to another \prefix{$\BBB$}permutation. We address these two cases in turn.

\begin{figure}[ht]
  $$
  \begin{tikzpicture}[scale=0.225,line join=round]
    \draw [very thick] (2,17)--(1,12)--(4,16);
    \draw [very thick] (5,15)--(1,12)--(6,14);
    \draw [very thick] (7,13)--(1,12);
    \draw [blue,very thick] (8,11)--(3,2)--(9,9);
    \draw [blue,very thick] (10,8)--(3,2)--(11,5);
    \draw [blue,very thick] (12,4)--(3,2);
    \draw [blue,very thick] (16,3)--(3,2);
    \draw [blue,very thick] (9,9)--(13,10)--(10,8);
    \draw [blue,very thick] (11,5)--(13,10)--(12,4)--(14,7)--(11,5)--(15,6)--(12,4);
    \draw [red,thin] (5,15)--(3,2)--(4,16);
    \draw [red,thin] (7,13)--(3,2)--(6,14);
    \plotpermnobox{}{12,17, 2,16,15,14,13,11, 9, 8, 5, 4,10, 7, 6, 3}
    \draw [thin] (1,12)  circle [radius=0.4];
    \draw [thin] (3,2) circle [radius=0.4];
    \node[above] at (13,9.6) {${}^\bigstar$};
  \end{tikzpicture}
  \qquad\qquad
  \begin{tikzpicture}[scale=0.225,line join=round]
    \draw [very thick] (2,17)--(1,12)--(4,16);
    \draw [very thick] (5,15)--(1,12)--(11,14);
    \draw [very thick] (12,13)--(1,12);
    \draw [blue,very thick] (6,11)--(3,2)--(7,10);
    \draw [blue,very thick] (8,9)--(3,2)--(9,8);
    \draw [blue,very thick] (10,7)--(3,2)--(13,6);
    \draw [blue,very thick] (14,5)--(3,2)--(15,4);
    \draw [blue,very thick] (16,3)--(3,2);
    \draw [red,thin] (5,15)--(3,2)--(4,16);
    \draw [red,thin] (11,14)--(6,11)--(12,13)--(7,10)--(11,14)--(8,9)--(12,13)--(9,8)--(11,14)--(10,7)--(12,13);
    \plotpermnobox{}{12,17, 2,16,15,11,10, 9, 8, 7,14,13, 6, 5, 4, 3}
    \draw [thin] (1,12)  circle [radius=0.4];
    \draw [thin] (3,2) circle [radius=0.4];
    \node[above] at (11,13.6) {${}^\bigstar$};
  \end{tikzpicture}
  \qquad\qquad
  \begin{tikzpicture}[scale=0.225,line join=round]
    \draw [very thick] (2,17)--(1,12)--(4,16);
    \draw [very thick] (5,15)--(1,12)--(11,14);
    \draw [very thick] (12,13)--(1,12);
    \draw [blue,very thick] (6,11)--(3,2)--(7,9);
    \draw [blue,very thick] (8,8)--(3,2)--(9,5);
    \draw [blue,very thick] (10,4)--(3,2);
    \draw [blue,very thick] (16,3)--(3,2);
    \draw [blue,very thick] (7,9)--(13,10)--(8,8);
    \draw [blue,very thick] (9,5)--(13,10)--(10,4)--(14,7)--(9,5)--(15,6)--(10,4);
    \draw [red,thin] (5,15)--(3,2)--(4,16);
    \draw [red,thin] (11,14)--(6,11)--(12,13)--(7,9)--(11,14)--(8,8)--(12,13)--(9,5)--(11,14)--(10,4)--(12,13);
    \plotpermnobox{}{12,17, 2,16,15,11, 9, 8, 5, 4,14,13,10, 7, 6, 3}
    \draw [thin] (1,12)  circle [radius=0.4];
    \draw [thin] (3,2) circle [radius=0.4];
    \node[above] at (11,13.6) {${}^\bigstar$};
  \end{tikzpicture}
  $$
  \caption{Ways to create a \prefix{$\BBB$}permutation by adding a source graph to the bottom subtree of an \prefix{$\AAA$}permutation}
  \label{figABBuild}
\end{figure}
One way to create a \prefix{$\BBB$}permutation from an \prefix{$\AAA$}permutation is to
add a source graph from $\BBB_\ssS$, positioning its root to the left of zero or more of the leaves of the bottom subtree
of the \prefix{$\AAA$}permutation
and its non-root vertices to the right of the bottom subtree.
In this case, the new permutation inherits its spike from the added source graph.
This is illustrated in the left diagram in Figure~\ref{figABBuild}.
The generating function for this set of permutations is thus given by
$$
B_{\textsf{AB1}}(u) \;=\; B_\ssS(u)\+\frac{A(1)-u\+A(u)}{1-u}
.
$$
For simplicity, we choose not to present the expanded form of $B_{\textsf{AB1}}(u)$, or that of most subsequent expressions. They can all be represented in the form $(p+q\+\sqrt{1-4\+z})/r$ for appropriate polynomials $p$, $q$ and $r$.

The other possibility for creating a \prefix{$\BBB$}permutation from an \prefix{$\AAA$}permutation involves the positioning of some non-root vertices of the source graph to the left of some of the leaves in the bottom subtree, making one of the vertices in the original bottom subtree the spike.
The source graph may be drawn from either $\AAA_\ssS$ or $\BBB_\ssS$, as illustrated in the centre and right diagrams in Figure~\ref{figABBuild}.

In this situation, if the source graph has a spike,
it must be positioned to the right of all leaves in the bottom subtree, or else a $\mathbf{1234}$ would be created.
Furthermore, any source graph vertices placed to the left of leaves in the bottom subtree
must occur at the same position in the bottom subtree, or else a $\mathbf{13524}$ would be created.
This position may be chosen independently of where the root vertex is placed.

From these considerations, it can be determined that
the resulting set of permutations has the generating function defined by
$$
B_{\textsf{AB2}}(u) \;=\; \Big(B_\ssS(u)+\frac{z^2\+u^2}{(1-z)\+(1-z\+u)}\Big) \+ \frac{1}{1-u} \+ \Big(A\!'(1)-\frac{u}{1-u}\+\big(A(1)-A(u)\big)\Big) ,
$$
where the presence of the derivative $A\!'$ is a consequence of the independent choice of two positions in the bottom subtree.


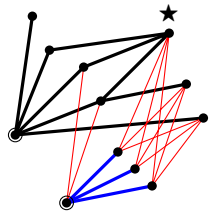
\begin{figure}[ht]
  $$
  \begin{tikzpicture}[scale=0.225,line join=round]
    \draw [very thick] (2,12)--(1,5)--(3,10)--(10,11);
    \draw [very thick] (5,9)--(1,5)--(6,7);
    \draw [very thick] (1,5)--(12,6);
    \draw [very thick] (11,8)--(6,7)--(10,11)--(5,9);
    \draw [blue,very thick] (7,4)--(4,1)--(8,3);
    \draw [blue,very thick] (9,2)--(4,1);
    \draw [red,thin] (5,9)--(4,1)--(6,7);
    \draw [red,thin] (10,11)--(7,4)--(11,8)--(8,3)--(10,11)--(9,2)--(11,8);
    \draw [red,thin] (7,4)--(12,6)--(8,3);
    \draw [red,thin] (9,2)--(12,6);
    \plotpermnobox{}{5,12,10, 1, 9, 7, 4, 3, 2,11, 8, 6}
    \draw [thin] (1,5)  circle [radius=0.4];
    \draw [thin] (4,1) circle [radius=0.4];
    \node[above] at (10,10.6) {${}^\bigstar$};
  \end{tikzpicture}
  $$
  \caption{Adding a source tree to the bottom subgraph of a \prefix{$\BBB$}permutation}
  \label{figBBBuild}
\end{figure}
Finally, we consider the addition of a source graph to a \prefix{$\BBB$}permutation.
As we noted in our key observation on page~\pageref{obsKey},
no vertex of the source graph may be positioned to the right of the spike in the bottom subgraph.
As a result, the new source graph may not contain a $\mathbf{123}$ or else a $\mathbf{1234}$ would be created, so the source graph must be a member of $\AAA_\ssS$ (a fan).
Moreover, the leaves of the source tree must be positioned \emph{immediately} to the left of the spike, or else a $\mathbf{1234}$ would be created.
See Figure~\ref{figBBBuild} for an illustration.

Note that, as a consequence of these restrictions, it is impossible for the addition of a source graph to a \prefix{$\BBB$}permutation to create a $\mathbf{13524}$ or $\mathbf{14523}$. So it is not possible to extend a \prefix{$\BBB$}permutation so as to create a \prefix{$\CCC$}permutation.

Thus
the bivariate generating function $B(u)$ of set $\BBB$ is defined by the following recursive functional equation:
\begin{equation*}
   B(u)  \;=\;  B_\ssS(u) + B_{\textsf{AB1}}(u) + B_{\textsf{AB2}}(u) \:+\: \frac{z\+u}{1-z\+u}\+\frac{B(1)-B(u)}{1-u} ,
\end{equation*}
where the final term reflects the addition of a source tree to a \prefix{$\BBB$}permutation.

This equation is amenable to the kernel method.
It can be rearranged to express $B(u)$ in terms of $B(1)$. The kernel can then be cancelled by setting $u=(1-\sqrt{1-4\+z})/2\+z$, which yields an expression for $B(1)$:
$$
B(1)
\;=\;
\frac{-1+8\+z-19\+z^2+12\+z^3 \:+\: (1-6\+z+9\+z^2-2\+z^3)\+\sqrt{1-4 z}}{2\+z^3\+(1-4\+z)}.
$$
Substitution then
results in an explicit algebraic expression for $B(u)$, which we refrain from presenting explicitly due to its size.

\vspace{9pt}
\begin{figure}[ht]
  $$
  \begin{tikzpicture}[scale=0.225,line join=round]
    \draw [blue,very thick] (2,15)--(1,1)--(3,14);
    \draw [blue,very thick] (1,1)--(4,10)--(6,13)--(5,9);
    \draw [blue,very thick] (5,9)--(1,1)--(7,7);
    \draw [blue,very thick] (8,4)--(1,1);
    \draw [blue,very thick] (12,8)--(11,3)--(1,1)--(15,2);
    \draw [blue,very thick] (9,12)--(4,10)--(10,11)--(5,9)--(9,12)--(7,7)--(10,11)--(8,4)--(9,12);
    \draw [blue,very thick] (7,7)--(12,8)--(8,4);
    \draw [blue,very thick] (13,6)--(8,4)--(14,5)--(11,3)--(13,6);
    \plotpermnobox{}{1,15,14,10, 9,13, 7, 4,12,11, 3, 8, 6, 5, 2}
    \draw [thin] (1,1)  circle [radius=0.4];
    \node[above] at (6,12.6) {${}^\bigstar$};
  \end{tikzpicture}
  $$
  \caption{A source graph in set $\CCC$}
  \label{figCSource}
\end{figure}
\subsubsection*{The structure of set $\CCC$}
\vspace{-6pt}

We begin our enumeration of $\CCC$ by counting its set of source graphs, which we denote~$\CCC_\ssS$.
Rather than doing this directly,
we enumerate all the source graphs that contain a $\mathbf{123}$ (i.e.~those in either $\BBB_\ssS$ or $\CCC_\ssS$) and then subtract those in $\BBB_\ssS$.
To begin, we consider how we might build an \emph{arbitrary} source graph in class $\FFF$ by adding vertices from left to right.

Suppose we have a partly formed source graph with at least one non-root vertex, whose rightmost vertex is $v$,
and we want to add further vertices to its right.
What are the options?
If $v$ is not the lowest of the non-root vertices, then any subsequent vertices must be placed lower than~$v$.
The only other restriction is that vertices must be positioned higher than the root.
If we use $y$ to mark the number of positions in which a vertex may be inserted, then the action of adding a new vertex can be seen to be reflected by the following linear operator: 
\begin{equation*}
  \oper_{\textsf{\L}}\big[f(y)\big] \;=\; z\+y^2\+\frac{f(1)-f(y)}{1-y}.
\end{equation*}
We choose to denote this operator $\oper_{\textsf{\L}}$ because it corresponds to the action used in building a {\L}uka\-sie\-wicz path.

Now let us consider source graphs 
whose rightmost vertex is a spike.
These are in~$\BBB_\ssS$, so let's call this set $\BBB_\textsf{S0}$.
As usual, we mark with $u$ those vertices to the left of the spike.
If, in addition, we mark with $y$ those vertices not above the spike, then $\BBB_\textsf{S0}$ is defined by the structural equation
$$
\BBB_\textsf{S0} \;=\; u\+\ZZZ \:\times\: \seq{u\+\ZZZ} \:\times\: \seqplus{u\+y\+\ZZZ} \:\times\: y\+\ZZZ .
$$
It is readily seen that $y$ correctly marks the number of positions in which an additional vertex may be inserted to the right.

Let $\DDD_\ssS=\BBB_\ssS\cup\CCC_\ssS$. Since every member of $\DDD_\ssS$ is built from an element of $\BBB_\textsf{S0}$ by applying $\oper_{\textsf{\L}}$ zero or more times, it follows that the generating function for $\DDD_\ssS$ is defined by the recursive functional equation
$$
D_\ssS(y) \;=\; \frac{z^3\+y^2\+u^2}{(1-z\+u)\+(1-z\+y\+u)} \:+\: z\+y^2\+\frac{D_\ssS(1)-D_\ssS(y)}{1-y} .
$$
This equation can be solved for $D_\ssS(1)$ by the kernel method,
using $y=(1-\sqrt{1-4\+z})/2\+z$ to
cancel the kernel.
The generating function for $\CCC_\ssS$ 
is then defined by
$$
C_\ssS(u) \;=\; D_\ssS(1) \:-\: B_\ssS(u) .
$$

We now study the process of building a \prefix{$\CCC$}permutation from a sequence of source graphs.
As with set $\BBB$, there are two cases. A permutation in $\CCC$ may result either from the addition of a source graph to an \prefix{$\AAA$}permutation, or else
from adding a source graph to another \prefix{$\CCC$}permutation.
(As we observed above, it is not possible to create a \prefix{$\CCC$}permutation by adding a source graph to a \prefix{$\BBB$}permutation.)
We address the two cases in turn.

\begin{figure}[ht]
  $$
  \begin{tikzpicture}[scale=0.225,line join=round]
    \draw [very thick] (2,19)--(1,14)--(4,18);
    \draw [very thick] (5,17)--(1,14)--(6,16);
    \draw [very thick] (7,15)--(1,14);
    \draw [blue,very thick] (8,13)--(3,1)--(9,10);
    \draw [blue,very thick] (10,9)--(3,1)--(11,8);
    \draw [blue,very thick] (12,7)--(3,1)--(13,5);
    \draw [blue,very thick] (15,4)--(3,1)--(16,3);
    \draw [blue,very thick] (3,1)--(19,2);
    \draw [blue,very thick] (14,12)--(10,9)--(17,11)--(9,10)--(14,12)--(13,5);
    \draw [blue,very thick] (11,8)--(14,12)--(12,7)--(17,11)--(11,8);
    \draw [blue,very thick] (17,11)--(13,5)--(18,6)--(15,4)--(17,11)--(16,3)--(18,6);
    \draw [red,thin] (4,18)--(3,1)--(5,17);
    \draw [red,thin] (6,16)--(3,1)--(7,15);
    \plotpermnobox{}{14,19, 1,18,17,16,15,13,10, 9, 8, 7, 5,12, 4, 3,11, 6, 2}
    \draw [thin] (1,14)  circle [radius=0.4];
    \draw [thin] (3,1) circle [radius=0.4];
    \node[above] at (14,11.6) {${}^\bigstar$};
  \end{tikzpicture}
  \qquad\qquad\qquad\quad
  \begin{tikzpicture}[scale=0.225,line join=round]
    \draw [very thick] (2,19)--(1,14)--(4,18);
    \draw [very thick] (8,17)--(1,14)--(10,16);
    \draw [very thick] (13,15)--(1,14);
    \draw [blue,very thick] (5,13)--(3,1)--(6,10);
    \draw [blue,very thick] (7,9)--(3,1)--(9,8);
    \draw [blue,very thick] (11,7)--(3,1)--(12,5);
    \draw [blue,very thick] (15,4)--(3,1)--(16,3);
    \draw [blue,very thick] (3,1)--(19,2);
    \draw [blue,very thick] (14,12)--(7,9)--(17,11)--(6,10)--(14,12)--(12,5);
    \draw [blue,very thick] (9,8)--(14,12)--(11,7)--(17,11)--(9,8);
    \draw [blue,very thick] (17,11)--(12,5)--(18,6)--(15,4)--(17,11)--(16,3)--(18,6);
    \draw [red,thin] (4,18)--(3,1);
    \draw [red,thin] (8,17)--(5,13)--(10,16)--(6,10)--(8,17)--(7,9)--(10,16)--(9,8)--(13,15);
    \draw [red,thin] (5,13)--(13,15)--(6,10);
    \draw [red,thin] (7,9)--(13,15)--(11,7);
    \draw [red,thin] (12,5)--(13,15);
    \plotpermnobox{}{14,19, 1,18,13,10, 9,17, 8,16, 7, 5,15,12, 4, 3,11, 6, 2}
    \draw [thin] (1,14)  circle [radius=0.4];
    \draw [thin] (3,1) circle [radius=0.4];
    \node[above] at (8,16.6) {${}^\bigstar$};
  \end{tikzpicture}
  $$
  \caption{Ways to create a \prefix{$\CCC$}permutation by adding a source graph to the bottom subtree of an \prefix{$\AAA$}permutation}
  \label{figACBuild}
\end{figure}
One way to create a \prefix{$\CCC$}permutation from an \prefix{$\AAA$}permutation is to
add a source graph from $\CCC_\ssS$, positioning its root to the left of zero or more of the leaves of the bottom subtree
of the \prefix{$\AAA$}permutation
and its non-root vertices to the right of the bottom subtree.
This is illustrated in the left diagram in Figure~\ref{figACBuild}.
The generating function for this set of permutations is thus
$$
C_{\textsf{AC1}}(u) \;=\; C_\ssS(u)\+\frac{A(1)-u\+A(u)}{1-u} .
$$

The other method for creating a \prefix{$\CCC$}permutation from an \prefix{$\AAA$}permutation involves the positioning of some non-root vertices of the source graph to the left of some of the leaves in the bottom subtree.
This is illustrated in the right diagram in Figure~\ref{figACBuild}.
In analysing this method, it is more convenient to look, more generally, at how an \prefix{$\AAA$}permutation can be extended to yield a permutation containing a $\mathbf{123}$, in either $\BBB$ or $\CCC$. We can then subtract those members of $\BBB$ that are enumerated by $B_{\textsf{AB2}}$.

We achieve the enumeration by adding vertices from left to right in four steps:
\begin{bulletnums}
\item The first step adds the root.
\item The second step adds the first non-root vertex, which determines the position of the new spike, and also any other vertices positioned to the left of the spike.
\item The third step adds any additional vertices to the right of the spike but to the left of some other leaves in the bottom subtree. The addition of such vertices creates occurrences of $\mathbf{13524}$.
\item Finally, the fourth step adds any vertices to the right of the bottom subtree.
\end{bulletnums}

Step 1: Permutations that result from the addition of the root vertex are enumerated by
$$
D_\textsf{1}(u) \;=\; z\+u\+\frac{A(1)-A(u)}{1-u} .
$$
Step 2: In this step, we insert the descending sequence of vertices that creates the new spike.
In the generating function for permutations resulting from this action, we introduce two additional catalytic variables that we require for steps 3 and 4.
For use in step 3, $v$ marks the number of source tree leaves to the right of the new spike.
For step 4, we use
$y$ to mark valid positions for the insertion of subsequent vertices, as we did previously.
The generating function is 
$$
D_\textsf{2}(v) \;=\; \frac{z\+y^2\+u^2}{1-z\+y\+u}\+\frac{D_\textsf{1}(v)-D_\textsf{1}(u)}{v-u} .
$$
Step 3: The effect of adding additional vertices to the right of the spike but to the left of some
other leaves in the bottom subtree is represented by the recursive functional equation
$$
D_\textsf{3}(y,v) \;=\; D_\textsf{2}(v) \:+\: z\+y\+v\+\frac{D_\textsf{3}(y,1)-D_\textsf{3}(y,v)}{1-v} .
$$
Again, the kernel method can be used to solve this for $D_\textsf{3}(y,1)$, the kernel being cancelled by setting $v=1/(1-z\+y)$.

Step 4: Finally, the addition of vertices to the right of the bottom subtree is reflected by the {\L}uka\-sie\-wicz operator $\oper_{\textsf{\L}}$, giving rise to the recursive functional equation
$$
D_\textsf{4}(y) \;=\; D_\textsf{3}(y,1) \:+\: z\+y^2\+\frac{D_\textsf{4}(1)-D_\textsf{4}(y)}{1-y} ,
$$
which can be solved for $D_\textsf{4}(1)$ by cancelling the kernel with $y=(1-\sqrt{1-4\+z})/2\+z$.

The generating function for the set of permutations resulting from the second way of creating a \prefix{$\CCC$}permutation from an \prefix{$\AAA$}permutation
is then defined by
$$
C_{\textsf{AC2}}(u) \;=\; D_\textsf{4}(1) - B_{\textsf{AB2}}(u) .
$$

Our work is almost complete. We only have to consider how a source graph may be added to a \prefix{$\CCC$}permutation.
In fact, the situation is extremely constrained.
First, as noted earlier, the source graph must be positioned to the left of the spike.
Furthermore, the presence of a $\mathbf{13524}$ or $\mathbf{14523}$ means that the addition of a source graph with even a single non-root vertex would create a $\mathbf{1234}$.
So the only possibility is the addition of a trivial (single vertex) source tree.
Thus the bivariate generating function $C(u)$ of set $\CCC$ is defined by the following recursive functional equation:
$$
C(u)  \;=\;  C_\ssS(u) + C_{\textsf{AC1}}(u) + C_{\textsf{AC2}}(u) \:+\: z\+u\+\frac{C(1)-C(u)}{1-u}.
$$
where the final term reflects the addition of a trivial source tree to a \prefix{$\CCC$}permutation.
This equation can be solved to yield the following expression for $C(1)$ by 
a sixth and final application of the kernel method,
cancelling the kernel by setting $u=1/(1-z)$:
$$
\frac{-1+10\+z-35\+z^2+52\+z^3-35\+z^4+12\+z^5 \:+\: (1-8\+z+21\+z^2-22\+z^3+11\+z^4-2\+z^5)\sqrt{1-4\+z}}{2\+z^3\+(1-4\+z)\+(1-3\+z+z^2)} .
$$

We now have all we need to prove Theorem~\ref{thmF} by
obtaining an explicit expression for the generating function
that enumerates class
$\FFF$.
Since $\FFF$ is the disjoint union of $\AAA$, $\BBB$ and $\CCC$,
its generating function
is
given by
$A(1)+B(1)+C(1)$.
Thus, by appropriate expansion and simplification,
the generating function for
$\av(\mathbf{1234},\mathbf{2341})$
can be shown to be equal to
$$
\frac
{2-10\+z+9\+z^2+7\+z^3-4\+z^4 \:-\: (2-8\+z+9\+z^2-3\+z^3)\+\sqrt{1-4\+z}}
{(1-3\+z+z^2)\+\big((1-5\+z+4\+z^2) \:+\: (1-3\+z)\+\sqrt{1-4\+z}\big)}.
$$
This has singularities at $z=\frac{1}{4}$, $z=\frac{1}{2}\+(3-\sqrt{5})$ and $z=\frac{1}{2}\+(3+\sqrt{5})$. Hence, the growth rate of $\av(\mathbf{1234},\mathbf{2341})$ is equal to 4, the reciprocal of the least of these.

The first twelve terms of the sequence $|\FFF_n|$ are 1, 2, 6, 22, 89, 376, 1611, 6901, 29375, 123996, 518971, 2155145.
More values
can be found at
\href{http://oeis.org/A165540}{A165540} in OEIS~\cite{OEIS}.

\cleardoublepage


\chapter{Avoiding 1243 and 2314}
\label{chapE}

In this chapter, we analyse the structure of the Hasse graphs of permutations in the class $\av(\mathbf{1243},\mathbf{2314})$ and prove the following theorem:

\thmbox{
\begin{thm}\label{thmE}
The class of permutations avoiding $\mathbf{1243}$ and $\mathbf{2314}$ has an algebraic generating function $F(z)$ which satisfies the cubic polynomial equation
$$
      (z-3\+z^2+2\+z^3)
\:-\: (1-5\+z+8\+z^2-5\+z^3)\+F(z)
\:+\: (2\+z-5\+z^2+4\+z^3)\+F(z)^2
\:+\: z^3\+F(z)^3
\;=\; 0
.
$$
Its growth rate is approximately 5.1955, the greatest real root of the quintic polynomial
$$
2 - 41\+z + 101\+z^2 - 97\+z^3 + 36\+z^4 -4\+z^5 .
$$
\end{thm}
}

The proof requires an unusual
simultaneous double application of the kernel method.

\vspace{9pt}
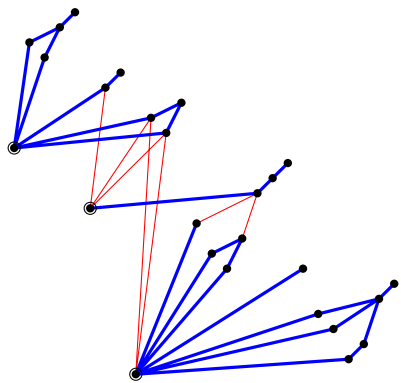
\begin{figure}[ht] 
  $$
  \begin{tikzpicture}[scale=0.20,line join=round]
    \draw [red,thin] (6,15)--(7,23);
    \draw [red,thin] (6,15)--(11,20)--(9,4)--(10,21)--(6,15);
    \draw [red,thin] (13,14)--(17,16)--(16,13);
    \draw [blue,very thick] (8,24)--(7,23)--(1,19)--(2,26)--(4,27)--(3,25)--(1,19);
    \draw [blue,very thick] (4,27)--(5,28);
    \draw [blue,very thick] (1,19)--(10,21)--(12,22)--(11,20)--(1,19);
    \draw [blue,very thick] (6,15)--(17,16)--(19,18);
    \draw [blue,very thick] (13,14)--(9,4)--(14,12)--(16,13)--(15,11)--(9,4);
    \draw [blue,very thick] (26,10)--(25,9)--(21,8)--(9,4)--(22,7)--(25,9)--(24,6)--(23,5)--(9,4)--(20,11);
    \plotpermnobox{}{19,26,25,27,28,15,23,24,4,21,20,22,14,12,11,13,16,17,18,11,8,7,5,6,9,10}
    \draw [thin] (1,19)  circle [radius=0.4];
    \draw [thin] (6,15) circle [radius=0.4];
    \draw [thin] (9,4) circle [radius=0.4];
  \end{tikzpicture}
  $$
  \caption{
           A permutation in class $\EEE$, spanned by three source graphs}
  \label{figEPerm}
\end{figure}

Let us 
use $\EEE$ to denote 
$\av(\mathbf{1243},\mathbf{2314})$.
What can we say about the structure of source graphs in $\EEE$?
Firstly, since $H_\mathbf{1243} =
\raisebox{-2.5pt}{\begin{tikzpicture}[scale=0.12,line join=round]
  \draw[] (1,1)--(2,2);
  \draw[] (3,4)--(2,2)--(4,3);
  \plotpermnobox{}{1,2,4,3}
\end{tikzpicture}}
$
may not occur as a subgraph, only the root of a
source graph may fork towards the upper right.
Secondly, each source graph in $\EEE$ is \emph{plane}.
This is the case because every non-plane graph
contains a
$H_\mathbf{2143} =
\raisebox{-2.5pt}{\begin{tikzpicture}[scale=0.12,line join=round]
  \draw[] (1,2)--(3,4)--(2,1)--(4,3)--(1,2);
  \plotpermnobox{}{2,1,4,3}
\end{tikzpicture}}
$, and, furthermore, any $\mathbf{2143}$ in a source graph occurs as part of a $\mathbf{13254}$ (where the $\mathbf{1}$ is the root of the source graph). But this is impossible in $\EEE$, since $\mathbf{13254}$
does not avoid $\mathbf{1243}$.

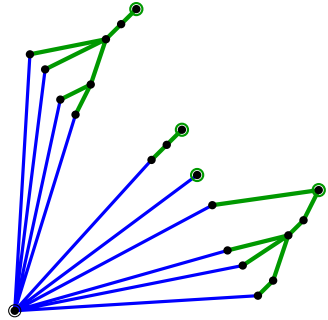
\begin{figure}[ht]
  $$
  \begin{tikzpicture}[scale=0.2,line join=round]
    \draw [blue,very thick] (2,18)--(1,1)--(3,17);
    \draw [blue,very thick] (4,15)--(1,1)--(5,14);
    \draw [blue,very thick] (10,11)--(1,1)--(13,10);
    \draw [blue,very thick] (14,8)--(1,1)--(15,5);
    \draw [blue,very thick] (16,4)--(1,1)--(17,2);
    \draw [ultra thick,black!40!green] (5,14)--(6,16)--(4,15);
    \draw [ultra thick,black!40!green] (6,16)--(7,19)--(3,17);
    \draw [ultra thick,black!40!green] (2,18)--(7,19)--(9,21);
    \draw [ultra thick,black!40!green] (10,11)--(12,13);
    \draw [ultra thick,black!40!green] (17,2)--(18,3)--(19,6)--(16,4);
    \draw [ultra thick,black!40!green] (15,5)--(19,6)--(20,7)--(21,9)--(14,8);
    \plotpermnobox{21}{1,18,17,15,14,16,19,20,21,11,12,13,10,8,5,4,2,3,6,7,9}
    \draw [thin] (1,1)  circle [radius=0.4];
    \draw [thick,black!40!green] (9,21)  circle [radius=0.4];
    \draw [thick,black!40!green] (12,13) circle [radius=0.4];
    \draw [thick,black!40!green] (13,10) circle [radius=0.4];
    \draw [thick,black!40!green] (21,9)  circle [radius=0.4];
  \end{tikzpicture}
  $$
  \caption{A source graph for class $\EEE$, constructed from four u-trees}
  \label{figESource}
\end{figure}
If we combine these two observations, we see that
the non-root vertices of a source graph consist of a sequence of
inverted subtrees whose roots are right-to-left maxima.
The avoidance of $H_\mathbf{2314} =
\raisebox{-2.5pt}{\begin{tikzpicture}[scale=0.12,line join=round]
  \draw[] (1,2)--(2,3)--(4,4)--(3,1);
  \plotpermnobox{}{2,3,1,4}
\end{tikzpicture}}
$
places restrictions on the structure of the subtrees, so that they must consist of a path at the lower right, which we call the \emph{trunk}, with pendant edges attached to its
left.
It is readily seen that these correspond to permutations in
$\av(\mathbf{132},\mathbf{231})$.
We call trees of this form \emph{u-trees}, short for
\emph{unbalanced} trees.
See Figure~\ref{figESource} for an illustration of a source graph constructed from u-trees.

The class $\UUU$ of u-trees satisfies the structural equation
$$
  \UUU \;=\; \ZZZ\times\seq{\ZZZ\times\seq{\ZZZ}}
$$
where the first term on the right represents the lowest leaf at the tip of the trunk and the second represents the remainder of the vertices in the trunk, each with a (possibly empty) sequence of pendant edges attached to the upper left.
Hence the generating function for $\UUU$ is
\begin{equation*}
  U(z) \;=\; \frac{z\+(1-z)}{1-2\+z}.
\end{equation*}
If we use $u$ to mark the number of u-trees, the class $\SSS$ of source graphs
satisfies the structural equation
$$
  \SSS \;=\; \ZZZ\times\seq{u\+\UUU}
$$
and thus has bivariate generating function
\begin{equation*}
  S(u) \;=\; S(z,u) \;=\; \frac{z\+(1-2\+z)}{1-(2+u)\+ z+u\+ z^2}.
\end{equation*}

Let us now examine how a permutation in $\EEE$ can be built from a sequence of source graphs.
Observe that, when a source graph is added,
no vertex of the source graph can be positioned between two vertices of a u-tree in the bottom subgraph, because otherwise
a $\mathbf{2314}$ would be created.
In addition,
there are strong constraints on when u-trees in the new source graph can be positioned to the left of a u-tree in the bottom subgraph.

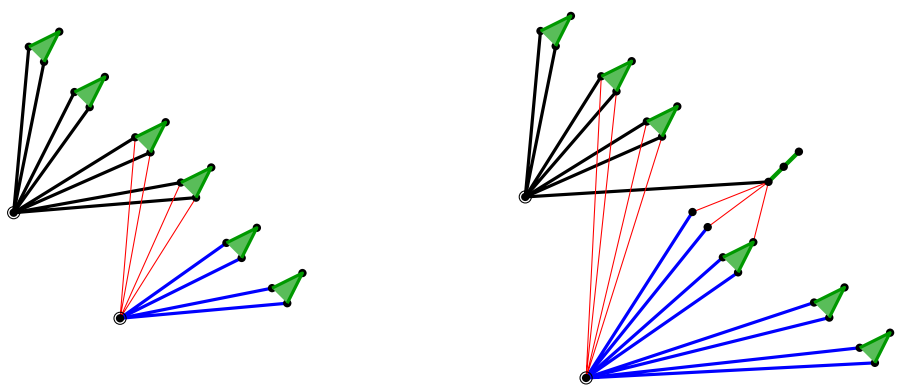
\begin{figure}[ht]
  $$
  \raisebox{22.5pt}{
  \begin{tikzpicture}[scale=0.2,line join=round]
    \draw [very thick] (2,19)--(1,8)--(3,18);
    \draw [very thick] (5,16)--(1,8)--(6,15);
    \draw [very thick] (9,13)--(1,8)--(10,12);
    \draw [very thick] (12,10)--(1,8)--(13,9);
    \draw [blue,very thick] (15,6)--(8,1)--(16,5);
    \draw [blue,very thick] (18,3)--(8,1)--(19,2);
    \draw [red,thin] (9,13)--(8,1)--(10,12);
    \draw [red,thin] (12,10)--(8,1)--(13,9);
      \plotpermnobox{20}{8,19,18,20,16,15,17,1,13,12,14,10,9,11,6,5,7,3,2,4}
    \draw [very thick,black!40!green,fill=black!40!green!65!white] (3,18)--(4,20)--(2,19);
    \draw [very thick,black!40!green,fill=black!40!green!65!white] (6,15)--(7,17)--(5,16);
    \draw [very thick,black!40!green,fill=black!40!green!65!white] (10,12)--(11,14)--(9,13);
    \draw [very thick,black!40!green,fill=black!40!green!65!white] (13,9)--(14,11)--(12,10);
    \draw [very thick,black!40!green,fill=black!40!green!65!white] (16,5)--(17,7)--(15,6);
    \draw [very thick,black!40!green,fill=black!40!green!65!white] (19,2)--(20,4)--(18,3);
    \draw [thin] (1,8)  circle [radius=0.4];
    \draw [thin] (8,1) circle [radius=0.4];
  \end{tikzpicture}
  }
  \qquad\qquad\qquad\quad
  \begin{tikzpicture}[scale=0.2,line join=round]
    \draw [very thick] (2,24)--(1,13)--(3,23);
    \draw [very thick] (6,21)--(1,13)--(7,20);
    \draw [very thick] (9,18)--(1,13)--(10,17);
    \draw [very thick] (1,13)--(17,14);
    \draw [ultra thick,black!40!green] (19,16)--(17,14);
    \draw [blue,very thick] (12,12)--(5,1)--(13,11);
    \draw [blue,very thick] (14,9)--(5,1)--(15,8);
    \draw [blue,very thick] (20,6)--(5,1)--(21,5);
    \draw [blue,very thick] (23,3)--(5,1)--(24,2);
    \draw [red,thin] (6,21)--(5,1)--(7,20);
    \draw [red,thin] (9,18)--(5,1)--(10,17);
    \draw [red,thin] (12,12)--(17,14)--(13,11);
    \draw [red,thin] (16,10)--(17,14);
      \plotpermnobox{25}{13,24,23,25,1,21,20,22,18,17,19,12,11,9,8,10,14,15,16,6,5,7,3,2,4}
    \draw [very thick,black!40!green,fill=black!40!green!65!white] (3,23)--(4,25)--(2,24);
    \draw [very thick,black!40!green,fill=black!40!green!65!white] (7,20)--(8,22)--(6,21);
    \draw [very thick,black!40!green,fill=black!40!green!65!white] (10,17)--(11,19)--(9,18);
    \draw [very thick,black!40!green,fill=black!40!green!65!white] (15,8)--(16,10)--(14,9);
    \draw [very thick,black!40!green,fill=black!40!green!65!white] (21,5)--(22,7)--(20,6);
    \draw [very thick,black!40!green,fill=black!40!green!65!white] (24,2)--(25,4)--(23,3);
    \draw [thin] (1,13)  circle [radius=0.4];
    \draw [thin] (5,1) circle [radius=0.4];
  \end{tikzpicture}
  $$
  \caption{The two methods for adding a source graph in class $\EEE$; u-trees are shown schematically as filled triangles}
  \label{figEBuild}
\end{figure}
These conditions result in there being two distinct ways in which a source graph may be added.
These are illustrated in Figure~\ref{figEBuild}.
In the first method, the root of the source graph is positioned to the left of zero or more u-trees in the bottom subgraph and the u-trees in the source graph are positioned to the right of the bottom subgraph.

The second method is more subtle. It is only applicable if
the rightmost u-tree of the bottom subgraph
is a path.
If that is the case, then an initial sequence of
u-trees in the source graph can be positioned to the left of this path subtree, as long as
each of them, except possibly the last,
consists of a single vertex.
If
the rightmost u-tree of the bottom subgraph
were not a path, then a $\mathbf{1243}$ would be created.
Similarly, if a non-final u-tree consisted of more than one vertex, then a $\mathbf{2314}$ would be created.

In order to handle this second method, we need to keep track of those source graphs in which the rightmost u-tree is a path.
Let $\SSS_\PP$ be the class of such graphs. It satisfies the structural equation
$$
  \SSS_\PP \;=\; \ZZZ\times\seq{u\+\UUU}\times u\+\seqplus{\ZZZ} ,
$$
where $u$ marks the number of u-trees as before.
This class thus has bivariate generating function
\begin{equation*}
  S_\PP(u) \;=\; S_\PP(z,u) \;=\; \frac{u\+ z^2\+ (1-2\+ z)}{(1-z)\+ \big(1-(2+u)\+ z+u\+ z^2\big)}.
\end{equation*}

In order to distinguish between those situations when the second method of adding a source graph is applicable and those when it isn't,
let us use $\PPP$ to denote the set of those permutations in $\EEE$ whose Hasse graphs have bottom subgraphs in which the rightmost u-tree is a path.

We are interested in determining the two bivariate generating functions
$E(u)=E(z,u)$ and
$P(u)=P(z,u)$ for $\EEE$ and $\PPP$ respectively, where $u$ marks the number of u-trees \emph{in the bottom subgraph}.
To do this, we establish four linear operators on these generating functions that reflect the different ways in which a source graph can be added.

The action of adding a source graph using the first method is readily seen to be reflected by the following linear operator:
\begin{equation*}
  \oper_{\EE\EE}\big[f(u)\big] \;=\; S(u)\+\frac{f(1)-u\+f(u)}{1-u}.
\end{equation*}
The first method creates a member of $\PPP$ from an arbitrary element of $\EEE$ whenever
the source graph is in $\SSS_\PP$ (i.e.
its rightmost u-tree is a path).
Thus the appropriate linear operator 
is
\begin{equation*}
  \oper_{\EE\PP}\big[f(u)\big] \;=\; S_\PP(u)\+\frac{f(1)-u\+f(u)}{1-u}.
\end{equation*}

Now let us determine the linear operators corresponding to the second method of adding a source graph.

The set, $\SSS^\star$, of source graphs that can be added using the second method
satisfies the structural equation
$$
  \SSS^\star \;=\; \ZZZ\times\seq{\ZZZ}\times u\+\UUU\times\seq{u\+\UUU},
$$
in which the third term on the right identifies the u-tree which is positioned immediately to the left of the rightmost (path) u-tree in the bottom subgraph.
This specification thus counts multiple times those source graphs that can be added in more than one way due to the presence of a non-empty initial sequence of single-vertex u-trees.
Note also that we don't mark the initial sequence of single-vertex u-trees with $u$.
The generating function for $\SSS^\star$ is
\begin{equation*}
  S^\star(u) \;=\; \frac{u\+ z^2}{1-(2+u)\+ z+u\+ z^2}.
\end{equation*}
The action of adding a source graph using the second method is then seen to be reflected
by the following linear operator:
\begin{equation*}
  \oper_{\PP\EE}\big[f_\PP(u)\big] \;=\; S^\star(u)\+\frac{f_\PP(1)-f_\PP(u)}{1-u} .
\end{equation*}
Finally, let us consider when adding a source graph to an arbitrary member of $\PPP$ creates another permutation in $\PPP$.
The second method creates an element of $\PPP$ if the source graph is in $\SSS_\PP$ and its rightmost (path) u-tree is added to the right of the bottom subgraph.
An element of $\PPP$ is also created if the source graph has a single path u-tree or consists of a single vertex (the root).
Thus the set, $\SSS_\PP^\star$, of source graphs, counted with multiplicity,
that can be added to create an element of $\PPP$
satisfies the structural equation
$$
  \SSS_\PP^\star \;=\; \ZZZ\times\seq{\ZZZ}\times\seqplus{u\+\UUU}\times u\+\seqplus{\ZZZ} \:+\: \ZZZ\times u\+\seq{\ZZZ}.
$$
Its generating function is
\begin{equation*}
  S_\PP^\star(u) \;=\; \frac{u\+ z\+ (1-2\+ z)\+(1-u\+z)}{(1-z)\+ \big(1-(2+u)\+ z+u\+ z^2\big)},
\end{equation*}
and the corresponding linear operator
is
\begin{equation*}
  \oper_{\PP\PP}\big[f_\PP(u)\big] \;=\; S_\PP^\star(u)\+\frac{f_\PP(1)-f_\PP(u)}{1-u}.
\end{equation*}

We are now in a position to derive the generating function 
for $\EEE$ and hence prove Theorem~\ref{thmE}.
From the analysis above, we know that
the bivariate generating function $E(u)=E(z,u)$ of class $\EEE$ is defined by the following pair of mutually recursive functional equations:
\begin{equation*}
  \begin{array}{rclcrcr}
   E(u) & = & S(u)     & \!+\! & \oper_{\EE\EE}\big[E(u)\big] & \!+\! & \oper_{\PP\EE}\big[P(u)\big] \\[3pt]
   P(u) & = & S_\PP(u) & \!+\! & \oper_{\EE\PP}\big[E(u)\big] & \!+\! & \oper_{\PP\PP}\big[P(u)\big]
  \end{array}  .
\end{equation*}
These can be expanded to give the following:
\begin{equation}\label{eqnE1}
  E(u) \;=\; z\+\frac{(1-2\+z)\+\big(1-u+E(1)-u\+E(u)\big) \:+\: u\+z\+\big(P(1)-P(u)\big)}{(1-u)\+ \big(1-(2+u)\+ z+u\+ z^2\big)} ,
\end{equation}
\begin{equation}\label{eqnE2}
  P(u) \;=\; u\+z\+(1-2\+z)\frac{z\+\big(1-u+E(1)-u\+E(u)\big)\:+\: (1-u\+z)\+\big(P(1)-P(u)\big)}{(1-u)\+ (1-z)\+ \big(1-(2+u)\+ z+u\+ z^2\big)} .
\end{equation}
An unusual simultaneous double application of the kernel method can then be used
to yield the
algebraic generating function for class $\EEE$ as follows.

First, we eliminate $P(u)$ from \eqref{eqnE1} and \eqref{eqnE2}, and express $E(u)$
in terms of $E(1)$ and $P(1)$ as a rational function.
Cancelling the resulting kernel,
\begin{equation*}
(1-3\+z+2\+z^2)
\:-\:
(2-7\+z+7\+z^2-z^3)\+u
\:+\:
(1-3\+z+3\+z^2)\+u^2
\:-\:
(z-3\+z^2+3\+z^3)\+u^3 ,
\end{equation*}
with the appropriate root
then gives us an equation relating $E(1)$ and $P(1)$.

Secondly, we eliminate $E(u)$ from \eqref{eqnE1} and \eqref{eqnE2}, and express $P(u)$ in terms of $E(1)$ and $P(1)$.
Cancelling the (same) kernel (using a different root)
gives a second equation relating $E(1)$ and $P(1)$.

Finally, we eliminate $P(1)$ from these two equations to yield the following explicit expression for $E(1)$:
\begin{equation*}
  E(1) \;=\;
  \frac
  {(1-\rho)   \!\left(
  (1-z)(1-2z)(1-\sigma+z\sigma^2(1-\rho))
  -\alpha\rho(1-\sigma)
  +z^2(1-z-z \sigma)\rho\sigma
  \right)}
  {(\alpha- z^2(1-z)\rho)\rho (1-\sigma)
   \:-\:
   (1-2z) \!\left( (1-z) (1-\sigma+z\sigma^2)- z (1-2 z)\rho \sigma^2  \right)},
\end{equation*}
where
\begin{equation*}
  \alpha \;=\;
  1-4\+ z+5\+ z^2-z^3 ,
\end{equation*}
and $\rho$ and $\sigma$ are the appropriate roots of the kernel:
\begin{eqnarray*}
  \rho   &=& \frac{1}{3\+z}\left(1 \:-\: \frac{2\+\omega\+\eta}{\xi^{1/3}} \:-\:
             \frac{\omega^2\+\xi^{1/3}}{2\+\beta}\right) \\[6pt]
  \sigma &=& \frac{1}{3\+z}\left(1 \:-\: \frac{2\+\eta}{\xi^{1/3}} \:-\: \frac{\xi^{1/3}}{2\+\beta} \right),
\end{eqnarray*}
with
\begin{eqnarray*}
  \xi    &=& 12\+ \sqrt{-3\+z^6\+ (1-2\+z)\+\beta^3 \+\theta}\:-\:4\+\beta^2\+\zeta  ,\\[3pt]
  \omega &=& \thalf\+(-1+i\+\sqrt{3}) ,\\[3pt]
  \beta  &=& 1 - 3\+z + 3\+z^2 ,\\[3pt]
  \eta   &=& 1 - 9\+ z + 24\+ z^2 - 21\+ z^3 + 3\+ z^4 ,\\[3pt]
  \theta &=& 4 - 36\+ z + 97\+ z^2 - 101\+ z^3 + 41\+ z^4 - 2\+ z^5 , \\[3pt] 
  \zeta  &=& 2 - 24\+ z + 96\+ z^2 - 144\+ z^3 + 63\+ z^4 .
\end{eqnarray*}

Thus, using a computer algebra system to handle the details of the algebraic manipulation, it can be determined that the generating function $F(z)=E(1)$ for 
$\av(\mathbf{1243},\mathbf{2314})$ has the minimal polynomial
$$
      (z-3\+z^2+2\+z^3)
\:-\: (1-5\+z+8\+z^2-5\+z^3)\+F(z)
\:+\: (2\+z-5\+z^2+4\+z^3)\+F(z)^2
\:+\: z^3\+F(z)^3
,
$$
and, by determining the location of the singularities, that the growth rate of the class is approximately 5.1955, the greatest real root of the quintic polynomial 
$$
-z^5\+\theta(z^{-1})
\;=\;
2 - 41\+z + 101\+z^2 - 97\+z^3 + 36\+z^4 -4\+z^5 ,
$$
as required.

The first twelve terms of the sequence $|\EEE_n|$ are 1, 2, 6, 22, 88, 367, 1571, 6861, 30468, 137229, 625573, 2881230.
More values
can be found at
\href{http://oeis.org/A165539}{A165539} in OEIS~\cite{OEIS}.

\cleardoublepage


\newcommand{\smallT}{\text{\small$\mathbf{T}$}}
\newcommand{\subT}{\mathbf{T}}
\newcommand{\smallF}{\text{\small$\mathbf{F}$}}
\newcommand{\subF}{\mathbf{F}}

\chapter{Avoiding 1324}
\label{chap1324}

\section{Introduction}\label{sectIntro}

The class $\av(\pdiamond)$, of permutations avoiding the pattern $\pdiamond$,
is the only class avoiding a single pattern of length four that is yet to be enumerated exactly.
There are three Wilf classes for single permutations of length four.

The classes $\av(\mathbf{1234})$, $\av(\mathbf{1243})$, $\av(\mathbf{1432})$ and $\av(\mathbf{2143})$ all have the same enumeration.
This fact is a consequence of the work of Babson \& West~\cite{BW2000}, which was later generalised by Backelin, West \& Xin~\cite{BWX2007} (see Section~\ref{sectWilfEquiv}).
Gessel~\cite{Gessel1990} derived an explicit form for the ({$D$-finite}) generating functions of the classes $\av(\mathbf{12\ldots k})$ in terms of determinants, for which Bousquet-M\'elou~\cite{Bousquet-Melou2011} later gave an alternative derivation using the kernel method. For $k=4$, the enumeration can be expressed explicitly:
$$
\av_n(\mathbf{1234})
\;=\;
\frac{1}{(n+1)^2\+(n+2)} \sum_{k=0}^n \binom{2k}{k} \binom{n+1}{k+1} \binom{n+2}{k+1}.
$$
The growth rate of this class is 9, a specific case of the general result of Regev~\cite{Regev1981} that $\gr(\av(\mathbf{12\ldots k}))=(k-1)^2$.

The permutation classes $\av(\mathbf{1342})$ and $\av(\mathbf{2413})$ constitute another Wilf equivalence class (see Stankova~\cite{Stankova1994}).
B\'ona~\cite{Bona1997a} determined the algebraic generating function of $\av(\mathbf{1342})$ to be
$$
\frac{32\+z}{1+20\+z-8\+z^2-(1-8\+z)^{3/2}}
,
$$
from which its growth rate of 8 can readily be obtained.

The remaining class, $\pdiamond$-avoiding permutations, remains unenumerated.
The notoriety of this problem is renowned.
Indeed, as reported in~\cite{EV2005}, at the Permutation Patterns conference in Florida in 2005,
Zeilberger made the unorthodox metaphysical conjecture that ``not even God knows $|\av_{1000}(\pdiamond)|$''.
Conway \& Guttmann~\cite{CG2015} recently presented evidence strongly suggesting that the generating function for this class does not have an algebraic singularity, and is thus not D-finite.
Even the growth rate of the $\pdiamond$-avoiders is currently unknown.

In this chapter, by considering certain large subsets of $\av(\pdiamond)$,
which consist of permutations with a particularly regular structure,
we prove that the growth rate of the class exceeds $9.81$.
This improves on a previous lower bound of $9.47$.
Central to our proof is an examination of the asymptotic
distributions of certain substructures in
the Hasse graphs of the permutations.
In this context,
we consider
occurrences of patterns in {\L}uka\-sie\-wicz paths and prove that
in the limit they exhibit a concentrated Gaussian distribution.

In a recent paper,
Conway \& Guttmann~\cite{CG2015}
calculate the number of permutations avoiding $\pdiamond$ up to length
$36$, building on earlier work by Johansson \& Nakamura~\cite{JN2014}.
They then analyse the sequence of values and
give an estimate for
the growth rate of $\av(\pdiamond)$
of~$11.60\pm0.01$.
However, rigorous bounds still differ from this value quite markedly.

The last few years have seen a steady reduction
in upper bounds on the growth rate, 
based on
a
colouring scheme of Claesson, Jel\'inek \& Stein\-gr\'imsson~\cite{CJS2012} which yields a 
value of~$16$.
B\'ona~\cite{Bona2014+}
has now reduced this to
$13.73718$ by employing
a refined counting argument.

As far as lower bounds go,
Albert, Elder, Rechnitzer, Westcott \& Zabrocki~\cite{AERWZ2006}
established that the growth rate is at least $9.47$,
by
using the \emph{insertion encoding} 
of
$\pdiamond$-avoiders
to construct a
sequence of finite automata that accept subclasses of $\av(\pdiamond)$.
The growth rate of a subclass is then determined from the
transition matrix of the corresponding automaton.
Our main result is an improvement on this lower bound:

\thmbox{
\begin{thm}\label{thm1324LowerBound}
$\gr(\av(\pdiamond)) > 9.81$.
\end{thm}
}

\begin{figure}[t]
\begin{center}
  \includegraphics[scale=0.55]{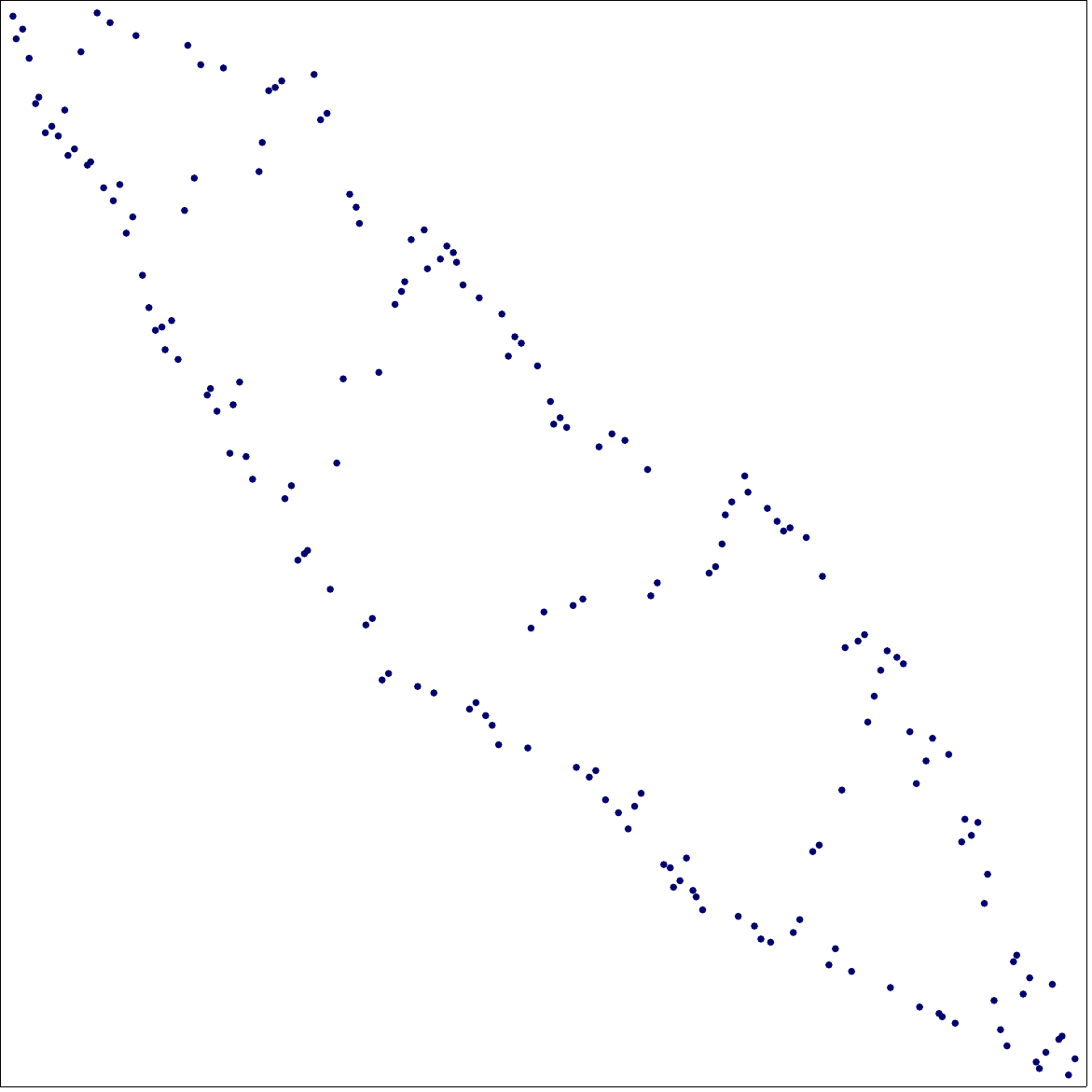}
  $
  \qquad
  $
  \includegraphics[scale=0.55]{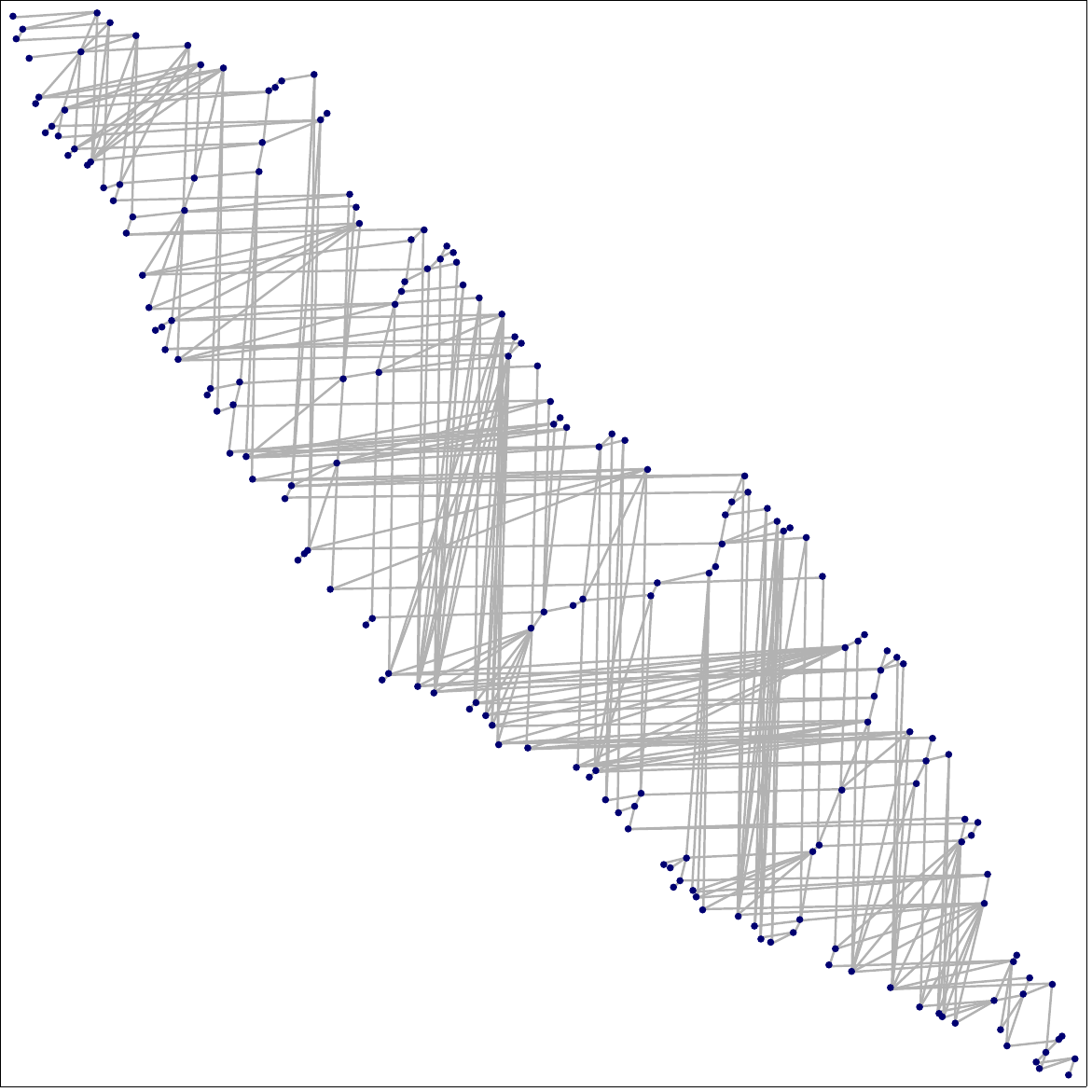}
\end{center}
  \caption{The plot of a $\pdiamond$-avoider of length 187 and its Hasse graph}
  \label{fig187}
\end{figure}

As observed in Section~\ref{sect1324Tethering},
the subgraph of the Hasse graph $H_\sigma$
induced by a left-to-right minimum of $\sigma$ and the points to its upper right is a tree.
By symmetry, this is also the case for the subgraph
induced by a right-to-left maximum of $\sigma$ and the points to its lower left.

What does a \emph{typical} $\pdiamond$-avoider look like?
Figures~\ref{fig187} and~\ref{figEinar1000} contain illustrations 
of large $\pdiamond$-avoiders.\footnote{The data for Figure~\ref{figEinar1000} was provided by Einar Steingr\'imsson from the investigations he describes in~\cite[Footnote~4]{Steingrimsson2013}.} As is noted by Flajolet \& Sedgewick (\cite[p.169]{FS2009}), the fact that a single example can be used to illustrate the asymptotic structure of a large random combinatorial object can be attributed to concentration of distributions, of which we make much use below in determining our lower bound.
Observe the cigar-shaped boundary regions consisting of numerous small subtrees, and also the relative scarcity of points in the interior, which tend to be partitioned into a few paths connecting the two boundaries.
Many questions concerning the
shape of a typical large $\pdiamond$-avoider
remain to be answered or even to be posed precisely.
The recent investigations of
Madras \& Liu~\cite{ML2010} and Atapour \& Madras~\cite{AM2014}
provide a starting point.

\begin{figure}[htp]
\begin{center}
  \includegraphics[scale=0.38]{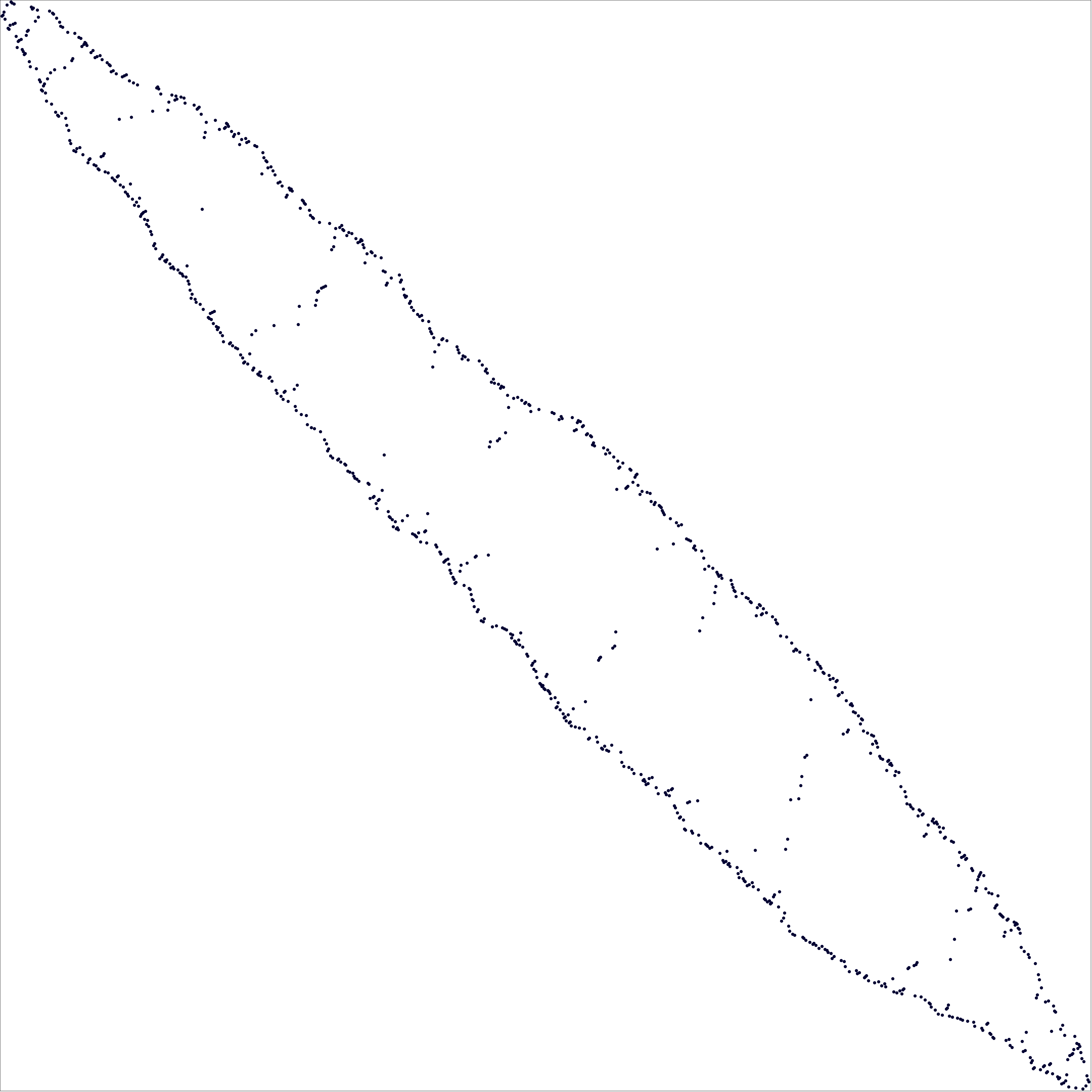}
\end{center}
\begin{center}
  \includegraphics[scale=0.36]{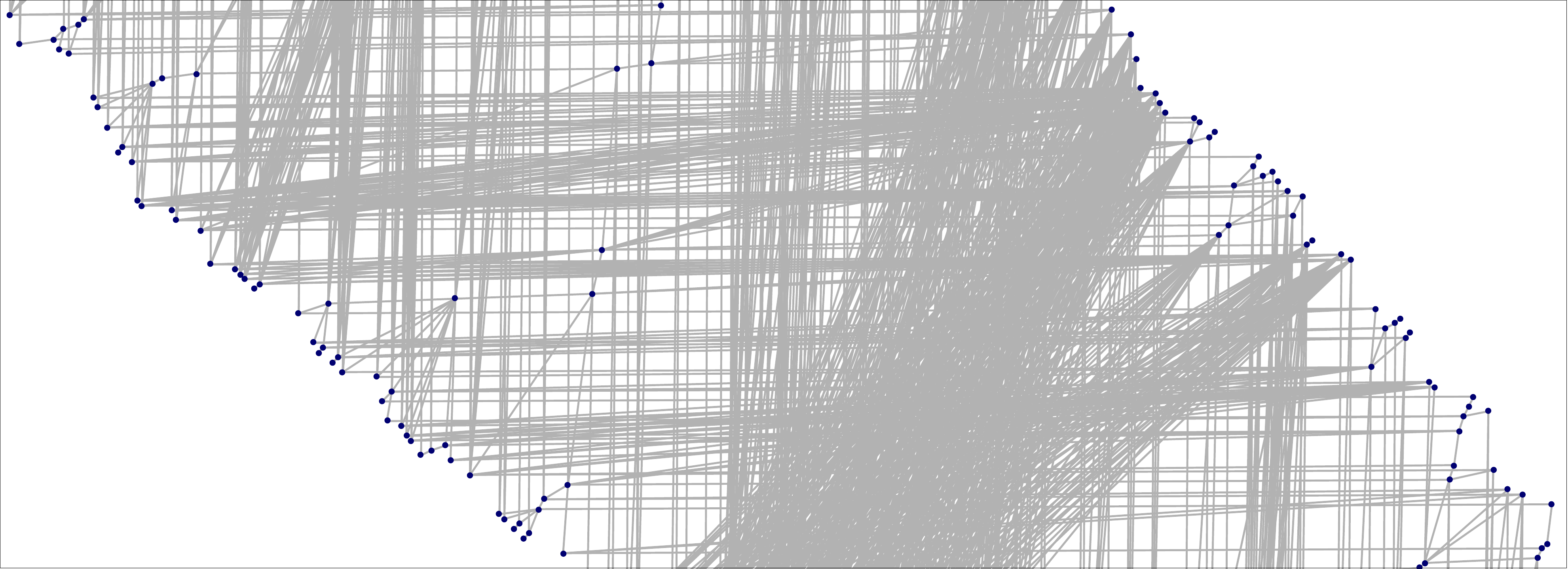}
\end{center}
\caption{The plot of a $\pdiamond$-avoider of length 1000 
and part of its Hasse graph}
\label{figEinar1000}
\end{figure}

We restrict our attention to
$\pdiamond$-avoiders whose Hasse graphs are spanned by a disjoint sequence of trees, rooted at alternate boundaries.
In our investigation of how these trees can interact, we consider the asymptotic distribution of certain substructures of the Hasse graphs.
In doing so, we
exploit the fact that plane trees
are in bijection with \emph{{\L}uka\-sie\-wicz paths}.
A {\L}uka\-sie\-wicz path of length $n$ is a
sequence of integers $y_0,\ldots,y_n$ such that $y_0=0$,
$y_i\geqslant1$ for $i\geqslant1$,
and each \emph{step} $s_i=y_i-y_{i-1}\leqslant1$.
Thus, at each step, a {\L}uka\-sie\-wicz path may rise by at most one, but may fall by any amount as long as it doesn't drop to zero or below.

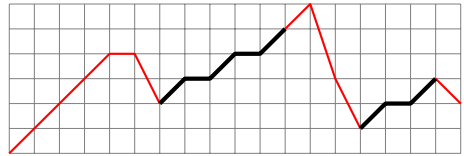
\begin{figure}[ht]
  $$
  \begin{tikzpicture}[scale=0.33,line join=round]
    \draw[help lines] (0,0) grid (18,6);
    \draw[red,thick] (0,0)--(4,4)--(5,4)--(6,2);
    \draw[red,thick] (11,5)--(12,6)--(13,3)--(14,1);
    \draw[red,thick] (17,3)--(18,2);
    \draw[ultra thick] (6,2)--(7,3)--(8,3)--(9,4)--(10,4)--(11,5);
    \draw[ultra thick] (14,1)--(15,2)--(16,2)--(17,3);
  \end{tikzpicture}
  $$
  \caption{The plot of a {\L}uka\-sie\-wicz path that contains three
  occurrences of the pattern $\mathbf{1,0,1}$, two of which overlap}
  \label{figLukaPath}
\end{figure}
In particular, we investigate the distribution of \emph{patterns} in {\L}uka\-sie\-wicz paths.
A pattern $\omega$ of length $m$ in such a path
is a sequence of steps $\omega_1,\ldots,\omega_m$
that occur contiguously in the path (i.e. there is some $k\geqslant0$ such that $\omega_j=s_{k+j}$ for $1\leqslant j\leqslant m$),
with the restriction that the \emph{height}
$\sum_{j=1}^i \omega_j$ after the $i$th step is positive
for $1\leqslant i\leqslant m$.
Note that multiple occurrences of a given pattern may overlap in a {\L}uka\-sie\-wicz path. See Figure~\ref{figLukaPath} for an illustration.

Under very general conditions, 
substructures of recursively defined combinatorial classes can be shown to be distributed normally in the limit.
By generalising the correlation polynomial of Guibas \& Odlyzko, 
and combining it with an application of the kernel method,
we prove that patterns in {\L}uka\-sie\-wicz paths also satisfy the conditions necessary for asymptotic normality:

\thmbox{
\begin{thm}\label{thmLukaPatternsGaussian}
The number of occurrences of a fixed pattern
in a {\L}uka\-sie\-wicz path of length $n$
exhibits
a Gaussian limit distribution
with mean
and standard deviation
asymptotically linear in $n$.
\end{thm}
}

In the next section, we introduce certain subsets of $\av(\pdiamond)$ for consideration, which consist of permutations having a particularly regular structure, and explore restrictions on their structure.
We follow this in Section~\ref{sectConcentration} by looking at a number of parameters that record the distribution of substructures in our permutations.
Key to our result is the fact that these are asymptotically concentrated, and
in this section
we
prove three of the four concentration results we need.
Section~\ref{sectLukasPatterns} is reserved for the proof of Theorem~\ref{thmLukaPatternsGaussian}, concerning the distribution of patterns in {\L}uka\-sie\-wicz paths. 
To conclude,
in Section~\ref{sectLowerBound1324},
we use Theorem~\ref{thmLukaPatternsGaussian} to prove our final concentration result, and then
pull everything together to calculate a lower bound for $\gr(\av(\pdiamond))$, thus proving
Theorem~\ref{thm1324LowerBound}.

\section{Permutations with a regular structure}\label{sectW}

In this section, we present the structure and substructures of the permutations that we investigate. 
Let $\WWW$ be the set of all
permutations avoiding $\pdiamond$
whose Hasse graphs are spanned by a sequence of trees
rooted alternately at the lower left and the upper right.
See Figure~\ref{fig157} for an example.

\begin{figure}[ht]
\begin{center}
  \includegraphics[scale=0.42]{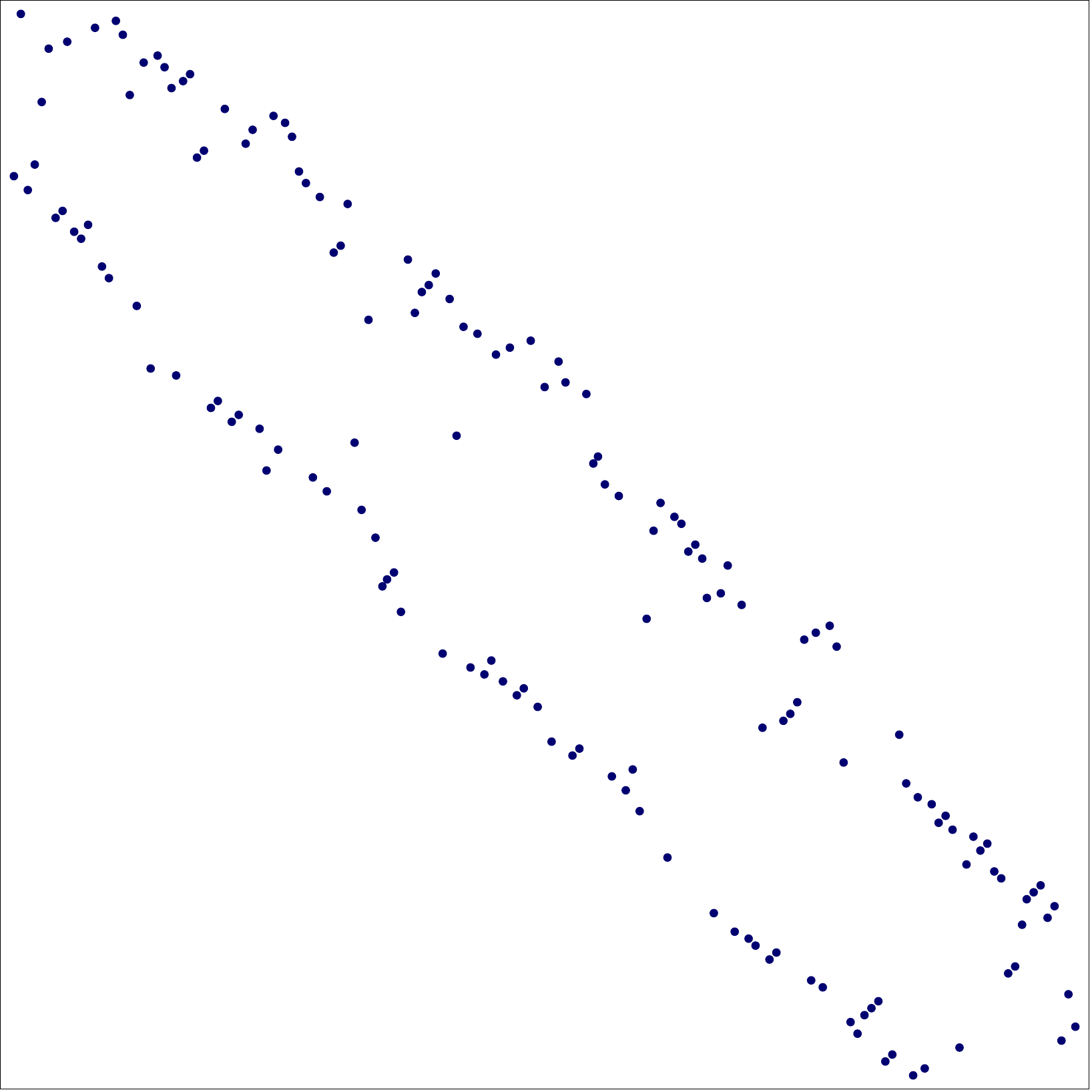}
  $
  \qquad
  $
  \includegraphics[scale=0.42]{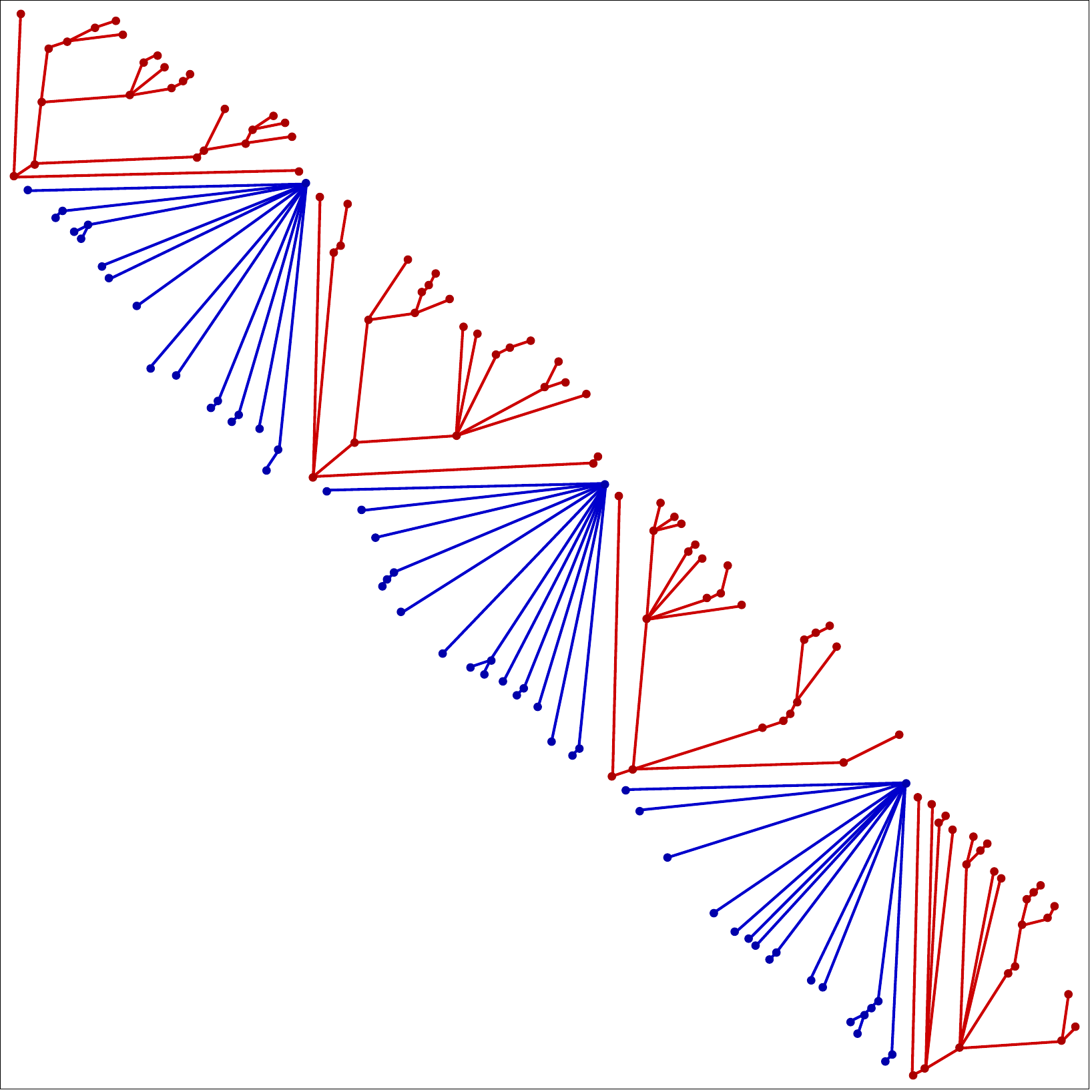}
\end{center}
  \caption{A permutation 
  in $\WWW(3,25,19,12)$ and the spanning of its Hasse graph by red and blue trees}
  \label{fig157}
\end{figure}
Trees rooted at a left-to-right minimum we colour red, and trees rooted at a right-to-left maximum we colour blue. We refer to these as \emph{red trees} and \emph{blue trees} respectively.
As a mnemonic, note that Red trees grow towards the Right and bLue trees grow towards the Left.

Observe that the root of each non-initial blue tree is the uppermost point below the root of the previous red tree, and the root of each non-initial red tree is the leftmost point to the right of the root of the previous blue tree.
Note that $\WWW$ does not contain every $\pdiamond$-avoider. For example, $\mathbf{2143}\notin\WWW$.

We consider elements of $\WWW$ with a particularly regular structure.
Each red tree has the same number of vertices.
Similarly, each blue tree has the same number of vertices. Moreover, every blue tree also has the same root degree.
Specifically,
for any positive $t$, $k$, $\ell$ and~$d$, let $\WWW(t,k,\ell,d)$ be the
set of those permutations in $\WWW$ which
satisfy the following four conditions:
\begin{bulletnums}
  \item Its Hasse graph is spanned by $t+1$ red trees and $t$ blue trees.
  \item Each red tree has $k$ vertices.
  \item Each blue tree has $\ell$ vertices.
  \item Each blue tree has root degree $d$ .
\end{bulletnums}
See Figure~\ref{fig157} for an illustration of a permutation in $\WWW(3,25,19,12)$.

To simplify our presentation,
we use the term \emph{blue subtree} to denote a principal subtree of a blue tree.
(The principal subtrees of a rooted tree are the connected components resulting from deleting the root.)
Thus each blue tree consists of a root vertex and a sequence of $d$ blue subtrees.
We also refer to the roots of blue subtrees simply as \emph{blue roots}.

Our goal is to determine a lower bound for the growth rate of
the union of all the $\WWW(t,k,\ell,d)$.
To achieve this, our focus is on sets in which the number and sizes of the trees
grow together along with
the root degree of blue trees.
Specifically, we consider
the parameterised sets
$$
\WWW_{\lambda,\delta}(k)
\;=\;
\WWW(k,k,\ceil{\lambda\+k},\ceil{\delta\+\lambda\+k}),
$$
for some $\lambda>0$ and $\delta\in(0,1)$,
consisting of $k+1$ $k$-vertex red trees and $k$ $\ceil{\lambda\+k}$-vertex blue trees each having root degree $\ceil{\delta\+\lambda\+k}$.
Thus,
$\lambda$ is the asymptotic ratio of the size of blue trees to red trees,
and $\delta$ is the limiting ratio of the root degree of each blue tree to its size.
Note that, asymptotically,
$1/\delta$
is the mean number of vertices in a blue subtree. Typically these subtrees are small.

Let $g(\lambda,\delta)$ denote the upper growth rate
of $\bigcup_k\! \WWW_{\lambda,\delta}(k)$:
$$
g(\lambda,\delta)
\;=\;
\liminfty[k]\big|\WWW_{\lambda,\delta}(k)\big|^{\nfrac{1}{n(k,\lambda)}},
$$
where
$n(k,\lambda) = k\+\big(k+\ceil{\lambda\+k}+1\big)$ is the length of each permutation in $\WWW_{\lambda,\delta}(k)$.
In order to prove Theorem~\ref{thm1324LowerBound},
we show that there is some
$\lambda$ and $\delta$ for which
$g(\lambda,\delta) > 9.81$.

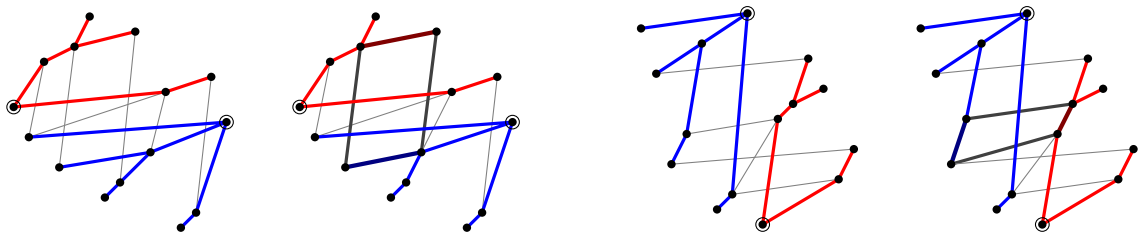
\begin{figure}[ht]
  $$
  \begin{tikzpicture}[scale=0.2,line join=round]
    \draw [black!50] (3,12)--(2,7)--(11,10);
    \draw [black!50] (5,13)--(4,5);
    \draw [black!50] (8,4)--(9,14);
    \draw [black!50] (10,6)--(11,10);
    \draw [black!50] (13,2)--(14,11);
    \draw [red,very thick] (6,15)--(5,13)--(3,12)--(1,9)--(11,10)--(14,11);
    \draw [red,very thick] (9,14)--(5,13);
    \draw [blue,very thick] (2,7)--(15,8)--(10,6)--(4,5);
    \draw [blue,very thick] (10,6)--(8,4)--(7,3);
    \draw [blue,very thick] (15,8)--(13,2)--(12,1);
    \plotpermnobox{15}{9,7,12,5,13,15,3,4,14,6,10,1,2,11,8}
    \draw [thin] (1,9) circle [radius=0.45];
    \draw [thin] (15,8) circle [radius=0.45];
  \end{tikzpicture}
  \qquad
  \begin{tikzpicture}[scale=0.2,line join=round]
    \draw [black!50] (3,12)--(2,7)--(11,10);
    \draw [black!75,very thick] (5,13)--(4,5);
    \draw [black!75,very thick] (9,6)--(10,14);
    \draw [black!50] (9,6)--(11,10);
    \draw [black!50] (13,2)--(14,11);
    \draw [red,very thick] (6,15)--(5,13)--(3,12)--(1,9)--(11,10)--(14,11);
    \draw [red!50!black,ultra thick] (10,14)--(5,13);
    \draw [blue,very thick] (2,7)--(15,8)--(9,6);
    \draw [blue!50!black,ultra thick] (9,6)--(4,5);
    \draw [blue,very thick] (9,6)--(8,4)--(7,3);
    \draw [blue,very thick] (15,8)--(13,2)--(12,1);
    \plotpermnobox{15}{9,7,12,5,13,15,3,4,6,14,10,1,2,11,8}
    \draw [thin] (1,9) circle [radius=0.45];
    \draw [thin] (15,8) circle [radius=0.45];
  \end{tikzpicture}
  \qquad\qquad
  \begin{tikzpicture}[scale=0.2,line join=round]
    \draw [black!50] (2,11)--(12,12);
    \draw [black!50] (4,7)--(10,8);
    \draw [black!50] (3,5)--(15,6);
    \draw [black!50] (10,8)--(7,3)--(14,4);
    \draw [blue,very thick] (1,14)--(8,15)--(5,13)--(2,11);
    \draw [blue,very thick] (5,13)--(4,7)--(3,5);
    \draw [blue,very thick] (8,15)--(7,3)--(6,2);
    \draw [red,very thick] (12,12)--(11,9)--(10,8)--(9,1)--(14,4)--(15,6);
    \draw [red,very thick] (13,10)--(11,9);
    \plotpermnobox{15}{14,11,5,7,13,2,3,15,1,8,9,12,10,4,6}
    \draw [thin] (9,1) circle [radius=0.45];
    \draw [thin] (8,15) circle [radius=0.45];
  \end{tikzpicture}
  \qquad
  \begin{tikzpicture}[scale=0.2,line join=round]
    \draw [black!50] (2,11)--(12,12);
    \draw [black!75,very thick] (3,5)--(10,7);
    \draw [black!75,very thick] (4,8)--(11,9);
    \draw [black!50] (3,5)--(15,6);
    \draw [black!50] (10,7)--(7,3)--(14,4);
    \draw [blue,very thick] (1,14)--(8,15)--(5,13)--(2,11);
    \draw [blue,very thick] (5,13)--(4,8);
    \draw [blue!50!black,ultra thick] (4,8)--(3,5);
    \draw [blue,very thick] (8,15)--(7,3)--(6,2);
    \draw [red,very thick] (12,12)--(11,9);
    \draw [red,very thick] (10,7)--(9,1)--(14,4)--(15,6);
    \draw [red!50!black,ultra thick] (11,9)--(10,7);
    \draw [red,very thick] (13,10)--(11,9);
    \plotpermnobox{15}{14,11,5,8,13,2,3,15,1,7,9,12,10,4,6}
    \draw [thin] (9,1) circle [radius=0.45];
    \draw [thin] (8,15) circle [radius=0.45];
  \end{tikzpicture}
  $$
  \caption{Valid and invalid horizontal interleavings,
            and valid and invalid vertical interleavings;
            occurrences of $\pdiamond$ are shown with thicker edges}
  \label{figInterleaveExample}
\end{figure}
$\WWW(t,k,\ell,d)$ consists precisely of those permutations that
can be built by starting with a $k$-vertex red tree
and repeating the following two steps exactly $t$ times (see Figure~\ref{fig157}):
\begin{bulletnums}
\item Place an $\ell$-vertex blue tree with root degree $d$
below the previous red tree (with its root to the right of the red tree),
horizontally interleaving its non-root vertices with the non-root vertices of the previous red tree in any way that avoids creating a $\pdiamond$.
\item Place a $k$-vertex red tree
to the right of the previous blue tree (with its root below the blue tree),
vertically interleaving its non-root vertices with the non-root vertices of the previous blue tree without creating a $\pdiamond$.
\end{bulletnums}
See Figure~\ref{figInterleaveExample} for illustrations of valid interleavings of the non-root vertices of red and blue trees, and also of invalid interleavings containing occurrences of $\pdiamond$.
The configurations that have to be avoided when interleaving are
shown schematically in Figure~\ref{fig1324Causes}.

\begin{figure}[ht]
$$
\begin{tikzpicture}[line join=round,scale=0.2]
  \path [fill=blue!15] (7,-6) -- (7,2.9) -- (-1.9,2.9) -- (7,-6);
  \node[left] at (7.5,-0.4) {\scriptsize BLUE};
  \node[left] at (7.5,-1.85) {\scriptsize TREE};
  \path [fill=red!15] (6.9,3.1) -- (-2,3.1) -- (-2,12) -- (6.9,3.1);
  \node[right] at (-2.5,7.85) {\scriptsize RED};
  \node[right] at (-2.5,6.4) {\scriptsize TREE};
  \draw [red,very thick] (2,4)--(4,5);
  \draw [blue,very thick] (1,1)--(3,2);
  \draw [black!50, thick] (1,1)--(2,4);
  \draw [black!50, thick] (3,2)--(4,5);
  \plotpermnobox[blue!30!black]{}{1,0,2,0}
  \plotpermnobox[red!30!black]{}{0,4,0,5}
  \fill[blue!30!black,radius=0.275] (6.5,2.4) circle;
  \fill[red!30!black,radius=0.275] (-1.5,3.6) circle;
  \draw [thin] (6.5,2.4)  circle [radius=0.45];
  \draw [thin] (-1.5,3.6) circle [radius=0.45];
\end{tikzpicture}
\qquad\qquad\qquad
\raisebox{19pt}{
\begin{tikzpicture}[line join=round,scale=0.2]
  \path [fill=blue!15] (-6,7) -- (2.9,7) -- (2.9,-1.9) -- (-6,7);
  \node[left] at (1.5,6.0) {\scriptsize BLUE};
  \node[left] at (1.5,4.55) {\scriptsize TREE};
  \path [fill=red!15] (3.1,6.9) -- (3.1,-2) -- (12,-2) -- (3.1,6.9);
  \node[right] at (4.5,0.45) {\scriptsize RED};
  \node[right] at (4.5,-1.0) {\scriptsize TREE};
  \draw [red,very thick] (4,2)--(5,4);
  \draw [blue,very thick] (1,1)--(2,3);
  \draw [black!50, thick] (1,1)--(4,2);
  \draw [black!50, thick] (2,3)--(5,4);
  \plotpermnobox[blue!30!black]{}{1,3}
  \plotpermnobox[red!30!black]{}{0,0,0,2,4}
  \fill[blue!30!black,radius=0.275] (2.4,6.5) circle;
  \fill[red!30!black,radius=0.275] (3.6,-1.5) circle;
  \draw [thin] (2.4,6.5)  circle [radius=0.45];
  \draw [thin] (3.6,-1.5) circle [radius=0.45];
\end{tikzpicture}
}
$$
\caption{Possible causes of a $\pdiamond$ when interleaving horizontally and vertically}
\label{fig1324Causes}
\end{figure}
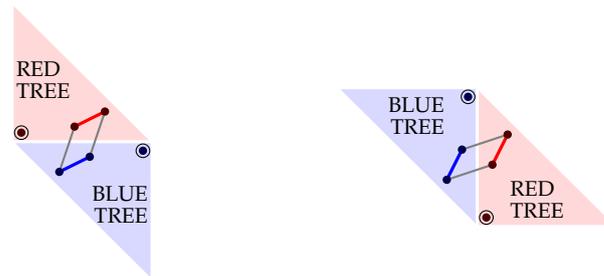

We simply call a valid interleaving of the non-root vertices of a red tree with those of a blue tree an \emph{interleaving} of the trees.
Note that the choice of interleaving at each step is completely independent of the interleaving at any previous or subsequent
step. The only requirement is that no $\pdiamond$ is created by any of the interleavings.

The key to our result is thus an analysis of how vertices of red and blue trees may be interleaved without forming a $\pdiamond$. The remainder of this chapter consists of this analysis.

In what follows, we work exclusively with interleavings of red and blue trees in elements of $\WWW_{\lambda,\delta}(k)$,
for some given
$\lambda>0$ and $\delta\in(0,1)$.
Thus, we assume, without restatement, that a red tree has $k$ vertices, and that a blue tree has $\ell=\ceil{\lambda\+k}$ vertices and is composed of $d=\ceil{\delta\+\lambda\+k}$ blue subtrees.

We now consider how to avoid creating a $\pdiamond$.
Without loss of generality, we limit our discussion to the horizontal case.

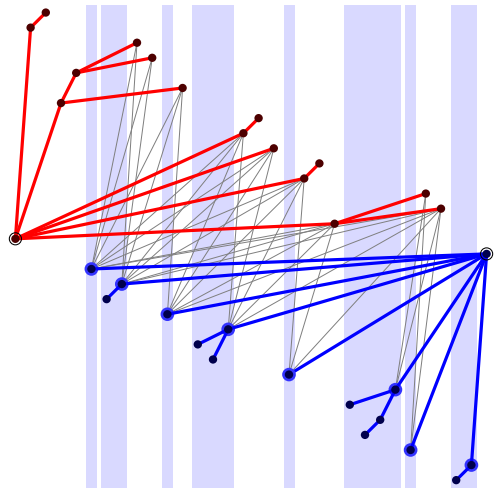
\begin{figure}[ht]
  $$
  \begin{tikzpicture}[scale=0.2,line join=round]
    \path [fill=blue!15] (5.65,0.5) rectangle (6.35,32.5);
    \path [fill=blue!15] (6.65,0.5) rectangle (8.35,32.5);
    \path [fill=blue!15] (10.65,0.5) rectangle (11.35,32.5);
    \path [fill=blue!15] (12.65,0.5) rectangle (15.35,32.5);
    \path [fill=blue!15] (18.65,0.5) rectangle (19.35,32.5);
    \path [fill=blue!15] (22.65,0.5) rectangle (26.35,32.5);
    \path [fill=blue!15] (26.65,0.5) rectangle (27.35,32.5);
    \path [fill=blue!15] (29.65,0.5) rectangle (31.35,32.5);
    \draw [gray,very thin] (6,15)--(9,30);
    \draw [gray,very thin] (6,15)--(10,29);
    \draw [gray,very thin] (6,15)--(12,27);
    \draw [gray,very thin] (6,15)--(16,24);
    \draw [gray,very thin] (6,15)--(20,21);
    \draw [gray,very thin] (6,15)--(22,18);
    \draw [gray,very thin] (6,15)--(29,19);
    \draw [gray,very thin] (8,14)--(9,30);
    \draw [gray,very thin] (8,14)--(10,29);
    \draw [gray,very thin] (8,14)--(12,27);
    \draw [gray,very thin] (8,14)--(16,24);
    \draw [gray,very thin] (8,14)--(20,21);
    \draw [gray,very thin] (8,14)--(22,18);
    \draw [gray,very thin] (8,14)--(29,19);
    \draw [gray,very thin] (11,12)--(12,27);
    \draw [gray,very thin] (11,12)--(16,24);
    \draw [gray,very thin] (11,12)--(20,21);
    \draw [gray,very thin] (11,12)--(22,18);
    \draw [gray,very thin] (11,12)--(29,19);
    \draw [gray,very thin] (15,11)--(16,24);
    \draw [gray,very thin] (15,11)--(20,21);
    \draw [gray,very thin] (15,11)--(22,18);
    \draw [gray,very thin] (15,11)--(29,19);
    \draw [gray,very thin] (20,21)--(19,8);
    \draw [gray,very thin] (22,18)--(19,8);
    \draw [gray,very thin] (29,19)--(19,8);
    \draw [gray,very thin] (26,7)--(28,20);
    \draw [gray,very thin] (26,7)--(29,19);
    \draw [gray,very thin] (27,3)--(28,20);
    \draw [gray,very thin] (27,3)--(29,19);
    \draw [gray,very thin] (6,15)--(18,23);
    \draw [gray,very thin] (8,14)--(18,23);
    \draw [gray,very thin] (11,12)--(18,23);
    \draw [gray,very thin] (15,11)--(18,23);
    \draw [red,very thick] (1,17)--(18,23);
    \draw [blue,very thick] (19,8)--(32,16);
    \fill [blue!80!white,thin] (6,15)  circle [radius=0.45];
    \fill [blue!80!white,thin] (8,14)  circle [radius=0.45];
    \fill [blue!80!white,thin] (11,12) circle [radius=0.45];
    \fill [blue!80!white,thin] (15,11) circle [radius=0.45];
    \fill [blue!80!white,thin] (19,8)  circle [radius=0.45];
    \fill [blue!80!white,thin] (26,7)  circle [radius=0.45];
    \fill [blue!80!white,thin] (27,3)  circle [radius=0.45];
    \fill [blue!80!white,thin] (31,2)  circle [radius=0.45];
    \draw [red,very thick] (1,17)--(2,31)--(3,32);
    \draw [red,very thick] (1,17)--(16,24)--(17,25);
    \draw [red,very thick] (1,17)--(4,26)--(5,28)--(9,30);
    \draw [red,very thick] (5,28)--(10,29);
    \draw [red,very thick] (4,26)--(12,27);
    \draw [red,very thick] (22,18)--(29,19);
    \draw [red,very thick] (1,17)--(20,21)--(21,22);
    \draw [red,very thick] (1,17)--(22,18)--(28,20);
    \draw [blue,very thick] (6,15)--(32,16);
    \draw [blue,very thick] (7,13)--(8,14)--(32,16);
    \draw [blue,very thick] (11,12)--(32,16);
    \draw [blue,very thick] (13,10)--(15,11)--(32,16);
    \draw [blue,very thick] (14,9)--(15,11);
    \draw [blue,very thick] (24,4)--(25,5)--(26,7)--(32,16);
    \draw [blue,very thick] (23,6)--(26,7);
    \draw [blue,very thick] (27,3)--(32,16);
    \draw [blue,very thick] (30,1)--(31,2)--(32,16);
    \draw [thin] (1,17)  circle [radius=0.4];
    \draw [thin] (32,16)  circle [radius=0.4];
    \plotpermnobox[red!30!black]{32}{17, 31, 32, 26, 28,  0,  0,  0, 30, 29,  0, 27,  0, 0,  0, 24, 25, 23, 0, 21, 22, 18, 0, 0, 0, 0, 0, 20, 19, 0, 0,  0}
    \plotpermnobox[blue!30!black]{32}{ 0,  0,  0,  0,  0, 15, 13, 14,  0,  0, 12,  0, 10, 9, 11,  0,  0,  0, 8,  0,  0,  0, 6, 4, 5, 7, 3, 0,  0, 1, 2, 16}
  \end{tikzpicture}
  $$
  \caption{An interleaving of red and blue trees in which no blue subtree (shown in a shaded rectangle) is split by a red vertex} 
  \label{figPartialInterleaving}
\end{figure}
One way to guarantee that no $\pdiamond$ is created when interleaving trees is to ensure that no blue subtree is split by a red vertex, since the pattern $\pdiamond$ is avoided in any interleaving in which no red vertex occurs between two blue vertices of the same blue subtree. See Figure~\ref{figPartialInterleaving} for an illustration.

Let $\WWW^0_{\lambda,\delta}(k)$ be the subset of $\WWW_{\lambda,\delta}(k)$ in which red vertices are interleaved with blue subtrees in this manner in each interleaving.

$\WWW^0_{\lambda,\delta}(k)$ is easy to enumerate since trees and interleavings can be chosen independently. Indeed,
$$
\big| \WWW^0_{\lambda,\delta}(k) \big| \;=\; R_k^{\,k+1} \times B_k^{\,k} \times P_k^{\,2k},
$$
where $R_k$ is the number of distinct red trees, $B_k$ is the number of distinct blue trees and $P_k$ is the number of distinct
ways of interleaving red vertices with blue subtrees.

$R_k=\frac{1}{k}\+\binom{2\+k-2}{k-1}$,
$B_k=\frac{d}{\ell-1}\+\binom{2\ell-3-d}{\ell-2}$ (see~\cite[Example~III.8]{FS2009}),
and
$P_k = \binom{k-1 + d}{d}$.
Hence, by applying Stirling's approximation we obtain the following expression for the growth rate of $\WWW^0_{\lambda,\delta}(k)$:
\begin{equation}\label{eqW0GR}
g_0(\lambda,\delta)
\;=\;
\liminfty[k]\big|\WWW^0_{\lambda,\delta}(k)\big|^{\nfrac{1}{n(k,\lambda)}}
\;=\;
E(\lambda,\delta)^{\nfrac{1}{(1+\lambda)}},
\end{equation}
where
\begin{equation}\label{eqEDef}
  E(\lambda,\delta) \;=\;
  4 \+ \frac{(2-\delta)^{(2-\delta)\lambda}}{(1-\delta)^{(1-\delta)\lambda}} \+ \frac{(1+\delta\+\lambda)^{2(1+\delta\lambda)}}{(\delta\+\lambda)^{2\delta\lambda}} .
\end{equation}
It is now elementary to determine the maximum value of this growth rate.
For fixed $\lambda$, $E(\lambda,\delta)$ is maximal when $\delta$ has the value
$$
  \delta_\lambda \;=\; \frac{2\+ \lambda -1+\sqrt{1+4\+ \lambda +8\+ \lambda ^2}}{2\+ \lambda\+  (2+\lambda )}.
$$
Thence,
numerically maximising $g_0(\lambda,\delta)$ by setting $\lambda\approx0.61840$ (with $\delta_\lambda\approx0.86238$) yields a preliminary lower bound for  $\gr(\av(\pdiamond))$ of 9.40399.
It is rather a surprise that such a simple construction exhibits a growth rate as large as this.

From this analysis, we see that we have complete freedom in choosing the positions of the blue roots (roots of blue subtrees) relative to the vertices of the red tree. In the light of this, we divide the process of interleaving into two stages:
\begin{bulletnums}
  \item Freely
        interleave the blue roots with the red vertices.
  \item Select 
        positions for the non-root vertices of each blue subtree, while avoiding the creation of a $\pdiamond$.
\end{bulletnums}
We call the outcome of the first stage 
a \emph{pre-interleaving}.
A pre-interleaving is thus a sequence consisting of $k-1$ red vertices and $d=\ceil{\delta\+\lambda\+k}$ blue vertices (the blue roots);
the non-root vertices of the blue subtrees play no role in a pre-interleaving.

Note that in the second stage,
each blue subtree can be considered independently since no $\pdiamond$ can contain vertices from more than one blue subtree.
We now consider where the non-root vertices may be positioned.

\begin{figure}[ht]
  $$
  \begin{tikzpicture}[scale=0.20,line join=round]
    \draw[gray,very thin] (3,5)--(4,21);
    \draw[gray,very thin] (3,5)--(5,12);
    \draw[gray,very thin] (8,3)--(9,19);
    \draw[gray,very thin] (10,4)--(11,18);
    \draw[gray,very thin] (10,4)--(12,14);
    \draw[gray,thin] (18,7)--(19,11);
    \draw[red!70!white,thick] (1,10)--(-3,9.3);
    \draw[red!70!white,thick] (1,10)--(2,20);
    \draw[red!70!white,thick] (1,10)--(5,12);
    \draw[red!70!white,thick] (1,10)--(19,11);
    \draw[red,ultra thick] (12,14)--(13,16);
    \draw[red,ultra thick] (12,14)--(15,15);
    \draw[red,ultra thick] (13,16)--(14,17);
    \draw[red,ultra thick] (2,20)--(4,21);
    \draw[red,ultra thick] (5,12)--(6,13);
    \draw[red,ultra thick] (6,13)--(9,19);
    \draw[red,ultra thick] (6,13)--(11,18);
    \draw[red,ultra thick] (6,13)--(12,14);
    \fill [blue!80!white,thin] (18,7)  circle [radius=0.4];
    \draw[blue!70!white,, thick] (18,7)--(25,9);
    \draw[blue,ultra thick] (3,5)--(16,6);
    \draw[blue,ultra thick] (8,3)--(10,4);
    \draw[blue,ultra thick] (10,4)--(16,6);
    \draw[blue,ultra thick] (16,6)--(18,7);
    \draw[blue,ultra thick] (17,2)--(18,7);
    \node[right] at (19,11.4) {$v$};
    \node[right] at (17.9,6.4) {$u$};
    \node[left] at (1.1,10.7) {$x$};
    \node at (14.5,3.7) {$\smallT$};
    \plotpermnobox[red!30!black]{}{10,20,0,21,12,13,0,0,19,0,18,14,16,17,15,0,0,0,11}
    \plotpermnobox[blue!30!black]{}{0,0,5,0,0,0,0,3,0,4,0,0,0,0,0,6,2,7,0,0,0,0,0,0,9}
    \draw [thin] (25,9) circle [radius=0.4];
  \end{tikzpicture}
  $$
  \caption{An interleaving of a blue subtree $\smallT$ with its two-component red forest}
  \label{figRedForest}
\end{figure}
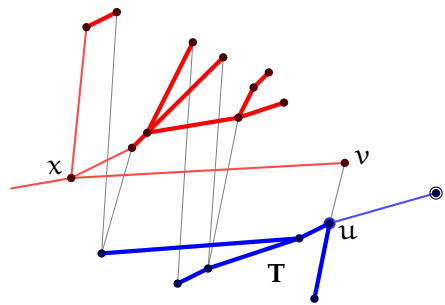
Our first observation is as follows:
Suppose $v$ is the nearest red vertex to the right of the root $u$ of some blue subtree $\smallT$.
Now let $x$ be the parent of $v$ in the red tree.
Then no vertex of~$\smallT$ can be positioned to the left of $x$, since otherwise a $\pdiamond$ would be created in which $xuv$ would be the $\mathbf{324}$.
Thus, vertices of~$\smallT$ can only be interleaved with those red vertices positioned between $u$ and $x$.
We call the graph induced by this set of red vertices (which may be empty) a \emph{red forest}.
See Figure~\ref{figRedForest} for an illustration.

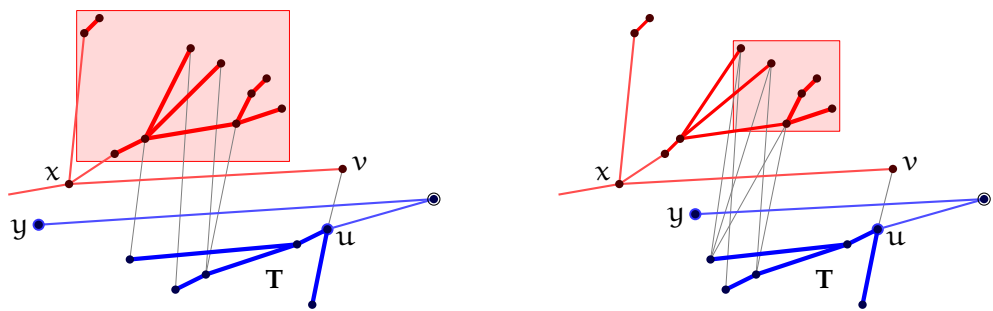
\begin{figure}[ht]
  $$
  \begin{tikzpicture}[scale=0.20,line join=round]
    \draw [red,very thin,fill=red!15] (1.5,11.5) rectangle (15.5,21.5);
    \draw[gray,very thin] (5,5)--(6,13);
    \draw[gray,very thin] (8,3)--(9,19);
    \draw[gray,very thin] (10,4)--(11,18);
    \draw[gray,very thin] (10,4)--(12,14);
    \draw[gray,thin] (18,7)--(19,11);
    \draw[red!70!white,thick] (1,10)--(-3,9.3);
    \draw[red!70!white,thick] (1,10)--(2,20);
    \draw[red!70!white,thick] (1,10)--(4,12);
    \draw[red!70!white,thick] (1,10)--(19,11);
    \draw[red,ultra thick] (12,14)--(13,16);
    \draw[red,ultra thick] (12,14)--(15,15);
    \draw[red,ultra thick] (13,16)--(14,17);
    \draw[red,ultra thick] (2,20)--(3,21);
    \draw[red,ultra thick] (4,12)--(6,13);
    \draw[red,ultra thick] (6,13)--(9,19);
    \draw[red,ultra thick] (6,13)--(11,18);
    \draw[red,ultra thick] (6,13)--(12,14);
    \draw[blue!70!white,, thick] (-1,7.3)--(25,9);
    \fill [blue!80!white,thin] (18,7)  circle [radius=0.4];
    \fill [blue!80!white,thin] (-1,7.3)  circle [radius=0.4];
    \draw[blue!70!white,, thick] (18,7)--(25,9);
    \draw[blue,ultra thick] (5,5)--(16,6);
    \draw[blue,ultra thick] (8,3)--(10,4);
    \draw[blue,ultra thick] (10,4)--(16,6);
    \draw[blue,ultra thick] (16,6)--(18,7);
    \draw[blue,ultra thick] (17,2)--(18,7);
    \node[right] at (19,11.4) {$v$};
    \node[right] at (17.9,6.4) {$u$};
    \node[left] at (1.1,10.7) {$x$};
    \node at (14.5,3.7) {$\smallT$};
    \plotpermnobox[red!30!black]{}{10,20,21,12,0,13,0,0,19,0,18,14,16,17,15,0,0,0,11}
    \plotpermnobox[blue!30!black]{}{0,0,0,0,5,0,0,3,0,4,0,0,0,0,0,6,2,7,0,0,0,0,0,0,9}
    \fill[blue!30!black,radius=0.275] (-1,7.3) circle;
    \node[left] at (-1,7.0) {$y$};
    \draw [thin] (25,9) circle [radius=0.4];
  \end{tikzpicture}
  \qquad\qquad
  \begin{tikzpicture}[scale=0.20,line join=round]
    \draw [red,very thin,fill=red!15] (8.5,13.5) rectangle (15.5,19.5);
    \draw[gray,very thin] (7,5)--(9,19);
    \draw[gray,very thin] (7,5)--(11,18);
    \draw[gray,very thin] (7,5)--(12,14);
    \draw[gray,very thin] (8,3)--(9,19);
    \draw[gray,very thin] (10,4)--(11,18);
    \draw[gray,very thin] (10,4)--(12,14);
    \draw[gray,thin] (18,7)--(19,11);
    \draw[red!70!white,thick] (1,10)--(-3,9.3);
    \draw[red!70!white,thick] (1,10)--(2,20);
    \draw[red!70!white,thick] (1,10)--(4,12);
    \draw[red!70!white,thick] (1,10)--(19,11);
    \draw[red,ultra thick] (12,14)--(13,16);
    \draw[red,ultra thick] (12,14)--(15,15);
    \draw[red,ultra thick] (13,16)--(14,17);
    \draw[red,very thick] (2,20)--(3,21);
    \draw[red,very thick] (4,12)--(5,13);
    \draw[red,very thick] (5,13)--(9,19);
    \draw[red,very thick] (5,13)--(11,18);
    \draw[red,very thick] (5,13)--(12,14);
    \draw[blue!70!white,, thick] (6,8)--(25,9);
    \fill [blue!80!white,thin] (18,7)  circle [radius=0.4];
    \fill [blue!80!white,thin] (6,8)  circle [radius=0.4];
    \draw[blue!70!white,, thick] (18,7)--(25,9);
    \draw[blue,ultra thick] (7,5)--(16,6);
    \draw[blue,ultra thick] (8,3)--(10,4);
    \draw[blue,ultra thick] (10,4)--(16,6);
    \draw[blue,ultra thick] (16,6)--(18,7);
    \draw[blue,ultra thick] (17,2)--(18,7);
    \node[right] at (19,11.4) {$v$};
    \node[right] at (17.9,6.4) {$u$};
    \node[left] at (1.1,10.7) {$x$};
    \node at (14.5,3.7) {$\smallT$};
    \plotpermnobox[red!30!black]{}{10,20,21,12,13,0,0,0,19,0,18,14,16,17,15,0,0,0,11}
    \plotpermnobox[blue!30!black]{}{0,0,0,0,0,0,5,3,0,4,0,0,0,0,0,6,2,7,0,0,0,0,0,0,9}
    \plotpermnobox[blue!30!black]{}{0,0,0,0,0,8}
    \node[left] at (6,7.8) {$y$};
    \draw [thin] (25,9) circle [radius=0.4];
  \end{tikzpicture}
  $$
\caption{Two interleavings of a blue subtree $\smallT$ with a red forest; the red fringes consist of the vertices in the shaded regions}
\label{figRedFringes}
\end{figure}
Our second (elementary) observation is
is as follows:
Suppose $u$ is the root of some blue subtree $\smallT$, and $y$ is the next blue root to the left of $u$.
Then all the non-root vertices of $\smallT$ must occur to the right of $y$ (else $\smallT$ would not be a tree).
Note that $y$ may occur either to the left of $x$ or to its right.
See Figure~\ref{figRedFringes} for illustrations of both of these situations.

These two observations
provide two independent constraints on the set of red vertices with which
the
non-root vertices of a blue subtree may be interleaved, the first determined by
the structure of the red tree
and the second by
the pre-interleaving.
This set consists of those vertices of the red forest situated to the right of both $x$ and $y$.
These red vertices induce a subgraph of the red forest which we call its \emph{red fringe}.
In the examples in Figure~\ref{figRedFringes}, the red fringes consist of those vertices in the shaded regions.
The key fact that motivates the rest of our analysis is that vertices of a blue subtree may only be interleaved with vertices of its red fringe.

The size of a red fringe depends on both the size of the corresponding red forest and also on the location of the next blue root to the left.
Let us call the number of
red vertices positioned between a blue root~$u$ and the next blue root to its left ($y$) the \emph{gap size} of $u$; the gap size may be zero.
The number of vertices in the red fringe is thus the smaller of the gap size and the number of vertices in the red forest.

If we combine this fact
with results concerning the limiting distributions of
blue subtrees and
red fringes,
then we can establish a lower bound for $g(\lambda,\delta)$. This is the focus of the next section.

\section{Concentration of distributions}\label{sectConcentration}

To determine our lower bound, we depend critically on the fact that the asymptotic distributions of substructures of permutations in $\WWW_{\lambda,\delta}(k)$ are \emph{concentrated}. In this section we
introduce certain parameters counting these substructures,
show how their concentration enables us to bound $g(\lambda,\delta)$ from below,
and prove three of the four concentration results we require.

It is frequently the case that
distributions of parameters counting the proportion of particular substructures in combinatorial classes
have a convergent mean
and
a variance that
vanishes asymptotically.
As a direct consequence of Chebyshev's inequality, such distributions have the following concentration property (see~\cite[Proposition~III.3]{FS2009}): 
\begin{prop}\label{propConcentration}
If $\xi_n$ is a sequence of random variables with means $\mu_n=\mathbb{E}[\xi_n]$
and variances $\upsilon_n=\mathbb{V}[\xi_n]$ satisfying the conditions
$$
\liminfty\mu_n=\mu,
\qquad\qquad
\liminfty \upsilon_n \;=\; 0,
$$
for some constant $\mu$,
then
$\xi_n$ \emph{is concentrated at} $\mu$ in the sense that,
for any $\varepsilon>0$,
given sufficiently large $n$,
$$
\mathbb{P}\big[\, |\xi_n-\mu| \:\leqslant\: \varepsilon \,\big] \:\;>\;\: 1-\varepsilon.
$$
\end{prop}
In practice this often means that we can work on the assumption that the value of any such parameter is entirely concentrated at its limiting mean.
This is the case for the
parameters in which we are interested.

We also make use of the following
result concerning \emph{multiple} concentrated parameters.
\begin{prop}\label{propConcentration2}
If $\xi_n$ and $\xi'_n$ are two sequences of random variables on the same sample space concentrated at $\mu$ and $\mu'$ respectively, then they are \emph{jointly concentrated}
in the sense that,
for any $\varepsilon>0$,
given sufficiently large~$n$,
$$
\mathbb{P}\big[\,
|\xi_n - \mu| \leqslant \varepsilon
\text{~~and~~}
|\xi'_n - \mu'| \leqslant \varepsilon
\,\big] \:\;>\:\; 1-\varepsilon.
$$
\end{prop}
\begin{proof}
For any $\eta>0$ and sufficiently large $n$,
the probability that $\xi_n$ differs from $\mu$ by less than $\eta$ exceeds $1-\eta$, and similarly for $\xi'$ with $\mu'$.
Hence the probability that both are simultaneously $\eta$-close to their asymptotic means is at least $1-2\+\eta$. Let $\eta=\nfrac{\varepsilon}{2}$.
\end{proof}

We now introduce the parameters we need:

\textbf{Blue subtrees} $\beta_k$:
For each plane tree $\smallT$, let $\beta_k(\smallT)$ be the random variable that records the
proportion of blue subtrees in a blue tree that are isomorphic to $\smallT$.

\textbf{Gap sizes} $\gamma_k$:
For each $j\geqslant0$,
let $\gamma_k(j)$ be the random variable that records the
proportion of blue roots in a pre-interleaving that have
gap size~$j$.
Also, let $\gamma_k(>\! j)$ record the
proportion of blue roots in a pre-interleaving whose gap size exceeds $j$.

\textbf{Red forests} $\rho_k$:
For each plane forest $\smallF$, let $\rho_k(\smallF)$ be the random variable that records the
proportion of positions in a red tree whose red forest is isomorphic to $\smallF$.
Also, let $\rho_k(\smallF^+)$ record the
proportion of positions in a red tree whose red forest has at least $|\smallF|$ vertices,
and for which the graph induced by the rightmost $|\smallF|$ vertices of the forest is isomorphic to $\smallF$.

Below, we prove that each of these parameters is concentrated, and calculate their asymptotic means.
First we describe how the parameters are combined.

Our first combined parameter counts red fringes.
Given
the combination of a red tree and a pre-interleaving of its vertices with a sequence of blue roots,
let $\varphi_k(\smallF)$
be the random variable that records the
proportion of blue roots whose red fringe is isomorphic to $\smallF$.
Now, occurrences of blue roots with a given gap size $j$ are spread almost uniformly across the positions in a red tree, non-uniformity only occurring for the $j$ leftmost positions.
This is also the case for the distribution of occurrences of blue roots whose gap size is at least $j$.
Hence, by the definition of a red fringe at the end of Section~\ref{sectW},
given any $\varepsilon>0$, if $k$ is large enough,
$\varphi_k(\smallF)$ differs from
\begin{equation}\label{eqChi}
\gamma_k(|\smallF|)\+\rho_k(\smallF^+)  \:+\:
\gamma_k(>\!|\smallF|)\+\rho_k(\smallF)
\end{equation}
by less than $\varepsilon$.

Our second combined parameter concerns pairs consisting of a blue subtree and a red fringe.
Given
a red tree, a blue tree and a pre-interleaving of their red vertices and blue roots,
let $\psi_k(\smallT,\smallF)$ be the random variable that records the proportion of blue subtrees that are isomorphic to $\smallT$ and have a red fringe that is isomorphic to~$\smallF$. We call such a blue subtree a \emph{$(\smallT,\smallF)$-subtree}. Given that occurrences of a given blue subtree are distributed uniformly across the blue roots, we have
\begin{equation}\label{eqZeta}
\psi_k(\smallT,\smallF) \;=\;
\beta_k(\smallT)\+\varphi_k(\smallF).
\end{equation}

Since, as we show below,
$\beta_k$, $\gamma_k$ and $\rho_k$ are concentrated, it follows that $\psi_k$ is also concentrated.
Let $\mu(\smallT,\smallF)$ denote the limiting mean of $\psi_k(\smallT,\smallF)$
as $k$ tends to infinity.

\begin{figure}[ht]
  $$
  \begin{tikzpicture}[scale=0.18,line join=round]
  \draw[gray,very thin] (3,3)--(4,7);
  \draw[gray,very thin] (3,3)--(5,5);
  \draw[red,very thick] (5,5)--(6,6);
  \draw[blue,very thick] (1,2)--(3,3);
  \draw[blue,very thick] (2,1)--(3,3);
  \draw[blue,very thick] (3,3)--(7,4);
  \plotpermnobox[red!30!black]{}{0,0,0,7,5,6}
  \plotpermnobox[blue!30!black]{}{2,1,3,0,0,0,4}
  \plotperm{7}{}
  \draw [thin] (7,4) circle [radius=0.33];
  \end{tikzpicture}
  \quad\;
  \begin{tikzpicture}[scale=0.18,line join=round]
  \draw[gray,very thin] (1,2)--(3,7);
  \draw[gray,very thin] (2,1)--(3,7);
  \draw[gray,very thin] (4,3)--(5,5);
  \draw[red,very thick] (5,5)--(6,6);
  \draw[blue,very thick] (1,2)--(4,3);
  \draw[blue,very thick] (2,1)--(4,3);
  \draw[blue,very thick] (4,3)--(7,4);
  \plotpermnobox[red!30!black]{}{0,0,7,0,5,6}
  \plotpermnobox[blue!30!black]{}{2,1,0,3,0,0,4}
  \plotperm{7}{}
  \draw [thin] (7,4) circle [radius=0.33];
  \end{tikzpicture}
  \quad\;
  \begin{tikzpicture}[scale=0.18,line join=round]
  \path [fill=gray!25] (0.5,0.5) rectangle (7.5,7.5);
  \draw[gray,very thin] (4,5)--(1,2)--(3,7);
  \draw[gray,very thin] (2,1)--(3,7);
  \draw[gray!50!black, thick] (4,5)--(2,1);
  \draw[gray!50!black, thick] (5,3)--(6,6);
  \draw[red!50!black,ultra thick] (4,5)--(6,6);
  \draw[blue,very thick] (1,2)--(5,3);
  \draw[blue!50!black,ultra thick] (2,1)--(5,3);
  \draw[blue,very thick] (5,3)--(7,4);
  \plotpermnobox[red!30!black]{}{0,0,7,5,0,6}
  \plotpermnobox[blue!30!black]{}{2,1,0,0,3,0,4}
  \plotperm{7}{}
  \draw [thin] (7,4) circle [radius=0.33];
  \end{tikzpicture}
  \quad\;
  \begin{tikzpicture}[scale=0.18,line join=round]
  \draw[gray,very thin] (1,2)--(3,7);
  \draw[gray,very thin] (1,2)--(4,5);
  \draw[gray,very thin] (2,1)--(3,7);
  \draw[gray,very thin] (2,1)--(4,5);
  \draw[red,very thick] (4,5)--(5,6);
  \draw[blue,very thick] (1,2)--(6,3);
  \draw[blue,very thick] (2,1)--(6,3);
  \draw[blue,very thick] (6,3)--(7,4);
  \plotpermnobox[red!30!black]{}{0,0,7,5,6}
  \plotpermnobox[blue!30!black]{}{2,1,0,0,0,3,4}
  \plotperm{7}{}
  \draw [thin] (7,4) circle [radius=0.33];
  \end{tikzpicture}
  \quad\;
  \begin{tikzpicture}[scale=0.18,line join=round]
  \draw[gray,very thin] (1,2)--(2,7);
  \draw[gray,very thin] (4,3)--(5,5);
  \draw[red,very thick] (5,5)--(6,6);
  \draw[blue,very thick] (1,2)--(4,3);
  \draw[blue,very thick] (3,1)--(4,3);
  \draw[blue,very thick] (4,3)--(7,4);
  \plotpermnobox[red!30!black]{}{0,7,0,0,5,6}
  \plotpermnobox[blue!30!black]{}{2,0,1,3,0,0,4}
  \plotperm{7}{}
  \draw [thin] (7,4) circle [radius=0.33];
  \end{tikzpicture}
  \quad\;
  \begin{tikzpicture}[scale=0.18,line join=round]
  \path [fill=gray!25] (0.5,0.5) rectangle (7.5,7.5);
  \draw[gray,very thin] (4,5)--(1,2)--(2,7);
  \draw[gray!50!black, thick] (4,5)--(3,1);
  \draw[gray!50!black, thick] (5,3)--(6,6);
  \draw[red!50!black,ultra thick] (4,5)--(6,6);
  \draw[blue,very thick] (1,2)--(5,3);
  \draw[blue!50!black,ultra thick] (3,1)--(5,3);
  \draw[blue,very thick] (5,3)--(7,4);
  \plotpermnobox[red!30!black]{}{0,7,0,5,0,6}
  \plotpermnobox[blue!30!black]{}{2,0,1,0,3,0,4}
  \plotperm{7}{}
  \draw [thin] (7,4) circle [radius=0.33];
  \end{tikzpicture}
  \quad\;
  \begin{tikzpicture}[scale=0.18,line join=round]
  \draw[gray,very thin] (1,2)--(2,7);
  \draw[gray,very thin] (1,2)--(4,5);
  \draw[gray,very thin] (3,1)--(4,5);
  \draw[red,very thick] (4,5)--(5,6);
  \draw[blue,very thick] (1,2)--(6,3);
  \draw[blue,very thick] (3,1)--(6,3);
  \draw[blue,very thick] (6,3)--(7,4);
  \plotpermnobox[red!30!black]{}{0,7,0,5,6}
  \plotpermnobox[blue!30!black]{}{2,0,1,0,0,3,4}
  \plotperm{7}{}
  \draw [thin] (7,4) circle [radius=0.33];
  \end{tikzpicture}
  $$
  $$
  \begin{tikzpicture}[scale=0.18,line join=round]
  \path [fill=gray!25] (0.5,0.5) rectangle (7.5,7.5);
  \draw[gray,very thin] (1,2)--(2,7);
  \draw[gray!50!black, thick] (3,5)--(1,2);
  \draw[gray!50!black, thick] (5,3)--(6,6);
  \draw[red!50!black,ultra thick] (3,5)--(6,6);
  \draw[blue,very thick] (1,2)--(5,3);
  \draw[blue!50!black,ultra thick] (4,1)--(5,3);
  \draw[blue,very thick] (5,3)--(7,4);
  \plotpermnobox[red!30!black]{}{0,7,5,0,0,6}
  \plotpermnobox[blue!30!black]{}{2,0,0,1,3,0,4}
  \plotperm{7}{}
  \draw [thin] (7,4) circle [radius=0.33];
  \end{tikzpicture}
  \quad\;
  \begin{tikzpicture}[scale=0.18,line join=round]
  \draw[gray,very thin] (1,2)--(2,7);
  \draw[gray,very thin] (1,2)--(3,5);
  \draw[gray,very thin] (4,1)--(5,6);
  \draw[red,very thick] (3,5)--(5,6);
  \draw[blue,very thick] (1,2)--(6,3);
  \draw[blue,very thick] (4,1)--(6,3);
  \draw[blue,very thick] (6,3)--(7,4);
  \plotpermnobox[red!30!black]{}{0,7,5,0,6}
  \plotpermnobox[blue!30!black]{}{2,0,0,1,0,3,4}
  \plotperm{7}{}
  \draw [thin] (7,4) circle [radius=0.33];
  \end{tikzpicture}
  \quad\;
  \begin{tikzpicture}[scale=0.18,line join=round]
  \draw[gray,very thin] (1,2)--(2,7);
  \draw[gray,very thin] (1,2)--(3,5);
  \draw[red,very thick] (3,5)--(4,6);
  \draw[blue,very thick] (1,2)--(6,3);
  \draw[blue,very thick] (5,1)--(6,3);
  \draw[blue,very thick] (6,3)--(7,4);
  \plotpermnobox[red!30!black]{}{0,7,5,6}
  \plotpermnobox[blue!30!black]{}{2,0,0,0,1,3,4}
  \plotperm{7}{}
  \draw [thin] (7,4) circle [radius=0.33];
  \end{tikzpicture}
  \quad\;
  \begin{tikzpicture}[scale=0.18,line join=round]
  \draw[gray,very thin] (4,3)--(5,5);
  \draw[red,very thick] (5,5)--(6,6);
  \draw[blue,very thick] (2,2)--(4,3);
  \draw[blue,very thick] (3,1)--(4,3);
  \draw[blue,very thick] (4,3)--(7,4);
  \plotpermnobox[red!30!black]{}{7,0,0,0,5,6}
  \plotpermnobox[blue!30!black]{}{0,2,1,3,0,0,4}
  \plotperm{7}{}
  \draw [thin] (7,4) circle [radius=0.33];
  \end{tikzpicture}
  \quad\;
  \begin{tikzpicture}[scale=0.18,line join=round]
  \path [fill=gray!25] (0.5,0.5) rectangle (7.5,7.5);
  \draw[gray,very thin] (4,5)--(2,2);
  \draw[gray!50!black, thick] (4,5)--(3,1);
  \draw[gray!50!black, thick] (5,3)--(6,6);
  \draw[red!50!black,ultra thick] (4,5)--(6,6);
  \draw[blue,very thick] (2,2)--(5,3);
  \draw[blue!50!black,ultra thick] (3,1)--(5,3);
  \draw[blue,very thick] (5,3)--(7,4);
  \plotpermnobox[red!30!black]{}{7,0,0,5,0,6}
  \plotpermnobox[blue!30!black]{}{0,2,1,0,3,0,4}
  \plotperm{7}{}
  \draw [thin] (7,4) circle [radius=0.33];
  \end{tikzpicture}
  \quad\;
  \begin{tikzpicture}[scale=0.18,line join=round]
  \draw[gray,very thin] (2,2)--(4,5);
  \draw[gray,very thin] (3,1)--(4,5);
  \draw[red,very thick] (4,5)--(5,6);
  \draw[blue,very thick] (2,2)--(6,3);
  \draw[blue,very thick] (3,1)--(6,3);
  \draw[blue,very thick] (6,3)--(7,4);
  \plotpermnobox[red!30!black]{}{7,0,0,5,6}
  \plotpermnobox[blue!30!black]{}{0,2,1,0,0,3,4}
  \plotperm{7}{}
  \draw [thin] (7,4) circle [radius=0.33];
  \end{tikzpicture}
  $$
  $$
  \begin{tikzpicture}[scale=0.18,line join=round]
  \path [fill=gray!25] (0.5,0.5) rectangle (7.5,7.5);
  \draw[gray!50!black, thick] (2,2)--(3,5);
  \draw[gray!50!black, thick] (5,3)--(6,6);
  \draw[red!50!black,ultra thick] (3,5)--(6,6);
  \draw[blue!50!black,ultra thick] (2,2)--(5,3);
  \draw[blue,very thick] (4,1)--(5,3);
  \draw[blue,very thick] (5,3)--(7,4);
  \plotpermnobox[red!30!black]{}{7,0,5,0,0,6}
  \plotpermnobox[blue!30!black]{}{0,2,0,1,3,0,4}
  \plotperm{7}{}
  \draw [thin] (7,4) circle [radius=0.33];
  \end{tikzpicture}
  \quad\;
  \begin{tikzpicture}[scale=0.18,line join=round]
  \draw[gray,very thin] (2,2)--(3,5);
  \draw[gray,very thin] (4,1)--(5,6);
  \draw[red,very thick] (3,5)--(5,6);
  \draw[blue,very thick] (2,2)--(6,3);
  \draw[blue,very thick] (4,1)--(6,3);
  \draw[blue,very thick] (6,3)--(7,4);
  \plotpermnobox[red!30!black]{}{7,0,5,0,6}
  \plotpermnobox[blue!30!black]{}{0,2,0,1,0,3,4}
  \plotperm{7}{}
  \draw [thin] (7,4) circle [radius=0.33];
  \end{tikzpicture}
  \quad\;
  \begin{tikzpicture}[scale=0.18,line join=round]
  \draw[gray,very thin] (2,2)--(3,5);
  \draw[red,very thick] (3,5)--(4,6);
  \draw[blue,very thick] (2,2)--(6,3);
  \draw[blue,very thick] (5,1)--(6,3);
  \draw[blue,very thick] (6,3)--(7,4);
  \plotpermnobox[red!30!black]{}{7,0,5,6}
  \plotpermnobox[blue!30!black]{}{0,2,0,0,1,3,4}
  \plotperm{7}{}
  \draw [thin] (7,4) circle [radius=0.33];
  \end{tikzpicture}
  \quad\;
  \begin{tikzpicture}[scale=0.18,line join=round]
  \draw[gray,very thin] (5,3)--(6,6);
  \draw[red,very thick] (2,5)--(6,6);
  \draw[blue,very thick] (3,2)--(5,3);
  \draw[blue,very thick] (4,1)--(5,3);
  \draw[blue,very thick] (5,3)--(7,4);
  \plotpermnobox[red!30!black]{}{7,5,0,0,0,6}
  \plotpermnobox[blue!30!black]{}{0,0,2,1,3,0,4}
  \plotperm{7}{}
  \draw [thin] (7,4) circle [radius=0.33];
  \end{tikzpicture}
  \quad\;
  \begin{tikzpicture}[scale=0.18,line join=round]
  \draw[gray,very thin] (3,2)--(5,6);
  \draw[gray,very thin] (4,1)--(5,6);
  \draw[red,very thick] (2,5)--(5,6);
  \draw[blue,very thick] (3,2)--(6,3);
  \draw[blue,very thick] (4,1)--(6,3);
  \draw[blue,very thick] (6,3)--(7,4);
  \plotpermnobox[red!30!black]{}{7,5,0,0,6}
  \plotpermnobox[blue!30!black]{}{0,0,2,1,0,3,4}
  \plotperm{7}{}
  \draw [thin] (7,4) circle [radius=0.33];
  \end{tikzpicture}
  \quad\;
  \begin{tikzpicture}[scale=0.18,line join=round]
  \draw[gray,very thin] (3,2)--(4,6);
  \draw[red,very thick] (2,5)--(4,6);
  \draw[blue,very thick] (3,2)--(6,3);
  \draw[blue,very thick] (5,1)--(6,3);
  \draw[blue,very thick] (6,3)--(7,4);
  \plotpermnobox[red!30!black]{}{7,5,0,6}
  \plotpermnobox[blue!30!black]{}{0,0,2,0,1,3,4}
  \plotperm{7}{}
  \draw [thin] (7,4) circle [radius=0.33];
  \end{tikzpicture}
  \quad\;
  \begin{tikzpicture}[scale=0.18,line join=round]
  \draw[red,very thick] (2,5)--(3,6);
  \draw[blue,very thick] (4,2)--(6,3);
  \draw[blue,very thick] (5,1)--(6,3);
  \draw[blue,very thick] (6,3)--(7,4);
  \plotpermnobox[red!30!black]{}{7,5,6}
  \plotpermnobox[blue!30!black]{}{0,0,0,2,1,3,4}
  \plotperm{7}{}
  \draw [thin] (7,4) circle [radius=0.33];
  \end{tikzpicture}
  $$
  \caption{
  $Q(\mathbf{2134},\mathbf{312})=15$; the five shaded interleavings contain a $\pdiamond$
  }
  \label{figTauPiInterleaving}
\end{figure}
Finally, given a blue subtree~$\smallT$ and a red fringe $\smallF$, let
$Q(\smallT,\smallF)$ denote the number of distinct ways of interleaving the non-root vertices of $\smallT$ and the vertices of $\smallF$ without creating a $\pdiamond$. 
See Figure~\ref{figTauPiInterleaving} for an example.


With all the relevant parameters defined, we are now in a position to present a lower bound on the value of $g(\lambda,\delta)$.
\begin{prop}\label{propLB}
Let $\SSS$ be any finite set of pairs $(\smallT,\smallF)$ composed of a plane tree $\smallT$ and a plane forest $\smallF$. Then
$$
g(\lambda,\delta)
\:\;\geqslant\;\:
E(\lambda,\delta)
^{\nfrac{1}{(1+\lambda)}}
\,\times\!
\prod_{(\subT,\subF)\,\in\,\SSS\,}\!\!\!
Q(\smallT,\smallF)^{2\delta\lambda\mu(\subT,\subF)
/(1+\lambda)
},
$$
where $E(\lambda,\delta)$ is as defined in equation~\eqref{eqEDef} on page~\pageref{eqEDef}.
\end{prop}
\begin{proof}
Consider
a red tree and a blue tree together with a pre-interleaving of their red vertices and blue roots.
By Propositions~\ref{propConcentration} and~\ref{propConcentration2},
for any $\varepsilon>0$, if $k$ is large enough, then
with probability exceeding $1-\varepsilon$,
it is the case that
$\big|\psi_k(\smallT,\smallF) - \mu(\smallT,\smallF)\big| \leqslant \varepsilon$
for every $(\smallT,\smallF)\in \SSS$.

So the proportion of pre-interleaved pairs of trees
with at least
$\ceil{\delta\+\lambda\+k}\!(\mu(\smallT,\smallF)-\varepsilon)$
occurrences of
$(\smallT,\smallF)$-subtrees
for every $(\smallT,\smallF)\in \SSS$
exceeds $1-\varepsilon$.

Elements of $\WWW_{\lambda,\delta}(k)$
are constructed by independently choosing
trees and interleavings.
Thus, the size of $\WWW_{\lambda,\delta}(k)$ is bounded below by
$$
\big|\WWW_{\lambda,\delta}(k)\big|
\:\;\geqslant\;\:
\big| \WWW^0_{\lambda,\delta}(k) \big| \,\times\,
\Bigg(\!
\prod_{(\subT,\subF)\,\in\,\SSS\,}\!\!\!
(1-\varepsilon)\+Q(\smallT,\smallF)^{\ceil{\delta\lambda k}(\mu(\subT,\subF)-\varepsilon)}
\Bigg) ^{\!2k}
.$$

Recall that
$$
g(\lambda,\delta)
\;=\;
\liminfty[k]\big|\WWW_{\lambda,\delta}(k)\big|^{\nfrac{1}{n(k,\lambda)}},
$$
where
$n(k,\lambda) = k\+\big(k+\ceil{\lambda\+k}+1\big)$ is the length of each permutation in $\WWW_{\lambda,\delta}(k)$.
The desired result follows after expanding and taking the limit, making use of~\eqref{eqW0GR}.
\end{proof}

To determine the asymptotic mean and variance of our parameters, we utilise bivariate generating functions.
The following standard result enables us to obtain the required moments directly
as long as we can extract coefficients.
We use $[z^n]\+f(z)$ to denote the coefficient of $z^n$ in the series expansion of $f(z)$; we also use $f_x$ for $\frac{\partial f}{\partial x}$ and $f_{xx}$ for $\frac{\partial^2 f}{\partial x^2}$.
\begin{prop}[{\cite[Proposition III.2]{FS2009}}]\label{propMoments}
Suppose $A(z,x)$ is the bivariate generating function for some combinatorial class, in which $z$ marks size and $x$ marks the value of a parameter $\xi$.
Then the mean
and
variance
of $\xi$ for elements of size $n$ are given by
$$
\mathbb{E}_n[\xi] \;=\; \frac{[z^n]\+ A_x(z,1)}{[z^n]\+ A(z,1)}
\qquad
\text{and}
\qquad
\mathbb{V}_n[\xi]
\;=\;
\frac{[z^n]\+ A_{xx}(z,1) 
}{[z^n]\+ A(z,1)}
\:+\: \mathbb{E}_n(\xi)
\:-\: {\mathbb{E}_n(\xi)}^2
$$
respectively. 
\end{prop}

The proofs of our first three concentration results each follow a similar pattern:
establish the generating function; extract the coefficients; apply Proposition~\ref{propMoments}; take limits
using Stirling's approximation; finally apply Proposition~\ref{propConcentration}.

First, we consider blue subtrees.
Recall that the random variable $\beta_k(\smallT)$ records
the proportion of principal subtrees in a
$\ceil{\lambda\+k}$-vertex plane tree with root degree $\ceil{\delta\+\lambda\+k}$
that are isomorphic to~$\smallT$.
\begin{prop}\label{propPsi}
Let $i=|\smallT|$.
$\beta_k(\smallT)$ is concentrated at
$$
\mu_\beta(\smallT) \;=\; \frac{(1-\delta)^{i-1}}{(2-\delta)^{2i-1}}.
$$
\end{prop}
\begin{proof}
Let
$
T(z) = \thalf \+ \big(1-\sqrt{1-4\+z}\big)
$
be the generating function for plane trees.
Then
the bivariate generating function for plane trees with root degree $d$, in which $z$ marks vertices and $u$ marks principal subtrees isomorphic to $\smallT$, is given by
$$
B(z,u) \;=\; z \left(T(z)+(u-1)\+z^i\right)^{\!d}.
$$
Extracting coefficients yields
$$
\begin{array}{lcl}
[z^\ell]\+ B(z,1) & = & \frac{d}{\ell-1} \+ \binom{2\+\ell - d - 3}{\ell-2}, \\[6pt]
[z^\ell]\+ B_u(z,1)
& = & \frac{d\+(d-1)}{\ell-i-1}\+\binom{2\+\ell-2\+i - d - 2}{\ell-i-2}, \\[6pt]
[z^\ell]\+ B_{uu}(z,1)
& = & \frac{d\+(d-1)\+(d-2)}{\ell-2\+i-1}\+\binom{2\+\ell-4\+i - d - 1}{\ell-2\+i-2}.
\end{array}
$$
Hence, with $\ell=\ceil{\lambda\+k}$ and $d=\ceil{\delta\+\lambda\+k}$, applying Proposition~\ref{propMoments}
and taking limits gives
$$\liminfty[k]\mathbb{E}[\beta_k(\smallT)] \;=\; \frac{(1-\delta)^{i-1}}{(2-\delta)^{2i-1}}
\qquad\text{and}\qquad
\liminfty[k] k \+ \mathbb{V}[\beta_k(\smallT)] \;=\; \upsilon_\beta(\smallT),$$
where $\upsilon_\beta(\smallT)$ is some rational function in $\delta$.
So,
$\beta_k(\smallT)$ satisfies the conditions for Proposition~\ref{propConcentration} and is thus concentrated at $\mu_\beta(\smallT)$ as required.
\end{proof}
Secondly, we consider gap size.
Recall that,
given a pre-interleaving of the non-root vertices of a $k$-vertex red tree and
$\ceil{\delta\+\lambda\+k}$ blue roots,
the random variable
$\gamma_k(j)$ records the
proportion of blue roots that have
gap size~$j$. Similarly, $\gamma_k(>\! j)$ records the proportion that have gap size exceeding $j$.
\begin{prop}\label{propPhi}
$\gamma_k(j)$ is concentrated at
$$\mu_\gamma(j) \;=\;
\frac{\delta\+\lambda}{(1+\delta\+\lambda)^{j+1}}
.
$$
Also,
$\gamma_k(>\! j)$ is concentrated at
$$\mu_\gamma(>\!j) \;=\;
\frac{1}{(1+\delta\+\lambda)^{j+1}}
.
$$
\end{prop}
\begin{proof}
The bivariate generating function for pre-interleavings containing $d$ blue roots, in which $z$ marks red vertices and $v$ marks gaps of size $j$, is given by
$$
G(z,v) \;=\; \tfrac{z}{1-z} \left(\tfrac{1}{1-z}+(v-1)\+z^j\right)^{\!d}.
$$
Extracting coefficients yields
$$
\begin{array}{lcl}
[z^k]\+ G(z,1) & = & \binom{k+d-1}{d}, \\[6pt]
[z^k]\+ G_v(z,1)
& = & d\+\binom{k-j+d-2}{d-1}, \\[6pt]
[z^k]\+ G_{vv}(z,1)
& = & d\+(d-1)\+\binom{k-2\+j+d-3}{d-2}.
\end{array}
$$
Hence, with $d=\ceil{\delta\+\lambda\+k}$, applying Proposition~\ref{propMoments}
and taking limits gives
$$\liminfty[k]\mathbb{E}[\gamma_k(j)] \;=\; \frac{\delta\+\lambda}{(1+\delta\+\lambda)^{j+1}}
\qquad\text{and}\qquad
\liminfty[k] k \+ \mathbb{V}[\gamma_k(j)] \;=\; \upsilon_\gamma(j),$$
where $\upsilon_\gamma(j)$ is some rational function in $\delta$ and $\lambda$.
So,
$\gamma_k(j)$ satisfies the conditions for Proposition~\ref{propConcentration} and is thus concentrated at $\mu_\gamma(j)$ as required.

Also, since
$$
\liminfty[k]\mathbb{E}[\gamma_k(>\!j)] \;=\; 1\:-\:\sum\limits_{i=0}^j\mu_\gamma(i) \;=\; \frac{1}{(1+\delta\+\lambda)^{j+1}}
,
$$ $\gamma_k(>\!j)$ is concentrated at $\mu_\gamma(>\!j)$ as required.
\end{proof}
Thirdly, we consider red forests.
Recall that
the random variable
$\rho_k(\smallF)$ records the
proportion of positions in a $k$-vertex red tree whose red forest is isomorphic to $\smallF$.
\begin{prop}\label{propOmega}
Let $m=|\smallF|$.
$\rho_k(\smallF)$ is concentrated at
$$
\mu_\rho(\smallF)
\;=\;
\frac{1}{2^{2m+1}}.
$$
\end{prop}
\begin{proof}
If $\smallF$ has $h$ components, then an occurrence of $\smallF$ in a red tree
comprises the leftmost $h$ subtrees of some vertex $x$ that has at least one additional child vertex to the right.
See Figure~\ref{figRedFringes} for an illustration.
Hence, if $\RRR$ is the class of red trees augmented by marking occurrences of $\smallF$ with~$w$, then $\RRR$ satisfies the structural equation
$$
\RRR \;=\; z \left( \seq{\RRR} \:+\: (w-1) \+ z^m \+ \seqplus{\RRR} \right) .
$$
So the corresponding bivariate generating function, $R(z,w)$, satisfies the functional equation
$$
R(z,w) \;=\; \frac{z \left( 1 \:+\: (w-1) \+ z^m \+ R(z,w) \right)}{1-R(z,w)} ,
$$
and hence
$$
R(z,w) \;=\; \thalf \+ \Big( 1 + (1-w)\+z^{m+1} - \sqrt{\big( 1 + (1-w)\+z^{m+1} \big)^2 - 4\+z} \,\Big) .
$$
Extracting coefficients then yields
$$
\begin{array}{lcl}
[z^k]\+ R(z,1) & = & \frac{1}{k} \binom{2k-2}{k-1}, \\[6pt]
[z^k]\+ R_w(z,1)
& = & \binom{2k-2m-3}{k-m-1}, \\[6pt]
[z^k]\+ R_{ww}(z,1)
& = & (k-2m-2)\+ \binom{2k-4m-4}{k-2m-2}.
\end{array}
$$
Hence, applying Proposition~\ref{propMoments}
and taking limits gives
$$\liminfty[k]\mathbb{E}[\rho_k(\smallF)] \;=\; \frac{1}{2^{2\+m+1}}
\qquad\text{and}\qquad
\liminfty[k] k \+ \mathbb{V}[\rho_k(\smallF)] \;=\; \upsilon_\rho(\smallF),$$
where $\upsilon_\rho(\smallF)$ depends only on $|\smallF|$.
So,
$\rho_k(\smallF)$ satisfies the conditions for Proposition~\ref{propConcentration} and is thus concentrated at $\mu_\rho(\smallF)$ as required.
\end{proof}

\begin{figure}[ht]
  $$
  \begin{tikzpicture}[scale=0.24,line join=round]
    \draw [thin,fill=gray!30!white] (9.6,12.4) to[out=45,in=166](10,12.5657) to[out=346,in=121](12.4,10.4) to[out=301,in=45](12.4,9.6)--(11.4,8.6) to[out=225,in=301](10.6,8.6) to[out=121,in=270](9.4343,12) to[out=90,in=225](9.6,12.4);
    \draw [thin,fill=gray!30!white] (7.6,13.4) to[out=45,in=180](8,13.5657) to[out=0,in=135](10.4,12.4) to[out=315,in=45](10.4,11.6)--(9.4,10.6) to[out=225,in=315](8.6,10.6) to[out=135,in=270](7.4343,13) to[out=90,in=225](7.6,13.4);
    \draw [thin] (9.6,12.4) to[out=45,in=166](10,12.5657) to[out=346,in=121](12.4,10.4) to[out=301,in=45](12.4,9.6)--(11.4,8.6) to[out=225,in=301](10.6,8.6) to[out=121,in=270](9.4343,12) to[out=90,in=225](9.6,12.4);
    \draw [thin,fill=gray!30!white] (1.6,18.4) to[out=45,in=180](2,18.5657) to[out=0,in=135](4.4,17.4) to[out=315,in=45](4.4,16.6)--(3.4,15.6) to[out=225,in=315](2.6,15.6) to[out=135,in=270](1.4343,18) to[out=90,in=225](1.6,18.4);
    \draw [ultra thick] (3,16)--(4,17);
    \draw [ultra thick] (9,11)--(10,12);
    \draw [ultra thick] (11,9)--(12,10);
    \draw [red,very thick] (2,18)--(1,15)--(3,16);
    \draw [red,very thick] (7,14)--(6,8)--(8,13);
    \draw [red,very thick] (9,11)--(6,8)--(5,1)--(11,9);
    \draw [red,very thick] (15,7)--(14,6)--(13,3)--(16,5);
    \draw [red,very thick] (17,4)--(13,3)--(5,1)--(18,2);
    \plotpermnobox[red!30!black]{18}{15,0,0,0,1,8,14,0,0,0,0,0,3,6,7,5,4,2}
    \plotpermnobox{18}{0,18,16,17,0,0,0,13,11,12,9,10}
  \end{tikzpicture}
  \qquad\qquad
  \raisebox{36pt}{
  \begin{tikzpicture}[scale=0.33,line join=round]
    \draw[help lines] (0,0) grid (18,6);
    \draw[red,thick] (0,0)--(4,4)--(5,4)--(6,2);
    \draw[red,thick] (11,5)--(12,6)--(13,3)--(14,1);
    \draw[red,thick] (17,3)--(18,2);
    \draw[ultra thick] (6,2)--(7,3)--(8,3)--(9,4)--(10,4)--(11,5);
    \draw[ultra thick] (14,1)--(15,2)--(16,2)--(17,3);
  \end{tikzpicture}
  }
  $$
  \caption{A partial red tree and the corresponding {\L}uka\-sie\-wicz path; 
  the three marked red fringes correspond to the occurrences of the pattern $\mathbf{1,0,1}$}
  \label{figRedTreePatterns}
\end{figure}
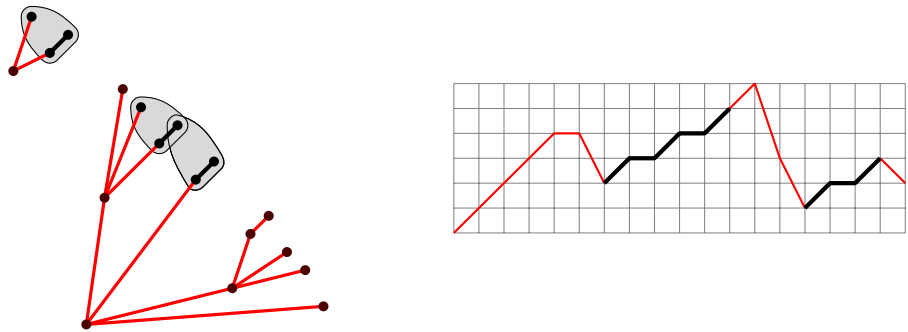
Our fourth and final concentration result concerns red fringes.
Recall that the random variable
$\rho_k(\smallF^+)$ records the
proportion of positions in a red tree whose red forest has at least $|\smallF|$ vertices,
and for which the graph induced by the rightmost $|\smallF|$ vertices of the forest is isomorphic to $\smallF$.

We would like to determine the bivariate generating function for red trees in which occurrences of the red fringe $\smallF$ are marked.
This is considerably less straightforward than was the case for the other parameters.
Primarily, this is because distinct occurrences of $\smallF$ may overlap. See the left of Figure~\ref{figRedTreePatterns} for an illustration.
To achieve our goal, it is convenient to rephrase our problem in terms of {\L}uka\-sie\-wicz paths.

Recall from Section~\ref{sectIntro} that a {\L}uka\-sie\-wicz path of length $n$ is a
sequence of integers $y_0,\ldots,y_n$ such that $y_0=0$,
$y_i\geqslant1$ for $i\geqslant1$,
and each step $s_i=y_i-y_{i-1}\leqslant1$.
It is easy to see that {\L}uka\-sie\-wicz paths
are in bijection with
red trees: visit the vertices of the tree from \emph{right to left} and let the height of the path be equal to the number of components in the forest induced by the vertices visited so far.
Thus, for each leaf vertex, the path contains an up-step, and for each internal vertex with $r$ children, the path contains a $(1\!-\!r)$-step.
See Figure~\ref{figRedTreePatterns} for an illustration.

Recall also that a pattern $\omega$ of length $m$ in a {\L}uka\-sie\-wicz path is a sequence of contiguous steps $\omega_1,\ldots,\omega_m$ in the path such that $\sum_{j=1}^i \omega_j > 0$
for $1\leqslant i\leqslant m$.
We do not consider sequences of steps for which the height drops to zero or below.
Thus, a pattern in a {\L}uka\-sie\-wicz path corresponds to an occurrence of a red fringe in a red tree.
Again, see Figure~\ref{figRedTreePatterns}, where this is illustrated.

\section{Patterns in {\L}uka\-sie\-wicz paths}\label{sectLukasPatterns}

The asymptotic distribution of patterns in \emph{words} has been investigated before.
For an exposition, see~\cite[Examples~I.12, III.26 and IX.13]{FS2009}.
The approach taken there makes use of the \emph{correlation polynomial} of a pattern, introduced by
Guibas \& Odlyzko in~\cite{GO1981} to analyse pattern-matching in strings, and also employs the
\emph{cluster method} of Goulden \& Jackson~\cite{GJ2004}.
We refine this approach for use with patterns in
{\L}uka\-sie\-wicz paths by utilising a generalisation of the correlation polynomial
and combining it with an application of the kernel method.

It is readily seen that the bivariate generating function, $L(z,y)$, for
{\L}uka\-sie\-wicz paths, in which $z$ marks length and $y$ marks height, satisfies the functional equation
\begin{equation}\label{eqLukaBGFHeightFuncEq}
  L(z,y) \;=\; z\+y \:+\: \frac{z\+y}{1-y}\big(L(z,1)-y\+L(z,y)\big) .
\end{equation}

Given a pattern $\omega=\omega_1,\ldots,\omega_m$,
let us use $h_i(\omega)=\sum_{j=1}^i \omega_j$ to denote the height after the $i$th step of~$\omega$, and let us call $h_m(\omega)$
the \emph{final height} of $\omega$.

The correlation polynomial of Guibas \& Odlyzko is univariate.
For our purposes, we define the \emph{bivariate} \emph{autocorrelation polynomial}, $\widehat{a}_\omega(z,y)$, for a pattern $\omega=\omega_1,\ldots,\omega_m$ in a {\L}uka\-sie\-wicz path as follows:
$$
\widehat{a}_\omega(z,y) \;=\; \sum_{i=1}^{m-1} c_i\+z^i\+y^{h_i(\omega)} ,
$$
where
$$
c_i  \;=\;
\begin{cases}
1, &\text{if $\omega_{i+1},\ldots,\omega_m = \omega_1,\ldots,\omega_{m-i}$;}\\
0, &\text{otherwise.}
\end{cases}
$$
Thus, $c_i$ records whether $\omega$ matches itself when shifted (left or right) by $i$, the variable $z$ marks the shift, and $y$ marks the height.
For example,
$\widehat{a}_\mathbf{1,1,0,1,1}(z,y) = z^3\+y^2+z^4\+y^3$.

Given a fixed pattern $\omega$ of length $m$ and final height $h$, we want to determine the trivariate generating function, $L_\omega(z,y,u)$, for {\L}uka\-sie\-wicz paths, where $u$ marks the number of occurrences of the pattern~$\omega$ in a path.
In order to achieve this, we first consider the class of {\L}uka\-sie\-wicz paths augmented by distinguishing an arbitrary selection of occurrences of $\omega$. Let $M_\omega(z,y,v)$ be the corresponding generating function, in which $v$ marks distinguished occurrences of the pattern in a path. By the standard inclusion-exclusion principle (see~\cite[p.208]{FS2009}), we know that
\begin{equation}\label{eqLMInclExcl}
  L_\omega(z,y,u) \;=\; M_\omega(z,y,u-1).
\end{equation}
In order to construct a functional equation for $M_\omega$, we consider subpaths each consisting of a maximal collection of overlapping distinguished occurrences of $\omega$. These collections are called \emph{clusters}.
It is readily seen that the generating function for clusters is
\begin{equation}\label{eqClusterGF}
  C_\omega(z,y,v) \;=\; \frac{z^m\+y^h\+v}{1-v\+\widehat{a}_\omega(z,y)},
\end{equation}
where $v$ is used to mark distinguished occurrences of $\omega$ in a cluster.

Furthermore, we have
\begin{equation}\label{eqMFuncEq}
  M_\omega(z,y,v) \;=\;
  z\+y
  \:+\: \frac{z\+y}{1-y}\big(M_\omega(z,1,v)-y\+M_\omega(z,y,v)\big)
  \:+\: M_\omega(z,y,v) \+ C_\omega(z,y,v),
\end{equation}
since a path grows either by adding an arbitrary step, as in~\eqref{eqLukaBGFHeightFuncEq}, or else by adding a
cluster.\footnote{This equation excludes distinguished occurrences of $\omega$ that begin with the first step of the path; this simplifies the algebra somewhat while having no effect on the asymptotics.}

Combining equations~\eqref{eqLMInclExcl}, \eqref{eqClusterGF} and \eqref{eqMFuncEq} and rearranging
gives us the following functional equation for $L_\omega(z,y,u)$:
$$
L_\omega(z,y,u) \;=\;
\frac
{z\+y\+\big(1 + (1-u)\+\widehat{a}_\omega(z,y)\big)\+\big(1 - y + L_\omega(z,1,u)\big)}
{z^m\+y^h\+(1-y)\+(1-u)+(1-y+z\+y^2)\+\big(1 + (1-u)\+\widehat{a}_\omega(z,y)\big)} .
$$
This equation is susceptible to the kernel method, so $L_\omega(z,1,u)=y_0(z,u)-1$, where $y_0$ is the appropriate root for $y$ of the denominator.
Rearranging, we obtain the following polynomial functional equation for $L=L(z,u)=L_\omega(z,1,u)$,
the bivariate generating function for {\L}uka\-sie\-wicz paths in which $u$ marks occurrences of $\omega$:
\begin{equation}\label{eqLukaBGFPatFuncEq}
  L \;=\; z\+(1+L)^2 \:-\: (1-u)\+\Big(z^m\+L\+(1+L)^h + \big(L - z\+(1+L)^2\big)\+\widehat{a}_\omega(z,1+L)\Big) .
\end{equation}
The fact that $L$ satisfies this equation enables us to demonstrate that patterns in {\L}uka\-sie\-wicz paths are concentrated, and moreover are
distributed normally in the limit.
The following proposition gives very general conditions for this to be the case for some parameter.
\begin{prop}[{\cite[Proposition~IX.17 with Theorem~IX.12]{FS2009}}; see also {\cite[Theorem~1]{Drmota1997}}]\label{propGaussian}
Let $F(z,u)$ be a bivariate function, analytic at $(0,0)$ and with non-negative Taylor coefficients, and
let $\xi_n$ be the sequence of random variables with probability generating functions
$$
\frac{[z^n]F(z,u)}{[z^n]F(z,1)} .
$$
Assume that $F(z,u)$ is a solution for $y$ of the equation
$$y\;=\;\Phi(z,u,y),$$
where $\Phi$ is a polynomial of degree at least two in $y$,
$\Phi(z,1,y)$ has non-negative Taylor coefficients and is analytic in some domain $|z|<R$ and $|y|<S$,
$\Phi(0,1,0)=0$,
$\Phi_y(0,1,0)\neq1$,
$\Phi_{yy}(z,1,y)\nequiv0$,
and there exist positive $z_0<R$ and $y_0<S$ satisfying the pair of equations
$$
\Phi(z_0,1,y_0)\;=\;y_0, \qquad\quad \Phi_y(z_0,1,y_0)\;=\;0.
$$
Then, as long as its asymptotic variance is non-zero, $\xi_n$
converges in law to a Gaussian distribution with mean and standard deviation asymptotically linear in~$n$.
\end{prop}
All that remains is to check that $L$ satisfies the relevant requirements.
\begin{repthm}{thmLukaPatternsGaussian}
The number of occurrences of a fixed pattern
in a {\L}uka\-sie\-wicz path of length $n$
exhibits
a Gaussian limit distribution
with mean
and standard deviation
asymptotically linear in~$n$.
\end{repthm}
\begin{proof}
From~\eqref{eqLukaBGFPatFuncEq}, it can easily be seen that
$L(z,u)$ satisfies the requirements of Proposition~\ref{propGaussian}, with $\Phi(z,1,y)=z\+(1+y)^2$, $z_0=\frac{1}{4}$ and $y_0=1$.
\end{proof}

\section{Summing up}\label{sectLowerBound1324}

Since patterns in {\L}uka\-sie\-wicz paths are in bijection with
red fringes in red trees,
$L(z,u)$ is also the bivariate generating function
for red trees in which $u$ marks occurrences of the red fringe $\smallF$ corresponding to the pattern $\omega$, with
$m=|\smallF|$ and $h$ the number of components of $\smallF$.
Thus, we know that $\rho_k(\smallF^+)$ is concentrated.
It remains for us to determine the limiting mean.
\begin{prop}\label{propOmegaPlus}
Let $m=|\smallF|$ and $h$ be the number of components of
$\smallF$.
$\rho_k(\smallF^+)$ is concentrated at
$$
\mu_\rho(\smallF^+) \;=\; \frac{1}{2^{2m-h}}.
$$
\end{prop}
\begin{proof}
Let $F(z)=\tfrac{1}{2z} \+ \big(1-\sqrt{1-4\+z}\big)$ be the generating function for plane forests.

Solving~\eqref{eqLukaBGFPatFuncEq} with $u=1$ gives $L(z,1)=F(z)-1$ (as expected).

Similarly, differentiating~\eqref{eqLukaBGFPatFuncEq} with respect to $u$, setting $u=1$, and solving the resulting equation gives
$$
L_u(z,1) \;=\; \frac{z^m\+F(z)^h\+\big(1-(1-2\+z)\+F(z)\big)}{1-4\+z}.
$$
Then, extracting coefficients yields
$$
\begin{array}{lcl}
[z^k]\+ L(z,1) & = & \frac{1}{k+1} \binom{2k}{k}, \\[6pt]
[z^k]\+ L_u(z,1) & = & \binom{2k-2m+h}{k-m-1}.
\end{array}
$$
Hence, applying Proposition~\ref{propMoments}
and taking limits,
$$
\liminfty[k]\mathbb{E}[\rho_k(\smallF^+)] \;=\; \frac{1}{2^{2m-h}}
.
$$
Concentration follows from Theorem~\ref{thmLukaPatternsGaussian}.
\end{proof}

We are finally in a position to compute a lower bound for the growth rate of the class of permutations avoiding $\pdiamond$, proving our main theorem.
\begin{repthm}{thm1324LowerBound}
$\gr(\av(\pdiamond)) > 9.81$.
\end{repthm}
\begin{proof}
We calculate the contribution to the growth rate from pairs consisting of a tree and a forest of bounded size.
From Proposition~\ref{propLB}, we know that, for each $N>0$, the growth rate is at least
$$
g_N(\lambda,\delta)
\;=\;
E(\lambda,\delta)^{\nfrac{1}{(1+\lambda)}}
\:\times\:
\prod\limits_{|\subT|+|\subF|\leqslant N}
\!\!
\,
Q(\smallT,\smallF)^{2\+\delta\+\lambda\+\mu(\subT,\subF)/(1+\lambda)},
$$
where
$$
\mu(\smallT,\smallF) \;=\;
\mu_\beta(\smallT)\+
\big(
\mu_\gamma(|\smallF|)\+\mu_\rho(\smallF^+)  \:+\:
\mu_\gamma(>\!|\smallF|)\+\mu_\rho(\smallF)
\big) ,
$$
as follows from~\eqref{eqChi} and~\eqref{eqZeta} and Propositions~\ref{propPsi}, \ref{propPhi}, \ref{propOmega} and~\ref{propOmegaPlus}.

Using \emph{Mathematica}~\cite{Mathematica} to evaluate $Q(\smallT,\smallF)$ and $\mu(\smallT,\smallF)$ and then to apply numerical maximisation over values of $\lambda$ and $\delta$ yields
$$
g_{14}(\lambda,\delta)
\;>\;
9.81056
$$
with $\lambda\approx0.69706$ and $\delta\approx0.75887$.
\end{proof}
The determination of this value requires the processing of more than 1.6 million pairs consisting of a tree and a forest.
Larger values of $N$ would require more sophisticated programming techniques.
However, increasing $N$ is unlikely to lead to a significantly improved lower bound;
although the rate of convergence at $N\!=\!14$ is still quite slow, numerical analysis of the computational data suggests that $\liminfty[N]\max\limits_{\lambda,\delta}g_N(\lambda,\delta)$ is probably not far from 9.82.

We conclude with the observation that in the construction that gives our bound, the mean number of vertices in a blue subtree, $1/\delta$, is less than 1.32.
We noted earlier that
the cigar-shaped boundary regions
of
a typical $\pdiamond$-avoider
contain numerous small subtrees (although it is not immediately obvious how one should identify such a boundary tree).
Is it the case that
the mean size of these
subtrees is asymptotically bounded?
Perhaps, on the contrary, their average size grows unboundedly (but very slowly), and understanding how (and the rate at which) this occurs would lead to an improved lower bound.
In the meantime, the following question might be somewhat easier to answer:
\begin{question}
Asymptotically, what proportion of the points in a typical $\pdiamond$-avoider are left-to-right minima or right-to-left maxima?
\end{question}

\cleardoublepage


\backmatter

\bibliographystyle{plain}
{\footnotesize\bibliography{mybib}}

\end{document}